\documentclass[reqno, 10pt]{amsart}
\usepackage[usenames,dvipsnames]{color}
\usepackage[pdfauthor={Andrea Appel, Valerio Toledano Laredo}, pdftitle={Monodromy of the Casimir connection of a symmetrisable Kac--Moody algebra}]{hyperref}
\hypersetup{
    colorlinks=true,       			
    linkcolor=blue,          			
    citecolor=blue,		 		
    filecolor=blue,      				
    urlcolor=blue           			
}
\usepackage{amsmath, amssymb, hyperref, empheq,enumitem}
\usepackage[multiple,bottom]{footmisc}
\usepackage{tikz}
\usepackage{tikz-cd}
\usepackage{mathrsfs}
\usetikzlibrary{shapes,backgrounds,fit}
\usepackage[all,cmtip,knot]{xy}
\xyoption{arc}
\usepackage{stmaryrd}
\usepackage{accents}

\usepackage{mathtools}
\mathtoolsset{showonlyrefs,showmanualtags}

\newtheorem*{theorem}{Theorem}
\newtheorem*{proposition}{Proposition}
\newtheorem*{corollary}{Corollary}
\newtheorem*{lemma}{Lemma}

\theoremstyle{definition}
\newtheorem*{definition}{Definition}

\numberwithin{equation}{section}

\newenvironment{pf}{\paragraph{{\sc Proof}}}{\qed\par\medskip}

\newcommand{\remark}{{\bf Remark.}}
\newcommand{\remarks}{{\bf Remarks.}}

\newcommand {\example}{{\bf Example.\ }}

\newcommand {\IC}{\mathbb{C}}

\newcommand {\IN}{\mathbb{N}}                          
\newcommand {\IP}{\mathbb{P}}
\newcommand {\IR}{\mathbb{R}}

\newcommand {\IZ}{\mathbb{Z}}

\newcommand {\bfI}{\mathbf I}

\newcommand {\A}{\mathcal A}
\newcommand {\B}{\mathcal B}
\newcommand {\C}{\mathcal C}
\newcommand {\D}{\mathcal D}
\newcommand {\E}{\mathcal E}
\newcommand {\F}{\mathcal F}
\newcommand {\G}{\mathcal G}
\renewcommand {\H}{\mathcal H}

\newcommand {\K}{\mathcal K}

\newcommand {\M}{\mathcal M}
\newcommand {\N}{\mathcal N}
\renewcommand {\O}{\mathcal O}
\renewcommand {\P}{\mathcal P}
\newcommand {\Q}{\mathcal Q}

\newcommand {\U}{\mathcal U}
\newcommand {\V}{\mathcal V}
\newcommand {\W}{\mathcal W}

\renewcommand {\a}{\mathfrak a}
\renewcommand {\b}{\mathfrak b}
\renewcommand {\c}{\mathfrak c}

\newcommand {\g}{\mathfrak{g}}
\newcommand {\h}{\mathfrak h}

\newcommand {\n}{\mathfrak n}

\renewcommand {\r}{\mathfrak r}
\newcommand {\z}{\mathfrak z}

\renewcommand {\SS}{\mathfrak S}

\newcommand {\sfA}{\mathsf{A}}
\newcommand {\sfB}{\mathsf{B}}
\newcommand {\sfC}{\mathsf{C}}

\newcommand {\sfG}{\mathsf{G}}
\newcommand {\sfX}{\mathsf{X}}

\newcommand {\ie}{{\it i.e., }}
\newcommand {\eg}{{\it e.g.}, }

\newcommand {\fd}{finite--dimensional }

\newcommand {\wrt}{with respect to }

\newcommand {\ol}{\overline}
\newcommand {\wt}{\widetilde}
\newcommand {\wh}{\widehat}

\newcommand {\mns}{maximal nested set }
\newcommand {\mnss}{maximal nested sets }

\newcommand {\supp}{\operatorname{supp}}
\newcommand {\zsupp}{\operatorname{\mathfrak{z}supp}}

\newcommand {\DCP}{De Concini--Procesi }
\renewcommand {\DJ}{Drinfeld--Jimbo }

\newcommand {\nEK}{Etingof--Kazhdan }

\newcommand {\KM}{Kac--Moody }

\newcommand{\ctp}{\hat{\otimes}}
\newcommand{\ten}{\otimes}

\def\iip#1#2{\langle{#1},{#2}\rangle}

\def\slantfrac#1#2{\kern.2em\raise.2em\hbox{$#1$}\kern-.2em\left/\lower.4em\hbox{$#2$}\kern-.2em\right.}
\def\backslantfrac#1#2{\kern.2em\lower.4em\hbox{$#1$}\kern-.3em\left\backslash\raise.4em\hbox{$#2$}\right.}

\DeclareMathOperator{\End}{End}
\DeclareMathOperator{\Ker}{Ker}

\DeclareMathOperator{\Res}{Res}
\DeclareMathOperator{\id}{id}

\DeclareMathOperator{\rank}{rank}

\renewcommand {\sl}[1]{\mathfrak{sl}_{#1}}

\newcommand {\Ug}{U\g}

\newcommand {\Uhg}{U_{\hbar}\g}

\newcommand {\nablak}{\nabla_{\K}}
\newcommand {\wtnablak}{\wt{\nabla}_{\K}}
\newcommand {\nablawk}{\nabla_{:\K:}}

\newcommand {\reg}{_{\scriptstyle{\operatorname{reg}}}}
\newcommand {\ess}{^{\scriptstyle{\operatorname{e}}}}
\newcommand{\hess}{\h\ess}
\newcommand {\hreg}{\h\reg}

\renewcommand{\1}{\mathbf{1}}

\newcommand{\fmls}{[\negthinspace[\hbar]\negthinspace]}
\newcommand{\hext}[1]{{#1}\fmls}

\renewcommand{\DJ}{U_{\hbar}}

\newcommand{\Oint}{\O^{\sint}}
\newcommand{\Oinf}{\O_{\infty}}

\newcommand{\cU}[2]{\mathcal{U}_{#1}^{#2}} 

\newcommand{\hRes}{\Res^{\hbar}}
\newcommand{\hOint}{\O^{\sint}}
\newcommand{\hOinf}{\O_{\infty}}
\newcommand{\hOinfint}{\O_{\infty,\Uhg}^{\sint}}

\newcommand {\Fi}{F_{\{i\}}}

\newcommand {\fml}{[\negthinspace[\hbar]\negthinspace]}

\newcommand {\half}[1]{\frac{#1}{2}}

\newcommand {\cor}[1]{h_{#1}}
\renewcommand {\root}[1]{\alpha_{#1}}
\newcommand {\cow}[1]{\lambda^{\vee}_{#1}}
\newcommand{\hinv}[1]{t_{#1}}

\newcommand {\aand}{\qquad\text{and}\qquad}

\newcommand {\ul}[1]{\underline{#1}}

\newcommand {\Alt}{\operatorname{Alt}}

\newcommand {\op}{^{\scriptscriptstyle{\operatorname{op}}}}

\newcommand {\<}{\langle}
\renewcommand {\>}{\rangle}

\newcommand {\Br}[1]{\mathcal{B}_{#1}}

\newcommand{\sfF}{\mathsf{F}}
\newcommand{\sfM}{\mathsf{M}}

\newcommand{\Mns}[1]{\sfMns(#1)}

\newcommand{\Omit}[1]{}
\newcommand{\summary}[1]{}

\newcommand{\sfMns}{\mathsf{Mns}}

\DeclareMathOperator{\Rep}{Rep}

\DeclareMathOperator{\vect}{{\mathsf{Vect}}}
\DeclareMathOperator{\ad}{ad}
\newcommand {\sfk}{\mathsf{k}}

\newcommand{\LBA}{\underline{LBA} }

\newcommand{\sfEnd}[1]{\mathsf{End}\left(#1\right)}

\newcommand{\sfAd}[1]{\mathsf{Ad}(#1)}
\newcommand{\sfAut}[1]{\mathsf{Aut}(#1)}

\newcommand{\rootsys}{\mathsf{\Delta}} 
\newcommand{\Rs}[1]{\rootsys_{#1}}
\newcommand{\Rp}{\rootsys_+}
\newcommand{\Rm}{\rootsys_-}

\newcommand{\re}{^{\scriptscriptstyle{\operatorname{re}}}}

\newcommand{\sfad}{\mathsf{ad}}

\newcommand{\Eq}[1]{\E_{#1}}

\newcommand{\DrY}[1]{\mathsf{DY}_{#1}}
\newcommand{\ff}{\mathsf{f}}

\newcommand{\preLA}{{\sf LA}}

\renewcommand{\LBA}{\ul{\sf LBA}}
\newcommand{\PROP}{\sf PROP}
\newcommand{\uPROP}{\PROP}

\newcommand{\gaugeJ}[2]{(#1)^{-1}_{12}\cdot#2\cdot(#1)_1\cdot(#1)_2}
\newcommand{\gaugeJi}[2]{(#1)_{12}\cdot#2\cdot(#1)_1^{-1}\cdot(#1)_2^{-1}}
\newcommand{\gaugeF}[2]{(#1)_1\cdot(#1)_2\cdot #2\cdot(#1)_{12}^{-1}}
\newcommand{\gaugeFi}[2]{(#1)_1^{-1}\cdot(#1)_2^{-1}\cdot #2\cdot(#1)_{12}}

\newcommand{\gb}{\g_{\b}}
\newcommand{\ga}{\g_{\a}}

\newcommand{\VDY}[1]{\ul{\mathsf{V}}_{#1}} 
\newcommand{\VCDY}[1]{\VDY{#1}} 
\newcommand{\ACDY}[1]{[#1]} 

\newcommand{\pEnd}[2]{{\mathsf{End}}_{#1}\left(#2\right)}

\newcommand{\GCM}[1]{\mathsf{#1}}

\newcommand{\DCPA}[2]{\Upsilon_{#1#2}}
\newcommand{\DCPAC}[3]{\Upsilon_{#2#3}^{#1}}

\newcommand {\sfQ}{\mathsf Q}

\newcommand {\sfa}{\mathsf a}

\newcommand {\sfP}{{\mathsf P}}

\newcommand{\sfR}{\mathsf{R}}
\newcommand{\fdim}[1]{\mathring{#1}}
\newcommand{\prr}{\Rs{+}\re}
\newcommand{\fpr}{\fdim{\sfR}_+}

\newcommand{\Diffd}[1]{d\left(\frac{#1}{\delta}\right)}
\newcommand{\Sfd}[1]{S\left(\frac{#1}{\delta}\right)}
\newcommand{\Tfd}[1]{T\left(\frac{#1}{\delta}\right)}

\newcommand {\IH}{\mathbb H}

\newcommand {\olg}{{\ol{\mathfrak{g}}}}
\newcommand {\olh}{{\ol{\mathfrak{h}}}}
\newcommand {\olC}{\ol{\C}}
\newcommand {\olD}{\ol{\dgr}}
\newcommand {\olmu}{{\ol{\mu}}}

\newcommand {\olF}{\ol{F}}
\newcommand {\olcalF}{\ol{\F}}

\newcommand{\Aut}{\mathsf{Aut}}
\newcommand{\Ima}{\operatorname{Im}}

\newcommand {\omitvaleriocomment}[1]{}

\newcommand {\IV}{\mathbb V}
\renewcommand {\O}{\mathcal O}

\newcommand{\nablah}{\mathsf{h}}
\newcommand{\hdef}{\scsop{\hbar}}

\usepackage{stmaryrd}
 
\newcommand{\real}{^{\operatorname{re}}}

\newcommand{\DCPHA}[1]{\mathfrak{t}_{#1}} 
\newcommand{\DCPHAH}[1]{\wh{\mathfrak{t}}_{#1}} 
\newcommand{\Kh}[2]{\mathsf{t}_{#1}^{#2}} 
\newcommand{\KhC}[1]{{\kappa}_{#1}} 
\newcommand{\DKHA}[2]{\mathfrak{t}_{#1}^{\scs{#2}}} 
\newcommand{\DKHAH}[2]{\wh{\mathfrak{t}}_{#1}^{\scs{#2}}} 
\newcommand{\Th}[1]{\mathsf{t}^{#1}} 
\newcommand{\DBLHA}[2]{\mathfrak{t}_{#1}^{\scs{#2}}} 
\newcommand{\DBLHAH}[2]{\wh{\mathfrak{t}}_{#1}^{\scs{#2}}} 

\newcommand{\Kdh}[2]{\mathsf{K}_{#1}^{#2}} 
\newcommand{\Tdh}[2]{\mathsf{\Omega}^{#1}_{#2}} 
\newcommand{\Trdh}[2]{\mathsf{r}^{#1}_{#2}} 
\newcommand{\Kp}[2]{\boldsymbol{\kappa}^{#1}_{#2}} 
\newcommand{\rp}[2]{\mathbf{r}^{#1}_{#2}} 
\newcommand{\Tp}[2]{\boldsymbol{\Omega}^{#1}_{#2}} 
\newcommand{\Ku}[2]{{\mathcal K}_{#1}^{#2}} 
\newcommand{\ku}[1]{{\mathcal K}_{#1}} 
\newcommand{\Cu}[1]{{\mathcal C}_{#1}} 
\newcommand{\sff}{\mathsf{f}}

\newcommand {\calkalpha}{{\mathcal K}_\alpha}

\newcommand{\cc}[1]{\mathsf{conn}(#1)}
\newcommand{\Ns}[1]{\mathsf{Ns}(#1)}
\newcommand{\Nsr}[2]{\mathsf{Ns}_{#2}(#1)}

\newcommand {\DYt}{Drinfeld--Yetter }

\newcommand {\ulm}{\ul{m}}

\newcommand {\BDm}{\Br{\dgr}^{\ulm}}
\newcommand {\BBm}{\Br{B}^{\ulm}}
\newcommand {\BBpm}{\Br{B'}^{\ulm}}

\newcommand {\brac}[1]{\mathsf{br}_{#1}}

\newcommand {\sfS}{\mathsf{S}}
\newcommand{\scs}{\scriptscriptstyle}

\newcommand{\mLBA}[1]{{\LBA}_{#1}}

\newcommand{\Fun}{\mathsf{Fun}}
\newcommand{\bfb}{\mathbf{b}}

\newcommand{\dLBA}[1]{{\LBA}_{#1}}

\newcommand{\dmp}[1]{\theta_{#1}}

\newcommand{\MDY}[2]{\ul{\mathsf{DY}}_{#1}^{\scs{#2}}}

\newcommand{\DYA}[2]{\U_{#1}^{#2}}  

\newcommand{\hDrY}[2]{\mathsf{DY}_{{#1}}^{#2}}

\newcommand{\Diag}{\operatorname{Diag}}
\newcommand{\symd}[1]{\mathsf{d}_{#1}} 
\newcommand{\symdi}[1]{\mathsf{d}_{#1}^{-1}} 

\newcommand{\sint}{{\scs\operatorname{int}}} 

\newcommand {\two}{^{(2)}}
\newcommand {\gtwo}{\g\two}
\newcommand {\btwo}{\b\two}

\newcommand{\RDY}[2]{\ul{\mathsf{DY}}^{#2}_{#1}}

\newcommand{\DUA}[2]{\mathbf{U}_{#1}^{#2}}
\newcommand{\CDUA}[2]{\wh{\mathbf{U}}_{#1}^{#2}} 

\newcommand{\hextsub}[1]{{#1}^{\hbar}}

\newcommand{\CRDYUA}[2]{\wh{\mathsf{U}}_{#1}^{#2}} 
\newcommand{\RDYUA}[2]{\mathsf{U}_{#1}^{#2}}

\newcommand{\drc}[1]{\delta_{#1}}

\newcommand{\bp}[1]{\b^+_{#1}}
\newcommand{\bm}[1]{\b^-_{#1}}
\newcommand{\bpm}[1]{\b^{\pm}_{#1}}
\newcommand{\bmp}[1]{\b^{\mp}_{#1}}

\newcommand{\np}[1]{\n^+_{#1}}

\newcommand {\res}{^{\scs\operatorname{res}}}
\newcommand{\grb}{\g_{\b}\res}

\newcommand{\Kar}[1]{\mathsf{Kar}(#1)}
\newcommand{\cKar}[1]{\ul{#1}}

\newcommand{\PD}{\P(\dgr)}
\newcommand{\BPD}{\P_2(\dgr)}

\newcommand{\DYrho}[2]{\rho^{#2}_{#1}} 

\newcommand{\sfi}{\mathsf{i}}
\newcommand{\adm}{\scs\operatorname{adm}}
\newcommand{\aDrY}[1]{\mathsf{DY}^{\scs\operatorname{adm}}_{#1}}

\newcommand{\PT}{\mathscr{P}}
\newcommand{\aw}{\mathscr{A}}
\newcommand{\bw}{\mathscr{B }}
\newcommand{\cw}{\mathscr{C}}
\newcommand{\PB}{\mathscr{P}_{\bw}}
\newcommand{\PBp}{\mathscr{P}_{\bw'}}

\newcommand{\scsop}[1]{{\scriptscriptstyle\operatorname{#1}}}

\newcommand{\trv}[2]{v_{#2}^{#1}}
\newcommand{\sfs}{\mathsf{s}}
\newcommand{\rht}{\mathsf{ht}}

\newcommand{\exph}{\exp(\h)}
\newcommand{\exphp}{\exp(\h')}

\newcommand{\bwc}{\bw_{\aw}}

\newcommand{\dgr}{\mathbb{D}}

\newcommand{\redasso}[2]{\mathsf{a}^{#1}_{#2}}

\newcommand{\twF}[1]{K_{#1}}
\newcommand{\pCox}[1]{\mathfrak{C}_{#1}} 
\newcommand{\sCox}[1]{\mathbf{C}_{#1}} 
\newcommand{\ACox}[1]{\mathscr{A}_{#1}} 
\newcommand{\tCox}[1]{{T}_{#1}} 
\newcommand{\gCox}[1]{{{a}}_{#1}} 
\newcommand{\cCox}[1]{\mathscr{C}_{#1}} 
\newcommand{\mCox}[1]{\mathbf{H}_{#1}} 
\newcommand{\nCox}[1]{\mathbf{v}_{#1}} 

\newcommand{\UCoxDY}[2]{\mathscr{U}_{#1}^{\scs{\hdef,#2}}} 
\newcommand{\UCoxDYz}[2]{\mathscr{U}_{#1}^{\scs{#2}}} 
\newcommand{\UCoxO}[2]{\mathscr{U}_{#1}^{\scs{\hdef,#2}}} 
\newcommand{\UCoxDYint}[2]{\mathscr{U}_{#1}^{\scs{\hdef,\sint,#2}}} 
\newcommand{\UCoxOint}[2]{\mathscr{U}_{{#1}}^{\scs{\hdef,\sint,#2}}} 
\newcommand{\UOint}[2]{\mathcal{U}_{#1}^{\scs{\hdef,\sint,#2}}} 
\newcommand{\UOintz}[2]{\mathcal{U}_{#1}^{\scs{\sint,#2}}} 
\newcommand{\UDY}[2]{\mathcal{U}_{#1}^{\scs{\hdef,#2}}} 
\newcommand{\UDYz}[2]{\mathcal{U}_{#1}^{\scs{#2}}} 
\newcommand{\UDYint}[2]{\mathcal{U}_{#1}^{\scs{\hdef,\sint,#2}}} 

\newcommand{\Udiag}[2]{\mathbf{U}_{#1}^{\scs{#2}}}

\newcommand{\OCox}[2]{\mathscr{O}_{#1}^{\scs{#2}}} 

\newcommand{\DYCox}[2]{\mathscr{DY}_{#1}^{\scs{#2}}} 

\newcommand{\gstr}{{\scsop{\Upsilon\text{-}str}}}  
\newcommand{\astr}{{\scsop{\sfa\text{-}str}}} 

\newcommand{\twistAJ}[2]{{#1}_{23}^{-1}\cdot{#1}^{-1}_{1,23}\cdot #2\cdot{#1}_{12,3}\cdot{#1}_{12}}
\newcommand{\twistRJ}[2]{{#1}_{21}^{-1}\cdot #2\cdot #1}

\newcommand{\twistAF}[2]{{#1}_{23}\cdot{#1}_{1,23}\cdot #2\cdot{#1}_{12,3}^{-1}\cdot{#1}_{12}^{-1}}

\newcommand{\Ki}{\K[i]}
\newcommand{\Kip}{\K'[i]}
\newcommand{\Kj}{\K[j]}

\newcommand{\rsm}[1]{\mathsf{m}_{#1}}

\newcommand{\topS}[1]{\mathscr{S}_{#1}}
\newcommand{\topT}[1]{\mathtt{\rho}_{#1}}     

\newcommand{\CoxS}[3]{S_{#3 #2}^{#1}} 
\newcommand{\CoxSk}[3]{\wt{S}_{#3 #2}^{#1}} 
\newcommand{\texp}[1]{\wt{s}_{#1}}
\newcommand{\qWS}[1]{\mathbf{S}_{#1}} 
\newcommand{\qWSk}[1]{\wt{\mathbf{S}}_{#1}}   
\newcommand{\Rmx}{\mathbf{R}}     

\newcommand{\trunc}[3]{{#1}^{#2}_{#3}}

\newcommand{\DCPS}[1]{\Psi_{#1}}

\newcommand{\admu}[2]{\mathsf{ad}^{(#2)}(#1)}

\newcommand{\Poid}[1]{\mathbf{\Pi}_1{#1}}

\newcommand{\WPoid}[1]{W\ltimes\Poid{#1}}

\newcommand{\Phig}[3]{\Phi^{#1#3}_{#2}}
\newcommand{\Rg}[3]{R^{#1#3}_{#2}}
\newcommand{\Jg}[3]{J^{#1#3}_{#2}}
\newcommand{\DCg}[4]{\DCPAC{#1#4}{#2}{#3}}
\newcommand{\Sg}[4]{\CoxS{#1#4}{#2}{#3}}

\newcommand{\rda}[3]{{#1}^{#2}_{#3}} 
\newcommand{\rdm}[5]{{#1}^{#2#3}_{#5#4}}

\newcommand{\tfV}{\vect_{\hbar}}

\newcommand{\rperp}{\perp_{\dgr}}

\newcommand {\realre}{^{\operatorname{re}}}

\newcommand {\DYO}{\mathscr{DY}}

\newcommand{\nc}{\newcommand}

\nc{\Uhb}{U_{\hbar}\b}
\nc{\Uhbp}[1]{U_{\hbar}\b^+_{#1}}
\nc{\Uhbm}[1]{U_{\hbar}\b^-_{#1}}
\nc{\Uhbpm}[1]{U_{\hbar}\b^{\pm}_{#1}}

\nc{\Uhnp}{U_{\hbar}\n^+}
\nc{\Uhnm}{U_{\hbar}\n^-}
\nc{\Uhnpm}{U_{\hbar}\n^{\pm}}

\nc{\wgt}[1]{\mathsf{wt}(#1)}
\newcommand {\bil}{\<\cdot,\cdot\>}

\newcommand {\tail}{b}

\newcommand {\At}{\tail} 
\newcommand {\Ati}{\tail_i} 
\newcommand {\whAt}{\wh{\tail}} 

\newcommand {\wick}[1]{:\negthinspace\negthinspace #1\negthinspace\negthinspace:}
\newcommand {\Oh}{\mathcal O_\hbar}
\newcommand {\Ohint}{\O_\hbar^{\sint}}
\newcommand {\sfh}{\mathsf h}

\title[Monodromy of the Casimir connection]
{Monodromy of the Casimir connection of a symmetrisable Kac--Moody algebra}
\author[A. Appel]{Andrea Appel}
\address{Dipartimento di Scienze Matematiche, Fisiche e Informatiche,
Universit\`a di Parma, INdAM - GNSAGA, and INFN Gruppo Collegato di Parma,
Parco Area delle Scienze 53/A, 
Parma PR 43124, Italy}
\email{andrea.appel@unipr.it}

\author[V. Toledano Laredo]{Valerio Toledano Laredo}
\address{Department of Mathematics,
Northeastern University,
360 Huntington Avenue,
Boston MA 02115, USA}
\email{V.ToledanoLaredo@neu.edu}
\thanks{The first author was partially supported by the ERC Grant 637618
and by the Programme {FIL 2022} of the University of Parma and co-sponsored by Fondazione Cariparma. 
The second author was supported in part through the NSF grants
DMS--1505305 and DMS--1802412}


\begin{document}

\begin{abstract}
Let $\g$ be a symmetrisable \KM algebra and $V$ an integrable $\g$--module in
category $\O$. We show that the monodromy of the (normally ordered) rational
Casimir connection on $V$ can be made equivariant \wrt the Weyl group $W$ of
$\g$, and therefore defines an action of the braid group $\Br{W}$ 
on $V$. We then prove that this action is canonically equivalent to the quantum Weyl
group action of $\Br{W}$ on a quantum deformation of $V$, that is an integrable,
category $\O$ module $\V$ over the quantum group $\Uhg$ such that $\V/\hbar
\V$ is isomorphic to $V$. This extends a result of the second author which is valid
for $\g$ semisimple.
\end{abstract}
\maketitle
\setcounter{tocdepth}{1} 
\vspace{-2em}
\tableofcontents

\section{Introduction}

\subsection{}

Let $\g$ be a complex, semisimple Lie algebra, $(\cdot,\cdot)$ an invariant inner
product on $\g$, $\h\subset\g$ a Cartan subalgebra, and $\rootsys\subset\h^*$
the corresponding root system. Set $\hreg=\h\setminus\bigcup_{\alpha\in\rootsys}
\ker(\alpha)$, and let $V$ be a \fd representation of $\g$.

The Casimir connection
of $\g$ is the flat connection on the holomorphically trivial vector bundle $\IV$
over $\hreg$ with fibre $V$ given by
\begin{equation}\label{eq:nablak}
\nablak=d-
\half{\sfh}\negthickspace\sum_{\alpha\in\Rs{+}}\frac{d\alpha}{\alpha}\cdot\calkalpha
\end{equation}
Here, $\sfh$ is a complex deformation parameter, $\Rs{+}\subset\rootsys$ a
chosen system of positive roots,\footnote{$\nablak$ is independent of the
choice of $\Rs{+}$ since $d\log\alpha=d\log(-\alpha)$ and $\calkalpha=
{\mathcal K}_{-\alpha}$.}
and $\calkalpha\in U\g$ the truncated Casimir operator
of the three--dimensional subalgebra $\sl{2}^\alpha\subset\g$ corresponding
to $\alpha$ given by $\calkalpha=x_\alpha x_{-\alpha}+x_{-\alpha} x_\alpha$,
where $x_{\pm\alpha}\in\g_{\pm\alpha}$ are root vectors such that $(x_\alpha,
x_{-\alpha})=1$ \cite{MTL,vtl-2,procesi-96,FMTV}. 

Although the Weyl group $W$ of $\g$ does not act on $V$ in general, the
action of its Tits extension $\wt{W}$ can be used to twist $(\IV,\nablak)$
into a $W$--equivariant, flat vector bundle $(\wt{\IV},\wtnablak)$ on $\hreg$
\cite{MTL,vtl-2}. This gives rise to a one--parameter family of actions $\mu
_\sfh$ of the braid group $\Br{W}=\pi_1(\hreg/W)$ on $V$ which deforms
the action of $\wt{W}$.

A theorem of the second author, originally conjectured by De Concini around
1995 and independently in \cite{vtl-2}, asserts that the monodromy of $\wtnablak$
is described by the quantum group $\Uhg$, with deformation parameter given
by $\hbar=2\pi\sqrt{-1}\sfh$ \cite{vtl-2,vtl-3,vtl-4,vtl-6}. Specifically, if $\V$ is a quantum
deformation of $V$, that is a $\Uhg$--module which is topologically free over
$\IC\fml$ and such that $\V/\hbar\V\cong V$ as $U\g$--modules, the action
of $\Br{W}$ on $V\fml$ given by the formal Taylor series of $\mu_\sfh$ at $\sfh
=0$ is equivalent to that on $\V$ given by the quantum Weyl group operators
of $\Uhg$.

\subsection{}

The goal of the present paper is to extend the description of the monodromy
of the Casimir connection in terms of quantum Weyl groups to the case of an
arbitrary symmetrisable \KM algebra $\g$. This extension requires several
new ideas, which are described below. They lead to a far stronger result, even
for a \fd $\g$, namely a constructive proof of the existence of a {\it canonical}
equivalence between these representations.\footnote{By way of comparison,
the results in \cite{vtl-3,vtl-4,vtl-6} establish cohomologically that the set of
such equivalences is non--empty if $\g$ is semisimple.} We conjecture in
fact that our equivalence can be specialised to non--rational values of $\sfh$,
and plan to return to this problem in future work.

\subsection{} 

When the root system 
is infinite, the sum in \eqref{eq:nablak} does not
converge. This is easily overcome, however,
by replacing each Casimir 
by its normally ordered version
\[\wick{\calkalpha}\,=
2\sum x_{-\alpha}^{(i)}x_\alpha^{(i)}
=
\calkalpha-m_\alpha t_\alpha
\]
where $\rsm{\alpha}=\dim\g_\alpha$, $\{x_{\pm\alpha}^{(i)}\}_{i=1}^{\rsm{\alpha}}$
are dual bases of the root spaces $\g_{\pm\alpha}$, and $t_\alpha=\nu^{-1}(\alpha)$,
with $\nu:\h\to\h^*$ the identification induced by the inner product. Although still
infinite, the sum in
\begin{equation}\label{eq:nablank}
\nablawk=
d-\half{\sfh}\negthickspace\sum_{\alpha\in\Rs{+}}\frac{d\alpha}{\alpha}\wick{\calkalpha}
\end{equation}
is now locally finite, provided the representation $V$ lies in category $\O$.
Moreover, the connection $\nablawk$ is flat \cite{FMTV} (we give an
alternative proof of this, along the lines of its \fd counterpart, in Section \ref{s:Casimir}).

\subsection{} 

Although it restores convergence, normal ordering 
breaks the $W$--equivariance of $\nablak$. If $\g$ is finite--dimensional,
equivariance can be restored by reverting to the original connection
\eqref{eq:nablak}, that is adding to $\nablawk$ the $\h$--valued, closed
$1$--form 
\begin{equation}
\nablak-\nablawk
=
-\half{\sfh}\sum_{\alpha\in\Rs{+}}\frac{d\alpha}{\alpha}\hinv{\alpha}
\end{equation}
Beyond finite type, it is therefore desirable to {\it resum} the divergent 1--form\footnote
{The 1--form $\wh{\tail}$ may be thought of as a differential analogue of the half--sum
$\wh{\rho}$ of all positive roots, and its resummation as parallel to Kac's construction
of an element $\rho\in\h$ with the same formal properties as $\wh{\rho}$. Note also
that, since positive imaginary roots are invariant under $W$, it is equivalent to resum
$\wh{\tail}^{\real}=\half{1}\sum d\log\alpha\,\hinv{\alpha}$, where the sum is restricted
to positive real roots. Correspondingly, both $\wh{\tail}$ and $\wh{\tail}^{\real}$ satisfy
$w^*\wh{\beta}=\wh{\beta}-\sum_{\alpha\in\Rs{+}\cap
w^{-1}\Rs{-}}d\log\alpha\,\hinv{\alpha}$
for any $w\in W$.}
\begin{equation}\label{eq:b hat}
\wh{\tail}=\half{1}\sum_{\alpha\in\Rs{+}}\rsm{\alpha}
\frac{d\alpha}{\alpha}\hinv{\alpha}
\end{equation}
Such an explicit resummation is carried out in Appendix~\ref{s:coda} when
$\g$ is affine. Its construction relies on the well--known resummation of the
series $\sum_{n\geqslant 0}(z+n)^{-1}$ via the logarithmic derivative $\Psi$
of the Gamma function, through its expansion
\[\Psi(z)=\frac{1}{z}+\sum_{n\geqslant 1}\left(\frac{1}{z+n}-\frac{1}{n}\right)\]

\subsection{} 

At present, it is not clear how to carry out such a resummation for an arbitrary
symmetrisable \KM algebra. We therefore opt for an alternative route: rather
than altering $\nablawk$, we modify its monodromy $\mu_{:\K:}$ as follows.

The lack of equivariance of $\mu_{:\K:}$ is measured by a 1--cocycle
$\A=\{\A_w\}$ on $W$. We show in Section \ref{s:holo-Cox-rep} that $\A_w$
is the monodromy of the abelian connection $d-a_w$, where
\[a_w=w^*\nablawk-\nablawk
=-\sfh\negthickspace\negthickspace\negthickspace\sum_{\alpha\in\Rs{+}\cap w^{-1}\Rs{-}}
\frac{d\alpha}{\alpha}\cdot\hinv{\alpha}\]
By relying on van der Lek's presentation of the fundamental groupoid of the
complexified Tits cone $\sfX\subset\hreg$ \cite{vdL}, we then prove that $\A$
is the coboundary of an explicit abelian cochain $\B$. As a consequence, the
monodromy of $\nablawk$ multiplied by $\B$ gives rise to a canonical 1--parameter
family of actions of the braid group $\Br{W}=\pi_1(\sfX/W)$ on any integrable,
category $\O$ module $V$.

We also prove that if $\tail$ is a resummation of the divergent 1--form $\wh{\tail}$
\eqref{eq:b hat}, the cochain $\B$ is the monodromy of the abelian connection
$d-\sfh\tail$, thus showing in particular that our two approaches coincide
when $\g$ is finite--dimensional or affine.

\subsection{}

Our main result can now be formulated as follows.

\begin{theorem}\label{th:intro thm}
The ($W$--equivariant) monodromy of $\nablawk$ on a category $\O$ integrable
$\g$--module $V$ is canonically equivalent to the quantum Weyl group action of
the braid group $\Br{W}$ on a quantum deformation of $V$.
\end{theorem}

Our strategy is patterned on that of \cite{vtl-3,vtl-4,vtl-6}, and hinges on the notion
of {\it braided Coxeter category} developed in \cite{ATL1-2}. Informally speaking,
such a category is a braided tensor category which carries commuting actions of
Artin's braid groups $B_n$, and of a given generalised braid group $\Br{W}$, on
the tensor powers of its objects. For $\Uhg$, such a structure arises on the category
$\Ohint$ of integrable, highest weight modules from the $R$--matrix and quantum
Weyl group operators.

For the category $\Oint$ of integrable, highest weight $\g$--modules, we 
prove in Section \ref{s:braided-Cox-holo} that such a structure arises from
the {\it dynamical coupling} of the KZ and Casimir connections of $\g$ \cite{FMTV}.
This is analogous to the fact that the monodromy of the KZ equations gives
rise to a braided tensor category structure on category $\O$ \cite{drinfeld-91},
and the fact that the canonical fundamental solutions of the Casimir equations
constructed by Cherednik and \DCP \cite{cherednik-90,DCP} give rise to a
Coxeter structure on category $\O$ \cite{vtl-4}. One crucial difference, however,
is that the joint Casimir--KZ system has {\it irregular singularities} when the
differences $z_i-z_j$ of the evaluation points tend to infinity. We address the
corresponding {\it Stokes phenomena} by adapting the argument of \cite{vtl-6},
and construct canonical solutions of the joint KZ--Casimir system with prescribed
asymptotics when $z_i-z_j\to\infty$ for any $i\neq j$.

\subsection{}\label{ss:EK equiv} 

Once the monodromy of the Casimir connection of $\g$ (resp. the quantum Weyl
group operators of $\Uhg$) are understood as arising from a braided Coxeter
structure on $\Oint$ (resp. $\Ohint$), Theorem \ref{th:intro thm} is deduced by
proving that $\Oint$ and $\Ohint$  are equivalent as braided Coxeter categories.

Such a statement presupposes in particular that $\Oint$ and $\Ohint$ are equivalent
as abelian categories. When $\g$ is finite--dimensional, this follows from the
fact that $U\g\fml$ and $\Uhg$ are isomorphic as algebras. While this is no longer
true for an arbitrary $\g$, an equivalence of abelian categories can be obtained via
{\it \nEK quantisation} \cite{ek-1,ek-2,ek-6}.

The EK equivalence relies on embedding category $\O$ (resp. $\Oh$) into the
category $\DYO_{\bm{}}$ of {\it \DYt} modules over the negative Borel
subalgebra $\bm{}$ (resp. the category $\DYO_{\Uhbm{}}$ of admissible \DYt
modules over $U_\hbar\bm{}$), which follows from the 
fact that 
$\g$ is a quotient of the restricted Drinfeld double of $\bm{}$. 

Given an associator $\Phi$,
\nEK consider the braided tensor category $\DYO_{\bm{}}^\Phi$ with underlying abelian
category $\DYO_{\bm{}}$, and commutativity and associativity constraints given by
$e^{\hbar\Omega}$ and $\Phi$. They construct a tensor functor $\sff_{\bm{}}:\DYO
_{\bm{}}^\Phi\to\vect$, and prove that it lifts to an
equivalence $\wt{\sff}_{\bm{}}: \DYO_{\bm{}}^\Phi\to\DYO_{\Uhbm{}}$
\cite{ek-2,ATL1-1}.

\subsection{}\label{ss:EK transfer} 

An equivalence of braided Coxeter categories further requires that the EK equivalence
be compatible with restriction to standard Levi subalgebras. To establish this, we constructed
in \cite{ATL1-1} a {\it relative} version of \nEK quantisation, which takes as input a pair
of Lie bialgebras $\a\subseteq\b$. This yields in particular a tensor functor $\sff_{\a,\b}:
\DYO^\Phi_\b\to\DYO^\Phi_\a$ which is isomorphic to restriction, is equal to $\sff_\b$
when $\a=0$, and to the identity when $\a=\b$.
We also proved that $\sff_{\a,\b}$ is compatible with the Tannakian equivalences
$\wt{\sff}_\a,\wt{\sff}_\b$, in that there is a natural isomorphism $v_{\a,\b}$ which fits
in the commutative diagram
\begin{equation}\label{eq:restriction}
\xymatrix@C=2cm{
\DYO^\Phi_{\b} \ar[r]^{\wt{\sff}_\b} \ar[d]_{\sff_{\a,\b}}	& \DYO_{U_\hbar\b} \ar[d]^{\res} \\
\DYO^\Phi_{\a} \ar[r]_{\wt{\sff}_\a} \ar@{=>}[ru]^{v_{\a,\b}}				& \DYO_{U_\hbar\a}			\\
}\end{equation}

In \cite{ATL1-2}, we used the data $\{\wt{\sff}_\b,\sff_{\a,\b},v_{\a\b}\}$, where $\a\subseteq\b$
range over the Borel subalgebras of all standard Levi subalgebras of $\g$, to {\it transfer} the
braided Coxeter structure on $\Ohint$ arising from the $R$--matrix and quantum Weyl group
to one on $\Oint$.

\subsection{}\label{ss:diff start}

To show that the transferred structure is equivalent to the one arising from the 
Casimir--KZ system, 
we rely on a {\it rigidity} result
according to which there is, up to equivalence, a unique braided Coxeter structure
on $\Oint$ with prescribed restriction functors, commutativity constraints, and local
monodromies.

When $\g$ is finite--dimensional, rigidity is proved in \cite{vtl-3,vtl-4} by relying
on the well--known computation of the Hochschild (coalgebra) cohomology of $U
\g$ in terms of the exterior algebra of $\g$, as well as an appropriately defined {\it
Dynkin diagram cohomology} designed to deal with secondary obstructions.

For an arbitrary $\g$, the cobar complex $\Ug^{\otimes\bullet}$ needs to be replaced
by its {\it completion} $\U^{\bullet}_{\g}$ \wrt category $\O$. This is so because $\Ug$ and $\Ug^{\otimes 2}$ do not contain the
Casimir operator $C$ of $\g$ and the invariant tensor $2\Omega=\Delta(C)-C\otimes
1-1\otimes C$ respectively, and are therefore not appropriate receptacles for the
coefficients of the Casimir and KZ connections. Unfortunately, $\U^{\bullet}_{\g}$
has an unwieldy and, to the best of our knowledge, unknown Hochschild cohomology.

\subsection{} 

To remedy this, we replace $\U^{\bullet}_{\g}$ with a suitable cosimplicial
subalgebra, which is big enough to contain the data describing the braided
Coxeter structures coming from $\Uhg$ and the Casimir--KZ connection,
yet small enough to have a manageable Hochschild cohomology.
This algebra is a refinement of Enriquez's {\it universal algebra} \cite{e1}
which we introduced in \cite{ATL2}, and arises as follows.

We first embed category $\O$ into the larger category of \DYt modules over
$\bm{}$, as explained in \ref{ss:EK equiv}. This yields a smaller algebra of endomorphisms
$\U^{\bullet}_{\bm{}}$, together with a canonical map $\U^{\bullet}_{\bm{}}\to\U^{\bullet}_{\g}$.
We then consider the subalgebra $\DUA{}{\bullet}\subset\U^{\bullet}_{\bm{}}$ consisting of all
\emph{universal} endomorphisms, \ie those obtained by compositions of iterated action and
coaction maps. Finally, taking into account the root space decomposition of $\bm{}$, we consider
the refinement $\DUA{}{\bullet}\subset\DUA{\rootsys}{\bullet}\subset\U^{\bullet}_{\bm{}}$
generated by the homogeneous components of universal endomorphisms.

\subsection{}

The Hoschschild cohomology of $\DUA{\rootsys}{\bullet}$ can be computed via the
calculus of Schur functors developed by Enriquez in \cite{e1}, and shown to be given
by a universal version of the exterior algebra of $\g$ \cite{ATL2}. In particular, $\DUA{\rootsys}{\bullet}$ behaves
like an (uncompleted) enveloping algebra, with the added 
feature that it does not contain primitive elements. This leads to a strong rigidity statement,
namely the fact that two braided Coxeter structures on $\Oint$ which are {\it universal},
that is such that their structure constants lie in $\DUA{\rootsys}{\bullet}$, are {\it uniquely}
equivalent. It also entirely bypasses the use of Dynkin diagram cohomology since the
secondary obstructions are primitive, and therefore zero in $\DUA{\rootsys}{\bullet}$.

\subsection{}\label{ss:diff end}

To conclude the proof of Theorem \ref{th:intro thm}, there remains to show that the braided
Coxeter structures on $\Oint$ coming from the joint KZ--Coxeter system and the transfer
from $\Uhg$ are universal. The first statement is proved in Sections \ref{s:double-holo},
\ref{s:km-holo} and \ref{s:main-thm}. It follows from the fact that an appropriate {\it double
holonomy algebra} $\DBLHA{\rootsys}{\bullet}$ underlying the KZ and Coxeter connections
admits a map to $\DUA{\rootsys}{\bullet}$.

The second statement is proved in Section \ref{s:diag-HA}. It follows from the construction
of the transfer of braided Coxeter from $\Ohint$ to $\Oint$ described in \ref{ss:EK transfer}.
The latter implies that the structure constants of the transferred structure lie in a subalgebra
$\DUA{\dgr}{\bullet}\subset\DUA{\rootsys}{\bullet}$ generated by the {\it diagrammatic} homogeneous components
of universal endomorphisms. By definition, these are the components corresponding to the
subalgebras of $\bm{}$ generated by $\{h_j,f_j\}_{j\in J}$, where $J$ is a subset of the
simple roots.

The following summarises the relations between the cosimplicial algebras described in
\ref{ss:diff start}--\ref{ss:diff end}
\begin{equation*}
\xymatrix{
\DBLHA{\rootsys}{\bullet}\ar[rd]&&&\\
&\DUA{\Delta}{\bullet}\ar[r]&\U^{\bullet}_{\bm{}}\ar[r]&\U^{\bullet}_{\g}\\
\DUA{\dgr}{\bullet}\ar[ru]&&&
}
\end{equation*}

\subsection{} 

In \cite{ATL5}, we obtain an analogue of Theorem~\ref{th:intro thm} for the actions of the
{\it pure} braid group $\P_W\subset\Br{W}$ on (not necessarily integrable) modules in $\O$
and $\Oh$. Specifically, we show that the quantum Weyl group operators of $\Uhg$ give
rise to a canonical action of $\P_W$ on any $\Uhg$--module $\V\in\Oh$. By relying on 
the methods developed in the present paper, we then show that this action describes the
monodromy of $\nablawk$ on the $\g$--module $V\in\O$ corresponding to $\V$ under the
Etingof--Kazhdan equivalence. We also extend these results to yield equivalent
representations of parabolic pure braid groups on parabolic category $\O$ for $\Uhg$ and $\g$.

\subsection{Outline of the paper}

The paper is divided in four parts.

In Part \ref{part-one}, we prove that the monodromy of the normally ordered Casimir connection
can be modified by an abelian cochain to make it $W$--equivariant. We also review the definition
of a {\it Coxeter algebra} following \cite{vtl-4,ATL1-2}. By adapting the construction of fundamental
solutions of the holonomy equations due to Cherednik and \DCP \cite{cherednik-90,DCP} to
infinite hyperplane arrangements, we then show that this modified monodromy arises from a
Coxeter algebra structure on the holonomy algebra $\DCPHA{\rootsys}$ of the root arrangement of $\g$.

In Part \ref{part-two}, we introduce the double holonomy algebra $\DBLHA{\rootsys}{\bullet}$
of $\g$, a cosimplicial algebra which contains both $\DCPHA{\rootsys}$ and the tower of holonomy
algebras of type $\sfA_n$. We review the definition of a {\it braided Coxeter algebra} \cite{vtl-4,ATL1-2},
and show that the dynamical coupling of the Casimir and KZ equations gives rise to a braided
Coxeter structure on $\DBLHA{\rootsys}{\bullet}$.

In Part \ref{part-three}, we review the definition of a braided Coxeter category following \cite{ATL1-2}.
We 
show that a braided Coxeter structure on the double holonomy algebra $\DBLHA{\rootsys}
{\bullet}$ gives rise to a braided Coxeter structure on the category $\Oint$ of integrable, highest weight
modules over $\g$. By Part \ref{part-two}, this implies
that the coupled Casimir--KZ system yields
a braided Coxeter category $\OCox{\g,\nabla}{\scsop{\hdef, \sint}}$ with underlying abelian
category $\Oint$. We also point out that the quantum Weyl group operators and $R$--matrix of 
$\Uhg$ give rise to a braided Coxeter category $\OCox{\Uhg,\Rmx, \qWS{}}{\scsop{\sint}}$ with 
underlying abelian category $\Ohint$.

The final Part \ref{part-four} contains the proof of our main result, namely the equivalence of the braided
Coxeter categories $\OCox{\g,\nabla}{\scsop{\hdef, \sint}}$ and $\OCox{\Uhg,\Rmx, \qWS{}}{\scsop
{\sint}}$. We first show that the braided Coxeter structure on $\OCox{\g,\nabla}{\scsop{\hdef, \sint}}$
can be extended to 
the category of \DYt modules over $\bm{}$. 
The corresponding structure $\DYCox{\bm{},\nabla}{\hdef,\sint}$ is {\it universal}, that is
arises from a
$\uPROP$ $\dLBA{\rootsys}$ describing a Lie bialgebra $[\bm{}]$ with the 
root decomposition
of 
$\bm{}$. Specifically, we prove that the double holonomy algebra $\DBLHA{\rootsys}{\bullet}$
maps to the endomorphisms of the tensor product of \DYt modules over $[\bm{}]$.

In a parallel vein, we show that the braided Coxeter structure $\OCox{\Uhg,\Rmx, \qWS{}}{\scsop{\sint}}$
can be extended to the category of admissible \DYt modules over $U_\hbar\bm{}$ and then, using the 
2--categorical extension of EK quantisation obtained in \cite{ATL1-1,ATL1-2}, transferred to a braided
Coxeter category $\DYCox{\bm{},\Rmx,\qWS{}}{\hbar,\sint}$ on integrable \DYt modules over $\bm{}$. The latter is also universal in that it comes
from a coarsening $\dLBA{\dgr}$ of the $\uPROP$ $\dLBA{\Delta}$, which only records the standard
subalgebras of $\bm{}$ generated by simple root vectors. Finally, we use the rigidity of universal braided Coxeter
algebra structures obtained in \cite{ATL2} to obtain the equivalence of $\DYCox{\bm{},\nabla}{\hdef,\sint}$
and $\DYCox{\bm{},\Rmx,\qWS{}}{\hbar,\sint}$.

\subsection{Acknowledgments} It is a pleasure, as ever, to thank Pavel Etingof for his insightful comments
and interest in the present work.

\part{The Casimir connection}\label{part-one}

\section{\KM algebras}\label{s:KMAs}

\summary{\color{magenta}{
		In this section,
		\begin{itemize}
			\item \ref{ss:km-recap}: KM algebras in short (necessary - root system)
			\item \ref{ss:sym-km}: symmetrisable KM in short (necessary - Weyl group)
			\item \ref{ss:diag-KM}: diagrammatic KM 
		\end{itemize}
}}

\subsection{Realisations of matrices}\label{ss:km-realisations}

In Sections \ref{ss:km-realisations}--\ref{ss:sym-km}, we mostly follow \cite{Ka}.
Let $\bfI$ be a finite set, $\sfk$ a field of characteristic zero, 
and $\GCM{A}=(a_{ij})_{i,j\in\bfI}$ a matrix with entries in $\sfk$.
A {\it realisation} of $\sfA$ is a triple $(\h,\Pi,\Pi^\vee)$, 
where
\begin{itemize}
	\item $\h$ is a \fd vector space over $\sfk$\footnote{Note that, unlike \cite{Ka},
	we do not require $\h$ to be of (minimal) dimension $2|\bfI|-\rank(\sfA)$.}
	\item $\Pi=\{\alpha_i\}_{i\in\bfI}$ is a linearly independent subset of $\h^*$
	\item $\Pi^\vee=\{\cor{i}\}_{i\in\bfI}$ is a linearly independent subset of $\h$
	\item $\alpha_i(\cor{j})=a_{ji}$ for any $i,j\in\bfI$ 
\end{itemize}

Given a realisation $(\h,\Pi,\Pi^\vee)$ of $\sfA$, we
denote by
\begin{equation}\label{eq:h' hess}
	\h'=\langle h_i\rangle_{i\in\bfI}\subset\h
	\aand
	\h\ess=\h/\Pi^\perp
\end{equation}
the $|\bfI|$--dimensional subspace and quotient of $\h$
determined by $\Pi^\vee$ and the annnihilator of $\Pi$ respectively. Note that $\h',\h\ess$ do not 
depend upon the choice of $\h$.

\subsection{Kac--Moody algebras}\label{ss:km-recap}

Let $\wt{\g}$ be the Lie algebra generated by $\h$,
$\{e_i, f_i\}_{i\in\bfI}$ with relations $[h,h']=0$, for any $h,h'\in\h$, and
\[
[h,e_i]=\alpha_i(h)e_i
\qquad
[h,f_i]=-\alpha_i(h)f_i
\qquad
[e_i,f_j]=\drc{ij}\cor{i}
\]
The Kac--Moody algebra corresponding to $\GCM{A}$ is the Lie algebra
$\g=\wt{\g}/\r$, where $\r$ is the sum of all two--sided ideals in $\wt{\g}$
having trivial intersection with $\h\subset\wt{\g}$. 
If $\GCM{A}$ is a generalised Cartan matrix (\ie $a_{ii}=2$, $a_{ij}\in\IZ_{\leqslant 0}$, $i\neq j$,
and $a_{ij}=0$ implies $a_{ji}=0$), the ideal $\r$ is generated by the Serre relations 
$\mathsf{ad}(e_i)^{1-a_{ij}}(e_j)=0=\mathsf{ad}(f_i)^{1-a_{ij}}(f_j)$ for any $i\neq j$.
The following is straightforward.

\begin{lemma}\hfill
	\begin{enumerate}\itemsep0.25cm
		\item The center of $\g$ 
		is $\z(\g)=\Pi^\perp$, and $\dim\z(\g)\cap\h'=|\bfI|-\rank(\GCM{A})$.
		\item $\h'=\h\cap[\g,\g]$ and $\h\ess=\h/\z(\g)$.
	\end{enumerate}
\end{lemma}
We refer to $\h'$ and $\h\ess$ as the {\em derived} and {\em essential} Cartan,
respectively. Set ${\mathsf{Q}}_+=\bigoplus_{i\in\bfI}\IZ_{\geqslant0}{\alpha}_i\subseteq{\h}^*$,
so that $\g$ has the root space decomposition $\g=\n_-\oplus\h\oplus\n+$, where
\[\n_\pm=\bigoplus_{\alpha\in\mathsf{Q}_+\setminus\{0\}}\g_{\pm\alpha}
\aand
{\g}_{{\alpha}}=\{X\in{\g}\;|\;[h,X]=\alpha(h)X,\;\forall h\in{\h}\}\]
Denote by
$\Rs{+}=\{\alpha\in\mathsf{Q}_+\;|\; {\g}_{\alpha}\neq0\}$ the set
of positive roots of ${\g}$ and set $\Rs{}=\Rs{+}\sqcup(-\Rs{+})$. 
For any root $\alpha\in\Rs{}$,  the root multiplicity $\rsm{\alpha}=
\dim\g_{\alpha}$ is finite. Moreover, if $\GCM{A}$ is a generalised Cartan matrix, 
the Weyl group $W$ of $\g$ preserves the root multiplicities, \ie, for any $\alpha\in\Rs{}$ 
and $w\in W$, $\rsm{\alpha}=\rsm{w\alpha}$.

\subsection{Symmetrisable Kac--Moody algebras}\label{ss:sym-km}

Let $\GCM{A}$ be a symmetrisable generalised Cartan matrix and fix a 
decomposition $\GCM{B}=\GCM{DA}$, where $\GCM{D}=\Diag(\symd
{i})_{i\in\bfI}$ is an invertible diagonal matrix with coprime entries $\symd
{i}\in\IZ_{>0}$ such that $\GCM{B}$ is symmetric. 

Let $\bil$ be a symmetric, non--degenerate bilinear form on $\h$ such
that\footnote{Such a form always exists, see \eg \cite[Prop. 11.4]{ATL1-2}.}
\begin{equation}\label{eq:di hi}
\<\cor{i},-\>=\symdi{i}\alpha_i
\end{equation}
Then, $\bil$ uniquely extends to an invariant symmetric
bilinear form on $\wt{\g}$, and $\iip{e_i}{f_j}=\symdi{i}\drc{ij}$. The kernel of this form is $\r$, 
so that $\bil$ descends to a nondegenerate form on $\g$.
Set $\bpm{}=\h\oplus\bigoplus_{\alpha\in\Rs{+}}\g_{\pm\alpha}\subset\g$. 
The bilinear form induces a canonical isomorphism of graded vector spaces 
$\bp{}\simeq(\bm{})^{\star}$, where 
$(\bm{})^{\star}=\h^*\oplus\bigoplus_{\alpha\in\Rs{+}}\g_{-\alpha}^*$.

\Omit{\noindent\remark\;
Through the identification $\bp{}\simeq(\bm{})^{\star}$, we obtain 
the {\em standard} Lie bialgebra structure on $\g$ with cobracket 
$\delta:\g\to\g\wedge\g$ given by (cf.~\ref{ss:km-sLBA})
\begin{equation}
	\delta|_{\h}=0
	\qquad
	\delta(e_i)=\symd{i} \cor{i}\wedge e_i
	\qquad
	\delta(f_i)=\symd{i} \cor{i}\wedge f_i
\end{equation}
}
We denote by $\nu:\h\to\h^*$ the isomorphism induced by $\iip{\cdot}{\cdot}$
%
and, for any $\beta\in\sfQ$, we set $\hinv{\beta}=\nu^{-1}(\beta)$. Recall that,
by \cite[Thm.~2.2]{Ka}, for any $x\in\g_{\alpha}$ and $y\in\g_{-\alpha}$, we
have $[x,y]=\iip{x}{y}\cdot\hinv{\alpha}$.

\subsection{Diagrammatic Kac--Moody algebras}\label{ss:diag-KM}

Let $\GCM{A}$ be a generalised Cartan matrix and $\dgr$ the Dynkin diagram of
$A$, \ie the undirected graph having $\bfI$ as its vertex set and an edge between
$i\neq j$ unless $a_{ij}=0=a_{ji}$.
For any subset of vertices $B\subseteq \dgr$, let $\sfA_B$ be the $|B|\times|B|$
matrix $(a_{ij})_{i,j\in B}$, $\Pi_B=\{\alpha_i\}_{i\in B}\subseteq\Pi$ and $\Pi_B^\vee=\{\alpha_i^\vee\}_{i\in B}\subseteq\Pi^\vee$. 

\begin{definition}\label{de:diagr real}\hfill
\begin{enumerate}
\item A realisation $(\h,\Pi,\Pi^\vee)$ of $\sfA$ is {\em diagrammatic} if it is endowed
with a collection of subspaces $\{\h_B\}_{B\subseteq\dgr}$ of $\h$ such that $\h_\dgr
=\h$, and the following holds
	\vspace{0.25cm}
	\begin{itemize}\itemsep0.25cm
		\item $\Pi_B^\vee\subset\h_B$ and $(\h_B, \Pi_B|_{\h_B},\Pi_B^\vee)$ is a realisation of $\sfA_B$ for any $B\subseteq\dgr$
		\item $\h_{B'}\subseteq\h_B$ whenever $B'\subseteq B$
		\item $\h_{B_1\sqcup B_2}=\h_{B_1}\oplus\h_{B_2}$ and $\h_{B_1}\subseteq\Pi_{B_2}^\perp$ whenever $B_1\perp B_2$\footnote{
			Two subdiagrams $B_1,B_2\subseteq\dgr$ are {\it orthogonal} if they
			have no vertices in common, and no two vertices $i\in B_1$, $j\in B
			_2$ are joined by an edge in $\dgr$ (cf.~\ref{ss:ns}).
		}
	\end{itemize}
\item If $A$ is symmetrisable, a diagrammatic realisation $(\h,\Pi,\Pi^\vee)$ is
additionally required to be endowed with a non--degenerate symmetric bilinear
form $\bil$ such that \eqref{eq:di hi} holds, and its restriction to 
each $\h_B$ is non--degenerate.\footnote{Following \cite[Secs.~5 and 9]{ATL1-2}, one could consider a more general definition, where a diagrammatic realisation of $\sfA$ is a {collection} of realisations $(\h_B, \Pi_B, \Pi_B^\vee)$ of $\sfA_B$ ($B\subseteq\dgr$) equipped with a system of linear maps $i_{BB'}\colon\h_{B'}\to\h_B$ ($B'\subseteq B$) satisfying natural compatibility conditions with respect to subdiagrams and orthogonal diagrams. When $\sfA$ is symmetrisable, however, the maps $i_{BB'}$ are required to be isometries, and thus embeddings. We have therefore opted to identify $\h_B$ with a subspace of $\h=\h_\dgr$ in Definition \ref{de:diagr real}.}

\item A (symmetrisable) Kac--Moody algebra is \emph{diagrammatic} if the underlying realisation is.
\end{enumerate}
\end{definition}
\Omit{Moreover, if $\GCM{A}$ is symmetrisable, the inclusions $\h_{B'}\subseteq\h_B$, $B'\subseteq B$, are further required to be isometries.}

\noindent\remark\;\hfill
\begin{enumerate}\itemsep0.25cm
	\item
	Any symmetrisable generalised Cartan matrix $\GCM{A}$ has a diagrammatic realisation.
	Namely, if $\sfA$ is of finite, affine or hyperbolic type, its minimal realisation is clearly
	diagrammatic. This is not always true for Cartan matrices of indefinite type. However, we
	proved in \cite[Prop.~12.4]{ATL1-2} that a canonical (non--minimal) diagrammatic realisation
	with $\dim\h=2|\bfI|$ always exists.
	\item 
	Note that a diagrammatic symmetrisable Kac--Moody algebra $\g$ is naturally endowed
	with \emph{diagrammatic} Lie subalgebras $\g_B=\<\{e_i,f_i\}_{i\in B},\h_B\>\subseteq\g$, $B\subseteq\dgr$, such that $\g_{B'}\subseteq\g_{B}$ if 
	$B'\subseteq B$ and $[\g_{B_1},\g_{B_2}]=0$ if $B_1\perp B_2$. In particular, $U\g$ has a natural
	structure of diagrammatic algebra in the sense of Definition~\ref{ss:functorial-diag-alg}.
\end{enumerate}

\section{The Casimir connection}\label{s:Casimir}

We review the definition of the Casimir connection of a symmetrisable
Kac--Moody algebra, introduced by De Concini (cf.~\cite{procesi-96} where the
Casimir connection is briefly mentioned in the introduction),
Millson--Toledano Laredo
\cite{vtl-2, MTL}, and Felder--Markov--Tarasov--Varchenko \cite{FMTV}, and provide an alternative proof of its flatness.

Henceforth, we fix a symmetrisable generalised Cartan matrix $\GCM{A}$, a diagrammatic 
realisation $(\h_{\IR, B},\Pi_B|_{\h_{\IR,B}},\Pi_B^\vee)_{B\subseteq\dgr}$ over $\IR$, the diagrammatic
realisation over $\IC$ given by its complexification $(\h_{B},\Pi_B|_{\h_B},\Pi_B^\vee)_{B\subseteq\dgr}$, with 
$\h_B=\IC\ten_{\IR}\h_{\IR,B}$, the corresponding Kac--Moody algebra $\g$ over $\IC$
and the diagrammatic subalgebras $\g_B\subseteq\g$, $B\subseteq\dgr$.

\subsection{Fundamental group of root system arrangements}\label{ss:fund-group}
Let $\GCM{A}$ be a symmetrisable generalised Cartan matrix, $(\h_{\IR},\Pi,\Pi^\vee)$
a realisation of $\GCM{A}$ over $\IR$, and $(\h=\IC\ten_{\IR}\h_{\IR},\Pi,\Pi^\vee)$ its complexification.
Let $\Pi^\perp\subset\h$ be the annihilator of $\Pi$, set $\h\ess=\h/\Pi^\perp$, and note
that $\h\ess$ is independent of the realisation of $\GCM{A}$. Let
\begin{equation}
	\C=\{h\in\h\ess_{\IR}\;|\;\forall i\in\bfI,\,\root{i}(h)>0\}
\end{equation}
be the fundamental Weyl chamber in $\h\ess_{\IR}$, and ${\sf Y}_{\IR}=\bigcup_{w\in W}w(\ol{\C})$
the 
Tits cone. ${\sf Y}_{\IR}$ is a convex cone, and the Weyl group
$W$ acts properly discontinuously on its interior $\mathring{{\sf Y}}_{\IR}$ and complexification
${\sf Y}=\mathring{{\sf Y}}_{\IR}+\iota\h\ess_{\IR}\subseteq\h\ess$, where $\iota=\sqrt{-1}$ \cite{Lo,V1}. The regular points of this action are given by
\[
\sfX={\sf Y}\setminus\bigcup_{\alpha\in\Rs{+}}\Ker(\alpha)
\]
The action of $W$ on $\sfX$ is proper and free, and the space $\sfX/W$ is a complex
manifold.\\

Recall that the braid group of $W$  is the group $\Br{W}$ presented on the generators $\topS{1},\dots, \topS{|\bfI|}$, with relations given by
\begin{equation}
	\underbrace{\topS{i}\cdot\topS{j}\cdot\topS{i}\cdot\cdots}_{m_{ij}}=
	\underbrace{\topS{j}\cdot\topS{i}\cdot\topS{j}\cdot\cdots}_{m_{ij}}
\end{equation}
for any $i,j\in\bfI$ such that $m_{ij}<\infty$, where $m_{ij}$ is the order of $s_is_j$ in $W$. The pure braid group $\P_W\subset\Br{W}$ is the kernel of the standard projection $\Br{W}\to W$.

The following result is due to van der Lek \cite{vdL}, and generalises Brieskorn's Theorem \cite{Br} to the case of an 
arbitrary Weyl group.

\begin{theorem}
	The fundamental groups of $\sfX/W$ and $\sfX$ are isomorphic to $\Br{W}$ and $\P_W$ respectively.
\end{theorem}

The generators $\{\topS{i}\}_{i\in\bfI}$ of $\Br{W}$ may be described as follows. Let $p:\sfX\to\sfX/W$
be the canonical projection, fix a point $x_0\in\C$ and use $[x_0]=p(x_0)$ as a base point in
$\sfX/W$. For any $i\in\bfI$, choose an open disk $D_i$ in 
$x_0+\IC\cor{i}$, centered in $x_0-\frac{\root{i}(x_0)}{2}\cor{i}$, and 
such that $\ol{D}_i$
does not intersect any root hyperplane other than $\Ker(\alpha_i)$.
Let $\gamma_i:[0,1]\to x_0+\IC\cor{i}$ be the path from $x_0$ to $s_i(x_0)$ in $\sfX$ determined by $\gamma_i\vert_{[0,1/3]\cup[2/3,1]}$ is affine and lies in 
$x_0+\IR\cor{i}\setminus D_i$, the points $\gamma_i(1/3),\gamma_i(2/3)$ are in $\partial\ol{D}_i$, and
$\gamma_i|_{[1/3,2/3]}$ is a semicircular arc in $\partial\ol{D}_i$,
positively oriented with respect to the natural orientation of $x_0+\IC\cor{i}$. Then, $\topS{i}= p\circ\gamma_i$.

\subsection{The Casimir connection}\label{ss:casimir-conn}

For any positive root $\alpha\in\Delta_+$, let $\{e_{\pm\alpha}^{(i)}\}_{i=1}
^{\rsm{\alpha}}$ be bases of $\g_{\pm\alpha}$ which are dual with respect
to $\iip{\cdot}{\cdot}$, and
\begin{equation}\label{eq:K +}
\Ku{\alpha}{+}=\sum_{i=1}^{\rsm{\alpha}}e_{-\alpha}^{(i)}e_\alpha^{(i)}
\end{equation}
the corresponding truncated and normally ordered Casimir operator. Let
$V$ be a $\g$--module in category $\O$ and
$\IV= \sfX\times V$ the holomorphically trivial vector bundle over 
$\sfX$ with fibre $V$ (cf.~\ref{ss:cat-O-diag}). 
Finally, let $\nablah\in\IC$ be a 
complex parameter.

\begin{definition}
	The Casimir connection of $\g$ is the connection on $\IV$ given by
	\begin{equation}\label{eq:Casimir}
		\nablak=
		d-\nablah\sum_{\alpha\in\Rs{+}}\frac{d\alpha}{\alpha}\cdot\Ku{\alpha}{+}
	\end{equation}
\end{definition}
\noindent
The Casimir connection for a semisimple Lie algebra was discovered by De
Concini around '95 (unpublished, though the connection is referenced in \cite
{procesi-96}) and, independently, Millson--Toledano Laredo \cite{vtl-2,
	MTL} and Felder--Markov--Tarasov--Varchenko \cite{FMTV}.
In \cite{FMTV}, the case of an arbitrary symmetrisable \KM algebra is considered. We give an alternative proof of flatness in this more general case, along the lines of \cite{vtl-2,MTL} in Section \ref{ss:flatness}.

\subsection{Local finiteness}\label{ss:local-finite}

The sum in \eqref{eq:Casimir} is locally finite even if $\rootsys$
is infinite since, for any $v\in V$, $\Ku{\alpha}{+} v=0$ for all but finitely
many $\alpha\in\Rs{+}$. Differently said, let $\rht:\sfQ_+\to\IZ_{\geqslant}$
be the height function on the positive root lattice
given by $\rht(\sum_{i\in\bfI}k_i\alpha_i)=\sum_{i\in\bfI}k_{i}$.
Then, $\rht^{-1}(n)$ is finite for any $n\in\IZ_{\geqslant 0}$.
Let $\lambda_1,\ldots,\lambda_p\in\h^*$ be such that the set of weights
of $V$ is contained in $\bigcup_{i=1}^p D(\lambda_i)$
where
$D(\lambda_i)=\{\mu\in\h^*|\mu\leqslant\lambda_i\}$ and
$\mu\leqslant\lambda$ iff $\lambda-\mu\in\sfQ_+$.
For $n\in\IZ_{>0}$, set
\begin{equation}
	V^n=
	\bigoplus_{\substack{
			\mu\in\h^*:\\[.3 ex]
			\rht(\lambda_i-\mu)\leqslant n,\\[.3 ex]
			\forall i:\medspace \mu\in D(\lambda_i)}}
	V[\mu]
\end{equation}
where $V[\mu]$ is the weight space of $V$ of weight $\mu$. Then,
$\displaystyle{V=\lim_{\longrightarrow}V^n}$, each $V^n$ is invariant
under the operators $\Ku{\alpha}{+}$, and $\Ku{\alpha}{+}$ acts as zero
on $V^n$ if $\rht(\alpha)>n$. Thus, if $\IV^n= \sfX\times V^n$ is
the trivial vector bundle over $\sfX$ with fibre $V^n$, then
$\displaystyle{\IV=\lim_{\longrightarrow}\IV^n}$ and $\displaystyle
{\nablak=\lim_{\longrightarrow}\nablak^n}$ where
\begin{equation}\label{eq:nablak n}
	\nablak^n=
	d-\nablah\sum_{\alpha\in\Rs{+}^{\leqslant n}}
	\frac{d\alpha}{\alpha}\cdot\Ku{\alpha}{+}
	\qquad\text{with}\qquad
	\Rs{+}^{\leqslant n}=
	\left\{\alpha\in\Rs{+}\left|\thinspace\rht(\alpha)\leqslant n\right.\right\}
\end{equation}
Note also that the pair $(\IV^n,\nablak^n)$ descends to a (trivial)
vector bundle with connection on the complement $\sfX^n$ of
the hyperplanes $\Ker(\alpha)$, $\alpha\in\Rs{+}^{\leqslant n}$, in the \fd
vector space
\begin{equation}\label{eq:h n}
	\mathsf{Y}^n=\mathsf{Y}/(\Rs{+}^{\leqslant n})^\perp
\end{equation}

\vspace{0.25cm}
\noindent\remark\;
Let $\Ku{\alpha}{}=\sum_{i=1}^{\rsm{\alpha}} e_{-\alpha}^{(i)}e_\alpha^{(i)}+
e_\alpha^{(i)}e_{-\alpha}^{(i)}$ be the truncated Casimir operator corresponding
to $\alpha\in\Delta_+$. Since $\Ku{\alpha}{}=2\Ku{\alpha}{+}+\rsm{\alpha}\hinv
{\alpha}$, the connection defined by $\{\Ku{\alpha}{}\}_{\alpha\in\Delta_+}$ can
be thought of as a Cartan extension of $\nablak$ since
\begin{equation}
	\frac{\nablah}{2}\sum_{\alpha\in\Rs{+}}\frac{d\alpha}{\alpha}\cdot\Ku{\alpha}{}=
	\nablah\sum_{\alpha\in\Rs{+}}\frac{d\alpha}{\alpha}\cdot\Ku{\alpha}{+}+
	\frac{\nablah}{2}\sum_{\alpha\in\Rs{+}}\frac{d\alpha}{\alpha}\cdot\rsm{\alpha}\hinv{\alpha}
\end{equation}
However, if $|\Rs{}|=\infty$, the second sum is not locally finite on category $\O$ modules,
in contrast with the case of $\nablak$.

\subsection{Flatness}\label{ss:flatness}

\begin{theorem}\label{th:kohno K}
	The connection $\nablak$ is flat for any $\nablah\in\IC$.
\end{theorem}
\begin{pf}
	It suffices to prove that the connection $\nablak^n$ defined by
	\eqref{eq:nablak n} is flat for any $n$. Since $\nablak^n$ is pulled
	back from the \fd vector space $\h^n$ \eqref{eq:h n}, Kohno's
	lemma \cite{Ko} implies that the flatness of $\nablak^n$ is equivalent
	to proving that, for any two--dimensional subspace $U\subset\h^*$
	spanned by a subset of $\Rs{+}^{\leqslant n}$, the following holds on $V^n$
	for any $\alpha\in U\cap\Rs{+}^{\leqslant n}$
	\begin{equation}
		\left[\Ku{\alpha}{+},\sum_{\beta\in U\cap\Rs{+}^{\leqslant n}}\Ku{\beta}{+}\right]=0
	\end{equation}
	Since $\Ku{\beta}{+}$ acts as $0$ on $V^n$ if $\rht(\beta)>n$, this
	amounts to proving that, on $V^n$
	\begin{equation}\label{eq:K-tt}
		\left[\Ku{\alpha}{+},\sum_{\beta\in U\cap\Rs{+}}\Ku{\beta}{+}\right]=0
	\end{equation}
	Let $\g_U=\h\oplus\bigoplus_{\alpha\in U\cap\rootsys}\g_\alpha$
	be the subalgebra spanned by $\h$ and the root subspaces
	corresponding to the elements of $U\cap\rootsys$. Then $\g_U$
	is a generalized Kac--Moody algebra and, modulo terms in
	$U\h$, the operator $\sum_{\beta\in U\cap\Rs{+}}\Ku{\beta}{+}$
	is proportional to the Casimir operator. Since any element in $U\h$ commutes
	with $\Ku{\alpha}{+}$, the above commutator is therefore zero.
\end{pf}

\subsection{Equivariance}

It is well known that the Weyl group $W$ of $\g$ does not act on
an integrable $\g$--module $V\in\O$ in general, but that the triple
exponentials
\begin{equation}\label{eq:triple}
	\texp{i}=\exp(e_{i})\exp(-f_{i})\exp(e_{i})
\end{equation}
give rise to an action of an extension $\wt{W}$ of $W$ by the sign
group $\IZ_2^r$, which is a quotient of $\Br{W}$ \cite{Ti}.
	
However, the connection $\nablak$ is not $\Br{W}$--equivariant and
therefore does not {\it a priori} yield a monodromy representation of
$\Br{W}$ on $V$. Indeed, for 
any $\alpha\in\Rs{+}$, $\wt{w}\in\Br{W}$ and $w\in W$ such that
$\wt{w}\mapsto w$ under the morphism $\Br{W}\to W$, 
we have
\begin{equation}\label{eq:W-on-K}
	\wt{w}\,\Ku{\alpha}{+}\,\wt{w}^{-1}=
	\left\{
	\begin{array}{lcc}
		\Ku{w\alpha}{+} & \mbox{if} & w\alpha>0\\
		\Ku{-w\alpha}{+} +\hinv{w\alpha}& \mbox{if} & w\alpha<0
	\end{array}
	\right.
\end{equation}
where $t_\beta=\nu^{-1}(\beta)\in\h'$ (cf.~\ref{ss:sym-km}), and we used 
the fact that if $\alpha\in\Rs{+}\cap w^{-1}\Rs{-}$, then $\alpha$ is real,
and $\rsm{\alpha}=1$. 
The lack of equivariance of $\nablak$ will be addressed in Section \ref
{s:holo-Cox-rep}.

\subsection{The holonomy algebra $\DCPHA{\rootsys}$}\label{ss:holonomy}

Let $\sfF_{\rootsys}$ be the free associative algebra with generators $\{\Kh{\alpha}
{}\}_{\alpha\in\Rp}$. For any $m\in\IZ_{\geqslant 0}$, let $J_m\subset\sfF_{\rootsys}$
be the two--sided ideal generated by $\Kh{\alpha}{}$, with $\alpha\not\in\Rs{+}^{\leqslant
m}$, and set $\ol{\sfF}_{\rootsys}=\lim_m\sfF_{\rootsys}/J_m$. Note that $\sum_{\beta
\in\Rs{+}}\Kh{\beta}{}$ is a well--defined element in $\ol{\sfF}_{\rootsys}$.

\begin{definition}
	The holonomy algebra $\DCPHA{\rootsys}$ is the associative algebra given
	by the quotient of $\ol{\sfF}_{\rootsys}$ by the $tt$--relations
	\begin{equation}\label{eq:tt-relns}
		\left[\Kh{\alpha}{}, \sum_{\beta\in\Psi_{\alpha}\cap\Rs{+}}\Kh{\beta}{}\right]=0
	\end{equation}
	where $\Psi_{\alpha}\subset\rootsys$ is any root subsystem of rank $2$ containing $\alpha$.
\end{definition}

\noindent\remark\; Let $\wt{J}_m$ be the two sided ideal generated by $J_m$ and the elements
\[
\left[\Kh{\alpha}{}, \sum_{\beta\in\Psi_{\alpha}\cap\Rs{+}^{\leqslant m}}\Kh{\beta}{}\right]
\]
where $\Psi_{\alpha}\subset\rootsys$ is as before. Set 
$\DCPHA{\rootsys}^{(m)}=\sfF_{\rootsys}/\wt{J}_m$.
Then, $\DCPHA{\rootsys}$ is 
isomorphic to $\lim_m\DCPHA{\rootsys}^{(m)}$.

\subsection{The holonomy algebra $\DCPHAH{\Rs{},\mathfrak{h}}$}\label{ss:ext-holo-alg}

The holonomy algebra $\DCPHA{\rootsys}$ is $\IN$--graded by $\deg(\Kh{\alpha}
{})=1$, $\alpha\in\Rs{+}$. Let $\DCPHAH{\Rs{},\mathfrak{h}}=\DCPHA{\rootsys}\wh
{\ten}{S}\mathfrak{h}$ be the completion of $\DCPHA{\rootsys}{\ten} S\mathfrak{h}$
\wrt the total grading.

The action of $W$ on $\h'$ extends to one on $\DCPHAH{\Rs{},\mathfrak{h}}$ 
patterned on \eqref{eq:W-on-K}, by setting 
\begin{equation}\label{eq:W-on-t}
	w(\Kh{\alpha}{})=
	\left\{
	\begin{array}{lcc}
		\Kh{w\alpha}{} & \mbox{if} & w\alpha>0\\
		\Kh{-w\alpha}{}+\hinv{w\alpha}& \mbox{if} & w\alpha<0
	\end{array}
	\right.
\end{equation}
where $w\in W$, $\alpha\in\Rs{+}$, and $\hinv{w\alpha}=\nu^{-1}(w\alpha)\in\h'$
(cf.~\ref{ss:sym-km}). 
Indeed, for $u,v\in W$, $\alpha\in\Rs{+}$, one has 
\begin{equation}
	u(v(\Kh{\alpha}{}))=
	\left\{
	\begin{array}{lcc}
		\Kh{uv\alpha}{} 
		& \mbox{if} & v\alpha>0 ,\  uv(\alpha)>0\\
		\Kh{-uv\alpha}{} + \hinv{uv\alpha}
		& \mbox{if} & v\alpha>0,\ uv(\alpha)<0\\
		\Kh{uv\alpha}{} + \left( u(\hinv{v\alpha}) 
		-\hinv{uv\alpha}\right)
		& \mbox{if} & v\alpha<0,\  uv(\alpha)>0\\
		\Kh{-uv\alpha}{} + u(\hinv{v\alpha})  
		& \mbox{if} & v\alpha<0,\  uv(\alpha)<0
	\end{array}
	\right.
\end{equation}
and therefore $uv(\Kh{\alpha}{})=u(v(\Kh{\alpha}{}))$.\\

\noindent\remark\; Note that any representation $V$ of $\g$ and choice of $\sfh\in\IC$
give rise to an action
\[\rho:\DCPHA{\Rs{},\mathfrak{h}}\to\mathsf{End}_\h(V)\]
by $\rho(\Kh{\alpha}{})=\sfh\cdot\Ku{\alpha}{+}$ and $\rho(h)=\sfh\cdot h$ for
$\alpha\in\Rs{+}$ and $h\in\mathfrak{h}$.


\subsection{The universal Casimir connection}\label{ss:univ-Casimir}

\newcommand{\nablat}{\nabla_{\Kh{}{}}}

\begin{definition}
	The universal Casimir connection is the formal connection on $\sfX$ 
	\begin{equation}\label{eq:univ-Casimir}
		\nablat=d-\sum_{\alpha\in\Rp}\frac{d\alpha}{\alpha}\cdot\Kh{\alpha}{}
	\end{equation}
\end{definition}
The flatness of $\nablat$ is proved as in \ref{ss:flatness}. Thus, any 
representation $\rho:\DCPHAH{\Rs{},\h}\to\sfEnd{V}$ 
gives rise to a flat connection 
\begin{equation*}
	\nabla_{\mathsf{t}, \rho}=d-\sum_{\alpha\in\Rp}\rho(\Kh{\alpha}{})\cdot\frac{d\alpha}{\alpha}
\end{equation*}
on the trivial vector bundle over $\sfX$ with fiber $V$.\\

\noindent\remark\;
We shall consider only solutions of the \emph{holonomy equation}
\begin{equation}\label{eq:holoeq}
	d\DCPS{}=\sum_{\alpha\in\Rp}\frac{d\alpha}{\alpha}\Kh{\alpha}{}\DCPS{}
\end{equation}
which are holomorphic functions in their domain of definition 
with values in $\DCPHA{\rootsys}\subset\DCPHAH{\Rs{},\h}$. 
The analytic computations performed with functions with values in 
$\DCPHA{\rootsys}$ are justified by the fact that the latter
is the inverse limit of the finite dimensional algebras $\sfF_{\rootsys}/J_{k, n}$, 
where $J_{k,n}$ is the ideal of the elements of degree $\geqslant n$ in 
$\sfF_{\rootsys}/\wt{J}_k$. In particular, a function $G$ with values in 
$\DCPHA{\rootsys}$ is determined by a sequence of compatible 
functions in the finite dimensional algebras $\sfF_{\rootsys}/J_{{k}, n}$.

\section{Equivariant monodromy}\label{s:holo-Cox-rep}

In this section, we prove that the monodromy of the universal Casimir connection
can be made equivariant \wrt to the Weyl group by multiplying it by an explicit
abelian cochain on $W$, and that it then gives rise to a representation of the
generalised braid group $\Br{W}$. 

\subsection{The orbifold fundamental groupoid of $\sfX$}
\label{ss:rep-groupoid}

\newcommand {\orb}{\mathsf{O}}

Let $\Poid(\sfX;Wx_0)$ be the fundamental groupoid of $\sfX$ based at the $W
$--orbit of $x_0$. Then, $\Poid{(\sfX/W;[x_0])}$ is equivalent to the orbifold
fundamental groupoid $\WPoid{(\sfX; Wx_0)}$, which is defined as follows.
\vskip .2cm

\begin{itemize}\itemsep0.1cm
	\item Its set of objects is $Wx_0$.
	\item A morphism between $x,y\in Wx_0$ is a pair $(w,\gamma)$,
	where $w\in W$ and $\gamma$ is a path in $\sfX$ from $x$ to $w^{-1}y$.
	\item The composition of $(w,\gamma):x\to y$ and $(w',\gamma')
	: y\to z$ is given by
	\[(w',\gamma')\circ (w,\gamma)=
	(w'w, w^{-1}(\gamma')\circ\gamma):x\to z\]
\end{itemize}

The projection functor
\begin{equation}\label{eq:projection functor}
	P:\WPoid{(\sfX; Wx_0)}\longrightarrow\Poid{(\sfX/W; [x_0])}
\end{equation}
given by $P(wx_0)=[x_0]$ and $P(w,\gamma)=[\gamma]$ is fully faithful since, for any
given $x,y\in Wx_0$, a loop $
[\gamma]\in \Poid{(\sfX/W; [x_0])}$ lifts uniquely to a path $\gamma: x\to w^{-1}y$, for
a unique $w\in W$. Any $x\in Wx_0$ therefore
determines a right inverse $\E_{x}$ of $P$ given by $\E_{x}([x_0])=x$ and $\E_{x}([\gamma])
=(w,\gamma)$, where $\gamma$ is the lift of $[\gamma]$ through $x$, and $w$
is such that $\gamma(1)=w^{-1}x$.


\subsection{Obstruction to $W$--equivariance}
\label{ss:W-equiv-obs}

\newcommand {\sfO}{\mathsf O}
\newcommand {\tO}{\mathfrak{\wh{t}}_{\Delta,\h}}
\newcommand {\tOp}{\mathfrak{\wh{t}}_{\Delta,\h'}}
\newcommand {\dsing}{d^{\operatorname{sing}}}

In what follows, we shall repeatedly identify an algebra $A$ such as $\tO$, $\tOp$ and
their semi--direct product with $W$, with the category with one object and morphisms
given by $A$, and abusively denote the latter by the same symbol. 

The universal Casimir connection $\nablat$ gives rise to a functor
\[\PT:\Poid{(\sfX; Wx_0)}\to\tOp\]
which maps a path $\gamma$ to its parallel transport $\PT(\gamma)$.
The lack of equivariance of $\nablat$ implies that of the functor $\PT$ \wrt the action
of $W$ on $\DBLHAH{\Rs{},\h'}{}$ defined in \ref{ss:ext-holo-alg}. 
Define the obstruction 
\[
\aw\colon W\to\mathsf{Hom_{grpd}}(\Poid{(\sfX; Wx_0)},\DCPHAH{\Rs{},\h'})
\quad\mbox{by}\quad
\aw_w(\gamma)=\PT(\gamma)^{-1}\cdot w^{-1}\PT(w\gamma)
\]
for $w\in W$ and $\gamma\in\Poid{(\sfX; Wx_0)}$.
The following shows that $\aw_{w}(\gamma)$ takes values in the abelian group
$\exphp\subset\DCPHAH{\Rs{},\h'}$.

\begin{lemma}\label{le:abelian obstruction}
For any $\gamma\in\Poid{(\sfX; Wx_0)}$ and $w\in W,$ $\aw_w(\gamma)\in\exphp$.
\end{lemma}
\begin{pf}
$w^{-1}(\PT(w\gamma))=w^*\PT(\gamma)$ is the parallel transport 
of the connection
	\[w^*\nablat=\nablat- A_w 
	\qquad\text{where}\qquad
	A_w=\sum_{\alpha\in\Rp\cap w^{-1}\Rm}
	\frac{d\alpha}{\alpha}\hinv{\alpha}
	\]
where the	sum involves only real roots, since the set of positive
	imaginary roots is $W$--invariant \cite[Prop.~5.2]{Ka}.
Since $\nablat$ and $A_w$
	commute, $\aw_{w}(\gamma)$ is the parallel transport along $\gamma$ of the abelian
	connection 
	\begin{equation}\label{eq:defect-connection}
		\nablat^{\scsop{ab},w}= d- A_w
	\end{equation}
	and therefore takes values in $\exphp$ since $\hinv{\alpha}=\nu^{-1}(\alpha)\in\h'$. 		
\end{pf}

\subsection{Restoring equivariance}

Let $\sfM'$ be the abelian group defined by
\[\sfM'=\mathsf{Hom_{grpd}}(\Poid{(\sfX; Wx_0)},\exphp)\]
and consider the action of $W$ on $\sfM'$ given by $(w\cdot f)(\gamma)=w(f(w^{-1}\gamma))$.

\begin{proposition}\label{pr:cocycle}
The following holds
\begin{enumerate}\itemsep0.3cm
\item $\aw=\{\aw_w\}_{w\in W}$ is a $1$--cocycle for $W$ with values in $\sfM'$, that is satisfies
\begin{equation}\label{eq:W-cocycle}
\aw_{vw}=(w^{-1}\cdot\aw_{v})\aw_{w}
\end{equation}
\item Assume that $\aw=d\bw$ for some $\bw\in\sfM'$, where $d\bw_w=\bw(w^{-1}\bw)^{-1}$.
Then, there is a functor
\[\PB:W\ltimes\Poid{(\sfX; Wx_0)}\to W\ltimes\tOp\]
which is uniquely defined by
\[ w_x\to w\aand \gamma\to\PT(\gamma)\cdot\bw(\gamma)\]  
\end{enumerate}
\end{proposition}
\begin{pf}
(1) By Lemma \ref{le:abelian obstruction}, $\aw_{w}$ takes values in $\exphp$ and
satisfies $\aw_{w}(\gamma'\circ\gamma)=\aw_{w}(\gamma')\aw_{w}(\gamma)$ since
it is the monodromy of the connection \eqref{eq:defect-connection}. Moreover, for any
$v,w\in W$, and $\gamma$ in $\Poid{(\sfX; Wx_0)}$
\begin{align*}
\aw_{vw}(\gamma)=&\PT(\gamma)^{-1} w^{-1}v^{-1}\PT(vw\gamma)\\
=&\PT(\gamma)^{-1}w^{-1}\left(\PT(w\gamma)\right)w^{-1}\left(\aw_v(w\gamma)\right) \\
=&\aw_{w}(\gamma)w^{-1}(\aw_{v}(w\gamma))
\end{align*}
as claimed.

(2) The restriction of $\PB$ is a functor $\Poid{(\sfX; Wx_0)}\to\tOp$ for any $\bw\in
\sfM'$ since $\exphp$ lies in the center of $\DCPHAH{\Rs{},\h'}$. Moreover, it is
$W$--equivariant if and only if $d\bw=\aw$ since, for any $\gamma\in\Poid{(\sfX;
Wx_0)}$ and $w\in W$,
\[w^{-1}(\PB(w\gamma))=
w^{-1}(\PT(w\gamma)\cdot\bw(w\gamma))=
\PB(\gamma)(\aw_w(\gamma) d\bw_w(\gamma)^{-1})\]
\end{pf}

\subsection{Natural transformations}\label{ss:nat tran} 

Let $\bw,\bw'\in\sfM'$ be such that $d\bw=\aw=d\bw'$, and
\[\PB,\PB':W\ltimes\Poid{(\sfX; Wx_0)}\to W\ltimes\tOp\]
be the corresponding functors. We shall consider natural isomorphisms $\PB
\Rightarrow\PBp$ which are given by a collection of elements $c=\{c_x\}_{x\in
Wx_0}$, with $c_x\in\exphp\subset W\ltimes\tOp$. The relation
\begin{equation}\label{eq:cob-nat-trans}
c_y\PB(w,\gamma)=\PT_{\bw'}(w,\gamma)c_x
\end{equation}
for any $(w,\gamma):x\to y$ implies in particular that $c_{wx_0}=w(c_{x_0})$, and
therefore that $c$ is uniquely determined by $c_{x_0}\in\exphp$.

\begin{proposition}\label{pr:nat iso}
	An element $\cw\in\exphp$ determines an isomorphism 
	$\PT_{\bw}\Rightarrow\PT_{\bw'}$ if and only if $\bw'=\bw\cdot \dsing\cw$, where
	$\dsing\cw\in\sfM'$ is given by\footnote{Note that $d(\dsing\cw)=1$.} 
	\[\dsing\cw(\gamma)=w_2(\cw)w_1(\cw)^{-1}\]
	for any $\gamma:w_1x_0\to w_2x_0$.
\end{proposition}

\begin{pf}
	If $\cw\in\exphp$ determines an isomorphism $c\colon \PT_{\bw}\Rightarrow\PT_{\bw'}$, then, for $(\id,\gamma):x\to y$,
	the relation \eqref{eq:cob-nat-trans} gives $\bw'(\gamma)=\bw(\gamma)c_yc_x^{-1}$.
	Thus, for any $\gamma:w_1x_0\to w_2x_0$, one has
	$\bw'(\gamma)=\bw(\gamma)w_2(\cw)w_1(\cw)^{-1}$, \ie $\bw'=\bw\cdot \dsing\cw$. The converse is clear.
\end{pf}

\noindent\remark\;
The assignment $(w,\gamma)\mapsto\aw_w(\gamma)$ can equivalently be thought
of as a 2--cocycle on the groupoid $\WPoid{(\sfX; Wx_0)}$ with values in $\exphp$,
which is normalised to vanish on $W$ and $\Poid{(\sfX; Wx_0)}$. Similarly, $\bw$
and $\cw$ can be thought of as $1$ and $0$--cocycles, respectively. Then, the result
above is simply stating that the equivalence classes of the representations $\PT_{\bw}$
for $\bw\in\sfM'$ such that $d\bw=\aw$ are controlled by the first
cohomology group.

\subsection{Presentation of $\Poid{(\sfX; Wx_0)}$}\label{ss:vdL}

Assume henceforth that the basepoint $x_0$ lies in $\imath\C$. For each $i\in\bfI$,
let $\gamma_i$ be a fixed elementary path in $\sfX$ from $x_0$ to $s_i(x_0)$ {\em above}
the wall $\alpha_i=0$, \ie is such that its real part lies in the half--space $\{\alpha_i\geqslant 0\}$.
For any $i\in\bfI$ and $w\in W$ set
\[\gamma_{w,i}=w\gamma_i: wx_0\longrightarrow ws_ix_0\]
Note that $\Poid{(\sfX; Wx_0)}$ is generated by $\{\gamma_{w,i}\}_{w\in W,i\in\bfI}$. For
instance, the elementary path from $x_0$ to $s_ix_0$ {\em below} the wall $\alpha_i=0$
is given by $\gamma_{s_i, i}^{-1}$. 

We shall consider the following class of paths depending upon the choice of a reduced
expression of a given element $v$ in $W$, which we refer to as {\em minimal Tits paths}.  
Let $\ul{s}=(s_{i_1},\dots, s_{i_{\ell}})$ be a reduced expression of $v$, set $\trv{}{k}=
s_{i_1}\cdots s_{i_{k}}$, $1\leqslant k\leqslant \ell$, and denote by $\gamma_{\ul{s}}$
the path
\begin{equation}\label{eq:steps}
	\begin{tikzcd}[column sep=1.3cm]
	x_0 \arrow[r, "\gamma_{i_1}"] &
	\trv{}{1}x_0 \arrow[r, "\gamma_{\trv{}{1}, i_2}"] &
	\trv{}{2}x_0\arrow[r, "\gamma_{\trv{}{2}, i_3}"] &
	\cdots
	\trv{}{\ell-1}x_0 \arrow[r, "\gamma_{\trv{}{\ell-1},i_\ell}"] &
	\trv{}{\ell}x_0
	\end{tikzcd}
\end{equation}
Then, a minimal Tits path
is an element of the form $\gamma_{w,\ul{s}}= w\gamma_{\ul{s}}$, where $w\in W$
and $\ul{s}$ is a reduced expression of some $v\in W$. 
Note that two minimal Tits paths $\gamma_{w,\ul{s}}$ and $\gamma_{w',\ul{s}'}$ have the same endpoints 
if and only if $w=w'$ and $\ul{s}$, $\ul{s}'$ are reduced expressions of the same element $v$. The following
is due to van der Lek \cite{vdL}.

\begin{theorem}
The homotopy relations in $\Poid{(\sfX; Wx_0)}$ are generated by the equivalence relation identifying
minimal Tits paths with the same endpoints.
\end{theorem}

\begin{pf}
For the reader's convenience, we provide a brief account of van der Lek's proof. The latter hinges on the
combinatorial description of $\Poid{(\sfX; Wx_0)}$ in terms of {\em signed galleries} in the root
hyperplane arrangement (cf.~\cite[Thm.~I-4.10]{vdL}). A {\em Tits gallery} is a sequence of chambers $\C_0,\C_1,\dots, \C_\ell$
such that, for any $i=0,\dots, \ell-1$, $\C_i$ and $\C_{i+1}$ are separated by a single hyperplane
$M_i$. Let $\H_i^+,\H_i^-$ be the halfspaces determined by $M_i$, with $\C_i\subset \H_i^+$
and $\C_{i+1}\subset\H_i^-$. Then, a {\em signed gallery} is a sequence $\C_0^{\epsilon_1}\C_1^{\epsilon_2}\cdots\C_{\ell-1}^{\epsilon_{\ell}}\C_{\ell}$,
where $\C_0,\C_1,\dots, \C_\ell$ is a Tits gallery and the signs $\epsilon_i\in\{\pm\}$ denote a choice of the
half--spaces $\H_i^{\pm}$.

Chambers and signed galleries are interpreted, respectively, as the objects and the morphisms of the category $\mathbf{Gal}(\sfX;\Delta)$. 
Note that $\mathbf{Gal}(\sfX;\Delta)$ is naturally endowed with an action of $W$\footnote{Indeed, note that, given a chamber $\C$ with a wall $M$, 
if $\H_M(\C)$ denotes the half--space determined by $M$ and containing $\C$, one has $w(\H_M(\C))=\H_{w(M)}(w(\C))$ for any $w\in W$. 
Therefore, $W$ preserves the signs of the signed galleries.}.
Then, $\Poid{(\sfX; Wx_0)}$ is isomorphic to the quotient of $\mathbf{Gal}(\sfX;\Delta)$ 
by the following equivalence relations
\begin{itemize}\itemsep0.25cm
	\item {\bf Cancel relations.} For any two adjacent chambers $\C_0,\C_1$, the signed gallery $\C_0^{\pm}\C_1^{\mp}\C_0$
	is equivalent to the gallery $\C_0$.
	\item {\bf Flip relations.} Let $\ul{\C}=\C_0^{\epsilon_1}\C_1^{\epsilon_2}\cdots\C_{\ell-1}^{\epsilon_{\ell}}\C_{\ell}$ and 
	$\ul{\D}=\D_0^{\epsilon'_1}\D_1^{\epsilon'_2}\cdots\D_{\ell-1}^{\epsilon'_{\ell}}\D_{\ell}$ be two minimal signed galleries 
	such that $\C_0=\D_0$ and $\C_{\ell}=\D_{\ell}$ are opposite chambers with respect to a codimension 2 facet.
	Then, $\ul{\C}$ and $\ul{\D}$ are equivalent.
\end{itemize}
Note that, by \cite[Rmk.~I-5.3 and 5.4]{vdL}, the sequences of signs appearing in the flip relations must satisfy 
$\epsilon_i=\epsilon'_{k-i+1}$ and admit at most one change of sign. It follows that is enough
to consider only flip relations with $\epsilon_i=+$ for any $i$. Moreover, the minimal Tits galleries have a 
simple combinatorial description (cf.~\cite[Prop.~II-2.16]{vdL}). Let $\C_0$ be a chamber and $w_0\in W$
the unique element such that $\C_0=w_0\C$. Then, for any ${w}\in W$, the minimal Tits gallery from 
$\C_0$ to ${w}\C_0$ are in bijection with the reduced expressions of $w_0^{-1}{w}w_0$, \ie, 
if $\ul{s}=(s_{i_1},s_{i_2},\dots,s_{i_{\ell}})$ is a reduced expression of $w_0^{-1}{w}w_0$, the sequence
\[
\C_0, \C_1= w_0w_1w_0^{-1}\C_0,\cdots, \C_{\ell}= w_0w_{\ell}w_0^{-1}\C_0={w}\C_0
\]
where $w_r= s_{i_1}s_{i_2}\dots s_{i_{r}}$, is a minimal Tits gallery. Clearly, every minimal Tits gallery 
from $\C_0$ to $\wt{w}\C_0$ arises in this way and it is the image through $w_0$ of
a minimal Tits gallery starting in the fundamental chamber $\C$.

Finally, the isomorphism between the two groupoids is induced by a $W$--equivariant full functor 
$\phi:\mathbf{Gal}(\sfX;\Delta)\to\Poid{(\sfX; Wx_0)}$ mapping the fundamental chamber to $x_0$ 
and the step one galleries $\C^{\pm}s_i\C$ to the elementary paths $\gamma_{i}^{\pm}$, 
where $\gamma_i^+=\gamma_i$ and $\gamma_i^{-}=\gamma_{s_i,i}^{-1}$ (cf.~\cite[Rmk.~II-3.10]{vdL}).
The result follows.
\end{pf}

\subsection{Normalised cochains}\label{ss:norm-bw}

\newcommand {\bfa}{\mathbf{a}}
\newcommand {\realr}{{\scriptstyle{\operatorname{re}}}}

Let $\sfM\supset\sfM'$ be the abelian group given by
\[\sfM=\mathsf{Hom_{grpd}}(\Poid{(\sfX; Wx_0)},\exph)\]

\begin{lemma}\label{le:db=1}\hfill
\begin{enumerate}
		\item Let $\bw\in\sfM$ be such that $d\bw=1$. Then, $\bw$ is uniquely determined by the values 
		$\bw(\gamma_i)\in\exph$, $i\in\bfI$.
		\item For any collection of complex numbers $\bfa=\{a_i\}_{i\in\bfI}$, there is a unique $\bw_\bfa\in\sfM'$ such
		that $d\bw_\bfa=1$ and $\bw_\bfa(\gamma_i)=\exp(a_i \hinv{\alpha_i})$.
\end{enumerate}
\end{lemma}
\begin{pf}
	(1) follows from $\bw(\gamma'\circ\gamma)=\bw(\gamma')\bw(\gamma)$, and the fact that
	the relation $d\bw=1$ reads $\bw(w\gamma)=w(\bw(\gamma))$.
	
	(2) As above, the relation $d\bw_\bfa=1$ implies that, for any $w\in W$ and $i\in\bfI$,
	$\bw_\bfa(\gamma_{w,i})=w(\bw_\bfa(\gamma_i))$. It is therefore sufficient to show that the
	assignment $\gamma_{w,i}\mapsto\exp(a_i\hinv{w\alpha_i})$ is constant on minimal
	Tits paths with the same endpoints.
	
Let $w,v\in W$, $\ul{s}=(s_{i_1},\dots, s_{i_\ell})$ a reduced expression of $v$,
	and set $\trv{ }{k}=s_{i_1}\cdots s_{i_{k}}$, $k\leqslant \ell$. Then,
	\[
	\bw_{\bfa}(\gamma_{w,\ul{s}})=
	\prod_{k=1}^\ell\bw_{\bfa}(\gamma_{w\trv{ }{k-1},i_k})
	=
	w\prod_{k=1}^\ell\exp(a_{i_k}\hinv{\trv{}{k-1}\alpha_{i_k}})
	=
	w\prod_{\substack{\alpha>0\\ v^{-1}\alpha<0}}\exp(a_{i_{\alpha},\ul{s}}\hinv{\alpha})
	\]
where $i_{\alpha},\ul{s}\in\bfI$ is the unique index $k$ such that $\alpha=\trv{}{k-1}\alpha_{i_k}$.
To check that this is independent of the reduced decomposition of $v$, it is sufficient to consider
the case when $v$ is the longest element in a rank 2 Weyl group. If $W$ is of Coxeter type $\sfA
_1\times\sfA_1$, $\sfB_2$ or $\sfG_2$, this follows because a given positive root $\alpha$ is $W
$--conjugate to a unique simple root $\alpha_i$, namely the one of the same length of $\alpha$.
If $W$ is of type $\sfA_2$, with $v=s_1s_2s_1=s_2s_1s_2$, the independence on the reduced
decomposition amounts to the identity
\[a_1\hinv{\alpha_1}+
		a_2\hinv{\alpha_1+\alpha_2}+a_1\hinv{\alpha_2}=a_2\hinv{\alpha_2}+
		a_1\hinv{\alpha_1+\alpha_2}+a_2\hinv{\alpha_1}\]
which clearly holds. The uniqueness of $\bw_\bfa$ follows from (1).
\end{pf}

\subsection{Triviality of the obstruction $\aw$} 

\begin{theorem}\label{th:coboundary}
There is a unique $\bw\in\sfM'$ such that
\[\aw=d\bw\aand\bw(\gamma_i)=1\]
for any $i\in\bfI$.
\end{theorem}
\begin{pf}
	The uniqueness of $\bw$ follows from Lemma \ref{le:db=1}. The relation $d\bw=\aw$ together
	with the normalisation of $\bw$ are equivalent to the requirement that, for any $w\in W$ and 
	$i\in\bfI$, $\bw(\gamma_{w,i})=w(\aw_w(\gamma_i))^{-1}$. By \ref{ss:vdL}, it is therefore sufficient
	to show that the assignment $\gamma_{w,i}\mapsto w(\aw_w(\gamma_i))^{-1}$ is constant on
	minimal Tits paths with the same endpoints.
	
	Let $w,v\in W$, $\ul{s}=(s_{i_1},\dots, s_{i_l})$ a reduced expression of $v$, set $\trv{ }{k}
	=s_{i_1}\cdots s_{i_{k}}$, $k\leqslant \ell$, and retain the notation used in \eqref{eq:steps}.
	Note that, since $\aw$ satisfies the cocycle identity \eqref{eq:W-cocycle}, one has
	\begin{align*}
		\prod_{k=1}^{\ell}w\trv{ }{k-1}(\aw_{w\trv{ }{k-1}}(\gamma_{i_k}))^{-1}&=\prod_{k=1}^{\ell}w(\aw_{w}(\gamma_{\trv{ }{k-1},i_k}))^{-1}\cdot w\trv{ }{k-1}(\aw_{\trv{ }{k-1}}(\gamma_{i_k}))^{-1}\\
		&=w(\aw_{w}(\gamma_{\ul{s}}))^{-1}\cdot w\left(\prod_{k=1}^{\ell}\trv{ }{k-1}(\aw_{\trv{ }{k-1}}(\gamma_{i_k}))\right)^{-1}
	\end{align*}
	where the first equality follows from $d\aw=1$.
	Since $\aw_w$ is the parallel transport of the abelian connection  \eqref{eq:defect-connection},
	$w(\aw_{w}(\gamma_{\ul{s}}))$ only depends on the endpoints of $\gamma_{\ul{s}}$, and is
	therefore independent of the reduced decomposition of $v$. For the second factor, we can
	ignore $w$ and consider
		\[\prod_{k=1}^{\ell}\trv{ }{k-1}(\aw_{\trv{ }{k-1}}(\gamma_{i_k}))^{-1}
		=
		\prod_{k=1}^{\ell}
		v_{k-1}\left(\prod_{\alpha\in I_{k-1}} 
		s_{i_k} \alpha(x_0)^{\hinv{\alpha}}\cdot
		\alpha(x_0)^{-\hinv{\alpha}}	
		\right)
		\]
		where $I_{k-1}^{ }=\{\alpha>0\;|\trv{ }{k-1} \alpha<0\}$. Since $s_{i_k}I_{k-1}^{ }=I_k^{ }\setminus\{\alpha_{i_k}\}$, 
		this is equal to 
\begin{multline*}
		\prod_{k=1}^\ell\prod_{\alpha\in I_k}\alpha(x_0)^{\hinv{v_k\alpha}}\cdot \alpha_{i_k}(x_0)^{-\hinv{v_k\alpha_{i_k}}}
		\cdot\prod_{k=1}^\ell\prod_{\alpha\in I_{k-1}}\alpha(x_0)^{-\hinv{v_{k-1}\alpha}}\\
		=
		\prod_{\alpha\in I_\ell}\alpha(x_0)^{\hinv{v\alpha}}
		\cdot
		\prod_{k=1}^\ell\alpha_{i_k}(x_0)^{-\hinv{v_k\alpha_{i_k}}}	
\end{multline*}
It therefore remains to show that 
\[
A=
\prod_{k=1}^\ell\alpha_{i_k}(x_0)^{\hinv{\trv{}{k-1}\alpha_{i_k}}}
=
\prod_{\alpha\in I_\ell}\alpha_{i_{\alpha,\ul{s}}}(x_0)^{\hinv{\alpha}}
\]
is independent of the reduced expression of $v$, where for each $\alpha\in I_\ell$, $i_{\alpha,\ul{s}}\in\bfI$
is the unique index $k$ such that $\alpha=\trv{}{k-1}\alpha_{i_k}$. As in the proof of part (2) of Lemma \ref
{le:db=1}, this reduces to the case when $W$ is of type $\sfA_2$, and $v$ is the longest element of $W$.
In that case,$\sfs=(1,2,1), \sfs'=(2,1,2)$, and\footnote{Theorem \ref{th:coboundary} and Lemma \ref{le:db=1}
reduce to the same verification because they are special cases of the more general statement that, for any
collection of complex numbers $\bfa=\{a_i\}_{i\in\bfI}$, there is a (unique) $\bw_\bfa\in\sfM'$ such that $d\bw
_\bfa=\aw$ and $\bw_\bfa(\gamma_i)=\exp(a_i \hinv{\alpha_i})$.}
			\[
			A_{\sfs}=
			\alpha_1(x_0)^{\hinv{\alpha_1}}
			\alpha_2(x_0)^{\hinv{\alpha_1+\alpha_2}}
			\alpha_1(x_0)^{\hinv{\alpha_2}}
			=
			\alpha_2(x_0)^{\hinv{\alpha_2}}
			\alpha_1(x_0)^{\hinv{\alpha_1+\alpha_2}}
			\alpha_2(x_0)^{\hinv{\alpha_1}}
			= A_{\sfs'}
			\]
\end{pf}

\subsection{Monodromy representations of $\WPoid{(\sfX; Wx_0)}$}\label{ss:holo-Cox-correction}

Let $\bfa=\{a_i\}_{i\in\bfI}$ be a collection of complex numbers, and $\bw_{\bfa},\bwc\in
\sfM'$ the elements determined by Lemma~\ref{le:db=1} and Theorem~\ref{th:coboundary}
respectively.

Since $d\bw_{\bfa}=1$ and $d\bwc=\aw$, it follows from Proposition~\ref{pr:cocycle} (2)
that there is a functor $\PT_\bfa:\WPoid{(\sfX; Wx_0)}\to W\ltimes\tOp$ which is the
identity on the morphisms $\{w_x\}$, and maps a path $\gamma\in\Poid{(\sfX;Wx_0)}$ to
\[\PT_{\bfa}(\gamma)=\PT(\gamma)\cdot\bwc(\gamma)\cdot\bw_{\bfa}(\gamma)\]

\begin{proposition}\label{prop:pt-rep}
	For any collections $\bfa,\bfa'$, regard $\PT_\bfa,\PT_{\bfa'}$ as functors
	\[\WPoid{(\sfX; Wx_0)}\to W\ltimes\tO\]
	via the embedding $\tOp\subset\tO$. Then, there is an isomorphism $\PT
	_\bfa\Rightarrow\PT_{\bfa'}$ determined by an element $\cw\in\exph$ (cf.
	\ref{ss:nat tran}).
\end{proposition}
\begin{pf}
Since $\bwc(\gamma_i)=1$ for any $i\in\bfI$, Proposition~\ref{pr:nat iso} implies
that it is enough to find $\cw\in\exph$ such that $\bw_{\bfa'}(\gamma_i)=\bw_{\bfa}(\gamma_i)\cdot
s_i(\cw)\cdot\cw^{-1}$ for any $i\in\bfI$. If we assume $\cw=\exp(c)$ for some $c\in{\h}$,
this condition reduces to the set of equations $\alpha_i(c)=a_i-a_i'$, $i\in\bfI$, which always
possess a solution since $\{\alpha_i\}$ are linearly independent in $\h^*$. 
\end{pf}

\subsection{Equivariance via resummation}\label{ss:resummation}

We describe below an alternative way of restoring the equivariance of the universal
Casimir connection $\nablat$ by resumming the formal $\h'$--valued 1--form on $\sfX$
given by
\[\whAt=
\half{1}
\sum_{\alpha\in\Rs{+}\realre}
\frac{d\alpha}{\alpha}\cdot\hinv{\alpha}\]

\begin{definition}
A {\it resummation} of $\whAt$ is a closed, holomorphic $1$--form $\At$ on $\sfX$ with values
in $\h\supset\h'$ such that 
\begin{itemize}
\item For any $w\in W$,
\[w^*\At=\At-
\negthickspace\negthickspace\negthickspace\negthickspace
\sum_{\alpha\in\Rp\cap w^{-1}\Rm} \frac{d\alpha}{\alpha}\cdot \hinv{\alpha}\]
so that $\nablat-\At$ is an integrable, $W$--equivariant connection.
\item For any $i\in\bfI$, $\At$ has a logarithmic singularity on the hyperplane $\alpha_i=0$,
with residue $\hinv{\alpha_i}/2$.
\end{itemize}
\end{definition}

The existence of a resummation is clear if $\g$ is finite--dimensional, and is proved in
Appendix~\ref{s:coda} for $\g$ affine.\footnote{In that case, a resummation only exists
if it is taken with values in $\h$ rather than $\h'$.}

\begin{proposition}
Let $\At$ be a resummation of $\whAt$.
\begin{itemize}
\item The parallel transport of the connection $\nablat-\At$ is $W$--equivariant,
and given on generators by 
\[ \PT_{\nablat-\At}(\gamma_i)
=
\PT_{\nablat}(\gamma_i)\cdot\exp(a_i\cdot\hinv{\alpha_i}) \]
for some complex numbers $\{a_i\}$.
\item The corresponding functor
\[\PT_{\nablat-\At}:\WPoid{(\sfX; Wx_0)}\to W\ltimes\tOp\]
coincides with the functor $\PT_\bfa$ given by Proposition \ref{prop:pt-rep}, with $\bfa=
\{a_i\}_{i\in\bfI}$.
\end{itemize}
\end{proposition}
\begin{pf}
The $W$--equivariance of $\PT_{\nablat-\At}$ follows from that of ${\nablat-\At}$. Moreover, if
$\bw\in\sfM$ is the parallel transport of the connection $d-\At$ then, for any $\gamma\in\Pi
_1(\sfX;Wx_0)$, $\PT_{\nablat-\At}(\gamma)$ is equal to $\PT_{\nablat}(\gamma)\cdot\bw
(\gamma)$. It follows that $d\bw=\aw$, and therefore that $\bw$ is uniquely determined by
the values $\bw(\gamma_i)$ by Lemma \ref{le:db=1}. In particular, $\bw$ coincides with
the cochain $\bwc(\gamma)\cdot\bw_{\bfb}(\gamma)$ of Proposition \ref{prop:pt-rep} provided
that, for any $i\in\bfI$, $\bw(\gamma_i)$ is of the form $\exp(a_i\cdot\hinv{\alpha_i})$ for
some $a_i\in\IC$.

Let $\Ati$ be the $\h$--valued 1--form given by $\Ati=\At-\frac{1}{2}d\log\alpha_i\cdot
\hinv{\alpha_i}$. Clearly, $\bw(\gamma_i)=\bw_i(\gamma_i)\cdot\exp(\iota\pi\hinv{\alpha
_i}/2)$, where $\bw_i$ is the parallel transport of $d-\Ati$. Note that $\Ati$ is equivariant
under $s_i$ and regular on $\Ker\alpha_i$. Let $f$ be an $\exp(\h)$--valued 
fundamental solution of $df=\Ati f $. It suffices to show that $g(x)=f(s_ix)\cdot f(x)^{-1}$
takes values in $\exp(\IC\hinv{\alpha_i})$. $g$ satisfies
\[ dg(x)=\left(\Ati(s_ix)-\Ati(x)\right)g(x)=(s_i-i)(\Ati(x))g(x) \]
where the second identity follows from the $s_i$--equivariance of $\Ati$. Moreover,
if $x$ lies on $\Ker\alpha_i$, then $g(x)=1$, from which the conclusion follows.
\end{pf}

\noindent
\remark\,x
	In the following, we shall need to further adjust the monodromy representations $\PT_{\bfa}$
	by elements in $S^2\h$. More precisely, one checks easily that, for any $W$--invariant
	functions $\bfa,\bfb$, there is a unique solution of $d\bw=1$ such that $\bw(\gamma_i)=\exp(a_i\hinv{\root{i}}+b_i\hinv{\root{i}}^2)$, 
	yielding a monodromy representation $\PT_{\bfa,\bfb}$. 
	Note that $\PT_{\bfa,\bfb}$ and $\PT_{\bfa',\bfb'}$
	are equivalent if and only if $\bfb=\bfb'$.
	
\subsection{Monodromy representations of $\Br{W}$}\label{ss:mon rep}

Denote by $\PT_\bfa^{x_0}$ the composition
\[\begin{tikzcd}[column sep=.7cm]
\Br{W}=
\Pi_1(\sfX/W;[x_0]) \arrow[r,"P_{x_0}^{-1}"] &
W\ltimes\Poid(\sfX;Wx_0)
\arrow[r,"\PT_\bfa"] &
W\ltimes\tOp
\end{tikzcd}\]
where $P_{x_0}$ is the restriction of the equivalence \eqref{eq:projection functor}
to automorphisms of $x_0$, and $\PT_\bfa$ the functor given by Proposition \ref{prop:pt-rep}.

The homomorphism $\PT_\bfa^{x_0}$ is given by 
\[\PT_\bfa^{x_0}(\gamma)=
(w,\PT_\bfa(\wt{\gamma}))\]
where $\wt{\gamma}:[0,1]\to\sfX$ is the unique lift of $\gamma$ such that $\wt{\gamma}
(0)=x_0$, and $w\in W$ the unique element such that $\wt{\gamma}(1)=w^{-1}\wt{\gamma}
(0)$. Moreover, any representation $\rho:W\ltimes\DBLHAH{\Rs{},\h'}{}\to\sfEnd{V}$
gives rise to an action of the braid group $\Br{W}$ on $V$.

\subsection{Monodromy representations of $\Br{W}$ on category $\Oint$}
\label{ss:twisting-bundle}

The above mechanism is not appropriate to construct representations of $\Br
{W}$ on integrable category $\O$ modules, since $W$ does not act on them. 
To remedy this, we introduce the following

\begin{definition}\label{def:ext-holo-alg}
	The \emph{extended holonomy algebra} is the semidirect product
	$\Br{W}\ltimes\DCPHAH{\Rs{},\h}$, where the braid group $\Br{W}$ 
	acts on $\DCPHAH{\Rs{},\h}$ through the morphism $\Br{W}\to W$.
\end{definition}

Then, we simply lift
\[\PT_\bfa^{x_0}:\Br{W}\to W\ltimes\DBLHAH{\Rs{},\h'}{}
\qquad\text{to}\qquad
\wt{\PT}_\bfa^{x_0}:\Br{W}\to\Br{W}\ltimes\DBLHAH{\Rs{},\h'}{}\]
and use instead a representation of $\Br{W}\ltimes\DBLHAH{\Rs{},\h'}{}$.
This yields the following.

\begin{theorem}
Let $\bfa=\{a_i\}_{i\in\bfI}$ be a collection of complex numbers.
\begin{enumerate}\itemsep0.25cm
\item The parallel transport of the universal Casimir connection $\nablat$
gives rise to a homomorphism
\[\wt{\PT}_\bfa^{x_0}:\Br{W}\to\Br{W}\ltimes \tOp
\qquad\text{given by}\qquad
b\longrightarrow(b,\PT_\bfa(\wt{b}))\]
where $\wt{b}\in\Poid(X;Wx_0)$ is the unique lift of $b$ through $x_0$, and 
\[\PT_\bfa(\wt{b})=\PT(\wt{b})\cdot\bwc(\wt{b})\cdot\bw_{\bfa}(\wt{b})\]
is defined in \ref{ss:holo-Cox-correction}.
\item
Let $V$ be a category $\O$ integrable $\g$--module, equipped with the action of $\Br{W}$ 
given by triple exponentials (cf.~Remark \ref{ss:flatness}). The parallel transport of the Casimir connection
$\nablak$ gives rise to a homomorphism $\wt{\PT}_{\bfa,V}^{x_0}\colon\Br{W}\to\operatorname{GL}(V)$
given on generators by
\begin{equation}\label{eq:monodromy-rep} \wt{\PT}_{\bfa,V}^{x_0}(\topS{i})=\wt{s}_i\cdot\PT({\gamma_{i}})\cdot\exp(a_i\hinv{\alpha_i})
\end{equation}
\end{enumerate}	
\end{theorem}

\subsection{Twisting}
	Let $A$ be a resummation of the formal 1--form $\wh{A}$ (cf. \ref{ss:resummation}).
	The representation $\wt{\PT}_{\bfa,V}^{x_0}$ can be equivalently obtained 
	from the following topological construction, twisting the $\Br{W}$--equivariant vector bundle
	$\IV$ into a $W$--equivariant vector bundle $\wt{\IV}$ on $\sfX$ (cf.~\ref{ss:casimir-conn}). 
	Since $\wt{W}$ is a quotient of the braid group $\Br{W}$, the latter acts on the vector bundle $p^*\IV$ on $\wt{\sfX}$. By definition, $\wt{\IV}$ is the quotient $p^*\IV/\P_W$, where $\P_W$ 
	is the pure braid group corresponding to $W$, and carries a residual action of $W=\Br{W}/\P_W$.
	As in Proposition~\ref{ss:resummation}, it follows that $\wt{\PT}_{\bfa,V}^{x_0}$ coincides with the representation induced through parallel transport by the twisted connection  
	on $\wt{\IV}$.


\section{Diagrams and nested sets}\label{s:diagrams}

In this section, we review the definition of (relative) nested sets on a diagram
$\dgr$ (cf.~\cite{DCP,vtl-4}). We follow the exposition given in \cite[Sec.~2]
{ATL1-2}.

\subsection{Nested sets on diagrams}\label{ss:ns}

A {\it diagram} is an undirected graph $\dgr$ with no multiple edges or
loops. A {\it subdiagram} $B\subseteq\dgr$ is a full subgraph of $\dgr$,
that is, a graph consisting of a (possibly empty) subset of vertices of
$\dgr$, together with all edges of $\dgr$ joining any two elements of it.
We will often identify $B$ and its set of vertices, and denote by $|B|$
the cardinality of the latter.

Two subdiagrams $B_1,B_2\subseteq\dgr$ are {\it orthogonal} if they
have no vertices in common, and no two vertices $i\in B_1$, $j\in B
_2$ are joined by an edge in $\dgr$. We denote by $B_1\sqcup B_2$
the disjoint union of orthogonal subdiagrams. Two subdiagrams $B_1,
B_2\subseteq\dgr$ are {\it compatible} if either one contains the other
or they are orthogonal.

A {\it nested set} on $\dgr$ is a collection $\H$ of pairwise compatible,
connected subdiagrams of $\dgr$ which contains the empty subdiagram
and $\cc{\dgr}$, where $\cc{\dgr}$ denotes the set of connected components 
of $\dgr$. It is easy to see that the cardinality of any maximal nested set
on $\dgr$ is equal to $|\dgr|+1$. 

Let $\Ns{\dgr}$ be the set of nested sets on $\dgr$, and $\Mns{\dgr}$ that
of maximal nested sets. Every (maximal) nested set $\H$ on $\dgr$
is uniquely determined by a collection $\{\H_i\}_{i=1}^r$ of (maximal)
nested sets on the connected components $\dgr_i$ of $\dgr$. We therefore
obtain canonical identifications
\[\Ns{\dgr}=\prod_{i=1}^r \Ns{\dgr_i}\qquad\text{and}
\qquad
\Mns{\dgr}=\prod_{i=1}^r\Mns{\dgr_i}.\]

\subsection{Relative nested sets}\label{ss:rel-ns}

If $B'\subseteq B\subseteq\dgr$ are two subdiagrams of $\dgr$, a nested
set on $B$ {\it relative} to $B'$ is a collection of subdiagrams of $B$
which contains $\cc{B}$ and $\cc{B'}$, and in which every element is
compatible with, but not properly contained in any of the connected
components of $B'$.  We denote by $\Ns{B,B'}$ and $\Mns{B,B'}$
the collections of nested sets and maximal nested sets on $B$ relative
to $B'$. In particular, 
\[\Ns{B,\emptyset}=\Ns{B}\aand\Mns{B,\emptyset}=\Mns{B}\]

Relative nested sets are endowed with the following operations, which
preserve maximal nested sets.
\vspace{0.25cm}\begin{enumerate}\itemsep0.25cm
	\item {\bf Vertical union.}
	For any $B''\subseteq B' \subseteq B$, there is an embedding
	\[
	\cup:\Ns{B,B'}\times\Ns{B',B''}\to\Ns{B,B''},
	\]
	given by the union of nested sets. Its image is the collection $\Nsr{B,B''}
	{B'}\subseteq\Ns{B,B''}$ of relative nested sets which contain $\cc{B'}$.
	\item {\bf Orthogonal union.}
	For any $B_1'\subseteq B_1\perp B_2\supseteq B_2'$, there is
	a bijection
	\[\Ns{B_1,B_1'}\times\Ns{B_2,B_2'}\to\Ns{B_1\sqcup B_2,B'_1\sqcup B_2'},\] 
	mapping $(\H_1,\H_2)\mapsto\H_1\cup\H_2$.
\end{enumerate}

\subsection{Elementary sequences}\label{ss:elem-seq}

\begin{definition}
	\hfill
	\begin{enumerate}\itemsep0.25cm 
		\item \label{def:truncation}
		Let $B'\subseteq C'\subseteq C\subseteq B$, and $\F\in\Mns{B,B'}$ a maximal nested
		set such that $\cc{C'}, \cc{C}\subset\F$. The \emph{truncation of $\F$ at $(C,C')$}
		is the element of $\Mns{C,C'}$ defined by
		\[
		\trunc{\F}{C}{C'}=\left\{C''\in\F\left|\right. C''\subseteq C\;\mbox{and}\;
		\forall\,\wt{C}\in\cc{C'},\; C''\not\subset\wt{C}\right\}
		\]
		We set $\trunc{\F}{C}{}=\trunc{\F}{C}{B'}$ and
		$\trunc{\F}{}{C}=\trunc{\F}{B}{C}$.
		\item
		Let $B'\subseteq B$, and $\F,\G\in\Mns{B,B'}$. The \emph{support} and \emph{central support} 
		of the pair $(\F,\G)$ are the subdiagrams of $B$ defined by
		\begin{align*}
			\supp(\F,\G)&=\min_C\{B'\subseteq C\subseteq B\;|\; \cc{C}\subseteq\F\cap\G, \trunc{\F}{}{C}=\trunc{\G}{}{C}\}\\
			\zsupp(\F,\G)&=\max_C\{B'\subseteq C\subseteq B\;|\; \cc{C}\subseteq\F\cap\G, \trunc{\F}{C}{}=\trunc{\G}{C}{}\}
		\end{align*}
		\item  
		Two ordered pairs $(\F,\G)$, $(\F',\G')$  in $\Mns{B,B'}$ 
		are {\em equivalent} if 
		\[\F\setminus\G=\F'\setminus\G'
		\aand
		\G\setminus\F=\G'\setminus\F'\]
		If this is the case, then
		\[\supp(\F,\G)=\supp(\F',\G')
		\aand\zsupp(\F,\G)=\zsupp(\F',\G')\]
		\item
		An ordered pair $(\F,\G)$  in $\Mns{B,B'}$ is {\em elementary} if 
		$\F$ and $\G$ differ only by one element.
	\end{enumerate}
\end{definition}

We shall make use of the following result \cite[Prop. 3.26]{vtl-4}.\footnote{In \cite
	{vtl-4}, this result is proved only for elements in $\Mns{\dgr}$. However, it generalises
	immediately to the case of relative nested sets.}

\begin{proposition}
	\hfill
	\begin{enumerate}\itemsep0.25cm
		\item For any $B'\subseteq B$ and $\F,\G\in\Mns{B,B'}$, there is a sequence
		\[\F=\H_1,\dots, \H_l=\G\]
		in $\Mns{B,B'}$ and the following holds for any $i$
		\begin{itemize}
			\item $(\H_i,\H_{i+1})$ is an elementary pair
			\item $\F\cap\G\subseteq\H_i\cap\H_{i+1}$
			\item $\supp(\H_i,\H_{i+1})\subseteq\supp(\F,\G)$
			\item	For any component $C$ of $\zsupp(\F,\G)$, either
			\[C\perp\supp(\H_i,\H_{i+1})\qquad\text{or}\qquad C\subseteq\zsupp(\H_i,\H_{i+1})\]
		\end{itemize}
		\item If $(\F,\G)$, $(\F',\G')$ are equivalent pairs in $\Mns{B,B'}$, then 
		the corresponding elementary sequences
		\[
		\F=\H_1,\dots, \H_l=\G
		\aand
		\F'=\H'_1,\dots, \H'_m=\G'
		\]
		can be chosen such that $l=m$ and, for any $i=1,\dots, l-1$, 
		$(\H_i,\H_{i+1})$ is equivalent to $(\H'_{i}, \H'_{i+1})$.  
	\end{enumerate}
\end{proposition}


\section{Diagrammatic algebras}\label{s:diag-algebras}
We review in this section the notions of diagrammatic and bidiagrammatic algebras,
which are essential to the definition of a Coxeter algebra in Section~\ref{s:Cox-algebras}.

\subsection{Diagrammatic algebras}\label{ss:functorial-diag-alg}

Let $\dgr$ be a diagram. A \emph{diagrammatic} structure of type $\dgr$ on an algebra
$A$ is a collection of subalgebras $A _B\subseteq A$ indexed by subdiagrams of $\dgr$
which is compatible with nesting \ie such that $A_{B'}\subseteq A_{B}$ if $B'\subseteq B$,
and $[A_{B_1}, A_{B_2}]=0$ if $B_1\perp B_2$ \cite{vtl-4}. We formalise a slightly more
general version as follows \cite[Sec. 5]{ATL1-2}.

Let $\PD$ the category whose objects are the subdiagrams of $\dgr$, and morphisms
$B'\to B$ the inclusions $B'\subseteq B$. The union $\sqcup$ of orthogonal diagrams is a
(symmetric, strict) partial tensor product on $\PD$, with the empty diagram $\emptyset$ as
unit object.

Recall that a {lax} monoidal structure on a functor $F:\C\to\D$ between two monoidal
categories is the datum of a morphism $u: 1_{\D}\to F(1_{\C})$ and a natural transformation
$J: F(\cdot)\ten_{\D} F(\cdot)\rightarrow F(\cdot\ten_{\C}\cdot)$, which satisfies the
cocycle identity and is compatible with the unit objects through $u$. A monoidal
structure is a lax monoidal structure with $u$ and $J$ invertible.

\begin{definition}
	Let $\sfk$ be a commutative unital ring and $(\mathsf{Alg(k)}, \ten)$ the category of 
	$\sfk$--algebras, with monoidal structure given by the tensor product and $\sfk$ as
	unit object. 
	\begin{enumerate}\itemsep0.25cm
		\item A (lax) diagrammatic algebra is a (lax) monoidal functor $\PD\to\mathsf{Alg(k)}$.
		\item A morphism of (lax) diagrammatic algebras is a natural transformation
		of the corresponding (lax) monoidal functors.
	\end{enumerate}
\end{definition}

Note that for any lax monoidal functor $F:\PD\to\mathsf{Alg(k)}$ the morphism $u:\sfk\to
F(\emptyset)$ is the unit of $F(\emptyset)$.

\subsection{Alternative description of diagrammatic algebras}\label{ss:diag-alg}

The following gives a more concrete description of diagrammatic algebras \cite[Prop. 5.14]{ATL1-2}.

\begin{proposition}
	\hfill
	\begin{enumerate}\itemsep0.25cm
		\item 
		A lax diagrammatic algebra $\ACox{}$ is the same as the datum of
		\vspace{0.25cm}
		\begin{itemize}\itemsep0.25cm
			\item for any $B\subseteq\dgr$, a $\sfk$--algebra $A_B$
			\item for any $B'\subseteq B$, a morphism of algebras
			$i_{BB'}:A_{B'}\to A_{B}$
			\item for any $B_1\perp B_2$, a morphism of algebras
			$j_{B_1B_2}:A_{B_1}\ten A_{B_2}\to A_{B_1\sqcup B_2}$
		\end{itemize}
		\vspace{0.25cm}
		such that the following properties hold.
		\vspace{0.25cm}
		\begin{itemize}\itemsep0.25cm
			\item {\bf Normalisation.} For any $B\subseteq\dgr$, $i_{BB}=\id_{A_B}$
			\item {\bf Composition.} For any $B''\subseteq B'\subseteq B$, $i_{BB'}\circ i_{B'B''}=i_{BB''}$
			\item {\bf Naturality.} For any $B_1'\subseteq B_1\perp B_2\supseteq B_2'$, the following
			diagram is commutative
			\begin{equation}\label{eq:naturality-j-1-tikzcd}
				\begin{tikzcd}
					\rda{A}{}{B_1}\ten\rda{A}{}{B_2}
					\arrow[r, "\rdm{j}{}{}{B_2}{B_1}"] 
					& \rda{A}{}{B_1\sqcup B_2}\\
					\rda{A}{}{B_1'}\ten\rda{A}{}{B_2'} 
					\arrow[u, "{\rdm{i}{}{}{B_1'}{B_1}\ten\rdm{i}{}{}{B_2'}{B_2}}"] 
					\arrow[r,"{\rdm{j}{}{}{B_2'}{B_1'}}"'] 
					& \rda{A}{}{B_1'\sqcup B_2'} 
					\arrow[u, "{\rdm{i}{}{}{B_1'\sqcup B_2'}{B_1\sqcup B_2}}"']
				\end{tikzcd}
			\end{equation}
			\item {\bf Associativity. } For any pairwise orthogonal subdiagrams $B_1, B_2, B_3$, 
			the following diagram is commutative: 
			\begin{equation}\label{eq:naturality-j-2-tikzcd}
				\begin{tikzcd}
					&\rda{A}{}{B_1\sqcup B_2}\ten \rda{A}{}{B_3} \arrow[dr, "{\rdm{j}{}{}{B_3}{B_1\sqcup B_2}}"]&\\
					\rda{A}{}{B_1}\ten\rda{A}{}{B_2}\ten \rda{A}{}{B_3} 
					\arrow[ur, "\rdm{j}{}{}{B_2}{B_1}\ten\id_{\rda{A}{}{B_3}}"] 
					\arrow[dr, "\id_{\rda{A}{}{B_1}}\ten\rdm{j}{}{}{B_3}{B_2}"']
					& 
					& 
					\rda{A}{}{B_1\sqcup B_2\sqcup B_3}\\
					& \rda{A}{}{B_1}\ten \rda{A}{}{B_2\sqcup B_3} 
					\arrow[ur, "{\rdm{j}{}{}{B_2\sqcup B_3}{B_1}}"'] &       
				\end{tikzcd}
			\end{equation}
			\item {\bf Unit.} For any $B$, $\rdm{j}{}{}{\emptyset}{B}|_{\rda{A}{}{B}\ten 1}=
			\id_{\rda{A}{}{B}}=
			\rdm{j}{}{}{B}{\emptyset}|_{1\ten\rda{A}{}{B}}$.
		\end{itemize}
		\item 
		$\ACox{}$ is diagrammatic if and only if 
		the morphisms $\rdm{j}{}{}{B_2}{B_1}$ are invertible. 
		\item 
		A morphism of lax diagrammatic algebras $\varphi:\ACox{}\to\ACox{}'$ is the same as a collection of
		homomorphisms $\varphi_B:\rda{A}{}{B}\to\rda{A'}{}{B}$ such that 
		\[\varphi_B\circ\rdm{i}{}{}{B'}{B}=\rdm{i'}{}{}{B'}{B}\circ\varphi_{B'}^{}\] 
		for any $B'\subseteq B$, and
		\[\varphi_{B_1\sqcup B_2}\circ\rdm{j}{}{}{B_2}{B_1}=
		\rdm{j'}{}{}{B_2}{B_1}\circ\varphi_{B_1}\ten\varphi_{B_2}\]
		for any $B_1\perp B_2$.
	\end{enumerate}
\end{proposition}

\begin{pf}
A functor $\ACox{}:\PD\to\mathsf{Alg(k)}$ is the same as a collection of algebras $\rda{A}{}{B}=\ACox{}(B)$
and morphisms $\rdm{i}{}{}{B'}{B}=\ACox{}(B'\subseteq B)$ which respect the composition of morphisms in
$\PD$. A lax monoidal structure on $\ACox{}$ is then a collection of morphisms $\rdm{j}{}{}{B_2}{B_1}$ which
are natural with respect to the morphisms $\rdm{i}{}{}{B'}{B}$, associative as in the diagram above, and
compatible with the unit $u:\sfk\to\rda{A}{}{\emptyset}$.
\end{pf}

\subsection{}\label{ss:diag-alg-co}

\begin{corollary}\label{co:j redundant}
Let $\ACox{}$ be a lax diagrammatic algebra. For any $B_1\perp B_2$,
\[\rdm{j}{}{}{B_2}{B_1}=m_B\circ\rdm{i}{}{}{B_1}{B}\ten\rdm{i}{}{}{B_2}{B}\]
where $B=B_1\sqcup B_2$, and $m_B$ denotes the product in $A_{B}$.
In particular, the images of $A_{B_1}$ and $A_{B_{2}}$ in $A_{B}$ commute.
\end{corollary}

\begin{pf}
	For any $(b_1,b_2)\in A_{B_1}\times A_{B_2}$, one has 
	\begin{align*}
		\rdm{j}{}{}{B_2}{B_1}(b_1\ten b_2)
		=\rdm{j}{}{}{B_2}{B_1}(b_1\ten 1)\rdm{j}{}{}{B_2}{B_1}(1\ten b_2)
		=\rdm{i}{}{}{B_1}{B}(b_1)\rdm{i}{}{}{B_2}{B}(b_2)
	\end{align*}
	where the second equality follows by naturality and 
	compatibility with the unit.
\end{pf}

\noindent\remark\,
It follows from Corollary \ref{co:j redundant} that the morphisms $j_{B_1B_2}$ are
redundant. In fact, it is easy to see that the morphisms $m_B\circ\rdm{i}{}{}{B_1}{B}
\ten\rdm{i}{}{}{B_2}{B}$ satisfy the properties of naturality, associativity, and unit listed 
above. We shall nevertheless retain the collection $\{j_{B_1B_2}\}_{B_1\perp B_2}$
as part of datum since their redundancy does not hold in the bidiagrammatic case (cf. Corollary
\ref{co:j non redundant}).\\

\noindent\example
Let $\g$ be a diagrammatic Kac--Moody algebra with Dynkin diagram $\dgr$
and diagrammatic Lie subalgebras $\g_B\subseteq\g$, $B\subseteq\dgr$ (cf.
~\ref{ss:diag-KM}). Then, the universal enveloping algebra $A= U\g$ is a
diagrammatic algebra with $A_B= U\g_B$.

\subsection{Bidiagrammatic algebras}\label{ss:bi-diag-alg}\label{ss:functorial-bidiag-alg}

We now refine a diagrammatic algebra by including, for any pair of subdiagrams
$C\subseteq B$, an {\it algebra of invariants} $\rda{A}{C}{B}$ which maps to $A_B$. In
a number of relevant examples, $\rda{A}{C}{B}$ is a subalgebra of the centraliser of
$\rdm{i}{}{}{C}{B}(A_C)$ in $A_B$ (cf. Prop. \ref{pr:inv-bidiag}), though this does
not hold in general (cf. Example \ref{ex:Ug ten}).

Let $\BPD$ be the category whose objects are pairs $(C,B)$ of subdiagrams of $\dgr$
such that $C\subseteq B$, and the morphisms $(C',B')\to (C,B)$ are given by inclusions of
the form $C\subseteq C' \subseteq B'\subseteq B$.

Two pairs $(C_1,B_1)$ and $(C_2,B_2)$ are orthogonal if $B_1\perp B_2$.
The componentwise union of orthogonal pairs is a (symmetric, strict) partial
tensor product on $\BPD$, with $(\emptyset,\emptyset)$ as unit object.

\begin{definition}
	\hfill
	\begin{enumerate}\itemsep0.25cm
		\item A (lax) bidiagrammatic algebra is a (lax) monoidal functor $\BPD\to\mathsf{Alg(k)}$.
		\item A morphism of (lax) bidiagrammatic algebras is a natural transformation of the corresponding
		(lax) monoidal functors.
	\end{enumerate}
\end{definition}


\subsection{Alternative description of bidiagrammatic algebras}

\begin{proposition}
	\hfill
	\begin{enumerate}\itemsep0.25cm
		\item 
		A {\it lax bidiagrammatic algebra} $\ACox{}$ is the same as the datum of
		\vspace{0.25cm}
		\begin{itemize}\itemsep0.25cm
			\item for any $C\subseteq B\subseteq\dgr$, a $\sfk$--algebra $\rda{A}{C}{B}$
			\item for any 
			$C\subseteq C'\subseteq B'\subseteq B$, a morphism of algebras
			$\rdm{i}{C}{C'}{B'}{B}:\rda{A}{C'}{B'}\to\rda{A}{C}{B}$
			\item for any $C_1\subseteq B_1\perp B_2\supseteq C_2$, a morphism of algebras
			\[\rdm{j}{C_1}{C_2}{B_2}{B_1}:\rda{A}{C_1}{B_1}\ten\rda{A}{C_2}{B_2}\to\rda{A}{C_1\sqcup C_2}{B_1\sqcup B_2}\]
		\end{itemize}
		\vspace{0.25cm}
		such that the following properties hold.
		\vspace{0.25cm}
		\begin{itemize}\itemsep0.25cm
			\item {\bf Normalisation.} For any $C\subseteq B\subseteq\dgr$, $\rdm{i}{C}{C}{B}{B}=\id_{\rda{A}{C}{B}}$
			\item {\bf Composition.} For any \[C\subseteq C'\subseteq C''\subseteq B''\subseteq B'\subseteq B\]
			the following holds: $\rdm{i}{C}{C'}{B'}{B}\circ \rdm{i}{C'}{C''}{B''}{B'}=\rdm{i}{C}{C''}{B''}{B}$
			
			\item {\bf Naturality.} For any
			\[C_1\subseteq C'_1\subseteq B'_1\subseteq B_1\perp B_2\supseteq B_2' \supseteq C'_2\supseteq C_2\]
			the following diagram is commutative 
			\begin{equation}\label{eq:binaturality}
				\xymatrix@C=2cm@R=1cm{
					\rda{A}{C_1}{B_1}\ten\rda{A}{C_2}{B_2}\ar[r]^(.55){\rdm{j}{C_1}{C_2}{B_2}{B_1}} 
					& \rda{A}{C_1\sqcup C_2}{B_1\sqcup B_2}\\
					\rda{A}{C'_1}{B'_1}\ten\rda{A}{C'_2}{B'_2} \ar[u]^{\rdm{i}{C_1}{C_1'}{B_1'}{B_1}\ten\rdm{i}{C_1}{C'_2}{B'_2}{B_2}} \ar[r]_(.55){\rdm{j}{C'_1}{C'_2}{B'_2}{B'_1}} 
					& \rda{A}{C'_1\sqcup C'_2}{B'_1\sqcup B'_2} \ar[u]_{\rdm{i}{C_1\sqcup C_2}{C'_1\sqcup C'_2}{B_1\sqcup B_2}{B'_1\sqcup B'_2}}
				}
			\end{equation}
			\item {\bf Associativity.} For any pairwise orthogonal pairs $(C_i,B_i)$, $1\leq i\leq 3$,
			\begin{equation}
				\rdm{j}{C_1\sqcup C_2}{C_3}{B_3}{B_1\sqcup B_2}\circ\rdm{j}{C_1}{C_2}{B_2}{B_1}\ten\id_{\rda{A}{C_3}{B_3}}=
				\rdm{j}{C_1}{C_2\sqcup C_3}{B_2\sqcup B_3}{B_1}\circ \id_{\rda{A}{C_1}{B_1}}\ten\rdm{j}{C_2}{C_3}{B_3}{B_2}
			\end{equation}
			as morphisms $\rda{A}{C_1}{B_1}\ten\rda{A}{C_2}{B_2}\ten \rda{A}{C_3}{B_3}\to
			\rda{A}{C_1\sqcup C_2\sqcup C_3}{B_1\sqcup B_2\sqcup B_3}$.
			\item {\bf Unit.} For any $C\subseteq B$, $\rdm{j}{C}{\emptyset}{\emptyset}{B}|_{\rda{A}{C}{B}\ten 1}=\id_{\rda{A}{C}{B}}=
			\rdm{j}{\emptyset}{C}{B}{\emptyset}|_{1\ten\rda{A}{C}{B}}$. 
		\end{itemize}
		\item 
		$\ACox{}$ is a bidiagrammatic algebra if and only if the morphisms $j$'s
		are invertible.
		\item 
		A morphism of lax bidiagrammatic algebras $\varphi:\ACox{}\to\ACox{}'$ is a collection
		of homomorphisms $\varphi_B^{C}:\rda{A}{C}{B}\to\rda{(A')}{C}{B}$ such that 
		\[\varphi_B^C\circ\rdm{i}{C}{C'}{B'}{B}=\rdm{(i')}{C}{C'}{B'}{B}\circ\varphi_{B'}^{C'}\]
		for any $C\subseteq C'\subseteq B'\subseteq B$, and
		\[\varphi_{B_1\sqcup B_2}^{C_1\sqcup C_2}\circ\rdm{j}{C_1}{C_2}{B_2}{B_1}=
		\rdm{(j')}{C_1}{C_2}{B_2}{B_1}\circ\varphi_{B_1}^{C_1}\ten\varphi_{B_2}^{C_2}\]
		for any $C_1\subseteq B_1\perp B_2\supseteq C_2$.
	\end{enumerate}
\end{proposition}

The following is an analogue of Corollary \ref{co:j redundant}. Note that, contrary to the diagrammatic
case, the datum of the morphisms $j$'s is essential.

\begin{corollary}\label{co:j non redundant}
	Let $\ACox{}$ be a lax bidiagrammatic algebra. For any
	$C_1\subseteq B_1\perp B_2\supseteq C_2$, set 
	\[\rdm{\ell}{C_1}{C_2}{B_2}{B_1}= \rdm{j}{C_1}{C_2}{B_2}{B_1}|_{\rda{A}{C_1}{B_1}\ten 1}\aand
	\rdm{r}{C_1}{C_2}{B_2}{B_1}=\rdm{j}{C_1}{C_2}{B_2}{B_1}|_{1\ten\rda{A}{C_2}{B_2}}\]
	Then, 
	\[\rdm{j}{C_1}{C_2}{B_2}{B_1}=m^{C_1\sqcup C_2}_{B_1\sqcup B_2}
	\circ\rdm{\ell}{C_1}{C_2}{B_2}{B_1}\ten\rdm{r}{C_1}{C_2}{B_2}{B_1}\]
	where $m^{C_1\sqcup C_2}_{B_1\sqcup B_2}$ denotes the product in $\rda{A}{C_1\sqcup C_2}{B_1\sqcup B_2}$.
\end{corollary}

\subsection{Remarks}

\begin{enumerate}\itemsep0.25cm
	\item
	There is a symmetric functor $\PD\to\BPD$ given by the assignment 
	$B\mapsto (\emptyset, B)$. This induces a forgetful functor
	$(-)^{\scsop{0}}$ from the category of (lax) bidiagrammatic algebras to that of 
	(lax) diagrammatic algebras. Explicitly, this maps
	$\left(\rda{A}{C}{B}, \rdm{i}{C}{C'}{B'}{B},\rdm{j}{C_1}{C_2}{B_2}{B_1}\right)$ to
	$\left(\rda{A}{\scsop{0}}{B},\rdm{i}{\scsop{0}}{}{B'}{B},\rdm{j}{\scsop{0}}{}{B_2}{B_1}\right)$,
	where
	\[\rda{A}{\scsop{0}}{B}=\rda{A}{\emptyset}{B}\qquad
	\rdm{i}{\scsop{0}}{}{B'}{B}=\rdm{i}{\emptyset}{\emptyset}{B'}{B}
	\aand
	\rdm{j}{\scsop{0}}{}{B_2}{B_1}=\rdm{j}{\emptyset}{\emptyset}{B_2}{B_1}\]
	\item Conversely, there is a symmetric functor $\BPD\to\PD$ given
	by the projection $(C,B)\mapsto B$. This induces a trivial extension functor
	$(-)^{\scsop{triv}}$ from the category of (lax) diagrammatic algebras to that 
	of (lax) bidiagrammatic algebras. This maps 
	$\left(\rda{A}{}{B},\rdm{i}{}{}{B'}{B},\rdm{j}{}{}{B_2}{B_1}\right)$ to
	$\left(\rda{(A^{\scsop{triv}})}{C}{B},
	\rdm{(i^{\scsop{triv}})}{C}{C'}{B'}{B},\rdm{(j^{\scsop{triv}})}{C_1}{C_2}{B_2}{B_1}\right)$
	where
	\[\rda{(A^{\scsop{triv}})}{C}{B}=\rda{A}{}{B}\qquad
	\rdm{(i^{\scsop{triv}})}{C}{C'}{B'}{B}= \rdm{i}{}{}{B'}{B}
	\qquad
	\rdm{(j^{\scsop{triv}})}{C_1}{C_2}{B_2}{B_1}=\rdm{j}{}{}{B_2}{B_1}\]
	\item 
	Note that, for any diagrammatic algebra $\ACox{}$, $(\ACox{}^{\scsop{triv}})^{\scsop{0}}=\ACox{}$.
\end{enumerate}

\subsection{Invariant subalgebras}\label{ss:inv-bidiag}

If $\ACox{}= (\rda{A}{}{B},\rdm{i}{}{}{B'}{B}, \rdm{j}{}{}{B_2}{B_1})$ is  a diagrammatic algebra,
and $C\subseteq B$, we denote by 
\[A_B^{A_C} = \left\{ a\in A_B\left|\,[a,\rdm{i}{}{}{C}{B}(A_C)]=0\right.\right\}\]
the centraliser of $\rdm{i}{}{}{C}{B}(A_C)$ in $A_B$. The following result shows that $\ACox{}$
is endowed with a canonical bidiagrammatic structure.

\begin{proposition}\label{pr:inv-bidiag}
	Set
	\begin{eqnarray*}
		\rda{(A^\flat)}{C}{B}&=& A_B^{A_C}\subseteq A_B\\
		\rdm{(i^\flat)}{C}{C'}{B'}{B}&=&\rdm{i}{}{}{B'}{B}|_{\rda{A}{\rda{A}{}{C'}}{B'}}\\
		\rdm{(j^\flat)}{C_1}{C_2}{B_2}{B_1}&=&\rdm{j}{}{}{B_2}{B_1}|_
		{\rda{A}{\rda{A}{}{C_1}}{B_1}\ten \rda{A}{\rda{A}{}{C_2}}{B_2}}
	\end{eqnarray*}
	Then $\ACox{}^\flat=\left(\rda{(A^\flat)}{C}{B},\rdm{(i^\flat)}{C}{C'}{B'}{B}, \rdm{(j^\flat)}{C_1}{C_2}{B_2}{B_1}\right)$ 
	is a bidiagrammatic algebra.
\end{proposition}

\begin{pf}
	It is enough to check that the morphisms 
	\[
	\rdm{(i^\flat)}{C}{C'}{B'}{B}:\rda{(A^\flat)}{C'}{B'}\to\rda{(A^\flat)}{C}{B}
	\aand
	\rdm{(j^\flat)}{C_1}{C_2}{B_2}{B_1}:\rda{(A^\flat)}{C_1}{B_1}\ten\rda{(A^\flat)}{C_2}{B_2}\to
	\rda{(A^\flat)}{C_1\sqcup C_2}{B_1\sqcup B_2}
	\]
	are well--defined. The other properties are clear. 
	
	Note that $i_{BB'}(A_{B'}^{A_{C'}})\subseteq A_{B}^{A_{C'}}$, 
	since $i_{BB'}\circ i_{B'C'}=i_{BC'}$, and $A_{B}^{A_{C'}}\subseteq A_{B}^{A_{C}}$,
	since $i_{C'C}(A_{C})\subseteq A_{C'}$. It follows that $\rdm{(i^\flat)}{C}{C'}{B'}{B}$
	is well--defined. Next, if $C_1\subseteq B_1\perp B_2\supseteq C_2$ and $B=B_1
	\sqcup B_2$, $C=C_1\sqcup C_2$, the identity
	$\rdm{j}{}{}{B_1}{B}\circ\rdm{i}{}{}{B_1}{C_1}\ten\rdm{i}{}{}{B_2}{C_2}
	=\rdm{i}{}{}{C}{B}\circ\rdm{j}{}{}{C_2}{C_1}$ implies that
	\[
	\rdm{j}{}{}{B_2}{B_1}\left(\rda{A}{\rda{A}{}{C_1}}{B_1}\ten \rda{A}{\rda{A}{}{C_2}}{B_2}\right)
	= \rda{A}{\rda{A}{}{C_1}\ten\rda{A}{}{C_2}}{B}
	= \rda{A}{\rda{A}{}{C}}{B}
	\]
	The morphisms $\rdm{(j^\flat)}{C_1}{C_2}{B_2}{B_1}$ are therefore well--defined and
	invertible. 
\end{pf}

\noindent\remark\,
Note that the proof only relies on the surjectivity of the morphisms $\rdm{j}{}{}{B_2}{B_1}$,
but not on their injectivity.\\

\noindent\example\label{ex:Ug ten} 
Let $\g$ be a diagrammatic Kac--Moody algebra (cf.~\ref{ss:diag-KM} and Example~\ref{ss:diag-alg-co}). 
Then, for any $n\geqslant 0$, $U\g^{\ten n}$ is bidiagrammatic with respect to the subalgebras $(U\g_B^{\ten n})^{\g_C}$,
$C\subseteq B\subseteq\dgr$.

\section{Coxeter algebras}\label{s:Cox-algebras}

\summary{\color{magenta}{
		In this section,
		\begin{itemize}
			\item \ref{ss:gen-braid-gp-diag}: generalized braid groups 
			\item \ref{ss:pre-cox-diag-alg}: definition of (pre-)Coxeter structure
			\item \ref{ss:Cox-to-braid-rep}: representations of braid groups from Coxeter structures
			\item \ref{ss:twist-gauge-Cox}: definition of twist and gauge
			\item \ref{ss:strict-pre-Cox}: definition of strict pre-Coxeter and \emph{strictification}
		\end{itemize}
}}

In this section, we review the definition of a Coxeter structure on a bidiagrammatic algebra
following \cite{vtl-4,ATL1-2}.

\subsection{Generalised braid groups}\label{ss:gen-braid-gp-diag}

\begin{definition}
	A {\it labeling} $\ulm$ of the diagram $\dgr$ is the assignment of an integer
	$m_{ij}\in\{2,3,\ldots,\infty\}$ to any pair $i,j$ of distinct vertices of $\dgr$ such
	that
	\[m_{ij}=m_{ji}\aand
	m_{ij}=2\,\,\,\text{if $i$ and $j$ are orthogonal}\]
	The \emph{generalised} braid group corresponding to $(\dgr,\ulm)$ is the group
	$\BDm$ generated by th elements $\topS{i}$, $i\in\dgr$, with relations
	\begin{equation}\label{eq:gen-braid}
		\underbrace{\topS{i}\cdot \topS{j}\cdot \topS{i}\;\cdots\;}_{m_{ij}}=
		\underbrace{\topS{j}\cdot \topS{i}\cdot \topS{j}\;\cdots\;}_{m_{ij}}
	\end{equation}
\end{definition}

\noindent\remark\; 
Let $\sfA$ be a symmetrisable Cartan matrix, $\dgr$ its Dynkin diagram,
and $m_{ij}$ the order of the element $s_is_j$ in the Weyl group $W$. We shall refer to ${\mathsf{Dyn}}=\{\mathsf{ord}(s_is_j)\}$ as the
standard labeling on the Dynkin diagram $\dgr$. Then, $\Br{\dgr}^{{\scsop
		{\mathsf{Dyn}}}}=\Br{W}$. 

\subsection{Coxeter algebras}\label{ss:pre-cox-diag-alg}

Let $\ACox{}=\left(\rda{A}{C}{B},\rdm{i}{C}{C'}{B'}{B}, \rdm{j}{C'}{C''}{B''}{B'}\right)$ be
a (lax) bidiagrammatic algebra such that
\begin{equation}\label{eq:cent condition}
	\rda{A}{C}{B}\subseteq\rda{A}{\rda{A}{}{C}}{B}\quad\text{for any}\quad C\subseteq B
\end{equation}

\begin{definition}\label{def:pre-cox-rda}
	\hfill
	\begin{enumerate}\itemsep0.25cm
		\item
		A \emph{pre--Coxeter structure} $(\DCPA{\F}{\G}, \redasso{\F}{\F'})$ on $\ACox{}$ consists
		of the following data.
		\vspace{0.25cm}
		\vspace{0.25cm}\begin{enumerate}[leftmargin=1em]\itemsep0.25cm
			\item 
			{\bf Generalised associators.}
			For any $B'\subseteq B$ and $\F,\G\in\Mns{B,B'}$, an invertible element
			$\DCPA{\G}{\F}\in\rda{A}{B'}{B}$ satisfying the following property.
			\vspace{0.25cm}
			\begin{itemize}\itemsep0.25cm
				\item {\bf Horizontal factorisation.} For any $\F,\G,\H\in\Mns{B,B'}$,
				\[\DCPA{\H}{\F}=\DCPA{\H}{\G}\cdot\DCPA{\G}{\F}\]
				In particular, $\DCPA{\F}{\F}=1$ and $\DCPA{\F}{\G}=\DCPA{\G}{\F}^{-1}$.
				\item {\bf Orthogonal factorisation.}
				For any $B_1^\prime\subseteq B_1\perp B_2\supseteq B'_2$, and pairs
				\[\phantom{XXXXXXX}(\G_1, \G_2), (\F_1,\F_2)\in\Mns{B_1,B'_1}\times\Mns{B_2,B'_2}=\Mns{B_1\sqcup B_2,B'_1\sqcup B'_2}\]
				the following holds
				\[	\DCPA{(\G_1,\G_2)}{(\F_1,\F_2)}
				=	\rdm{j}{B_1'}{B_2'}{B_2}{B_1}(\DCPA{\G_1}{\F_1}\ten\DCPA{\G_2}{\F_2})\]
			\end{itemize}
			\item 
			{\bf Vertical joins.}
			For any $B''\subseteq B'\subseteq B$, $\F\in\Mns{B,B'}$, and $\F'\in\Mns{B',B''}$, an
			invertible element $\redasso{\F}{\F'}\in\rda{A}{B''}{B}$ satisfying the following properties.
			\vspace{0.25cm}
			\begin{itemize}\itemsep0.25cm
				\item {\bf Normalisation.} For any $\F\in\Mns{B,B'}$,
				\[\redasso{\F}{B'}=1=\redasso{B}{\F}\]
				\item {\bf Compatibility with $\DCPA{}{}$ (\emph{vertical $\DCPA{}{}$--factorisation}).} 
				For any 
				$\F,\G\in\Mns{B,B'}$ and $\F',\G'\in\Mns{B',B''}$,
				\[\DCPA{(\G\cup\G')}{(\F\cup\F')}\cdot\redasso{\F}{\F'}=
				\redasso{\G}{\G'}\cdot
				\rdm{i}{B''}{B'}{B\phantom{'}}{B\phantom{''}}(\DCPA{\G}{\F})\cdot
				\rdm{i}{B''}{B''}{B'\phantom{'}}{B\phantom{''}}(\DCPA{\G'}{\F'})\]
				\item {\bf Associativity.} For any $B'''\subseteq B''\subseteq B'\subseteq B$, 
				$\F\in\Mns{B,B'}$, $\F'\in\Mns{B',B''}$, and $\F''\in\Mns{B'',B'''}$,
				\[\redasso{\F'\cup\F}{\F''}\cdot
				\rdm{i}{B'''}{B''}{B\phantom{''}}{B\phantom{'''}}(\redasso{\F}{\F'})=
				\redasso{\F}{\F''\cup\F'}\cdot
				\rdm{i}{B'''}{B'''}{B'\phantom{''}}{B\phantom{'''}}(\redasso{\F'}{\F''})\]
				\item {\bf Orthogonal factorisation.}
				For any
				\begin{gather*}
					\phantom{XXXX}B_1''\subseteq B_1^\prime\subseteq B_1\perp B_2\supseteq B'_2\supseteq B''_2\\
					\phantom{XXXX}(\F_1,\F_2)\in\Mns{B_1,B_1'}\times\Mns{B_2,B'_2}\\
					\phantom{XXXX}(\F_1',\F_2')\in\Mns{B'_1,B_1''}\times\Mns{B_2',B''_2}
				\end{gather*}
				the following holds
				\begin{eqnarray*}
					\redasso{(\F_1, \F_2)}{(\F'_1,\F'_2)}&=&
					\rdm{j}{B_1''}{B_2''}{B_2}{B_1}(\redasso{\F_1}{\F'_1}\ten\redasso{\F_2}{\F'_2})
				\end{eqnarray*}
			\end{itemize}
		\end{enumerate}
		\item 
		Let $\ulm$ be a labelling of $\dgr$. 
		A \emph{Coxeter structure} $(\DCPA{\F}{\G}, \redasso{\F}{\F'}, S_{i})$ of type $(\dgr,\ulm)$ on $\ACox{}$ consists of a pre--Coxeter structure $(\DCPA{\F}{\G}, \redasso{\F}{\F'})$ with the following
		additional data.
		\vspace{0.25cm}
		\begin{enumerate}\itemsep0.25cm
			\item{\bf Local monodromies.} 
			For any vertex $i$ of $\dgr$, an invertible element
			$\CoxS{}{}{i}\in\rda{A}{\emptyset}{i}$ satisfying
			the following property.
			\vspace{0.25cm}
			\begin{itemize}\itemsep0.25cm
				\item {\bf Braid relations.} 
				For any $B\subseteq\dgr$, $i\neq j\in B$ and \mnss $\Ki,\Kj$ on $B$ 
				with $\{i\}\in\Ki, \{j\}\in\Kj$,  the following holds in $\rda{A}{\emptyset}{B}$
				\begin{equation}\label{eq:braid-Cox-diag-alg}
					\underbrace{\mathsf{Ad}\left(\DCPA{j}{i}\right)(\CoxS{\redasso{}{}}{}{i})\cdot \CoxS{\redasso{}{}}{}{j}
						\cdot \mathsf{Ad}\left(\DCPA{j}{i}\right)(\CoxS{\redasso{}{}}{}{i})\cdots}_{m_{ij}}=
					\underbrace{\CoxS{\redasso{}{}}{}{j}\cdot\mathsf{Ad}\left(\DCPA{j}{i}\right)(\CoxS{\redasso{}{}}{}{i})\cdot \CoxS{\redasso{}{}}{}{j}\cdots}_{m_{ij}}
				\end{equation}
				where $\DCPA{j}{i}=\DCPA{\Kj}{\Ki}$, $\CoxS{\redasso{}{}}{}{i}=
				\sfAd{\redasso{\trunc{\Ki}{}{i}}{\trunc{\Ki}{i}{}}}(\CoxS{}{}{i})$, and $\trunc{\Ki}{}{i}$ and $\trunc{\Ki}{i}{}$ are, respectively, the lower and upper truncations of $\Ki$ at $\{i\}$.
			\end{itemize}
		\end{enumerate}
	\end{enumerate}
\end{definition}

\noindent\remark\; 
Whenever clear from the context, we may omit the reference to the
datum $(\dgr,\ulm)$ from the terminology. 

\subsection{Representations of generalised braid groups}\label{ss:Cox-to-braid-rep}

\begin{proposition}
	Let $\ACox{}$ be a Coxeter algebra.
	\vspace{0.25cm}\begin{enumerate}
		\item There is a family of representations of generalised braid groups
		\[\lambda_{\F}:\BBm\to(\rda{A}{\emptyset}{B})^{\times}\]
		where $B\subseteq\dgr$ and $\F\in\Mns{B}$, which is uniquely determined by
		the conditions
		\vspace{0.25cm}\begin{enumerate}
			\item $\lambda_\F(\topS{i})=\sfAd{\redasso{\trunc{\F}{}{i}}{\trunc{\F}{i}{}}}(S_i)$ if $\{i\}\in\F$.
			\item $\lambda_\G=\sfAd{\DCPA{\G}{\F}}\circ\lambda_\F$.
		\end{enumerate}
		\item For any $B'\subseteq B$, $\F\in\Mns{B}$ with $\cc{B'}\subseteq\F$,
		and $\F'\in\Mns{B'}$, the diagram
		\[
		\begin{tikzcd}
			\BBm \arrow[r, "\lambda_{\F}"]& (\rda{A}{\emptyset}{B})^{\times}\\
			\BBpm \arrow[r, "\lambda_{\F'}"'] \arrow[u] & (\rda{A}{\emptyset}{B'})^{\times} 
			\arrow[u, "\iota_{\F'\F}"']
		\end{tikzcd}
		\]
		where the left vertical arrow is the canonical inclusion $\BBpm\subseteq\BBm$
		and 
		\[\iota_{\F'\F}=
		\sfAd{\redasso{\trunc{\F}{}{B'}}{\trunc{\F}{B'}{}}\cdot\DCPA{\trunc{\F}{B'}{}}{\F'}}\]
		is commutative.
	\end{enumerate}
\end{proposition}

\begin{pf}
	(1) For any $i\in\dgr$, we choose $\Ki\in\Mns{B}$ such that $\{i\}\in\Ki$.
	We claim that the assignment 
	\[\lambda_{\F}(\topS{i})=\sfAd{\DCPA{\F}{\Ki}}
	(\CoxS{\redasso{}{}}{}{i})\]
	where $\CoxS{\redasso{}{}}{}{i}=\sfAd{\redasso{\trunc{\Ki}{}{i}}{\trunc{\Ki}{i}{}}}(S_i)$,
	provides a morphism of groups
	$\lambda_{\F}:\BBm\to(\rda{A}{B}{\emptyset})^{\times}$.
	Moreover, this is independent of the chosen maximal nested sets $\Ki$'s.
	Finally, we observe that the morphisms $\{\lambda_{\F}\}_{\F\in\Mns{B}}$ 
	satisfy the conditions (a), (b) and are uniquely determined by them. 
	\vspace{0.25cm}
	\begin{itemize}\itemsep0.25cm
		\item \emph{$\lambda_{\F}$ is a morphism of groups.} 
		We shall prove that the braid relations hold, \ie
		\begin{equation}\label{eq:braid-lambda}
			\begin{split}
				&\underbrace{
					\mathsf{Ad}\left(\DCPA{\F}{\Ki}\right)(\CoxS{\redasso{}{}}{}{i})\cdot 
					\mathsf{Ad}\left(\DCPA{\F}{\Kj}\right)(\CoxS{\redasso{}{}}{}{j}))\cdot
					\mathsf{Ad}\left(\DCPA{\F}{\Ki}\right)(\CoxS{\redasso{}{}}{}{i}))\cdots}_{m_{ij}}=\\
				&\qquad\qquad\qquad=\underbrace{
					\mathsf{Ad}\left(\DCPA{\F}{\Kj}\right)(\CoxS{\redasso{}{}}{}{j}))\cdot 
					\mathsf{Ad}\left(\DCPA{\F}{\Ki}\right)(\CoxS{\redasso{}{}}{}{i}))\cdot
					\mathsf{Ad}\left(\DCPA{\F}{\Kj}\right)(\CoxS{\redasso{}{}}{}{j}))\cdots}_{m_{ij}}
			\end{split}
		\end{equation}
		By horizontal factorisation $\DCPA{\Kj}{\F}\DCPA{\F}{\Kj}=1$ and
		$\DCPA{\Kj}{\F}\DCPA{\F}{\Ki}=\DCPA{\Kj}{\Ki}$. Therefore, 
		the equations \eqref{eq:braid-lambda} and \eqref{eq:braid-Cox-diag-alg}
		are equivalent and obtained from each other through $\sfAd{\DCPA{\Kj}{\F}}$ .
		\item \emph{$\lambda_{\F}$ does not depend on the choice of $\Ki$'s.}
		Indeed, let $\Kip\in\Mns{B}$ be such that $\{i\}\in\Kip$. Thus, 
		$\trunc{\Kip}{i}{}=\trunc{\Ki}{i}{}$ and, by vertical factorisation,
		\[
		\DCPA{\Kip}{\Ki}\cdot \redasso{\trunc{\Ki}{}{i}}{\trunc{\Ki}{i}{}}=
		\redasso{\trunc{\Kip}{}{i}}{\trunc{\Kip}{i}{}}\cdot \DCPA{\trunc{\Kip}{}{i}}{\trunc{\Ki}{}{i}}
		\]
		Since $\DCPA{\trunc{\Kip}{}{i}}{\trunc{\Ki}{}{i}}\in\rda{A}{\{i\}}{B}$,
		it follows that $\sfAd{\DCPA{\trunc{\Kip}{}{i}}{\trunc{\Kip}{}{i}}}(S_i)=S_i$,
		and therefore
		\begin{align}
			\sfAd{\DCPA{\F}{\Ki}}
			\sfAd{\redasso{\trunc{\Ki}{}{i}}{\trunc{\Ki}{i}{}}}(S_i)=&
			\sfAd{\DCPA{\F}{\Kip}}\sfAd{\DCPA{\Kip}{\Ki}\cdot\redasso{\trunc{\Ki}{}{i}}{\trunc{\Ki}{i}{}}}(S_i)\\
			=&\sfAd{\DCPA{\F}{\Kip}}\sfAd{\redasso{\trunc{\Kip}{}{i}}{\trunc{\Kip}{i}{}}\cdot \DCPA{\trunc{\Kip}{}{i}}{\trunc{\Ki}{}{i}}}(S_i)\\
			=&\sfAd{\DCPA{\F}{\Kip}}\sfAd{\redasso{\trunc{\Kip}{}{i}}{\trunc{\Kip}{i}{}}}(S_i)
		\end{align}
		where the first and second equalities follows, respectively, from horizontal and
		vertical factorisations.
		\item {\em  The morphisms $\{\lambda_{\F}\}_{\F\in\Mns{B}}$ 
			satisfy the conditions (a), (b).} Let $\F,\G\in\Mns{B}$. Then, $\DCPA{\G}{\Ki}=\DCPA{\G}{\F}\DCPA{\F}{\Ki}$ 
		and we get
		\[
		\sfAd{\DCPA{\G}{\F}}\circ\lambda_{\F}(\topS{i})=
		\sfAd{\DCPA{\G}{\F}\cdot \DCPA{\F}{\Ki}}(\CoxS{\redasso{}{}}{}{i})
		=\sfAd{\DCPA{\G}{\Ki}}(\CoxS{\redasso{}{}}{}{i})=\lambda_{\G}(\topS{i})
		\]
		Moreover, if $\{i\}\in\F$, we can choose $\Ki=\F$, so that
		$\lambda_{\F}(\topS{i})=\CoxS{\redasso{}{}}{}{i}=\sfAd{\redasso{\trunc{\F}{}{i}}{\trunc{\F}{i}{}}}(S_i)$. 
		\item {\em The morphisms $\{\lambda_{\F}\}_{\F\in\Mns{B}}$ are uniquely determined by
			(a) and (b).}
		Let $\{\wt{\lambda}_{\F}\}_{\F\in\Mns{B}}$
		be a collection of morphisms of groups satisfying (a), (b). Then, if we choose $\G=\Ki$, 
		we get
		\[
		\wt{\lambda}_{\F}(\topS{i})\stackrel{(b)}{=}
		\sfAd{\DCPA{\F}{\Ki}}\circ\wt{\lambda}_{\Ki}(\topS{i})\stackrel{(a)}{=}
		\sfAd{\DCPA{\F}{\Ki}}(\CoxS{\redasso{}{}}{}{i})=\lambda_{\F}(\topS{i})
		\]
	\end{itemize}
	(2) Let $B'\subseteq B$, $\F\in\Mns{B}$ with $\cc{B'}\subseteq \F$, and $\F'\in\Mns{B'}$.
	For any $i\in B'$, let $\Kip\in\Mns{B'}$ be such that $\{i\}\in\Kip$ and
	set $\Ki=\Kip\cup\trunc{\F}{}{B'}\in\Mns{B}$, so that $\trunc{\Ki}{B'}{}=\Kip$ 
	and $\trunc{\Ki}{}{B'}=\trunc{\F}{}{B'}$. Note that $\trunc{\Ki}{}{i}=
	\trunc{\Kip}{}{i}\cup\trunc{\F}{}{B'}$ and $\trunc{\Ki}{i}{}=\trunc{\Kip}{i}{}$. Thus, 
	\begin{align}
		\redasso{\trunc{\Ki}{}{i}}{\trunc{\Ki}{i}{}}\cdot\redasso{\trunc{\F}{}{B'}}{\trunc{\Kip}{}{i}}&
		=\redasso{\trunc{\F}{}{B'}}{\Kip}\cdot\redasso{\trunc{\Kip}{}{i}}{\trunc{\Kip}{i}{}}
	\end{align}
	and one has
	\begin{align*}
		\lambda_{\F}(\topS{i})&=\sfAd{\DCPA{\F}{\Ki}\cdot\redasso{\trunc{\Ki}{}{i}}{\trunc{\Ki}{i}{}}}(S_i)\\
		&=\sfAd{\DCPA{\F}{\Ki}\cdot\redasso{\trunc{\Ki}{}{i}}{\trunc{\Ki}{i}{}}\cdot \redasso{\trunc{\F}{}{B'}}{\trunc{\Kip}{}{i}}}(S_i)\\
		&=\sfAd{\DCPA{\F}{\Ki}\cdot\redasso{\trunc{\F}{}{B'}}{\Kip}\cdot
			\redasso{\trunc{\Kip}{}{i}}{\trunc{\Kip}{i}{}}}(S_i)\\
		&=\sfAd{\DCPA{(\trunc{\F}{}{B'}\cup\trunc{\F}{B'}{})}{(\trunc{\F}{}{B'}\cup\Kip)}
			\cdot\redasso{\trunc{\F}{}{B'}}{\Kip}\cdot
			\redasso{\trunc{\Kip}{}{i}}{\trunc{\Kip}{i}{}}}(S_i)\\
		&=\sfAd{\redasso{\trunc{\F}{}{B'}}{\trunc{\F}{B'}{}}\cdot
			\DCPA{\trunc{\F}{B'}{}}{\Kip}
			\cdot
			\redasso{\trunc{\Kip}{}{i}}{\trunc{\Kip}{i}{}}}(S_i)\\
		&=\sfAd{\redasso{\trunc{\F}{}{B'}}{\trunc{\F}{B'}{}}\cdot
			\DCPA{\trunc{\F}{B'}{}}{\F'}\cdot\DCPA{\F'}{\Kip}
			\cdot
			\redasso{\trunc{\Kip}{}{i}}{\trunc{\Kip}{i}{}}}(S_i)\\
		&=\sfAd{\redasso{\trunc{\F}{}{B'}}{\trunc{\F}{B'}{}}\cdot
			\DCPA{\trunc{\F}{B'}{}}{\F'}}(\lambda_{\F'}(\topS{i}))
	\end{align*}
	where the second equality follows by the invariance property $\redasso{\trunc{\F}{}{B'}}{\trunc{\Kip}{}{i}}\in\rda{A}{\{i\}}{B}$,
	the third one by the associativity of $\redasso{}{}$, the fourth one by construction, 
	the fifth one by vertical factorisation, and the sixth one by horizontal factorisation.
\end{pf}

\subsection{Twisting and gauging of Coxeter structures}\label{ss:twist-gauge-Cox}

\begin{definition}
	\hfill
	\begin{enumerate}\itemsep0.25cm
		\item A {\it twist} $u=\{u_{\F}\}$ in $\ACox{}$ is the datum, for 
		any $\F\in\Mns{B,B'}$, of an invertible element 
		$u_{\F}\in\rda{A}{B'}{B}$ such that, if $B_1^\prime\subseteq B_1\perp B_2\supseteq B'_2$,
		$(\F_1,\F_2)\in\Mns{B_1\sqcup B_2,B'_1\sqcup B'_2}$,
		\begin{eqnarray*}
			u_{(\F_1,\F_2)}&=&
			\rdm{j}{B_1'}{B_2'}{B_2}{B_1}(u_{\F_1}\ten u_{\F_2})
		\end{eqnarray*}
		\item The {\it twisting} of a Coxeter structure
		$\sCox{}=(\DCPA{\F}{\G}, \redasso{\F}{\F'}, \CoxS{}{}{i})$
		by a twist $u=\{u_{\F}\}$  is the Coxeter structure 
		\[\sCox{u}=((\DCPA{\F}{\G})_{u}, (\redasso{\F}{\F'})_{u}, (\CoxS{}{}{i})_{u})\]
		given by
		\begin{eqnarray*}
			(\DCPA{\F}{\G})_{u}		&=&u_{\F}^{-1}\cdot \DCPA{\F}{\G}\cdot u_{\G}\\
			(\redasso{\F}{\F'})_{u}	&=&u_{\F'\cup\F}^{-1}\cdot \redasso{\F}{\F'}\cdot u_{\F'}
			\cdot u_{\F}\\
			(\CoxS{}{}{i})_{u} &=& u_{\{i\}}^{-1}\cdot\CoxS{}{}{i}\cdot u_{\{i\}}
		\end{eqnarray*}
		We denote by $\ACox{u}$ the Coxeter algebra with twisted structure $\sCox{u}$. 
		\item A {\it gauge} $\gCox{}=\{a_B\}$ in $\ACox{}$ consists of
		an invertible element $a_B\in\rda{A}{B}{B}$ 
		for any $B\subseteq\dgr$, satisfying
		\[
		a_{B_1\sqcup B_2}=\rdm{j}{B_1}{B_2}{B_2}{B_1}(a_{B_1}\ten a_{B_2})
		\]
		\item The {\it gauging} of a twist $u=(u_{\F})$ by $a$ is the
		twist $u_{\gCox{}}=((u_{\F})_{\gCox{}})$ given by
		\begin{align*}
			(u_{\F})_{\gCox{}}&=\rdm{i}{B'}{B'}{B'}{B\phantom{'}}(a_{B'})\cdot u_\F
			\cdot \rdm{i}{B'}{B}{B}{B\phantom{'}}(a_{B})^{-1}
		\end{align*}
	\end{enumerate}
\end{definition}

The following is standard.

\begin{proposition}
	Let $\sCox{}$ be a Coxeter structure on $\ACox{}$, $u$ a twist and $a$ a gauge.
	Then,  $\sCox{u}=\sCox{u_{\gCox{}}}$. Moreover, the representations of the braid group $\lambda_{\F}^{\sCox{}}$ and $\lambda_{\F}^{\sCox{u}}$, arising, respectively, 
	from $\sCox{}$ and $\sCox{u}$, are equivalent.
\end{proposition}

\subsection{Strict Coxeter structures}\label{ss:strict-pre-Cox}

Let $\ACox{}$ be a Coxeter  algebra. We say that
\begin{itemize}\itemsep0.25cm
	\item $\ACox{}$ is $\DCPA{}{}$--\emph{strict} if $\DCPA{\F}{\G}=1$ for any $\F,\G\in\Mns{B,B'}$
	\item $\ACox{}$ is $\redasso{}{}$--\emph{strict} if $\redasso{\F}{\F'}=1$ for any $\F\in\Mns{B,B'}$ and $\F'\in\Mns{B',B''}$
\end{itemize}
The following result shows that we can always restrict to either of these cases.

\begin{proposition}\label{prop:strictness-univ-cox}
	\hfill
	\begin{enumerate}\itemsep0.25cm
		\item $\ACox{}$ is twist equivalent to a $\DCPA{}{}$--strict Coxeter algebra.
		\item $\ACox{}$ is canonically twist equivalent to an $\redasso{}{}$--strict Coxeter algebra.
	\end{enumerate}
\end{proposition}

Thanks to the condition \eqref{eq:cent condition}, the proof is identical to that of 
\cite[Prop. 9.10]{ATL1-2} and therefore omitted.\\

\noindent\remark\;
Note, however, that the latter result cannot be used to obtain a Coxeter structure 
which is both $\DCPA{}{}$--strict and $\redasso{}{}$--strict (cf. \cite[Sec. 9.]{ATL1-2}).



\newcommand{\ptF}[1]{\mathsf{pt}_{#1}}
\newcommand{\cF}[2]{c_{#1,#2}}
\newcommand{\pF}[2]{p_{#1,#2}}
\newcommand{\aF}[2]{\alpha_{#1,#2}} 
\newcommand{\mF}[2]{m_{#1,#2}} 
\newcommand{\uF}[2]{u_{#1,#2}}
\newcommand{\polF}[2]{P_{#1,#2}}
\newcommand{\polFm}[2]{P_{#1,#2}^{(m)}}
\newcommand{\resF}[2]{R_{#1,#2}}
\newcommand{\resFm}[2]{R^{(m)}_{#1,#2}}
\newcommand{\tF}[2]{\mathsf{t}_{#1,#2}}
\newcommand{\tFm}[2]{\mathsf{t}_{#1,#2}^{(m)}}
\newcommand{\Khm}[1]{\mathsf{t}_{#1}^{(m)}}

\section{Canonical fundamental solutions}\label{s:DCP}

We generalise to an arbitrary root system the construction of the 
fundamental solutions of the universal Casimir connection due to 
Cherednik \cite{cherednik-90, cherednik-91} and De Concini--Procesi \cite{DCP}.
More precisely, it is enough to observe that, by restriction to the truncated root system $\Rs{}^{\leqslant m}$ (cf. \ref{ss:local-finite} and Remark \ref{ss:holonomy}), the corresponding hyperplane arrangement is finite. Thus, the 
theory developed in \cite{DCP} (see also \cite[Sec.~1]{vtl-4}) applies and can be extended by limit to the full root system. In the following, we describe the construction of the fundamental solutions in this context, omitting all proofs, which can be recovered verbatim from \cite{DCP, vtl-4}.

\subsection{Diagrammatic structure on $\DCPHA{\rootsys}$}\label{ss:comp-diag-sub}

For any subdiagram $B\subseteq\dgr$, we denote by $\Rs{B}$ the corresponding
root subsystem, and by $\DCPHA{B}$ the holonomy algebra  $\DCPHA{\Rs{B}}$.
For any $B'\subseteq B$, the inclusion $\Rs{B'}\subseteq\Rs{B}$ induces a
morphism $i_{BB'}:\DCPHA{B'}\to\DCPHA{B}$, mapping $\Kh{\alpha}{}
\in\DCPHA{B'}$ to the same symbol in $\DCPHA{B}$.

\begin{lemma}
	The assignment
	\[
	p_{B'B}(\Kh{\alpha}{})=
	\left\{
	\begin{array}{cc}
		\Kh{\alpha}{} & \mbox{if } \alpha\in\Rs{B',+}\\
		0 & \mbox{otherwise}\\
	\end{array}
	\right.
	\]
	extends to a morphism of algebras $p_{B'B}:\DCPHA{B}\to\DCPHA{B'}$ such that $p_{B'B}\circ i_{BB'}=\id_{\DCPHA{B'}}$.
	In particular, $i_{BB'}$ is injective.
\end{lemma}

\begin{pf}
	It is enough to show that $p_{B'B}$ preserves the $tt$--relations 
	\begin{equation}\label{eq:tt-rel-B}
		\left[\Kh{\alpha}{}, \sum_{\beta\in\Psi}\Kh{\beta}{}\right]=0
	\end{equation}
	where $\alpha\in\Rs{B,+}$ and $\Psi\subseteq\Rs{B}$ is a subsystem of
	rank $2$, which contains $\alpha$. Denote by $X$ the left--hand side, and consider the following cases.
	\vspace{0.25cm}
	\begin{itemize}\itemsep0.25cm
		\item If $\alpha\not\in\Rs{B',+}$, then $p_{B'B}(X)=0$.
		\item If $\Psi\cap\Rs{B'}=\{\pm\alpha\}$, then $p_{B'B}(X)=[\Kh{\alpha}{},\Kh{\alpha}{}]=0$.
		\item If $\Psi\cap\Rs{B'}$ contains at least two linearly independent elements, then
		$\Psi\subseteq\Rs{B'}$, and the $tt$--relations
		in $\DCPHA{B'}$ imply that $p_{B'B}(X)=0$.
	\end{itemize}
\end{pf}

Finally, note that, if $B_1,B_2\subseteq B$ with $B_1\perp B_2$, multiplication induces an isomorphism of algebras
$\rdm{j}{}{}{B_2}{B_1}:\DCPHA{B_1}\ten\DCPHA{B_2}\to\DCPHA{B_1\sqcup B_2}$, with inverse given by reordering.
Thus, we have the following\footnote{Henceforth, for simplicity, we will denote a (bi)diagrammatic algebra by the
collection of its subalgebras only, omitting the structural morphisms $i$ and $j$.}

\begin{proposition}
	$\DCPHA{}=\{\DCPHA{B}\}$ is a split diagrammatic algebra with respect to the structural morphisms described above.
\end{proposition}

\noindent\remark\;
For any $B'\subseteq B$, let $\DCPHA{BB'}$ be the centraliser of $\DCPHA{B'}$ in $\DCPHA{B}$.
Then, by Proposition \ref{ss:inv-bidiag}, $\DCPHA{}^\flat=\{\DCPHA{BB'}\}$ is a bidiagrammatic algebra.\\

If $A=\bigoplus_{N\geqslant 0}A_N$ is an $\IN$--graded algebra,
we denote by $\wh{A}=\prod_{N\geqslant   0}A_N$ the completion of $A$
with respect to its grading. For any $B\subseteq\dgr$, let $\DCPHAH{B}$ be the 
completion of $\DCPHA{B}$ with respect to the grading $\deg(\Kh{\alpha}{})=1$,
$\alpha\in\Rs{B,+}$, and $\DCPHAH{BB'}$ the centraliser of $\DCPHA{B'}$ in $\DCPHAH{B}$,
$B'\subseteq B$. The results above extend to completions, in the sense that
$\DCPHAH{}=\{\DCPHAH{B}\}$ and
$\DCPHAH{}^{\flat}=\{\DCPHAH{BB'}\}$
are naturally \emph{lax} diagrammatic and bidiagrammatic algebras, respectively.\footnote
{The isomorphism $\rdm{j}{}{}{B_2}{B_1}:\DCPHA{B_1}\ten\DCPHA{B_2}\to\DCPHA{B_1\sqcup B_2}$
	extends to an injective, but not surjective, morphism
	$\rdm{\wh{j}}{}{}{B_2}{B_1}:\DCPHAH{B_1}\ten\DCPHAH{B_2}\to\DCPHAH{B_1\sqcup B_2}$.
	It is clear, however, that $\DCPHAH{}$  (resp. $\DCPHAH{}^{\flat}$) is
	diagrammatic (resp. bidiagrammatic) with respect to the completed tensor product $\ctp$.}

Henceforth, we will identify $\DCPHA{B}$ (resp. $\DCPHAH{B}$) with the 
subalgebra in $\DCPHA{}$ (resp. $\DCPHAH{}$) topologically generated 
by the elements $\Kh{\alpha}{}$, $\alpha\in\Rs{B}\subseteq\rootsys$.

\subsection{Commutation relations}\label{ss:hol-comm-rel}
We say that $\alpha,\beta\in\Rs{+}$ are strongly orthogonal, and we write $\alpha\rperp\beta$, if the root system generated by $\alpha$ and $\beta$ is $\{\pm\alpha,\pm\beta\}$. Note that this condition is indeed stronger than $\alpha$ and $\beta$ being orthogonal with respect to the inner product on $\h^*$.
However, it is equivalently stated in terms of orthogonality of subdiagrams, \ie $\alpha\rperp\beta$ if and only if $\supp(\alpha)\perp\supp(\beta)$. Therefore, with a slight abuse of notation and terminology, in the following we shall simply say 
that two roots are orthogonal and write $\alpha\perp\beta$. 
Moreover, we write $\alpha\perp B$ if $\supp(\alpha)\perp B$.\\

For any $B\subseteq\dgr$, set 
\begin{equation}\label{eq:t-diag}
	\Khm{B}=\sum_{\substack{\alpha\in\Rs{+}^{\leqslant m}\\\supp(\alpha)\subseteq B}}
	\Kh{\alpha}{}\in\DCPHA{B}^{(m)}
	\aand
	\Kh{B}{}=\sum_{\substack{\alpha\in\Rs{+}\\\supp(\alpha)\subseteq B}}
	\Kh{\alpha}{}\in\DCPHA{B}
\end{equation}

\begin{proposition}
	The following holds.
	\begin{enumerate}\itemsep0.25cm
		\item If $B_1\perp B_2$, then $[\Kh{B_1}{},\Kh{B_2}{}]=0$.
		\item If $\alpha\in\Rs{B,+}$, then $[\Kh{\alpha}{}, \Kh{B}{}]=0$.
		\item If $B'\subseteq B$, then $[\Kh{B'}{}, \Kh{B}{}]=0$.
	\end{enumerate}
Analogous results hold for the elements $\Khm{B}$ in $\DCPHA{B}^{(m)}$.	
\end{proposition}

\begin{pf}
	(1) is clear and (3) follows from (2). Note that, if $B_1\perp B_2$,then $\Kh{B_1\sqcup B_2}{}=\Kh{B_1}{}+\Kh{B_2}{}$. Therefore, it is enough to prove 
	$[\Kh{\alpha}{}, \Kh{B}{}]=0$ for $B$ connected and $\alpha\in\Rs{B,+}$.
	
	Let $\mathsf{C}_{\alpha}$ be the set of equivalence classes in $\Rs{B,+}\setminus\{\alpha\}$
	with respect to the equivalence relation given by $\beta\sim\gamma$ if they span the same
	line in $\h_B^*/\langle\alpha\rangle$. Then, 
	\[
	\Kh{B}{}=\Kh{\alpha}{}+\sum_{[\beta]\in\mathsf{C}_{\alpha}}\sum_{\beta\in[\beta]}\Kh{\beta}{}
	\] 
	By construction, the span of $\alpha$ and $\{\beta\;|\;\beta\in[\beta]\}$ is two--dimensional,
	therefore 
	\[
	[\Kh{\alpha}{},\Kh{B}{}]=\sum_{[\beta]\in\mathsf{C}_{\alpha}}
	\left[\Kh{\alpha}{}, \sum_{\beta\in[\beta]}\Kh{\beta}{}\right]=0
	\] 
	where the second equality follows from \eqref{eq:tt-relns}. The case $\Khm{}$ is
	identical.
\end{pf}

\noindent\remark\;
Note that the results above hold in greater generality. Let $S\subseteq\Rs{+}$
be a subset of positive roots, $\langle S\rangle\subseteq\h^*$ the subspace spanned
by $S$. Set $\Rs{S,+}=\langle S\rangle\cap\Rs{+}$ and
\[\Kh{\langle S\rangle}{}=\sum_{\beta\in\Rs{S,+}}\Kh{\beta}{}\]
Then, if $\alpha\in\Rs{S,+}$, one has 
$[\Kh{\alpha}{}, \Kh{\langle S\rangle}{}]=0$.

\subsection{Blow--up coordinates on $\sfX$}\label{ss:blow-up-coord}
Let $\F\in\Mns{\dgr}$ be a maximal nested set on $\dgr$. 
For any $\alpha\in\Rs{}$, let $\pF{\F}{\alpha}$ be the minimal element $B\in\F$
such that $\supp({\alpha})\subseteq B$. Then $\pF{\F}{\bullet}$ establishes a 
one to one correspondence between the simple roots $\{\alpha_1,\dots, \alpha_n\}$ 
and the elements in $\F$. For any $B\in\F$, we denote by $\aF{\F}{B}$ the simple 
root corresponding to $B$ under $\pF{\F}{\bullet}$. For any $B\in\F$, we denote by 
$\cF{\F}{B}$ the minimal element in $\F$ which contains properly $B$.\\ 

For any $B\in\F$, set $x_B=\sum_{i\in B}\alpha_i$. Then $\{x_B\}_{B\in\F}$ defines a set 
of coordinates on $\hess$. Set $U_{\F}=\IC^{|\F|}$ with 
coordinates $\{\uF{\F}{B}\}_{B\in\F}$. Let $\rho:U_{\F}\to\sfX$ be the map defined on 
the coordinates $\{x_B\}$ by $x_B=\prod_{B\subseteq C\in\F}\uF{\F}{C}$. 
Then, $\rho$ is birational with inverse
\begin{equation}\label{eq:blow-up-coord}
	\uF{\F}{B}=\left\{
	\begin{array}{cl}
		x_B & \mbox{if B is maximal in $\F$} \\
		{x_B}/{x_{\cF{\F}{B}}}& \mbox{otherwise}
	\end{array}
	\right.
\end{equation}
For any $\alpha\in\Rp$ set $\polF{\F}{\alpha}=\frac{\alpha}{x_{\pF{\F}{\alpha}}}$.\\

\noindent\remark\;
In the case of affine root systems, it is convenient to impose $x_\dgr=\sum_i a_i\alpha_i=\delta$.

\subsection{Solutions of the Casimir connection}\label{ss:dcp-solutions}\label{ss:DCPsol}

Following \cite{cherednik-90,cherednik-91,DCP}, we construct a collection of fundamental solutions of the 
universal Casimir connection \eqref{eq:univ-Casimir}, indexed by maximal nested
sets on $\dgr$.

Let $m\geqslant 0$. For any $\F\in\Mns{\dgr}$ and $B\in\F$, set $\resFm{\F}{B}=\sum_{\substack{\alpha\in\Rp^{\leqslant m}\\ \pF{\F}{\alpha}=B}}\Kh{\alpha}{}$.
Hence,
\begin{equation}\label{eq:residues1}
	\Khm{B}=\sum_{\substack{\alpha\in\Rp^{\leqslant m}\\ 
			\supp(\alpha)\subseteq B}}\Kh{\alpha}{}=
	\sum_{\substack{C\in\F\\C\subseteq B}}\resFm{\F}{C}
	\aand
	\sum_{B\in\F}\resFm{\F}{B}\frac{dx_B}{x_B}=
	\sum_{B\in\F}\Khm{B}\frac{d\uF{\F}{B}}{\uF{\F}{B}}
\end{equation}
and finally
\begin{equation}\label{eq:residues2}
	\prod_{B\in\F}\uF{\F}{B}^{\Khm{B}}=\prod_{B\in\F}\uF{\F}{B}^{\sum_{C\subseteq B}\resFm{\F}{C}}=
	\prod_{C\in\F}\prod_{B\supseteq C}\uF{\F}{B}^{\resFm{\F}{C}}
	=\prod_{C\in\F}x_C^{\resFm{\F}{C}}
\end{equation}

Let $\C_{\IC}$ be the 
complexification of the fundamental Weyl chamber.
For any $\F\in\Mns{\dgr}$, let $\U^{(m)}_{\F}\subset U$ be the 
complement of the zeros of the polynomials $\polF{\F}{\alpha}$, $\alpha\in\Rs{+}^{\leqslant m}$, 
and $\D^{(m)}\subset\U^{(m)}_{\F}\cap\C_{\IC}$ a simply connected set with
$\ptF{\F}=\cap_{B\in\F}\{\uF{\F}{B}=0\}\in\ol{\D^{(m)}}$. 
We have the following

\begin{theorem}\label{thm:DCPsol}\hfill
	\begin{enumerate}
		\item There is a unique holomorphic function $H^{(m)}_{\F}:\D^{(m)}\to {\DCPHAH{\rootsys}}^{(m)}$ such that $H^{(m)}_{\F}(\ptF{\F})=1$ and, for every determination of $\log(x_B)$, $B\in\F$, the multivalued function 
		\[
		\DCPS{\F}^{(m)}=
		H^{(m)}_{\F}\prod_{B\in\F}x_B^{\resFm{\F}{B}}=H^{(m)}_{\F}\prod_{B\in\F}\uF{\F}{B}^{\Khm{B}}
		\]
		is a solution of the holonomy equation $dG=A^{(m)}G$, 
		where
		\[A^{(m)}=\sum_{\alpha\in\Rs{+}^{\leqslant m}}\Kh{\alpha}{}d\log(\alpha)\]
		\item
		The sequence of solutions $\{\DCPS{\F}^{(m)}\}_{m\geqslant 0}$ uniquely determines a multivalued function
		$\DCPS{\F}$ with values in $\DCPHAH{\rootsys}$, satisfying the holonomy equation $dG=AG$, where
		$A=\sum_{\alpha\in\Rs{+}}\Kh{\alpha}{}d\log(\alpha)$.
	\end{enumerate}
\end{theorem}

\subsection{Asymptotics of the canonical solutions at infinity}\label{ss:DCP-asym}

We conclude this section with the study of the asymptotic behavior of fundamental 
solution $\DCPS{\F}$ as $\alpha_i\to\infty$ with $\{i\}\in\F$, which is 
a straightforward generalisation of \cite[Prop.~4.5, 4.6]{vtl-6}.

\subsubsection{} 

Fix $i\in\bfI$, let $\ol{\rootsys}\subset\rootsys$ be the root system
generated by the simple roots $\{\alpha_j\}_{j\neq i}$, $\olh\ess\subset\h\ess$
the corresponding essential Cartan subalgebra, and 
$\DCPHA{\ol{\rootsys}}\subset\DCPHA{\rootsys}$ the holonomy algebra.
The inclusion of root systems $\ol{\rootsys}\subset\rootsys$ gives rise to
a projection $\pi:\h\ess\to\olh\ess$ determined by the requirement that
$\alpha(\pi(h))=\alpha(h)$ for any $\alpha\in\ol{\rootsys}$. The kernel
of $\pi$ is the line $\IC\cow{i}$ spanned by the $i$th fundamental
coweight of $\h$. 
We shall coordinatise the fibres of $\pi$ by restricting
the simple root $\alpha_i$ to them. This amounts to
trivialising the fibration $\pi:\h\ess\to\ol{\h}\ess$ as $\h\ess\simeq\IC
\times\ol{\h}\ess$ via $(\alpha_i,\pi)$. The inverse of this
isomorphism is given by $(w,\olmu)\to w\cow{i}+\imath
(\olmu)$, where $\imath:\ol\h\to\h$ is the embedding
with image $\ker(\alpha_i)$ given by
\begin{equation}\label{eq:emb i}
	\imath(\ol{t})=\ol{t}-\alpha_i(\ol{t})\cow{i}
\end{equation}
Let
\begin{equation}\label{eq:K olK}
	\Kh{}{}=\sum_{\alpha\in\Rs{+}}\Kh{\alpha}{}\aand
	\ol{\Kh{}{}}=\sum_{\alpha\in\ol{\rootsys}_+}\Kh{\alpha}{}
\end{equation}
be the {\em universal Casimir operators} in $\DCPHA{\rootsys}$ and 
$\DCPHA{\ol{\rootsys}}$, respectively.\\

Set $\ol{\dgr}= \dgr\setminus\{i\}$. Fix $\ol{\mu}\in\ol{\h}\ess=\h_{\ol{\dgr}}\ess$, and
consider the fiber of $\pi:\h\ess\to\olh\ess$ at $\olmu$. Since
the restriction of $\alpha\in\rootsys$ to $\pi^{-1}(\olmu)$ is
equal to $\alpha(\cow{i})\alpha_i+\alpha(\imath(\olmu))$,
the restriction of the Casimir connection $\nablak$ to
$\pi^{-1}(\olmu)$ is equal to
\[\nabla_{i,\olmu}=
d-\sum_{\alpha\in\Rs{+}\setminus\ol{\rootsys}}
\frac{d\alpha_i}{\alpha_i-w_\alpha}\Kh{\alpha}{}\]
where $w_\alpha=-\alpha(\imath(\olmu))/\alpha(\cow{i})$.
Set $R_\olmu=\max\{|w_\alpha|\}_{\alpha\in\rootsys\setminus\ol{\rootsys}}$.

\begin{proposition}\label{pr:Fuchs infty}\hfill
	\begin{enumerate}
		\item For any $\olmu\in\olh\ess$, there is a unique holomorphic
		function
		\[H_\infty:\{w\in \IP^1|\,|w|>R_\olmu\}\to\DCPHA{\rootsys}\]
		such that $H_\infty(\infty)=1$ and, for any determination of
		$\log(\alpha_i)$, the function $\Upsilon_\infty=H_\infty(\alpha_i)
		\cdot \alpha_i^{\Kh{}{}-\ol{\Kh{}{}}}$ satisfies
		\[d\Upsilon_\infty=
		\sum_{\alpha\in\Rs{+}\setminus\ol{\rootsys}}
		\frac{d\alpha_i}{\alpha_i-w_\alpha}\Kh{\alpha}{}\,
		\Upsilon_\infty\]
		\item The function $H_\infty(\alpha_i,\olmu)$ is holomorphic on
		the simply--connected domain 
		$\D_\infty=\{(w,\ol\mu)|\,|w|>R_\olmu\}\subset\IP^1\times\ol{\h}$
		and, as a function on $\D_\infty$, $\Upsilon_\infty$ satisfies
		\[d\Upsilon_\infty=
		\sum_{\alpha\in\Rs{+}}\frac{d\alpha}{\alpha}
		\Kh{\alpha}{}\,\,\Upsilon_\infty
		-\Upsilon_\infty\sum_{\alpha\in\ol{\rootsys}_+}\frac{d\alpha}{\alpha}
		\Kh{\alpha}{}\]
	\end{enumerate}
\end{proposition}

\subsubsection{} 

Let $\F$ be a \mns on $\dgr$, set $\olcalF=\F\setminus\{\dgr\}$
and $\root{i}=\aF{\F}{\dgr}$, \ie $\root{i}$ is the only simple root
whose support is not contained in the maximal elements of $\F$. Let
\[\Psi_\F:\C\to\DCPHAH{\rootsys}
\aand
\DCPS{\olcalF}:\olC\to\DCPHAH{\ol{\rootsys}}\]
be the fundamental solutions of the Casimir connection for
$\rootsys$ and $\ol{\rootsys}=\Rs{\dgr\setminus\alpha_i}$ corresponding to
$\F$, $\olcalF$ respectively, with blow-up coordinates  
$\{x_B\}_{B\subseteq \dgr}$ (cf.~\ref{ss:blow-up-coord}). Regard $\DCPS{\olcalF}$ as being
defined on $\C$ via the projection $\pi:\h\ess\to\olh\ess$. The result
below expresses $\Psi_\F$ in terms of $\DCPS{\olcalF}$ and
the solution $\Upsilon_\infty$.

\begin{proposition}\label{pr:infty factorisation}
	$\Psi_\F=
	\Upsilon_\infty\cdot\DCPS{\olcalF}\cdot
	x_{\dgr}(\cow{i})^{\Kh{}{}-\ol{\Kh{}{}}}$.
\end{proposition}

Clearly, the same results holds for any $\DCPS{\F}$ with $\F\in\Mns{B}$ and $B\subseteq \dgr$.

\section{A Coxeter structure from holonomy}\label{s:holo-Cox}

We prove that the monodromy of the Casimir connection defines 
a {Coxeter structure} on the holonomy algebra $\DCPHA{\rootsys}$ of the root 
arrangement in $\h$.

\subsection{De Concini--Procesi associators}\label{ss:DCP}

Let $\F,\G\in\Mns{\dgr}$ be two \mnss and $\DCPS{\F},\DCPS{\G}$
be the corresponding fundamental solutions given by Theorem \ref {thm:DCPsol}. Define the {\it \DCP associator} $\DCPAC{\nabla}{\F}{\G}$ to be the invertible element of $\DCPHAH{\rootsys}$ defined by
\[\DCPS{\G}(x)=\DCPS{\F}(x)\cdot\DCPAC{\nabla}{\F}{\G}\]
where $x$ lies in the fundamental Weyl chamber. The following summarises the 
essential properties of these associators.

\begin{theorem}
	Let $\F,\G,\H\in\Mns{\dgr}$. Then, the following properties hold.
	\begin{enumerate}\itemsep0.25cm
		\item {\bf Transitivity:} $\DCPAC{\nabla}{\F}{\G}=\DCPAC{\nabla}{\F}{\H}\DCPAC{\nabla}{\H}{\G}$.
		\item {\bf Support:} $\DCPAC{\nabla}{\F}{\G}\in\DCPHAH{\supp(\F,\G)}^{\;\zsupp(\F,\G)}$.
		\item {\bf Forgetfulness:} $\DCPAC{\nabla}{\F}{\G}=\DCPAC{\nabla}{\F'}{\G'}$ whenever $(\F,\G), (\F',\G')$ are equivalent.
	\end{enumerate}
\end{theorem}

\begin{pf}
Transitivity follows directly from the definition of $\DCPAC{\nabla}{\F}{\G}$.
The proof of the properties of support and forgetfulness is identical to those in \cite[Thm. 3.6]{DCP} and \cite[Thm. 1.33, Prop. 1.38]{vtl-4} and is therefore omitted.
\end{pf}

\subsection{Pre--Coxeter structure}\label{ss:hol-pre-cox}

By Proposition \ref{ss:comp-diag-sub}, the holonomy algebra gives rise to a diagrammatic algebra $\DCPHA{}=\{\DCPHA{B}\}$ and
a bidiagrammatic algebra $\DCPHA{}^\flat=\{\DCPHA{BB'}\}$, where 
$\DCPHA{BB'}\subseteq\DCPHA{B}$ is the centraliser of $\DCPHA{B'}$.
Both structures are compatible with the grading, and we denote by
$\DCPHAH{}$ (resp. $\DCPHAH{}^\flat$) the lax diagrammatic  (resp. lax bidiagrammatic) algebras
corresponding to the algebras $\DCPHAH{B}$ (resp. $\DCPHAH{BB'}$).\\

\newcommand{\mnsext}[3]{{#1}^{#2}_{#3}}
\newcommand{\mnsextM}[1]{\mnsext{#1}{\M}{\M'}}
\newcommand{\mnsextN}[1]{\mnsext{#1}{\N}{\N'}}

Choose
$\M\in\Mns{\dgr,B}$, $\M'\in\Mns{B'}$, and let
\[\mnsextM{(-)}:\Mns{B,B'}\longrightarrow\Mns{\dgr}\]
be the map defined by $\mnsextM{\F}=\M\cup\F\cup\M'$. 
For any {\em relative} maximal nested sets $\F,\G\in\Mns{B,B'}$, we set $\DCPAC{\nabla}{\F}{\G}=\DCPAC{\nabla}{\mnsextM{\F}}{\mnsextM{\G}}$,
which we also refer to as a De Concini--Procesi associator, with a slight abuse of terminology. 
Note that, by the forgetfulness property, $\DCPAC{\nabla}{\F}{\G}$ is independent of the choice of
$\M$ and $\M'$, and it is therefore well--defined. Moreover, by the support property, $\DCPAC{\nabla}{\F}{\G}$ is an invertible 
element in $\DCPHAH{BB'}$. 
Finally, the following holds by construction.
\begin{itemize}\itemsep0.25cm
	\item[(a)] For any $B'''\subseteq B''\subseteq B'\subseteq B$, 
	$\M\in\Mns{B,B'}$, $\F,\G\in\Mns{B',B''}$, $\M''\in\Mns{B'',B'''}$,
	\[\DCPAC{\nabla}{(\M\cup\F\cup\M'')}{(\M\cup\G\cup\M'')}=\DCPAC{\nabla}{\F}{\G}\]
	\item[(b)] For any $B'\subseteq B\perp C$,
	\[\DCPAC{\nabla}{(\F\cup\{C\})}{(\G\cup\{C\})}=\DCPAC{\nabla}{\F}{\G}\]
\end{itemize}
\begin{theorem}\label{thm:HAqC}
	The collection of De Concini--Procesi associators $\DCPAC{\nabla}{\F}{\G}$, for any 
	$\F,\G\in\Mns{B,B'}$ with $B'\subseteq B\subseteq\dgr$, defines an $\sfa$--strict pre--Coxeter
	structure on $\DCPHAH{}^{\flat}$.
\end{theorem}

\begin{pf}
	We shall show that the associators $\DCPAC{\nabla}{\F}{\G}$
	satisfy the requirements of Definition \ref{ss:pre-cox-diag-alg}.
	\begin{itemize}\itemsep0.25cm
		\item {\bf Horizontal factorisation.} This follows from transitivity.
		\item {\bf Vertical factorisation.}
		For any $B''\subseteq B'\subseteq B$,
		\[\F,\G\in\Mns{B,B'}\aand\F',\G'\in\Mns{B',B''}\]
		one has
		\begin{align}
			\DCPAC{\nabla}{(\G\cup\G')}{(\F\cup\F')}=
			\DCPAC{\nabla}{(\G\cup\G')}{(\F\cup\G')}\cdot
			\DCPAC{\nabla}{(\F\cup\G')}{(\F\cup\F')}=
			\DCPAC{\nabla}{\G}{\F}\cdot\DCPAC{\nabla}{\G'}{\F'}
		\end{align}
		where the first identity follows from transitivity and the
		second one from (a) above. Note that, since $\DCPAC{\nabla}{\G}{\F}\in\DCPHAH{BB'}$ and 
		$\DCPAC{\nabla}{\G'}{\F'}\in\DCPHAH{B'}$, the order of factors does 
		not matter.

		\item {\bf Orthogonal factorisation.}
		For any $B_1^\prime\subseteq B_1\perp B_2
		\supseteq B'_2$, and pairs
		\begin{gather*}
			\phantom{XXXX}(\G_1, \G_2), (\F_1,\F_2)\in\Mns{B_1,B'_1}\times\Mns{B_2,B'_2}=\Mns{B_1\sqcup B_2,B'_1\sqcup B'_2}
		\end{gather*}
		one has $(\G_1,\G_2)=(\G_1, B_2)\cup(B'_1, \G_2)$ and
		$(\F_1,\F_2)=(\F_1, B_2)\cup(B'_1, \F_2)$, hence
		\[
		\DCPAC{\nabla}{(\G_1,\G_2)}{(\F_1,\F_2)}=
		\DCPAC{\nabla}{(\G_1, B_2)}{(\F_1, B_2)}\cdot\DCPAC{\nabla}{(B'_1, \G_2)}{(B'_1,\F_2)}
		=\DCPAC{\nabla}{\G_1}{\F_1}\cdot\DCPAC{\nabla}{\G_2}{\F_2}
		\]
		where the first identity follows from vertical factorisation and the second one
		from (b) above. Therefore, orthogonal factorisation holds. 
		Note that, since $[\DCPHA{B_1}, \DCPHA{B_2}]=0$, the order 
		of factors does not matter.
	\end{itemize}
\end{pf}

\subsection{Coxeter structure on $\Br{W}\ltimes\DCPHAH{\Rs{},\h}$}\label{ss:hol-cox}
As for $\DCPHAH{\Rs{}}$ and $\DCPHAH{\Rs{}}^\flat$, the extended holonomy algebra 
gives rise to a bidiagrammatic algebra, which we denote by $\Br{W}\ltimes\DCPHAH{\Rs{},\h}$ and is described 
by the collection of algebras
\begin{equation}
\rda{(\Br{W}\ltimes\DCPHAH{\Rs{},\h})}{}{BB'}=
\Br{W_B}\ltimes\left(\DCPHA{\Rs{},BB'}\wh{\ten}{S}\h'_B\right)
\end{equation}
where $B'\subseteq B\subseteq\dgr$,  $\h'_{B}=\mathsf{span}\{\cor{i}\;|\; i\in B\}$, 
$W_B=\langle s_i\;|\; i\in B\rangle\subseteq W$, and the action of $\Br{W_B}$ is given by Definition~\ref{def:ext-holo-alg}. 
Indeed, it is enough to observe that, if $B'\perp B''$, then 
$\DCPHAH{B'}$ is pointwise fixed by $W_{B''}$. Finally, we have the following

\begin{theorem}
\hfill
\begin{enumerate}\itemsep0.25cm
\item 
The De Concini--Procesi associators $\DCPAC{\nabla}{\F}{\G}\in\DCPHAH{\Rs{}}$ and the elements 
\begin{equation}\label{eq:loc-mono-univ}
\CoxSk{\nabla}{}{i}= \topS{i}\cdot\exp(\pi\iota\KhC{\alpha_i})
\in \Br{W_i}\ltimes\left(\DCPHA{\Rs{},i}\wh{\ten}{S}\h_i\right)
\end{equation}
where $\KhC{\alpha_i}=\Kh{\alpha_i}{}+\hinv{\root{i}}/2$
define an $\sfa$--strict Coxeter structure $\wt{\sCox{}}$ on $\Br{W}\ltimes\DCPHAH{\Rs{},\h}$ with respect to the
standard labelling on $\dgr$ (\ie $m_{ij}=\mathsf{ord}(s_is_j)$ in $W$).
\item
The action of $\Br{W}$ arising from $\wt{\sCox{}}$ coincides with the monodromy
of the universal Casimir connection from Section \ref{s:holo-Cox-rep}, given in terms 
of the fundamental solutions $\Psi_{\F}$.
\end{enumerate}
\end{theorem}

\begin{pf}
Let $\wt{\sfX}\stackrel{p}{\rightarrow}\sfX$ be the universal cover of $\sfX$,
fix a $\wt{x}_0\in\wt{\sfX}$ which lifts $x_0\in\sfX$ and a fundamental solution 
$G$ of $p^*\nablat$ valued in $\DCPHAH{\Rs{},\h}$. Let $\mu_G(\gamma)\in\DCPHAH{\Rs{},\h}$
be its monodromy along the path $\gamma$. Then, by Corollary~\ref{ss:twisting-bundle} (1), 
for any $W$--invariant function $\bfa$, we obtain a representation $\mu_{G,\bfa}^{\sharp}:\WPoid{(\sfX; Wx_0)}\to\Br{W}\ltimes\DCPHAH{\Rs{},\h}$,
given, for any $\gamma:x_0\to w^{-1}x_0$, by
	\begin{equation}
		\mu^{\sharp}_{G,\bfa}(w,\gamma)=(P(w,\gamma), \mu_{G,\bfa}(\gamma))
	\end{equation}
where $\mu_{G,\bfa}(\gamma)=\mu_{G}(\gamma)\bwc(\gamma)\bw_{\bfa}(\gamma)$.\\
		
Let $\F\in\Mns{B}$ with $\{i\}\in\F$. Then, by choosing $G=\Psi_{\F}$ and $\bfa$ such that 
$a_i=\pi\iota$,
it follows from \eqref{eq:monodromy-rep} that $\mu^\sharp_{\DCPS{\F}, \pi\iota}(s_i, \gamma_i)=\CoxSk{\nabla}{}{i}$.
Moreover, if $\G\in\Mns{B}$ with $\{j\}\in\G$, then
\[
\mu^\sharp_{\DCPS{\F}, \pi\iota}(s_j, \gamma_j)=
\DCPAC{\nabla}{\F}{\G}\cdot\mu^\sharp_{\DCPS{\G},\pi\iota}(s_j, \gamma_j)\cdot(\DCPAC{\nabla}{\F}{\G})^{-1}=
\sfAd{\DCPAC{\nabla}{\F}{\G}}(\CoxSk{\nabla}{}{j})
\]
In particular, the elements $\CoxSk{\nabla}{}{i}$ and the associators $\DCPAC{\nabla}{\F}{\G}$ satisfy
the braid relations \eqref{eq:braid-Cox-diag-alg} and define a Coxeter structure
on $\Br{W}\ltimes\DCPHAH{\Rs{},\h}$. The results follow.
\end{pf}

\noindent\remark\;
Following Remark~\ref{ss:holo-Cox-correction} (2), we further adjust the monodromy operators
by setting $\CoxS{\nabla}{}{i}=\CoxSk{\nabla}{}{i}\cdot\exp(\pi\iota\symdi{i}\hinv{\root{i}}^2/4)$. 
This yields another $\redasso{}{}$--strict Coxeter structure $\sCox{}=(\DCPAC{\nabla}{\F}{\G}, \CoxS{\nabla}{}{i})$ 
on $\Br{W}\ltimes\DCPHAH{\Rs{},\h}$, encoding the monodromy representation $\mu^\sharp_{\DCPS{\F}, \pi\iota,\pi\iota\iip{\root{i}}{\root{i}}}$.
Indeed, it is enough to observe that, since the associators $\DCPAC{\nabla}{\F}{\G}$ are supported on $\DCPHAH{\Rs{}}$, 
the generalised braid relations \eqref{eq:braid-Cox-diag-alg} still hold.  
We shall show in Section~\ref{s:km-holo} that, given a representation $V$ of $\g$, the operator $\CoxS{\nabla}{}{i}$  (resp. $\CoxSk{\nabla}{}{i}$)
specialises on $V$ to the local monodromy operators 
$\texp{i}\cdot\exp(\pi\iota\nablah\cdot C_{i}/2)$  (resp. $\texp{i}\cdot\exp(\pi\iota\nablah\cdot\Ku{i}{}/2)$) 
where  $C_{i}$ denote the Casimir operator of $U\sl{2}^i$ (resp. $\Ku{i}{}=C_{i}-\symdi{i}\hinv{\root{i}}^2/2$).

\part{The KZ--Casimir connection}\label{part-two}

\section{Cosimplicial diagrammatic algebras}\label{s:cosimplicial-diagrammatic}

We describe (bi)diagrammatic algebras endowed with a compatible 
cosimplicial structure.

\subsection{Cosimplicial algebras}\label{ss:cosim}
In the following, we shall consider a number of algebras, which are not
bialgebras, but fit instead in the more general setting of cosimplicial algebras.

A {cosimplicial algebra} is a cosimplicial object in the category
of algebras,
\ie the datum 
\[A^{\scsop{\bullet}}\colon\;
	\xymatrix@C=.5cm{A^0\ar@<-2pt>[r]\ar@<2pt>[r] & 
	A^1 \ar@<3pt>[r] \ar@<0pt>[r] \ar@<-3pt>[r] & 
	A^2 \ar@<-6pt>[r]\ar@<-2pt>[r]\ar@<2pt>[r]\ar@<6pt>[r] & A^3
	\quad\cdots}\]
of a collection of algebras $\{A^n\}_{n\geqslant 0}$ endowed with 
{\em face morphisms} $d^{n+1}_i: A^n\to A^{n+1}$, $i=0, 1, \dots, n+1$, and 
{\em degeneration morphisms} $\varepsilon^n_i: A^n\to A^{n-1}$, $i=1,\dots, n$, such that 
\[
\begin{array}{lr}
	d^{n+1}_jd^n_i=d_i^{n+1}d^n_{j-1} & (i<j) \\[6pt]
	\varepsilon^n_j\varepsilon^{n+1}_i=\varepsilon^n_i\varepsilon^{n+1}_{j+1} & (i\leqslant j)
\end{array}
\aand
\varepsilon^{n+1}_jd_i^{n+1}=\left\{
\begin{array}{cl}
	d^{n}_i\varepsilon^n_{j-1} & i<j\\[3pt]
	\id & i=j,j+1\\[3pt]
	d^{n}_{i-1}\varepsilon^n_j & i>j+1
\end{array}
\right.
\]
\noindent\remarks\;
\hfill
\begin{enumerate}\itemsep0.25cm
	\item For any $x\in A^n$, we write 
	$x_{1,\dots, ii+1, \dots, n+1}= d^{n+1}_{i}(x)\in A^{n+1}$, with $i=1,\dots, n$.
	Moreover, for any $m>0$, we denote by $x_{1,\dots, n}\in A^{n+m}$ the element
	$d_{n+m}^{n+m}\circ d_{n+m-1}^{n+m-1}\circ\cdots\circ d_{n+1}^{n+1}(x)$.
	\item
	For any $n\geqslant 1$, we consider denote by $d^{(n)}: A^1\to A^n$ the canonical 
	morphism recursively defined by $d^{(1)}=\id_{A^1}$ and $d^{(n)}= 
	d^n_1\circ d^{(n-1)}$.\\
\end{enumerate}

Henceforth, we shall further assume that $A^{\scsop{\bullet}}$ is endowed
with an action of the symmetric group in each degree such that, 
for any $x\in A^n$ and $1\leqslant i\leqslant n-1$, 
\[
d^{n+1}_j((i\,i+1)\cdot x)=
\left\{
\begin{array}{cl}
(i\, i+1)\cdot d^{n+1}_j(x) & j< i\\
(i\,i+1\,i+2)\cdot d^{n+1}_{i+1}(x) & j=i\\
(i\,i+2\,i+1)\cdot d^{n+1}_{i}(x) & j=i+1\\
(i+1\, i+2)\cdot d^{n+1}_j(x) & j> i+1
\end{array}
\right.
\]
and
\[
\varepsilon^{n}_j((i\,i+1)\cdot x)=
\left\{
\begin{array}{cl}
	(i\, i+1)\cdot \varepsilon^{n}_j(x) & j< i\\
	\varepsilon^{n}_{i+1}(x) & j=i\\
	\varepsilon^{n}_{i}(x) & j=i+1\\
	(i-1\, i)\cdot \varepsilon^{n}_j(x) & j> i+1
\end{array}
\right.
\]
For $x\in A^n$ and $\sigma\in\SS_n$,
we write $x_{\sigma(1),\dots, \sigma(n)}=\sigma\cdot x$.

\subsection{Examples}\label{ss:cosim-example}
Let $A$ be a bialgebra over a base ring $\sfk$ with 
coproduct $\Delta: A\to A^{\ten 2}$ and
counit $\varepsilon: A\to\sfk$. We provide two basic examples of cosimplicial algebras
associated to $A$.
\begin{enumerate}[leftmargin=2em]\itemsep0.25cm
\item 
The tower of algebras $A^n= A^{\ten n}$, $n\geqslant 0$, is a cosimplicial algebra with face morphisms 
\[
d_i^{n+1} (x)=
\left\{
\begin{array}{cc}
	1\ten x & i=0\\[1.1ex]
	\id_A^{\ten i-1}\ten\Delta\ten\id_A^{\ten n-i}(x) & 1\leqslant i \leqslant n\\[1.1ex]
	x\ten 1 & i=n+1
\end{array}
\right.
\]
and degeneration morphisms 
\[\varepsilon^n_i(x)=\id_A^{\ten i-1}\ten\varepsilon\ten\id_A^{\ten n-i}(x)\]
where $x\in A^{\ten n}$. In particular, $d^2_1=\Delta$ and $\varepsilon^1_1=\varepsilon$.

\item\label{ss:cosim-example-cat}
Let $\C\subseteq\Rep(A)$ be a tensor subcategory, and $\sff^{\boxtimes n}:\C^{\boxtimes n}\to\vect$ the $n$--fold forgetful functor given by $\sff^{\boxtimes n}(V_1,\dots, V_n)= V_1\ten\cdots\ten V_n$. The tower of algebras 
$\A^n=\sfEnd{\sff^{\boxtimes n}}$ gives rise to a cosimplicial algebra
with face and degeneration morphisms induced by the tensor product and the 
unit of $\C$ and defined as follows.
\vspace{0.25cm}
\begin{itemize}\itemsep0.25cm
	\item
	The face morphisms $d_i^{n+1}:\sfEnd{\ff^{\boxtimes n}}\to\sfEnd{\ff^{\boxtimes n+1}}$,
	$i=0,\dots, n+1$, are given by
	\[(d_0^1 \varphi)_X: 
	\xymatrix{
		\ff(X)\ar[r] & \ff(X)\ten{\bf 1} \ar[r]^{1\ten\varphi} & \ff(X)\ten{\bf 1} \ar[r] & \ff(X)
	}
	\]
	\[(d_1^1 \varphi)_X: 
	\xymatrix{
		\ff(X)\ar[r] & {\bf 1}\ten\ff(X) \ar[r]^{\varphi\ten1} & {\bf 1}\ten\ff(X)  \ar[r] & \ff(X)
	}
	\]
	where ${\bf 1}$ is the trivial module, $X\in\C$, $\varphi\in\sfk$, and, 
	\[
	(d_i^{n+1} \varphi)_{X_1,\dots, X_{n+1}}=
	\left\{
	\begin{array}{ll}
		\id\ten \varphi_{X_2,\dots, X_{n+1}} & i=0\\
		\varphi_{X_1,\dots, X_i\ten X_{i+1},\dots X_{n+1}} & 1\leqslant i \leqslant n\\
		\varphi_{X_1,\dots, X_n}\ten\id & i=n+1
	\end{array}
	\right.
	\]
	where $\varphi\in\sfEnd{\ff^{\boxtimes n}}$, $X_j\in\C$, $j=1,\dots, n+1$.
	\item
	The degeneration morphisms $\varepsilon^n_i:\sfEnd{\ff^{\boxtimes n}}\to\sfEnd{\ff^{\boxtimes n-1}}$, for 
	$i=1,\dots, n$, are
	\[(\varepsilon^n_i \varphi)_{X_1, \dots, X_{n-1}}=\varphi_{X_1,\dots, X_{i-1}, {\bf 1}, X_{i}, \dots, X_{n-1} }\]
	where $\varphi\in\sfEnd{\ff^{\boxtimes n}}$, $X_j\in\C$, $j=1,\dots, n-1$.
\end{itemize}
\end{enumerate}

\noindent\remark\; 
Note that there is a natural morphisms of cosimplicial algebras
$A^{\scsop{\otimes\bullet}}\to\A^{\scsop{\bullet}}$, so that the latter 
can be regarded as a topological completion of the former. However, in 
Sections~\ref{ss:cosimplicial-KZ} and \ref{ss:univ-cosimp}, we shall consider 
certain cosimplicial algebras which do not arise from topological bialgebras.

\subsection{Cosimplicial diagrammatic algebras}\label{ss:diag-cosim}

A cosimplicial (lax) diagrammatic algebra is a cosimplicial object in the category 
of (lax) diagrammatic algebras, \ie the datum of a collection of (lax) diagrammatic algebras
$\{A^n\}_{n\geqslant0}$ endowed with the face and degeneration maps, which are  further
required to be morphisms of diagrammatic algebras.

Given a (lax) diagrammatic {\em bialgebra}
$\ACox{}= (\rda{A}{}{B}, \rdm{i}{}{}{B'}{B}, \rdm{j}{}{}{B''}{B'})$, it is clear that
\[\ACox{}^{\ten n}=(\rda{A}{\ten n}{B}, \rdm{i}{\ten n}{}{B'}{B}, \rdm{j}{\ten n}{}{B''}{B'})\]
is a (lax) diagrammatic algebra for any given $n\geqslant 0$.\footnote{
	More precisely, $\rdm{j}{\ten n}{}{B_2}{B_1}:(\rda{A}{}{B_1}\ten\rda{A}{}{B_2})^{\ten n}\to\rda{A}{\ten n}{B_1\sqcup B_2}$. By abuse of notation, we omit the identification
	$(\rda{A}{}{B_1}\ten\rda{A}{}{B_2})^{\ten n}\simeq\rda{A}{\ten n}{B_1}\ten\rda{A}{\ten n}{B_2}$
	and we denote by $\rdm{j}{\ten n}{}{B_2}{B_1}$ the morphism 
	$\rda{A}{\ten n}{B_1}\ten\rda{A}{\ten n}{B_2}\to\rda{A}{\ten n}{B_1\sqcup B_2}$.}
Moreover, the collection of morphisms $\Delta_B:\rda{A}{}{B}\to\rda{A}{\ten 2}{B}$ and
$\varepsilon_B:\rda{A}{}{B}\to\sfk$, with $B\subseteq\dgr$, define a 
cosimplicial structure on $\{\ACox{}^{\ten n}\}_{n\geqslant 0}$, and we denote by 
$\ACox{}^{\ten \bullet}$ the resulting cosimplicial diagrammatic algebra.

\subsection{Cosimplicial bidiagrammatic algebras}\label{ss:bi-diag-cosim}

\begin{definition}
	A \emph{cosimplicial (lax) bidiagrammatic algebra} is a 
	cosimplicial object in the category of (lax) bidiagrammatic algebras 
	such that, for any $B,C\subseteq\dgr$, the map
	\begin{equation}
		m\circ\left( i_1\ten i_2\right)\circ(d^{(n)}_C\ten\id): \rda{(A^1)}{\emptyset}{C}\ten\rda{(A^{n+1})}{C}{B}\to\rda{(A^{n+1})}{\emptyset}{B}
	\end{equation}
	is a morphism of algebras,
	where $m$ is the multiplication in $\rda{(A^{n+1})}{\emptyset}{B}$,
	$i_1=\rdm{(i^{n+1})}{\emptyset}{\emptyset}{C}{B}$, $i_2=\rdm{(i^{n+1})}{\emptyset}{C}{B}{B}$,
	and
	$d^{(n)}_C: \rda{(A^1)}{\emptyset}{C}\to\rda{(A^{n+1})}{\emptyset}{C}$ is 
	defined as in Remark~\ref{ss:cosim} (2).
\end{definition}

The definition above generalises the following situation. Let 
$A$ be a bialgebra with a distinguished subbialgebra $A'\subseteq A$. 
By definition, the subalgebra $(A^{\ten n})^{A'}$ of (diagonal) $A'$--invariants
in $A^{\ten n}$ satisfies
\begin{equation}
	\left[\Delta^{(n)}(A'), (A^{\ten n})^{A'}\right]=0
\end{equation}
Indeed, we have the following

\begin{proposition}
	Let $\ACox{}=(\rda{A}{}{B},\rdm{i}{}{}{B'}{B}, \rdm{j}{}{}{B''}{B'})$ be a diagrammatic bialgebra.
	\begin{enumerate}\itemsep0.25cm
		\item For any $n\geqslant 0$, set
		\begin{eqnarray*}
			\rda{(A^{\ten n, \flat})}{C}{B}&=& (A^{\ten n}_B)^{A_C}\subseteq A^{\ten n}_B\\
			\rdm{(i^{\ten n, \flat})}{C}{C'}{B'}{B}&=&\rdm{i}{\ten n}{}{B'}{B}|_
			{(\rda{A}{\ten n}{B'})^{\rda{A}{}{C'}}}\\
			\rdm{(j^{\ten n, \flat})}{C'}{C''}{B''}{B'}&=&\rdm{j}{\ten n}{}{B''}{B'}|_
			{(\rda{A}{\ten n}{B'})^{\rda{A}{}{C'}}
				\ten 
				(\rda{A}{\ten n}{B''})^{\rda{A}{}{C''}}}
		\end{eqnarray*}
		where we regard $(\rda{A}{\ten n}{B'})^{\rda{A}{}{C'}}\ten 
		(\rda{A}{\ten n}{B''})^{\rda{A}{}{C''}}$ as a subalgebra in 
		$(\rda{A}{}{B'}\ten\rda{A}{}{B''})^{\ten n}$.
		Then
		\[\ACox{}^{\ten n, \flat}=(\rda{(A^{\ten n,\flat})}{C}{B},\rdm{(i^{\ten n, \flat})}{C}{C'}{B'}{B}, \rdm{(j^{\ten n, \flat})}{C'}{C''}{B''}{B'})\]
		is a bidiagrammatic algebra.
		\item The morphisms $\Delta_B:\rda{A}{}{B}\to\rda{A}{\ten 2}{B}$ and
		$\varepsilon_B:\rda{A}{}{B}\to\sfk$, with $B\subseteq\dgr$, define a 
		cosimplicial structure on $\{\ACox{}^{\ten n, \flat}\}_{n\geqslant 0}$, and we denote by 
		$\ACox{}^{\ten \bullet, \flat}$ the resulting cosimplicial bidiagrammatic algebra.
	\end{enumerate}
\end{proposition}

\noindent\remarks\;
\hfill
\begin{itemize}\itemsep0.25cm
	\item Note that $\ACox{}^{\ten n,\flat}$ contains, but does not coincide with, 
	the bidiagrammatic algebra $(\ACox{}^{\ten n})^{\flat}$ defined using Proposition \ref{ss:inv-bidiag}. 
	The difference is the same as that between the subalgebras $(A^{\ten n})^{A'}$ and 
	$(A^{\ten n})^{(A')^{\ten n}}$ in $A^{\ten n}$ for any bialgebra $A$ with a distinguished 
	subbialgebra $A'$.
	\item Let $A$ be a cocommutative bialgebra. 
	The canonical action of the symmetric group $\SS_n$ on $A^{\ten n}$ preserves 
	the subalgebra $(A^{\ten n})^{A'}$, since for any $\sigma\in\SS_{n}$ it holds
	$\sigma\circ\Delta^{(n-1)}=\Delta^{(n-1)}$. Similarly, whenever $\ACox{}$ is a bidiagrammatic 
	cocommutative bialgebra, the symmetric group $\SS_n$ acts on $\ACox{}^{\ten n, \flat}$ by bidiagrammatic automorphisms.
\end{itemize}

\noindent\example
Let $\g$ be a diagrammatic Kac--Moody algebra (cf.~\ref{ss:diag-KM}). Then, $U\g$ is a lax diagrammatic 
Hopf algebra and $\Udiag{\g}{\bullet}= U\g^{\ten \bullet, \flat}$ is the 
cosimplicial bidiagrammatic algebra with  face/degeneration maps induced by the Hopf 
algebra structure on $U\g$ and bidiagrammatic subalgebras $(U\g_B^{\ten n})^{\g_C}$, 
$C\subseteq B\subseteq\dgr$.


\section{Braided Coxeter algebras}\label{s:braided-Cox-alg}

The notion of a braided Coxeter algebra arises from the combination of a
quasitriangular and a Coxeter structure on a cosimplicial bidiagrammatic algebra.
In particular, it is naturally endowed with commuting actions of the (type $\sfA$)
braid groups $\Br{n}$ and a fixed generalised braid group $\Br{W}$.

\subsection{Braided Coxeter algebras}\label{ss:braid-cox-alg}

Let $(\dgr,\ulm)$ be a labelled diagram. Let $\ACox{}^{\scsop{\bullet}}$ be a cosimplicial 
(lax) bidiagrammatic algebra, satisfying the condition \eqref{eq:cent condition} 
in degree one.

\begin{definition}
	A \emph{braided Coxeter structure} $\sCox{}=(\Phi_B,R_B, \Jg{}{\F}{},\DCPA
	{\F}{\G},\redasso{\F}{\F'}, \CoxS{}{}{i})$ on $\ACox{}^{\scsop{\bullet}}$ consists of the following data. 
	\vspace{0.25cm}
	\vspace{0.25cm}\begin{enumerate}[label=(\alph*)]\itemsep0.25cm
		\item 
		{\bf Associators.} 
		For any $B\subseteq \dgr$, an invertible element $\Phi_B\in\rda{(A^3)}{B}{B}$
		satisfying the following properties
		\vspace{0.25cm}
		\begin{itemize}\itemsep0.25cm
			\item {\bf Pentagon relation.}
			\[
			(\Phi_B)_{1,2,34}(\Phi_B)_{12,3,4}=(\Phi_B)_{2,3,4}(\Phi_B)_{1,23,4}(\Phi_B)_{1,2,3}
			\]
			\item {\bf Degeneration.} For $i=1,2,3$, $\varepsilon^3_i(\Phi_B)=1_{\rda{(A^2)}{B}{B}}$.
			\item {\bf Orthogonal factorisation.}
			If $B_1\perp B_2$,
			\begin{eqnarray*}
				\Phi_{B_1\sqcup B_2}&=&
				\rdm{(j^3)}{B_1}{B_2}{B_2}{B_1}(\Phi_{B_1}\ten\Phi_{B_2})
			\end{eqnarray*}
		\end{itemize}
		\item 
		{\bf $R$--matrices.}
		For any $B\subseteq\dgr$, an invertible element $R_B\in\rda{(A^2)}{B}{B}$
		satisfying the following properties
		\vspace{0.25cm}
		\begin{itemize}\itemsep0.25cm
			\item {\bf Hexagon relations.}
			\begin{align*}
				&(R_B)_{12,3}=(\Phi_B)_{3,1,2}(R_B)_{13}(\Phi_B)_{1,3,2}^{-1}(R_B)_{23}(\Phi_B)_{1,2,3}\\
				&(R_B)_{1,23}=(\Phi_B)_{2,3,1}^{-1}(R_B)_{13}(\Phi_B)_{2,1,3}(R_B)_{12}(\Phi_B)_{1,2,3}^{-1}
			\end{align*}
			\item {\bf Degeneration.} For $i=1,2$, $\varepsilon^2_i(R_B)=1_{\rda{A}{B}{B}}$.
			\item {\bf Orthogonal factorisation.}
			If $B_1\perp B_2$,
			\begin{eqnarray*}
				R_{B_1\sqcup B_2}&=&
				\rdm{(j^2)}{B_1}{B_2}{B_2}{B_1}(R_{B_1}\ten R_{B_2})
			\end{eqnarray*}
		\end{itemize}
		\item 
		{\bf Relative twists.} For any $B^\prime\subseteq B$ and \mns $\F\in\Mns
		{B,B^\prime}$, an invertible element $\Jg{}{\F}{}\in\rda{(A^2)}{B'}{B}$ satisfying 
		the following properties.
		\vspace{0.25cm}
		\begin{itemize}\itemsep0.25cm
			\item {\bf Compatibility with associators.}
			The relative twist equation holds,
			\begin{equation}\label{eq:rel twist}
				J_{\F, 1,23}\cdot J_{\F, 23}\cdot \Phi_{B'} = \Phi_{B}\cdot J_{\F, 12,3}\cdot J_{\F, 12}
			\end{equation}
			\item {\bf Normalisation.} For any $B\subseteq D$, $J_B=1_{\rda{(A^1)}{B}{B}}$.\footnote
			{Here $B$ is identified with the unique element in $\Mns{B,B}$.}
			\item {\bf Degeneration.} For $i=1,2$, $\varepsilon_i^2(\Jg{}{\F}{})=1_{\rda{(A^1)}{B}{B'}}$.
			\item {\bf Orthogonal factorisation.}
			If $B_1^\prime\subseteq B_1\perp B_2\supseteq B'_2$,
			$(\F_1,\F_2)\in\Mns{B_1\sqcup B_2,B'_1\sqcup B'_2}$, 
			\begin{eqnarray*}
				J_{(\F_1,\F_2)}&=&
				\rdm{(j^2)}{B_1'}{B_2'}{B_2}{B_1}(J_{\F_1}\ten J_{\F_2})
			\end{eqnarray*}
		\end{itemize}
		\item 
		{\bf Generalised associators.}
		For any $B'\subseteq B$ and $\F,\G\in\Mns{B,B'}$, an invertible element
		$\DCPA{\G}{\F}\in\rda{(A^1)}{B'}{B}$ satisfying the properties from Definition
		\ref{def:pre-cox-rda} and the following
		\vspace{0.25cm}
		\begin{itemize}\itemsep0.25cm
			\item {\bf Compatibility with $J$.} For any $\F,\G\in\Mns{B,B'}$,  
			\[\Jg{}{\G}{}=\gaugeJ{\DCPA{\G}{\F}}{\Jg{}{\F}{}}\]
		\end{itemize}
		\item 
		{\bf Vertical joins.}
		For any $B''\subseteq B'\subseteq B$, $\F\in\Mns{B,B'}$, and $\F'\in\Mns{B',B''}$, an
		invertible element $\redasso{\F}{\F'}\in\rda{(A^1)}{B''}{B}$ satisfying the same properties 
		from Definition \ref{def:pre-cox-rda} and the following
		\vspace{0.25cm}
		\begin{itemize}\itemsep0.25cm
			\item {\bf Compatibility with $J$ (\emph{vertical $J$--factorisation}).}
			\[\Jg{}{\F'\cup\F}{}=\gaugeJi{\redasso{\F}{\F'}}
			{\rdm{(i^2)}{B''}{B'}{B}{B\phantom{''}}(\Jg{}{\F}{})\cdot 
				\rdm{(i^2)}{B''}{B''}{B'}{B\phantom{''}}(\Jg{}{\F'}{})}\]
		\end{itemize}
		\item {\bf Local monodromies.} For any vertex $i$ of $\dgr$,
		an invertible element $S_i\in\rda{(A^{1})}{\emptyset}{i}$  satisfying the braid 
		relations \eqref{eq:braid-Cox-diag-alg} and the following
		\vspace{0.25cm}
		\begin{itemize}\itemsep0.25cm
			\item {\bf Coproduct identity.} For any $i\in D$, 
			\begin{equation}\label{eq:coxcoprod-diag-alg}
				J_i^{-1}\cdot (S_i)_{12}\cdot J_i=
				J_i^{-1}\cdot R_{i, 21}\cdot J_{i,21}\cdot (S_i)_1\cdot (S_i)_2
			\end{equation}
		\end{itemize}
	\end{enumerate}
\end{definition}

\noindent\remarks\hfill 
\begin{itemize}\itemsep0.25cm
\item
The relations above readily imply the following.
\vspace{0.25cm}\begin{enumerate}\itemsep0.25cm
	\item If $B'\subseteq B\perp B''$ and $\F\in\Mns{B,B'}$, 
	\begin{eqnarray*}
		J_{(\F, B'')}&=&\rdm{(j^2)}{B'}{B''}{B''}{B\phantom{'}}(\Jg{}{\F}{}\ten 1_{\rda{(A^2)}{B''}{B''}})\\	
		\DCPA{(\F, B'')}{(\G,B'')}&=&
		\rdm{(j^1)}{B'}{B''}{B''}{B}(\DCPA{\F}{\G}\ten 1_{\rda{(A^1)}{B''}{B''}})
	\end{eqnarray*}
	\item If $B'_1\subseteq B_1\perp B_2\supseteq B'_2$, 
	$\F_1\in\Mns{B_1,B'_1}$, and $\F_2\in\Mns{B_2,B'_2}$, 
	\[\redasso{(\F_1,B_2)}{(B'_1,\F_2)}=
	1_{\rda{(A^1)}{B_1'\sqcup B_2'}{B_1\sqcup B_2}}=
	\redasso{(B_1,\F_2)}{(\F_1,B'_2)}\]
\end{enumerate}
\item
It is clear from the definition that a braided Coxeter algebra is a cosimplicial (lax) 
bidiagrammatic algebra with a Coxeter algebra in degree one and some further
compatible data in degree two and three.
\end{itemize}

\subsection{Representations of braid groups}\label{ss:braid-Cox-to-braid-rep}
Let $\Br{n}$ be the braid group associated to $\SS_n$, with 
generators $\topT{1}, \dots, \topT{n-1}$, and $\brac{n}$ the set of complete bracketing 
on the non--commutative monomial $x_1 x_2\cdots x_{n}$. The following is a straightforward 
generalisation of Proposition \ref{ss:Cox-to-braid-rep}.

\begin{proposition}
	Let $\ACox{}^{\scsop{\bullet}}$ be a braided Coxeter algebra. Then, there is a family of representations
	\[\lambda_{\F, b}:\BBm\times\Br{n}\to\sfAut{\rda{(A^n)}{\emptyset}{B}}\]
	labelled by $B\subseteq\dgr$, $\F\in\Mns{B}$, and $b\in\brac{n}$, which is uniquely determined by the conditions
	\vspace{0.25cm}\begin{enumerate}
		\item[(1)] $\lambda_{\F, b}(\topS{i})=\sfAd{\redasso{\trunc{\F}{}{i}}{\trunc{\F}{i}{}}}
		(\CoxS{}{}{i})_{1\dots n}$ if $\{i\}\in\F$.
		\item[(2)] $\lambda_{\G, b}=\sfAd{\DCPA{\G}{\F}}_{1\dots n}\circ\lambda_{\F, b}$.
	\end{enumerate}
	and
	\vspace{0.25cm}\begin{enumerate}
		\item[(3)] $\lambda_{\F, b}(\topT{i})=(i\ i+1)\circ (R_{B})_{i,i+1}$ if $b= x_1\cdots (x_ix_{i+1})\cdots x_n$.
		\item[(4)] $\lambda_{\F, b'}=\sfAd{\Phi_{B, b'b}}\circ\lambda_{\F, b}$.
	\end{enumerate}
\end{proposition}


\subsection{Twisting and gauging of braided Coxeter structures}\label{ss:twist-gauge-braided-Cox}
The notions of twisting and gauging of braided Coxeter structure extends those introduced in 
\ref{ss:twist-gauge-Cox}. 

\begin{definition}
	\hfill
	\vspace{0.25cm}\begin{enumerate}\itemsep0.25cm
		\item A {\it twist} $\tCox{}=(u_{\F},\twF{B})$ in $\ACox{}^{\scsop{\bullet}}$
		consists of the following data.
		\vspace{0.25cm}\begin{enumerate}\itemsep0.25cm
			\item For any $\F\in\Mns{B,B'}$, an invertible element 
			$u_{\F}\in\rda{(A^1)}{B'}{B}$ such that $\varepsilon^1_1(u_{\F})=1$
			and, if $B_1^\prime\subseteq B_1\perp B_2\supseteq B'_2$,
			$(\F_1,\F_2)\in\Mns{B_1\sqcup B_2,B'_1\sqcup B'_2}$,
			\begin{eqnarray*}
				u_{(\F_1,\F_2)}&=&
				\rdm{(j^1)}{B_1'}{B_2'}{B_2}{B_1}(u_{\F_1}\ten u_{\F_2})
			\end{eqnarray*}
			\item For any $B\subseteq\dgr$, an invertible element 
			$\twF{B}\in\rda{(A^1)}{B}{B}$ such that 
			$(\twF{B})_{21}=\twF{B}$, 
			$\varepsilon^2_i(\twF{B})=1_{\rda{(A^1)}{B}{B}}$, $i=1,2$,
			and, if $B_1\perp B_2$,
			\begin{eqnarray*}
				\twF{(B_1,B_2)}&=&
				\rdm{(j^2)}{B_1}{B_2}{B_2}{B_1}(\twF{B_1}\ten\twF{B_2})
			\end{eqnarray*}
		\end{enumerate}
		\item The {\it twisting} of a braided Coxeter structure
		$\sCox{}=(\Phi_B,R_B, \Jg{}{\F}{},\DCPA{\F}{\G}, \redasso{\F}{\F'}, \CoxS{}{}{i})$
		by a twist $\tCox{}=(u_{\F},\twF{B})$  is the braided Coxeter structure 
		\[\sCox{\tCox{}}=((\Phi_B)_{\tCox{}},(R_B)_{\tCox{}}, (\Jg{}{\F}{})_{\tCox{}},(\DCPA{\F}{\G})_{\tCox{}}, (\redasso{\F}{\F'})_{\tCox{}}, (S_i)_{\tCox{}})\]
		given by
		\begin{eqnarray*}
			(\Phi_B)_{\tCox{}}			&=&\twistAJ{(\twF{B})}{\Phi_B}\\
			(R_B)_{\tCox{}} &=& \twistRJ{(\twF{B})}{R_B}\\
			(\Jg{}{\F}{})_{\tCox{}}			&=&\gaugeJ{u_{\F}}{\rdm{i}{B'}{B}{B}{B}(\twF{B})^{-1}\cdot \Jg{}{\F}{}\cdot\rdm{i}{B'}{B'}{B'}{B}(\twF{B'})}
		\end{eqnarray*}
		and
		\begin{equation*}
			\begin{array}{rcccr}
				(\DCPA{\F}{\G})_{\tCox{}}		&=&u_{\F}^{-1}\cdot \DCPA{\F}{\G}\cdot u_{\G} &=& (\DCPA{\F}{\G})_{u}\\
				(\redasso{\F}{\F'})_{\tCox{}}	&=&u_{\F'\cup\F}^{-1}\cdot \redasso{\F}{\F'}\cdot u_{\F'}
				\cdot u_{\F} &=& (\redasso{\F}{\F'})_{u}\\
				(\CoxS{}{}{i})_{\tCox{}} &=& u_{\{i\}}^{-1}\cdot\CoxS{}{}{i}\cdot u_{\{i\}} &=& (\CoxS{}{}{i})_{u}
			\end{array}
		\end{equation*}
		We denote by $\ACox{\tCox{}}^{\scsop{\bullet}}$ the braided Coxeter algebra with twisted structure $\sCox{\tCox{}}$. 
		\item A {\it gauge} $\gCox{}=\{a_B\}$ in $\ACox{}^{\scsop{\bullet}}$ consists of
		an invertible element $a_B\in\rda{(A^1)}{B}{B}$ 
		for any $B\subseteq\dgr$, satisfying $\varepsilon_1^1(a_B)=1$
		and 
		\[
		a_{B_1\sqcup B_2}=\rdm{(j^1)}{B_1}{B_2}{B_2}{B_1}(a_{B_1}\ten a_{B_2})
		\]
		\item The {\it gauging} of a twist $\tCox{}=(u_{\F},\twF{B})$ by $\gCox{}$ is the
		twist $\tCox{\gCox{}}=((u_{\F})_{\gCox{}},(\twF{B})_{\gCox{}})$ given by
		\begin{align*}
			(u_{\F})_{\gCox{}}&=\rdm{(i^1)}{B'}{B'}{B'}{B\phantom{'}}(a_{B'})\cdot u_\F
			\cdot \rdm{(i^1)}{B'}{B}{B}{B\phantom{'}}(a_{B})^{-1}\\
			(\twF{B})_{\gCox{}}	&=\gaugeJ{a_B}{\twF{B}}
		\end{align*}
	\end{enumerate}
\end{definition}

The following is standard.

\begin{proposition}
	Let $\sCox{}$ be a braided Coxeter structure on $\ACox{}^{\scsop{\bullet}}$, $\tCox{}$ a 
	twist, and $a$ a gauge.
	Then, $\sCox{\tCox{}}=\sCox{\tCox{\gCox{}}}$. Moreover, the representations of the braid groups 
	$\lambda_{\F,b}^{\sCox{}}$ and $\lambda_{\F,b}^{\sCox{\tCox{}}}$, arising from $\sCox{}$ 
	and $\sCox{\tCox{}}$, respectively, are equivalent.
\end{proposition}


\section{The double holonomy algebra}\label{s:double-holo}

We proved in \ref{ss:hol-cox} that the holonomy algebra $\DCPHA{\rootsys}{}$ of
the Casimir connection $\nabla_{\sfC}$ is a bidiagrammatic algebra, and that it can
be endowed with a Coxeter structure encoding the monodromy of $\nabla_{\sfC}$.
In this section, we introduce the holonomy algebra $\DBLHA{\rootsys}{\bullet}$ of
the joint KZ--Casimir system 
and describe its cosimplicial bidiagrammatic structure. 

\subsection{The holonomy algebra of the KZ connection}\label{ss:KZ-holo}
Let $n\geqslant 2$.

\begin{definition}
	The holonomy algebra $\DKHA{}{n}$ is the associative algebra generated
	by the elements $\{\Th{ij} \;|\; 1\leqslant i\neq j\leqslant n\}$ with the following
	relations.
	\vspace{0.25cm}
	\begin{itemize}\itemsep0.25cm
		\item {\bf Symmetry.} For any $i\neq j$, 
		$\Th{ij}=\Th{ji}$.
		\item {\bf Locality.} For any distinct $i,j,k,l$ 
		$[\Th{ij},\Th{kl}]=0$.
		\item {\bf KZ relations.} For any distinct $i,j,k$,
		\begin{equation}\label{eq:KZ-holo-rels}
			[\Th{ij},\Th{ik}+\Th{jk}]=0 
		\end{equation}
	\end{itemize}
\end{definition}

\noindent\remark\;
The algebra $\DKHA{}{n}$ is 
the holonomy 
algebra $\DCPHA{\Rs{\sfA_{n-1}}}$ of 
the root system of type $\sfA_{n-1}$. 
Indeed, under the map $\Th{ij}\mapsto\Kh{\alpha_i+\cdots+\alpha_{j-1}}{}$, $i<j$,
the relations \eqref{eq:KZ-holo-rels} correspond precisely to the $tt$--relations
\eqref{eq:tt-relns}. For instance, in $\DCPHA{\sfA_3}$ one has
\[
\left[\Kh{\alpha_1}{}, \Kh{\alpha_2}{}+\Kh{\alpha_1+\alpha_2}{}\right]=0
\aand
\left[\Kh{\alpha_1}{}, \Kh{\alpha_3}{}\right]=0
\]
The grading and the completion of $\DKHA{}{n}$ are therefore defined as in 
\ref{ss:comp-diag-sub}.

\subsection{Cosimplicial structure on $\DKHA{}{\bullet}$}\label{ss:cosimplicial-KZ}
Set $\DKHA{}{1}=\sfk$.
The tower of algebras $\DKHA{}{\bullet}=\{\DKHA{}{n}\}_{n\geqslant 1}$ 
is a cosimplicial algebra (cf. \ref{ss:cosim}), with the face morphisms $d_n^k:\DKHA{}{n}\to\DKHA{}{n+1}$, 
$k=0,1,\dots, n+1$, given by
\begin{equation}\label{eq:coprod-holo-KZ-1}
	d_n^0(\Th{ij})=\Th{i+1,j+1}\qquad
	d_n^{n+1}(\Th{ij})=\Th{ij}
\end{equation}
and
\begin{equation}\label{eq:coprod-holo-KZ-2}
	d_n^k(\Th{ij})=\delta_{ki}(\Th{ij}+\Th{i+1,j})+\delta_{kj}(\Th{ij}+\Th{i,j+1})\qquad k=1,\dots, n
\end{equation}
while the degeneration homomorphisms 
$\varepsilon_n^k:\DKHA{}{n}\to\DKHA{}{n-1}$, $k=1,\dots, n$ 
are given by
\begin{equation}\label{eq:counit-holo-KZ}
	\varepsilon_n^k(\Th{ij})=(1-\delta_{ki}-\delta_{kj})\Th{ij}
\end{equation}

We shall describe several refinements of $\DKHA{}{\bullet}$, to which the
cosimplicial structure naturally extends. Their mutual relations are described
in Proposition \ref{pr:diamond} and diagram \eqref{eq:holonomy diamond} below.

\subsection{Diagrammatic refinement $\DBLHA{}{\bullet,\dgr}$ of $\DKHA{}{\bullet}$}\label{ss:D-KZ-holo}
Let $\dgr$ be a diagram. We construct a cosimplicial diagrammatic algebra 
by gluing together a copy of $\DKHA{}{\bullet}$ for any subdiagram $B\subseteq \dgr$. The algebra 
$\DBLHA{}{\bullet,\dgr}$ allows to simultaneously describe the monodromy of the KZ
equations corresponding to all diagrammatic subalgebras of a symmetrisable \KM algebra. 

\begin{definition}
	The algebra $\DBLHA{}{n,\dgr}$ is the associative
	algebra generated by the symbols $\{\Kh{B}{ij}\;|\; 1\leqslant i\neq j\leqslant n,\; B\subseteq\dgr\}$ with the following relations.
	\vspace{0.25cm}
	\begin{itemize}\itemsep0.25cm
		\item {\bf Symmetry.} For any $i\neq j$, and $B\subseteq \dgr$
		$\Kh{B}{ij}=\Kh{B}{ji}$.
		\item {\bf Locality.} For any distinct $i,j,k,l$, and $B,B'\subseteq \dgr$
		\begin{equation}\label{eq:OmegaB-locality}
			[\Kh{B}{ij}{},\Kh{B'}{kl}{}]=0
		\end{equation}
		\item {\bf KZ relations.} For any distinct $i,j,k$, and $B'\subseteq B\subseteq \dgr$
		\begin{equation}\label{eq:OmegaB-KZ}
			[\Kh{B}{ij}{},\Kh{B'}{ik}{}+\Kh{B'}{jk}{}]=0
		\end{equation}
		\item {\bf Orthogonality.} For any $i,j,k,l$ and 
		$B_1\perp B_2\subset\dgr$,  
		\begin{equation}\label{eq:OmegaB-perp}
			\Kh{B_1\sqcup B_2}{ij}{}=\Kh{B_1}{ij}{}+\Kh{B_2}{ij}{}
			\qquad\mbox{and}\qquad
			[\Kh{B_1}{ij}{},\Kh{B_2}{kl}{}]=0
		\end{equation}
	\end{itemize}
\end{definition}
\vspace{0.25cm}
\noindent\remark\; Note that, by \eqref{eq:OmegaB-perp}, it is enough to
assume \eqref{eq:OmegaB-locality} and \eqref{eq:OmegaB-KZ}
for connected subdiagrams only.

\subsection{Diagrammatic and cosimplicial structure}\label{ss:D-KZ-holo-cosimp}

\begin{proposition}\hfill
	\begin{enumerate}\itemsep0.25cm
		\item For any $B\subseteq\dgr$, there is an embedding $\iota^n_B
		:\DKHA{}{n}\to\DBLHA{}{n,B}$ given by $\Kh{}{ij}\mapsto\Kh{B}{ij}{}$.
		\item There is a unique cosimplicial structure on $\DBLHA{}{n,\dgr}$ such that
		$\{\iota_B^n\}$ is a morphism of 
		cosimplicial algebras $\iota_B\colon\DKHA{}{\bullet}\to\DBLHA{}{\bullet,\dgr}$
		for every $B\subseteq\dgr$.
		\item For any $B'\subseteq B\subseteq\dgr$, there is an embedding
		$i^n_{B'B}:\DBLHA{}{n,B'}\to\DBLHA{}{n,B}$ given by
		$\Kh{B''}{ij}\mapsto\Kh{B''}{ij}$ for any $B''\subseteq B'\subseteq B$. Moreover, 
		if $B_1\perp B_2$, multiplication induces 
		an isomorphism of algebras
		$\DBLHA{}{n,B_1}\ten\DBLHA{}{n,B_2}\to
		\DBLHA{}{n,B_1\sqcup B_2}$, so that
		\[\DKHA{}{n,\scsop{\dgr}}=(\DBLHA{}{n,B}, i^n_{B'B})\]
		is a diagrammatic algebra.
		\item The tower $\DKHA{}{\bullet,\scsop{\dgr}}=\{\DKHA{}{n,\scsop{\dgr}}\}_{n\geqslant 1}$
		is a cosimplicial diagrammatic algebra.
	\end{enumerate}
\end{proposition}

\subsection{Root refinement $\DKHA{}{\bullet,\Delta}$ of $\DKHA{}{\bullet}$}\label{ss:R-KZ-holo}

Let $\g$ be a symmetrisable \KM algebra, $\h$ its Cartan subalgebra, and $\rootsys\subset\h^*$ its
root system. We define a cosimplicial refinement of $\DKHA{}{\bullet}$ controlled by $\rootsys$,
which is suitable to describe the monodromy of the dynamical KZ equations of $\g$. 

\begin{definition}
	The algebra $\DKHA{}{n,\scsop{\Rs{}}}$ is the 
	associative algebra generated by the symbols 
	$\{\Tdh{ij}{0}, \Trdh{ij}{\alpha}\;|\;
	1\leqslant i\neq j\leqslant n, \alpha\in\rootsys\}$
	with the following relations.
	\vspace{0.25cm}
	\begin{itemize}\itemsep0.25cm
		\item {\bf Symmetry.} For any $i\neq j$ and $\alpha\in\rootsys$,
		\begin{equation}\label{eq:Omegaalpha-symmetry}
			\Trdh{ij}{\alpha}=\Trdh{ji}{-\alpha}
			\aand
			\Tdh{ij}{0}=\Tdh{ji}{0}
		\end{equation}
		\item {\bf Locality.} 
		For any distinct $i,j,k,l$ and $\alpha,\beta\in\rootsys$,
		\begin{equation}\label{eq:Omegaalpha-commute}
			[\Trdh{ij}{\alpha}, \Trdh{kl}{\beta}]=0
			\qquad
			[\Trdh{ij}{\alpha}, \Tdh{kl}{0}]=0
			\qquad
			[\Tdh{ij}{0},\Tdh{kl}{0}]=0
		\end{equation}
		\item {\bf KZ relations.} Set\footnote{If $|\rootsys|=\infty$, then the relation \eqref{eq:Omegadec} is to be 
			understood as in \ref{ss:holonomy}.}
		\begin{equation}\label{eq:Omegadec}
			\Tdh{ij}{}=\Tdh{ij}{0}+\sum_{\alpha\in\rootsys}\Trdh{ij}{\alpha}
		\end{equation}
		Then, for any distinct $i,j,k$, $[\Tdh{ij}{},\Tdh{ik}{}+\Tdh{jk}{}]=0$.
		\item {\bf Orthogonality.} For any 
		$i,j,k,l$, and $\alpha\perp\beta$, $[\Trdh{ij}{\alpha},\Trdh{kl}{\beta}]=0$.
		\item {\bf Weight zero.} For any 
		$i,j,k,l$, $[\Tdh{ij}{0},\Tdh{kl}{0}]=0$.
	\end{itemize}
\end{definition}
\vspace{0.25cm}
\noindent
Note that locality implies that $[\Tdh{ij}{},\Tdh{kl}{}]=0$ for any distinct $i,j,k,l$.

\begin{proposition}\hfill
	\begin{enumerate}\itemsep0.25cm
		\item For any $n\geq 2$, there is an embedding $\DKHA{}{n}\to\DKHA{}{n,\scsop{\Rs{}}}$ given by $\Th{ij}\mapsto\Tdh{ij}{}$. 
		\item The tower of algebras $\DKHA{}{\bullet,\scsop{\Rs{}}}=\{\DKHA{}{n,\scsop{\Rs{}}}\}_{n\geqslant1}$ 
		is endowed with a unique cosimplicial structure which extends that on $\DKHA{}{n}$, and is given by\footnote{By convention, $\Trdh{ii}{\alpha}=0$.}
		\[
		d_n^k(\Trdh{ij}{\alpha})=\delta_{ki}(\Trdh{ij}{\alpha}+\Trdh{i+1,j}{\alpha})+
		\delta_{kj}(\Trdh{ij}{\alpha}+\Trdh{i,j+1}{\alpha})\qquad k=1,\dots, n
		\]
	\end{enumerate}
\end{proposition}

\noindent
The algebra $\DKHA{}{n,\scsop{\Rs{}}}$ is acted upon by $\h^{\oplus n}$. 
For any $h\in\h$ and $1\leq k\leq n$, we set
\[
\admu{h}{k}\cdot\Trdh{ij}{\alpha}=(\delta_{ki}-\delta_{kj})\alpha(h)\Trdh{ij}{\alpha}
\]
Note that $\h^{\oplus n}$ does not preserve the elements $\Tdh{ij}{}$, and thus
the image of $\DKHA{}{n}$ in $\DKHA{}{n,\scsop{\Rs{}}}$.\\

\noindent\remark\;
Let $\dgr$ be the Dynkin diagram of $\rootsys$. For any $B\subseteq\dgr$, consider
the subsystem $\Rs{B}\subseteq \rootsys$ consisting of all $\alpha\in\rootsys$ with $\supp(\alpha)\subseteq B$ 
and define the subalgebra $\DKHA{B}{n,\scsop{\Rs{}}}\subseteq\DKHA{}{n,\scsop{\Rs{}}}$
generated by the symbols\footnote{Note that $\DKHA{B}{n,\scsop{\Rs{}}}$ does not coincide
with the root refinement of $\DKHA{}{n}$ corresponding to $\Rs{B}$, since the operators
$\Tdh{ij}{B}=\Tdh{ij}{0}+\sum_{\alpha\in\Rs{B}}\Trdh{ij}
{\alpha}$ are not required to satisfy the relations $[\Tdh{ij}{B},\Tdh{ik}{B}+\Tdh{jk}{B}]=0$.}
\[
\{\Tdh{ij}{0}, \Trdh{ij}{\alpha}\;|\;
1\leqslant i<j\leqslant n, \alpha\in\rootsys_B\}
\]
The cosimplicial
structure on $\DKHA{}{n,\scsop{\Rs{}}}$ restricts to one on $\DKHA{B}{n,\scsop{\Rs{}}}$
and, for any $B'\subseteq B$, we have $\DKHA{B'}{n,\scsop{\Rs{}}}\subseteq\DKHA{B}
{n,\scsop{\Rs{}}}$.

Note, however, that this does not give rise to a diagrammatic structure on $\DKHA{}{n,
\Delta}$. Indeed, if $B_1\perp B_2$, $\DKHA{B_1}{n,\scsop{\Rs{}}}$ and 
$\DKHA{B_2}{n,\scsop{\Rs{}}}$ do not commute in $\DKHA{B_1\sqcup B_2}{n,\scsop
{\Rs{}}}$ since the elements $\Tdh{ij}{0}$ do not distinguish between $\alpha\in B_1$ or
$\alpha\in B_2$. In order to obtain a diagrammatic structure, we need to further refine
the elements $\Tdh{ij}{0}$ in a way which is analogous to the refinement of the elements 
$\Th{ij}$ into $\Th{ij}_B$ in \ref{ss:D-KZ-holo}. 
We shall do so in the following section, integrating the diagrammatic and root refinements
of $\DKHA{}{\bullet}$ with the holonomy algebra $\DCPHA{\rootsys}$.

\subsection{The double holonomy algebra $\DBLHA{\rootsys}{\bullet}$}\label{ss:doubleholo}

Retain the notation of \ref{ss:R-KZ-holo}. 

\begin{definition}
	For $n\geqslant 1$, let $\DBLHA{\Delta}{n}$ be the $\IC$--algebra generated by the elements\footnote
	{The generators $\Tdh{ij}{0,B},\Trdh{ij}{\alpha}$ and $\Kdh{\alpha}{(n)}$ are included only if $n\geqslant 2$.} 
	\[\left\{\Tdh{ij}{0,B},\Trdh{ij}{\alpha}\right\}_{\substack{1\leqslant i\neq j\leqslant n\\ \alpha\in\Rs{}, B\subseteq \dgr}}
	\aand
	\left\{\Kdh{\alpha}{k},\Kdh{\alpha}{(n)}\right\}_{\substack{1\leq k\leq n\\ \alpha\in\Rs{+}}}
	\] 
	with the following relations
	\begin{itemize}[leftmargin=2em]\itemsep0.25cm
		\item {\bf Symmetry.} For any $i\neq j$, $B\subseteq \dgr$, and $\alpha\in\Rs{}$,
		\begin{equation}\label{eq:symmetry}
			\Tdh{ij}{0,B}=\Tdh{ji}{0,B}\aand\Trdh{ij}{\alpha}=\Trdh{ji}{-\alpha}
		\end{equation}
\Omit{
		\item {\bf Locality.} For any distinct $i,j,k,l$, $B,B'\subseteq\dgr$, $\alpha,\beta\in\Rs{}$, 
		and $\gamma,\delta\in\Rs{+}$
		\begin{gather}
			[\Trdh{ij}{\alpha},\Trdh{kl}{\beta}]=0
			\qquad
			[\Trdh{ij}{\alpha},\Tdh{kl}{0,B}]=0
			\qquad
			[\Trdh{ij}{\alpha},\Kdh{\gamma}{k}]=0
			\label{eq:locality-1}\\[2ex]
			[\Tdh{ij}{0,B}, \Tdh{kl}{0,B'}]=0
			\qquad
			[\Tdh{ij}{0,B}, \Kdh{\gamma}{k}]=0
			\label{eq:locality-2}\\[2ex]
			[\Kdh{\gamma}{i},\Kdh{\delta}{j}]=0
			\label{eq:locality-3}
		\end{gather}
}
		\item {\bf Locality.} For any distinct $i,j,k,l$, $B,B'\subseteq\dgr$, $\alpha,\beta\in\Rs{}$, 
		and $\gamma,\delta\in\Rs{+}$
		\begin{align}
			[\Trdh{ij}{\alpha},\Trdh{kl}{\beta}]&=0
			&\qquad
			[\Trdh{ij}{\alpha},\Tdh{kl}{0,B}]&=0
			&\qquad
			[\Trdh{ij}{\alpha},\Kdh{\gamma}{k}]&=0
			\label{eq:locality-1}\\
			\intertext{and}
			[\Tdh{ij}{0,B}, \Tdh{kl}{0,B'}]&=0
			&\qquad
			[\Tdh{ij}{0,B}, \Kdh{\gamma}{k}]&=0
			&\qquad
			[\Kdh{\gamma}{k},\Kdh{\delta}{l}]&=0
			\label{eq:locality-2}
		\end{align}
		\item {\bf KZ relations.} 
		For any  distinct $i,j,k$, and $B'\subseteq B\subseteq \dgr$
		\begin{equation}\label{eq:KZ-relations}
			[\Tdh{ij}{B},\Tdh{ik}{B'}+\Tdh{jk}{B'}]=0 
		\end{equation}
		where $\Tdh{ij}{B}=\Tdh{ij}{0,B}+\sum_{\alpha\in\Rs{B, +}}(\Trdh{ij}{\alpha}+\Trdh{ij}{-\alpha})$.
		\item {\bf Orthogonality.} 
		For any $i,j,k,l$, $B_1\perp B_2\subseteq\dgr$, $\alpha\in\Rs{B_1},\beta\in\Rs{B_2}$, and any $\gamma
		\in\Rs{B_1,+}, \delta\in\Rs{B_2,+}$
		\begin{align}\label{eq:orthogonality-1}
			[\Trdh{ij}{\alpha},\Trdh{kl}{\beta}]=0
			\qquad
			[\Trdh{ij}{\alpha},\Kdh{\gamma}{k}]=0
			\qquad
			[\Kdh{\gamma}{k},\Kdh{\delta}{l}]=0
		\end{align}
		and
		\begin{align}\label{eq:orthogonality-2}
			[\Tdh{ij}{0,B_1}, \Trdh{kl}{\beta}]=0
			\qquad
			[\Tdh{ij}{0,B_1}, \Kdh{\delta}{k}]=0
		\end{align}
		together with $\Tdh{ij}{0,B_1\sqcup B_2}=\Tdh{ij}{0,B_1}+\Tdh{ij}{0,B_2}$.
		\item {\bf Weight zero.} For any 
		$i,j,k,l$, and $B,B'\subseteq\dgr$, $[\Tdh{ij}{0,B},\Tdh{kl}{0,B'}]=0$.
		\item {\bf Casimir relations.} For any $1\leqslant k\leqslant n$, rank $2$ root subsystem
		$\Psi\subseteq\Rs{}$, and $\alpha\in\Psi\cap\Rs{+}$,
		\begin{equation}\label{eq:Casimir-relations}
			\left[\Kdh{\alpha}{k}, \sum_{\beta\in\Psi\cap\Rs{+}}\Kdh{\beta}{k}\right]=0
			\aand
			\left[\Kdh{\alpha}{(n)}, \sum_{\beta\in\Psi\cap\Rs{+}}\Kdh{\beta}{(n)}\right]=0
		\end{equation}
		\item {\bf Invariance relations.}
		For any $1\leqslant i\neq j\leqslant n$, $B\subseteq\dgr$, and $\alpha\in\Rs{B,+}$,
		\begin{equation}\label{eq:invariance}
			[\Tdh{ij}{B},\Kdh{\alpha}{(n)}]=0
		\end{equation}
		\item {\bf Coproduct relation.} For any $\alpha\in\Delta_+$,
		\begin{equation}\label{eq:Kn coproduct}
			\Kdh{\alpha}{(n)}=\sum_{i<j}(\Trdh{ij}{\alpha}+\Trdh{ij}{-\alpha})+
			\sum_{k=1}^n\Kdh{\alpha}{k}
		\end{equation}
	\end{itemize}
\end{definition}

\noindent\remark\; The coproduct relation \eqref{eq:Kn coproduct} implies that
the generators $\Kdh{\alpha}{(n)}$ are redundant. However, the relations \eqref
{eq:Casimir-relations} and \eqref{eq:invariance} are easier to formulate in terms
of $\Kdh{\alpha}{(n)}$ rather than the remaining generators.
Note also that $\DBLHA{\Rs{}}{1}$ is the holonomy algebra
$\DCPHA{\Rs{}}$ introduced in \ref{ss:holonomy}.

\subsection{Actions of $\SS_n$ and $\h^{\oplus n}$}\label{ss:doubleholo-actions}

The algebra $\DBLHA{\Delta}{n}$ is acted upon by $\SS_n\ltimes\h^{\oplus n}$.
The action of $\sigma\in\SS_n$ is defined by 
\begin{equation*}
	\sigma(\Tdh{ij}{0,B})=\Tdh{\sigma(i)\sigma(j)}{0,B}
\qquad
\sigma(\Trdh{ij}{\alpha})=\Trdh{\sigma(i)\sigma(j)}{\alpha}
\qquad
\sigma(\Kdh{\alpha}{i})=\Kdh{\alpha}{\sigma(i)}
\qquad
\sigma(\Kdh{\alpha}{(n)})=\Kdh{\alpha}{(n)}
\end{equation*}
The action of $\h^{\oplus n}$ is defined as follows.
For any $h\in\h$, we set 
\begin{gather*}
\admu{h}{k}\cdot\Tdh{ij}{0}=0=\admu{h}{k}\cdot\Kdh{\alpha}{\ell}\\
\intertext{and}
\admu{h}{k}\cdot\Trdh{ij}{\alpha}= (\delta_{ki}-\delta_{kj})\alpha(h)\Trdh{ij}{\alpha}
\qquad\qquad
\admu{h}{k}\cdot\Kdh{\alpha}{(n)}= \admu{h}{k}\cdot\left(\sum_{i<j}\Trdh{ij}{\alpha}+\Trdh{ij}{-\alpha}\right)
\end{gather*}
Note that this is consistent with the relation \eqref{eq:invariance}. Moreover, the action
of $\h^{\oplus n}$ on $\DBLHA{\Delta}{n}$
clearly factors through the essential Cartan $({\hess})^{\oplus n}$.

\subsection{Cosimplicial structure on $\DBLHA{\rootsys}{\bullet}$}\label{ss:doubleholo-cosimp}

Set $\DBLHA{\Rs{}}{0}=\sfk$. The tower of algebras 
$\DBLHA{\Delta}{\bullet}=\{\DBLHA{\Delta}{n}\}$
is endowed with a natural cosimplicial structure. 
The face morphisms 
\[
d^{n+1}_k:\DBLHA{\Delta}{n}\to\DBLHA{\Delta}{n+1}\qquad k=0,1,\dots, n+1
\]
are defined on $\Tdh{ij}{0,B}, \Trdh{ij}{\alpha}$ as in the case of $\DBLHA{}{n,\Rs{}}$ (see~ \ref{ss:R-KZ-holo}) and on $\Kdh{\alpha}{i}$
by
\begin{equation*}
	d^{n+1}_k(\Kdh{\alpha}{i})=
	\left\{
	\begin{array}{lcl}
		\Kdh{\alpha}{i+1}& \mbox{if} & k<i\\
		(\Kdh{\alpha}{i})^{(2)} & \mbox{if} & k=i\\
		\Kdh{\alpha}{i} & \mbox{if} & k>i
	\end{array}
	\right.
\end{equation*}
where $(\Kdh{\alpha}{i})^{(2)}=\Trdh{i,i+1}{\alpha}+\Trdh{i,i+1}{-\alpha}+
\Kdh{\alpha}{i}+\Kdh{\alpha}{i+1}$.

More generally, set
\begin{equation}\label{eq:coproduct-K}
	(\Kdh{\alpha}{k})^{(m)}=\sum_{k\leqslant i<j\leqslant m+k-1}\Trdh{ij}{\alpha}+\Trdh{ij}{-\alpha}+
	\sum_{l=k}^{m+k-1}\Kdh{\alpha}{\ell}
\end{equation}
so that $(\Kdh{\alpha}{1})^{(m)}=\Kdh{\alpha}{(m)}$ and $(\Kdh{\alpha}{k})^{(1)}=\Kdh{\alpha}{k}$. Then,
one has
\begin{equation}\label{eq:face-double-holo}
	d^{n+1}_k((\Kdh{\alpha}{i})^{(m)})=
	\left\{
	\begin{array}{lcl}
		(\Kdh{\alpha}{i+1})^{(m)} & \mbox{if} & k<i\\
		(\Kdh{\alpha}{i})^{(m+1)} & \mbox{if} & k=i,\dots, m+i-1\\
		(\Kdh{\alpha}{i})^{(m)} & \mbox{if} & k\geqslant m+i
	\end{array}
	\right.
\end{equation}

Similarly, the degeneration morphisms $\varepsilon^n_k:\DBLHA{\Delta}{n}\to\DBLHA{\Delta}{n-1}$, 
$k=1,\dots, n$ are defined as in \ref{ss:R-KZ-holo}, together with the additional requirement that
\[
\varepsilon^n_k((\Kdh{\alpha}{i})^{(m)})=
\left\{
\begin{array}{lcl}
	(\Kdh{\alpha}{i-1})^{(m)} & \mbox{if} & k<i\\
	(\Kdh{\alpha}{i})^{(m-1)} & \mbox{if} & k=i,\dots, m+i-1\\
	(\Kdh{\alpha}{i})^{(m)} & \mbox{if} & k\geqslant m+i
\end{array}
\right.
\]

\subsection{Cosimplicial bidiagrammatic structures}\label{ss:doubleholo-bidiagrammatic}

For any $B\subseteq\dgr$, we denote $\DBLHA{\Rs{B}}{n}$ by $\DBLHA{B}{n}$. The following
result describes a bidiagrammatic structure on $\DBLHA{\rootsys}{\bullet}$, and its relation with
the diagrammatic and root refinements $\DBLHA{}{\bullet,\dgr}$, $\DBLHA{}{\bullet,\Delta}$ of
$\DBLHA{}{\bullet}$ defined in \ref{ss:D-KZ-holo} and \ref{ss:R-KZ-holo}.

\begin{proposition}\label{pr:diamond}
	\hfill
	\begin{enumerate}\itemsep0.25cm
		\item For any $n\geq 2$, there is a morphism $\iota_{\rootsys}^n:\DKHA{}{n,\Rs{}}
		\to\DBLHA{\Rs{}}{n}$ given by 
		\begin{equation*}
			\Tdh{ij}{0}\mapsto\Tdh{ij}{0,\dgr}\aand\Trdh{ij}{\alpha}\mapsto\Trdh{ij}{\alpha}
		\end{equation*}
		$\iota_{\rootsys}^n$ is $\h^{\oplus n}$--equivariant, and gives rise to a morphism of
		cosimplicial algebras $\iota_{\rootsys}:\DKHA{}{\bullet,\rootsys}\to\DBLHA{\Delta}{\bullet}$.
		\item For any $B'\subseteq B$, there is an embedding $i_{B'B}^{n}:\DBLHA{B'}{n}\to\DBLHA{B}{n}$, which maps
		every generator in $\DBLHA{B'}{n}$ to the same symbol in $\DBLHA{B}{n}$. 
		Then, 
		$\DBLHA{\rootsys}{\bullet}=\{\DBLHA{\rootsys}{n}\}$
		is a cosimplicial diagrammatic algebra.
		\item For any $n\geq 2$, there is a  morphism $\iota_{\dgr}^n:\DKHA{}{n,\dgr}
		\to\DBLHA{\Rs{}}{n}$ given by
		\begin{equation*}
		\Kh{B}{ij}\mapsto\Tdh{ij}{B},\qquad B\subseteq\dgr
		\end{equation*}
		which gives rise to a morphism of diagrammatic cosimplicial
		algebras $\iota_{\dgr}\colon\DKHA{}{\bullet,\dgr}\to\DBLHA{\Delta}{\bullet}$.
		\item For any $B'\subseteq B\subseteq\dgr$,
		let $\DBLHA{BB'}{n}$ be the subalgebra of 
		$\DCPHA{B'}$--invariant elements in $\DBLHA{B}{n}$. Then,
		$\DBLHA{\rootsys}{\bullet,\flat}=\{\DBLHA{BB'}{n}\}$
		is a cosimplicial bidiagrammatic algebra, whose structure is obtained from 
		$\DBLHA{\rootsys}{\bullet}$ by restriction.
	\end{enumerate}	
\end{proposition}

\begin{pf}
	For (2) it is enough to observe that, if $B'''\subseteq B''\subseteq B'$, clearly
	$i_{B''B'}^{n}\circ i_{B'''B''}^{n}=i_{B'''B'}^{n}$. Moreover, if $B',B''\subseteq B$ 
	with $B'\perp B''$, the multiplication induces an isomorphism of algebras
	$j_{B_1B_2}:\DBLHA{B_1}{n}\ten\DBLHA{B_2}{n}\to
	\DBLHA{B_1\sqcup B_2}{n}$.
	(1) and (3) are clear. 
	(4) follows as in Proposition \ref{ss:bi-diag-cosim}.
\end{pf}

\noindent\remark\;
Combined with Propositions \ref{ss:D-KZ-holo-cosimp} and \ref{ss:R-KZ-holo}, together
with the fact that $\DBLHA{\rootsys}{1}=\DBLHA{\rootsys}{}$, the result above yields a
commutative diagram of holonomy algebras
\begin{equation}\label{eq:holonomy diamond}
	\begin{tikzcd}
		&\DBLHA{}{\bullet} \ar[dl,"\{\iota_B\}_{B\subseteq\dgr}"'] \ar[dr]&\\
		\DBLHA{}{\bullet,\dgr} \ar[dr, "\iota_{\dgr}"']&&\DBLHA{}{\bullet,\rootsys} \ar[dl, "\iota_{\rootsys}"]\\
		&\DBLHA{\rootsys}{\bullet}&\\
		&\DBLHA{\rootsys}{}\ar[u, hook, "n=1"']&
	\end{tikzcd}
\end{equation} 

\subsection{Grading completions}
We denote by $\DBLHAH{B}{n}$ the completion of $\DBLHA{B}{n}$ with respect 
to the grading $\deg(\Tdh{}{})=\deg(\Trdh{}{})=\deg(\Kdh{}{})=1$. Let $\DBLHAH{BB'}{n}$ be 
the subalgebra of $\DCPHA{B'}$--invariant elements in $\DBLHAH{B}{n}$. 
Note that, if $B_1'\subseteq B_1\perp B_2\supseteq B_2'$, we get isomorphisms 
$\DBLHAH{B_1}{n}\ten\DBLHAH{B_2}{n}\to\DBLHAH{B}{n}$
and
$\DBLHAH{B_1B_1'}{n}\ten\DBLHAH{B_2B_2'}{n}\to\DBLHAH{BB'}{n}$,
where $B= B_1\sqcup B_2$, $B'= B'_1\sqcup B'_2$, and $\ten$ denotes
the completion of the tensor product with respect to the grading.

\begin{corollary}\hfill
	\begin{enumerate}\itemsep0.25cm
		\item
		For any $n\geqslant 1$, 
		$\DBLHAH{\rootsys}{n}=\left(\DBLHAH{B}{n}, i_{B'B}^n\right)$
		is a diagrammatic algebra. 
		The face and degeneration morphisms of the cosimplicial structure on
		$\DBLHAH{\rootsys}{\bullet}=\{\DBLHAH{\rootsys}{n}\}$
		are morphisms of diagrammatic algebras. Thus, $\DBLHAH{\rootsys}{\bullet}$ is 
		a cosimplicial diagrammatic algebra.
		\item 
		$\DBLHAH{\rootsys}{\bullet,\flat}=
		\{\DBLHAH{BB'}{n}\}$ is a 
		cosimplicial bidiagrammatic algebra, whose structure is obtained from $\DBLHAH{\rootsys}{\bullet}$ by restriction. 
	\end{enumerate}
\end{corollary}



\section{A braided Coxeter structure from double holonomy}\label{s:braided-Cox-holo}

We prove that the monodromy data of the KZ and Casimir connections,
described in Sections~\ref{s:holo-Cox} and \ref{s:double-holo}, are encoded by a braided 
Coxeter structure with relative twists arising from the monodromy of the dynamical KZ equations.
The proof is a simple generalisation of \cite{vtl-6} at the level of the double holonomy algebra, 
which in turn applies to the case of infinite--dimensional Kac--Moody algebras.

\subsection{Monodromy of the KZ connection}\label{ss:KZ-associator}
We observed in Remark \ref{ss:KZ-holo} that $\DKHA{}{3}=\DCPHA{\sfA_2}$
and it is well--known that in this case the canonical solutions of the holonomy 
equation \eqref{eq:holoeq} are obtained by solving the $\mathsf{KZ}_3$ equation
\begin{equation}\label{eq:KZ}
	\frac{d}{du}\DCPS{}=\left(\frac{\Kh{12}{}}{u}+\frac{\Kh{23}{}}{1-u}\right)\DCPS{}
\end{equation}
at $u=0$ and $u=1$. Therefore, let $\F,\G$ be the only two elements of 
$\Mns{\sfA_2}$ with 
\[\{\alpha_1\}\subset\F\aand\{\alpha_2\}\subset\G\]
and set $\Phi^{\nabla}=\DCPA{\F}{\G}\in\DKHAH{}{3}$.

\begin{definition}
	An invertible element $\Phi\in\DKHAH{}{3}$ is called a \emph{Lie associator}
	if $\Phi$ is the exponential of a formal Lie series 
	in $\Th{12}$ and $\Th{23}$ and the following relations are satisfied.
	\footnote{We use the notation from \ref{ss:braid-cox-alg}.}
	\begin{itemize}
		\item {\bf Pentagon relation}
		\[
		\Phi_{1,2,34}\Phi_{12,3,4}=\Phi_{2,3,4}\Phi_{1,23,4}\Phi_{1,2,3}
		\]
		\item {\bf Hexagon relations}
		\begin{align*}
			&e^{\Th{12,3}/2}=\Phi_{3,1,2}e^{\Th{13}/2}\Phi_{1,3,2}^{-1}e^{\Th{23}/
				2}\Phi_{1,2,3}\\
			&e^{\Th{1,23}/2}=\Phi_{2,3,1}^{-1}e^{\Th{13}/2}\Phi_{2,1,3}e^{\Th{12}/
				2}\Phi_{1,2,3}^{-1}
		\end{align*}
		\item {\bf Duality}
		\[
		\Phi_{3,2,1}=\Phi_{1,2,3}^{-1}
		\]
		\item {\bf $2$--jet}
		\[
		\Phi=1+\frac{1}{24}[\Th{12},\Th{23}]\qquad\mod(\;\DKHAH{}{3})_{\geqslant 3}
		\]
	\end{itemize}
\end{definition}

The following result is well--known and due to Drinfeld \cite{drinfeld-91}.

\begin{theorem}
	The element $\Phi^{\nabla}\in\DKHAH{}{3}$ is a Lie associator.
\end{theorem}

\noindent\remark\; For any $B$, set 
\[
\Phig{\nabla}{B}{}= i_{B}^3(\Phi^{\nabla})
\aand
R^{\nabla}_{B}= i_{B}^2(\exp(\pi\iota\Kh{}{12}))
\]
The datum of $\Phig{\nabla}{B}{}\in\DKHAH{B}{3}$ and 
$R^{\nabla}_{B}\in\DKHAH{B}{2}$ satisfies the properties 
of associators and $R$--matrices listed in Definition \ref{ss:braid-cox-alg}. 
Note, in particular, that since $\Phi^{\nabla}$ is a Lie associator, 
then the invariance and orthogonal factorisation property of $\Phig{\nabla}{B}{}$ 
follow, respectively, from \eqref{eq:OmegaB-KZ} and \eqref{eq:OmegaB-perp}.

\subsection{A braided Coxeter structure on $\DBLHAH{\Rs{}}{\bullet,\scsop{ext}}$}\label{ss:hol-braid-cox}
In analogy with \ref{ss:ext-holo-alg} and \ref{ss:hol-cox}, we extend the double holonomy 
algebra $\DBLHA{\rootsys}{\bullet,\flat}$ with the parabolic braid groups $\Br{W_B}$. This yields
a cosimplicial bidiagrammatic algebra $\DBLHAH{\Rs{}}{\bullet,\scsop{ext}}=\{\DBLHAH{\Rs{}}{n,\scsop{ext}}\}$ 
where
\begin{equation}
	\rda{(\DBLHAH{\Rs{}}{n,\scsop{ext}})}{}{BB'}=
	\Br{W_B}\ltimes\left(\DKHA{BB'}{n}\wh{\ten}({S}\h'_B)^{\ten n}\right)
\end{equation}
$B'\subseteq B\subseteq\dgr$,  $\h'_{B}=\mathsf{span}\{\cor{i}\;|\; i\in B\}$, 
$W_B=\langle s_i\;|\; i\in B\rangle\subseteq W$, and the action of $\Br{W_B}$ extends that
on $\Br{W}\ltimes\DCPHAH{\Rs{},\h}$. The goal of this section is to prove the following

\begin{theorem}\label{thm:holo-cox}
	Let $(\Phig{\nabla}{B}{}, \Rg{\nabla}{B}{}, \DCPAC{\nabla}{\F}{\G}, \CoxS{\nabla}{}{i})$ be the 
	monodromy data of the KZ and Casimir connections defined in \ref{ss:D-KZ-holo} and 
	\ref{ss:hol-cox}, respectively. Then, the dynamical KZ equations give
	rise to a collection of relative twists $\Jg{\nabla}{\F}{}\in\DBLHAH{B'B}{2}$, 
	$\F\in\Mns{B,B'}$ such that the datum of 
	\[
	\sCox{\nabla}=(\Phig{\nabla}{B}{}, \Rg{\nabla}{B}{}, \Jg{\nabla}{\F}{}, \DCg{\nabla}{\F}{\G}{}, \Sg{\nabla}{}{i}{})
	\]
	defines an $\sfa$--strict braided Coxeter structure on the 
	cosimplicial bidiagrammatic algebra $\DBLHAH{\Rs{}}{\bullet,\scsop{ext}}$ with respect to the
	standard labeling on $\dgr$ (\ie $m_{ij}=\mathsf{ord}(s_is_j)$ in $W$).
\end{theorem}

The proof closely follows \cite[Sec. 3--7]{vtl-6}, and is outlined in this section. In \ref
{ss:diff-tw}--\ref{ss:relative twist}, we introduce the notion of a {\it differential twist} with
values in $\DBLHAH{\rootsys}{2}$. In \ref{ss:centraliser}--\ref{ss:all together}, we show that
a differential twist with the {\em centraliser property} induces a braided  Coxeter structure
on $\DBLHAH{\Rs{}}{\bullet,\scsop{ext}}$ compatible with the monodromy
data of the KZ and Casimir connections. Finally, in \ref{ss:DKZ2}--\ref{ss:differential twist}, we  show
that such a differential twist can be obtained as a regularised holonomy of the dynamical
KZ equations.

\subsection{Differential twist}\label{ss:diff-tw}

Let $\C_\IR=\{h\in\h\ess_\IR|\,\alpha_i(h)>0,\,\,\forall i\in\bfI\}$
be the fundamental chamber, and set $\C=\C_\IR+i\h\ess_\IR$. Let $\DBLHA
{\rootsys}{2}$ be the double holonomy algebra, 
and define $\wt{r}\in\DBLHA{\rootsys}{2}$ by
\[
\wt{\Trdh{}{}}=\frac{1}{2}\sum_{\alpha\in\Rs{+}}\left(\Trdh{12}{\alpha}-\Trdh{12}{-\alpha}\right)
\]

\begin{definition}
	A {\it differential twist} is a holomorphic map 
	$F:\C\to\DBLHAH{\rootsys}{2}$ such that
	\vspace{0.25cm}\begin{enumerate}
		\item\label{it:veps} $\varepsilon_2^1(F)=1=\varepsilon_2^2(F)$.
		\item\label{it:Phi} $(\Phi^{\nabla}_{\dgr})_F=1$ in $\DBLHAH{\rootsys}{3}$, where
		\[
		(\Phi^{\nabla}_{\dgr})_F=\twistAF{F}{\Phi}
		\]
		\item\label{it:norm} $F=1+f\mod(\DBLHA{\rootsys}{2})_{\geqslant 2}$,
		where $f\in(\DBLHA{\rootsys}{2})_1$ satisfies $\Alt_2 f=\wt{\Trdh{}{}}$.
		\item\label{it:PDE} F satisfies
		\begin{equation}\label{eq:2-Casimir}
			d F=\sum_{\alpha\in\Rp}\frac{d\alpha}{\alpha}
			\Bigl((\Kdh{\alpha}{1}+\Kdh{\alpha}{2})\cdot F-F\cdot \Kdh{\alpha}{(2)}\Bigr)
		\end{equation}
	\end{enumerate}
\end{definition}

\subsection{Compatibility with De Concini--Procesi associators}\label{ss:diffl DCP}

For any \mns $\F\in\Mns{\dgr}$, let $\DCPS{\F}:\C\to\DCPHAH{\rootsys}$ be the
fundamental solution of $\nablak$ corresponding to $\F$ (cf.~\ref{ss:DCPsol}),
and $\DCPAC{\nabla}{\G}{\F}=\DCPS{\G}^{-1}\cdot \DCPS{\F}$ the corresponding
associator. Let $F:\C\to\DBLHAH{\rootsys}{2}$ be a differential twist, and set
\begin{equation}\label{eq:diff-tw}
	F_\F=\gaugeFi{\DCPS{\F}}{F}
\end{equation}
The following is straightforward.

\begin{lemma}\label{le:diffl DCP}\hfill{}
	\begin{enumerate}
		\item $\varepsilon_2^1(F_\F)=1=\varepsilon_2^2(F_\F)$
		\item $(\Phi^{\nabla}_{\dgr})_{F_\F}=1$
		\item $F_\F=1+f_\F\mod(\DBLHA{\rootsys}{2})_{\geqslant 2}$,
		where $f_\F\in(\DBLHA{\rootsys}{2})_1$ satisfies $\Alt_2 f_\F=\wt{\Trdh{}{}}$.
		\item $F_\F$ is constant on $\C$
		\item The following holds for any $\F,\G\in\Mns{\dgr}$
		\[F_\F=\gaugeF{\DCPAC{\nabla}{\F}{\G}}{F_{\G}}\]
	\end{enumerate}
\end{lemma}

\subsection{Relative differential twists}\label{ss:relative twist}

We recall the settings of Section~\ref{ss:DCP-asym}.
Fix $i\in\bfI$, let $\ol{\rootsys}\subset\rootsys$ be the root subsystem
generated by the simple roots $\{\alpha_j\}_{j\neq i}$, $\olh\ess\subset\h\ess$
and $\DBLHA{\ol{\rootsys}}{n}\subset\DBLHA{\rootsys}{n}$
the corresponding essential Cartan and double holonomy subalgebras, respectively. Let 
$\pi:\h\ess\to\olh\ess$ be the projection determined by the requirement that
$\alpha(\pi(h))=\alpha(h)$ for any $\alpha\in\ol{\rootsys}$.

Let $F$ be a differential twist and $\Upsilon_\infty$ the solution of the Casimir equations
given by Proposition \ref{pr:Fuchs infty} with respect to the simple root $\alpha_i$, where we are using 
the standard determination of $\log$. Define $F_\infty:\C\to\DBLHAH{\rootsys}{2}$ by
\[F_\infty=\gaugeFi{\Upsilon_\infty}{F}\]
Then, the following holds
\vspace{0.25cm}\begin{enumerate}
	\item $\varepsilon_2^1(F_\infty)=1=\varepsilon_2^2(F_\infty)$
	\item $(\Phi^{\nabla}_{\dgr})_{F_\infty}=1^{\otimes 3}$
	\item $F_\infty=1+f_\infty\mod(\DBLHA{\rootsys}{2})_{\geqslant 2}$, where $f_\infty\in(\DBLHA{\rootsys}{2})_1$ satisfies
	$\Alt_2 f_\infty=\ol{r}$.
	\item $F_\infty$ satisfies
	\[d F_\infty=\sum_{\alpha\in\ol{\rootsys}_+}\frac{d\alpha}{\alpha}
	\Bigl((\Kdh{\alpha}{1}+\Kdh{\alpha}{2})\cdot F_\infty-F_\infty\cdot \Kdh{\alpha}{(2)}\Bigr)\]
\end{enumerate}

Let $\olC$ be the complexified chamber of $\olg$, and $\olF= F_{\ol{\rootsys}}:
\olC\to\DBLHAH{\ol{\rootsys}}{2}$ a differential twist for $\ol{\rootsys}$.
Since the projection $\pi:\h\ess\to\olh\ess$ maps $\C$ to $\olC$, we
may regard $\olF$ as a function on $\C$, and define
$F'_{(\dgr;\alpha_i)}:\C\to\DBLHAH{\rootsys}{2}$ by
\begin{equation}\label{eq:rel-twist-prime}
	F'_{(\dgr;\alpha_i)}=\olF^{-1}\cdot F_\infty
\end{equation}

\begin{proposition}\label{pr:relative twist}
	Set $\olD=\dgr\setminus\{i\}$. The following holds
	\vspace{0.25cm}\begin{enumerate}
		\item $\varepsilon_2^1(F'_{(\dgr;\alpha_i)})=1=\varepsilon_2^2(F'_{(\dgr;\alpha_i)})$
		\item $(\Phi^{\nabla}_{\dgr})_{F'_{(\dgr;\alpha_i)}}=\Phi^{\nabla}_{\olD}$
		\item $F'_{(\dgr;\alpha_i)}=1+f\mod(\DBLHA{\rootsys}{2})_{\geqslant 2}$, where $f\in(\DBLHA{\rootsys}{2})_1$ satisfies
		$\Alt_2 f=\wt{r}_{\dgr}-\wt{r}_{\olD}$.
		\item $F'_{(\dgr;\alpha_i)}$ satisfies
		\[d F'_{(\dgr;\alpha_i)}=\sum_{\alpha\in\ol{\rootsys}_+}\frac{d\alpha}{\alpha}
		[\Kdh{\alpha}{(2)},F'_{(\dgr;\alpha_i)}]\]
		In particular, if $F'_{(\dgr;\alpha_i)}$ is invariant under $\DCPHA{\ol{\rootsys}}$, then it is constant
		on $\C$.
	\end{enumerate}
\end{proposition}

\subsection{Centraliser property}\label{ss:centraliser}

Let $\{F_B:\C_B\to\DBLHAH{B}{2}\}_{B\subseteq\dgr}$ be a \emph{factorisable} collection of  
differential twists, \ie such that $F_B=\prod_i F_{B_i}$ if $B$ has connected components $\{B_i\}$.

\begin{definition}
	The collection $\{F_B\}$ has the {\it centraliser property} if, for any $i
	\in B\subseteq\dgr$, the relative twist $F'_{(B,\root{i})}$ defined by \eqref{eq:rel-twist-prime}
	is invariant under $\DCPHA{B\setminus\{i\}}$
	and therefore constant.
\end{definition}

Assume the centraliser property holds, let $i\in B\subseteq\dgr$, and set
\begin{equation}\label{eq:reltwist}
	F_{(B;\alpha_i)}=
	\left(x_B(\cow{i})^{-(\Kdh{B}{}-\Kdh{B\setminus\{i\}}{})}\right)^{\otimes 2}
	\cdot F'_{(B;\alpha_i)}\cdot
	d_1^1\left(x_B(\cow{i})^{\Kdh{B}{}-\Kdh{B\setminus\{i\}}{}}\right)
\end{equation}
where $\{x_B\}_{B\subseteq D}$ are the blow-up coordinates defined in 
\ref{ss:blow-up-coord}. The (constant)
twist $F_{(B;\alpha_i)}$ is invariant under $\DCPHA{B\setminus\{i\}}$,
and has the properties (1)--(3) given in Proposition \ref{pr:relative twist}. 
Moreover, 
\[
F_{(B;\alpha_i)}=1+f\mod(\DBLHA{\rootsys}{2})_{\geqslant 2}
\]
where $f\in(\DBLHA{\rootsys}{2})_1$ satisfies $\Alt_2 f=\wt{r}_{B}-\wt{r}_{B\setminus\alpha_i}$.
The following is a direct consequence of Proposition \ref{pr:infty factorisation}.

\begin{lemma}\label{le:factorisation}
	Let $\F$ be a \mns on $\dgr$, and $F_\F$ the twist defined in \eqref{eq:diff-tw}. 
	Then, the following holds
	\[F_\F=\stackrel{\longrightarrow}{\prod_{B\in\F}}F_{(B;\aF{\F}{B})}\]
	where the product is taken with $F_{(B;\aF{\F}{B})}$ to the right of $F_{(C;\aF{\F}{C})}$ if $B\supset C$
	\footnote{$\aF{\F}{B}$ denotes the only simple root whose support is not contained 
		in any maximal element of $\F^B$ (cf.~\ref{ss:blow-up-coord}).}.
\end{lemma}

\subsection{Braided Coxeter structure}\label{ss:all together}
The relative twists arising from a suitable collection of differential twists give rise to
a braided Coxeter structure encoding the monodromy data of the KZ and Casimir connections.
Specifically, we have the following

\begin{proposition}
	Let $\mathbf{F}=\{F_B:\C_B\to\DBLHAH{B}{2}\}$ be a factorisable collection of differential twists satisfying 
	the centraliser property.
	\vspace{0.25cm}\begin{enumerate}\itemsep0.25cm 
		\item 
		The elements $\{F_{(B;\root{i})}\}$ defined in \eqref{eq:reltwist} give rise to an $\redasso{}{}$--strict 
		braided pre--Coxeter structure $(\Phig{\nabla}{B}{}, \Rg{\nabla}{B}{}, J^{\mathbf{F}}_{\F}, \DCPAC{\nabla}{\F}{\G})$ on $\DBLHAH{\Rs{}}{\bullet,\scsop{ext}}$ with relative twists 
		\begin{equation}\label{eq:J-reltwist}
			J^{\mathbf{F}}_{\F}= \stackrel{\longleftarrow}{\prod_{B\in\F}}F_{(B;\aF{\F}{B})}^{-1}
		\end{equation}
		where $B'\subseteq B$ and $\F\in\Mns{B,B'}$.
		\item Assume that, for any $i\in\bfI$, the elementary differential twist $F_i$ satisfies $\sfAd{\texp{i}}(F_i)=F_{i,21}$.
		Then, $\sCox{\mathbf{F}}=(\Phig{\nabla}{B}{}, \Rg{\nabla}{B}{}, J^{\mathbf{F}}_{\F}, \DCPAC{\nabla}{\F}{\G}, \CoxS{\nabla}{}{i})$ is an $\redasso{}{}$--strict braided Coxeter structure on $\DBLHAH{\Rs{}}{\bullet,\scsop{ext}}$.
	\end{enumerate}
\end{proposition}

\begin{pf}
	(1) is a direct consequence of \ref{ss:diff-tw}, Proposition \ref{pr:relative twist}, and Lemma \ref{le:factorisation}.
	(2) amounts to prove the coproduct identity \eqref{eq:coxcoprod-diag-alg}. Namely, recall that 
	$\CoxS{\nabla}{}{i}=\texp{i}\exp(\pi\iota \mathsf{C}_i)$ with $\mathsf{C}_i=\Kh{\root{i}}+\hinv{i}+\symdi{i}\hinv{i}^2/2$.
	Since
	\[(\CoxS{\nabla}{}{i})_{12}=\exp(\pi\iota\Tdh{i}{})\cdot(\CoxS{\nabla}{}{i})_1(\CoxS{\nabla}{}{i})_2\]
	and $\mathsf{C}_{i,1}\mathsf{C}_{i,2}$ is central in $\DBLHAH{i}{2}$, 
	the coproduct identity for $\sCox{\mathbf{F}}$ reduces to the condition $\sfAd{\texp{i}}(J^{\mathbf{F}}_{i})=J^{\mathbf{F}}_{i,21}$, which then follows from the assumption on $F_{i}$.
\end{pf}

\subsection{The dynamical KZ equation}\label{ss:DKZ2}

The dynamical KZ equation is the connection on the trivial bundle over
$\IC^{\times}$ with fiber $\DBLHAH{\rootsys}{2}$ given by
\begin{equation}\label{eq:DKZ2}
	d-\left(\frac{\Tdh{}{}}{z}+\admu{\mu}{1}\right)dz
\end{equation}
It has a regular singularity at $z=0$, and an irregular singularity at $z=\infty$.
We shall exploit these singularities to produce a collection of differential twists 
satisfying the assumptions of Theorem~\ref{ss:all together}.

\subsubsection{Canonical fundamental solution at $z=0$}\label{ss:fund 0}

\begin{proposition}[\cite{vtl-6}]\label{pr:Fuchs 0}\hfill
	\begin{enumerate}
		\item For any $\mu\in\h$, there is a unique holomorphic function
		$H_0:\IC\to\DBLHAH{\rootsys}{2}$ 
		such that $H_0(0,\mu)\equiv 1$ and, for any determination
		of $\log(z)$, the $\sfEnd{\DBLHAH{\rootsys}{2}}$--valued function
		\[\Upsilon_0(z,\mu)=e^{z\admu{\mu}{1}}\cdot H_0(z,\mu)\cdot z^{\Tdh{}{}}\]
		is a fundamental solution of the dynamical KZ equations.
		\item $H_0$ and $\Upsilon_0$ are holomorphic functions of $\mu\in\h$,
		and $\Upsilon_0$ satisfies
		\[d_\h\Upsilon_0=
		\sum_{\alpha\in\Phi_+}\frac{d\alpha}{\alpha}
		\left[\Kdh{\alpha}{(2)},\Upsilon_0\right]
		+z\admu{d\mu}{1}
		\Upsilon_0\]
	\end{enumerate}
\end{proposition}

\subsubsection{Canonical fundamental solutions at $z=\infty$}\label{ss:fund infty}

Let $\IH_\pm=\{z\in\IC|\,\Ima(z)\gtrless 0\}$.

\begin{theorem}[\cite{vtl-6}]\label{th:Hn}\hfill
	\begin{enumerate}
		\item For any $\mu\in\C$, there is a unique holomorphic function
		$H_\pm:\IH_\pm\to\DBLHAH{\rootsys}{2}$ such that $H_\pm(z)$ tends to $1$ 
		as 
		\[z\to\infty\quad\text{with}\quad|\arg(z)|\in(\delta,\pi-\delta)\]
		$\delta>0$, and, for any determination of $\log(z)$, the $\sfEnd{\DBLHAH
			{\rootsys}{2}}$--valued function
		\[\Upsilon_\pm(z)=H_\pm(z)\cdot e^{z\admu{\mu}{1}}\cdot z^{\Tdh{}{0}}\]
		is a fundamental solution of the dynamical KZ equations.
		\item $H_\pm$ and $\Upsilon_\pm$ are smooth functions of $\mu\in\C$, and $\Upsilon_\pm$
		satisfies
		\[d_\h\Upsilon_\pm=
		\sum_{\alpha\in\Phi_+}\frac{d\alpha}{\alpha}
		\left(\Kdh{\alpha}{(2)}\Upsilon_\pm - \Upsilon_\pm (\Kdh{\alpha}{1}+\Kdh{\alpha}{2})\right)
		+z\admu{d\mu}{1}\Upsilon_\pm\]
	\end{enumerate}
\end{theorem}

\subsection{Differential twist from the dynamical KZ equation}\label{ss:differential twist}

Fix henceforth the standard determination of $\log z$ with a cut
along the negative real axis, and let $\Upsilon_0,\Upsilon_\pm$
be the corresponding fundamental solutions of the dynamical KZ
equations given in \ref{ss:fund 0} and \ref{ss:fund infty}
respectively.

Let $F_\pm:\C\to\End(\DBLHAH{\rootsys}{2})$ be 
the smooth function defined by
\[F_\pm=\Upsilon_\pm(z)^{-1}\cdot\Upsilon_0(z)\] 
where $z\in\IC\setminus\IR_{\leq 0}$. $F_\pm$ is a regularised holonomy
of the dynamical KZ equations from $z=0$ to $z=\pm\iota\infty$. The form of
$\Upsilon_0,\Upsilon_\pm(z)$ shows that 
\[F_\pm=
z^{-\Omega_0}\cdot
\exp(-z\admu{\mu}{1})\left(H_\pm^{-1}\right)\cdot
H_0(z)\cdot
z^{\Omega}
\]
so that $F_\pm$ acts by left multiplication. We henceforth identify
$F_\pm$ and $F_\pm(1)$, and consider the former as taking values
in $\DBLHAH{\rootsys}{2}$.

\begin{theorem}[\cite{vtl-6}]
	$F_\pm$ is a differential twist with values in $\DBLHAH{\rootsys,2}{}$, which satisfies the 
	centraliser property and the assumption of Proposition~\ref{ss:all together}(2).
\end{theorem}


\part{Braided Coxeter categories}\label{part-three}

\section{Braided Coxeter categories}\label{s:braided-Cox}

In this section, we briefly review the definition of a braided Coxeter category
introduced in \cite{ATL1-2}. Roughly, this is a monoidal category carrying 
commuting actions of a generalised braid group $\BDm$ and Artin's braid groups 
$\Br{n}$ on the tensor powers of its objects. Under the Tannakian formalism,
a braided Coxeter category is the categorical counterpart of a braided Coxeter
algebra.

\subsection{Braided Coxeter categories}\label{ss:cox-cat}\label{ss:braided-cox-category}

Let $\dgr$ be a diagram with a labelling $\ulm$. A braided Coxeter category of type $(\dgr,\ulm)$ is a tuple $\cCox{}=(\C_B, F_{\F}, \DCPA{\G}{\F}, \redasso{\F}{\F'},\CoxS{}{}{i})$ 
consisting of the following data.
\vspace{0.25cm}
\begin{itemize}\itemsep0.25cm
	\item {\bf Diagrammatic categories.} For any subdiagram $B\subseteq
	\dgr$, a braided monoidal category $\C_B$.
	\item {\bf Restriction functors.} For any pair of subdiagrams $B'\subseteq B$
	and relative \mns $\F\in\Mns{B,B'}$, a tensor functor $F_\F:\C_B\to\C_{B'}$ 
	($F_\F$ is not assumed to be braided).
	\item {\bf Generalised associators.} For any pair of subdiagrams $B'\subseteq B$ and 
	relative \mnss $\F,G\in\Mns{B,B'}$, an isomorphism of tensor functors
	$\DCPA{\G}{\F}:F_\F\Rightarrow F_\G$.
	\item {\bf Vertical joins.} For any chain of inclusions $B''\subseteq B'\subseteq B$, $\F\in\Mns{B,B'}$,
	and $\F'\in\Mns{B',B''}$, an isomorphism of tensor functors $\redasso{\F}{\F'}: F_{\F'}\circ F_{\F}
	\Rightarrow F_{\F'\cup\F}$.
	\item {\bf Local monodromies.} For any vertex $i$ of $\dgr$ with
	corresponding restriction functor $\Fi:\C_i\to\C_\emptyset$, a distinguished
	automorphism $\CoxS{}{}{i}\in\Aut(\Fi)$ ($\CoxS{}{}{i}$ is not assumed to be
	a monoidal automorphism).
\end{itemize}

\noindent
These data are assumed to satisfy the following properties.
\vspace{0.25cm}
\begin{itemize}\itemsep0.25cm
	\item {\bf Normalisation.} If $\F=\{B\}$ is the unique
	element in $\Mns{B,B}$, then $F_\F=\id_{\C_\F}$ with the trivial tensor structure.
	\item {\bf Transitivity.} For any $B'\subseteq B$ and $\F,\G,
	\H\in\Mns{B,B'}$, $\DCPA{\H}{\F}=\DCPA{\H}{\G}\circ\DCPA{\G}{\F}$
	as isomorphisms $F_\F\Rightarrow F_\H$. In particular, $\DCPA{\F}{\F}
	=\id_{F_\F}$ and $\Upsilon_{\G\F}=\Upsilon_{\F\G}^{-1}$.
	\item {\bf Associativity.} For any $B'''\subseteq B''\subseteq B'\subseteq B$, 
	$\F\in\Mns{B,B'}$, $\F'\in\Mns{B',B''}$, and $\F''\in\Mns{B'',B'''}$,
	\[\redasso{\F'\cup\F}{\F''}\cdot
	\redasso{\F}{\F'}=
	\redasso{\F}{\F''\cup\F'}\cdot
	\redasso{\F'}{\F''}\]
	as isomorphisms $F_{\F''}\circ F_{\F'}\circ F_{\F}\Rightarrow F_{\F''\cup\F'\cup\F}$.
	\item {\bf Vertical factorisation.} For any $B''\subseteq B'\subseteq B$,
	$\F,\G\in\Mns{B,B'}$ and $\F',\G'\in\Mns{B',B''}$, 
	\[
	\DCPA{(\G'\cup\G)}{(\F'\cup\F)}\circ\redasso{\F}{\F'}=\redasso{\G}{\G'}
	\circ
	\left(\begin{array}{l}\DCPA{\G}{\F}\\\phantom{00}\circ\\ \DCPA{\G'}{\F'}\end{array}\right)
	\]
	as isomorphisms $F_{\F'}\circ F_{\F}\Rightarrow F_{\G'}\circ F_{\G}$.
	\item {\bf Generalised braid relations.} 
	For any $B\subseteq\dgr$, $i\neq j\in B$ and \mnss $\Ki,\Kj$ on $B$ such that $\{i\}\in\Ki,
	\{j\}\in\Kj$,  the following holds in $\sfAut{F_{\Ki}}$
	\[
	\underbrace{\mathsf{Ad}\left(\DCPA{i}{j}\right)(\CoxS{\sfa}{}{j})\cdot \CoxS{\sfa}{}{i}
		\cdot \mathsf{Ad}\left(\DCPA{i}{j}\right)(\CoxS{\sfa}{}{j})\cdots}_{m_{ij}}=
	\underbrace{\CoxS{\sfa}{}{i}\cdot\mathsf{Ad}
		\left(\DCPA{i}{j}\right)(\CoxS{\sfa}{}{j})\cdot \CoxS{\sfa}{}{i}\cdots}_{m_{ij}}
	\]
	where $\DCPA{i}{j}=\DCPA{\Ki}{\Kj}$ and 
	$\CoxS{\sfa}{}{i}=\sfAd{\redasso{\trunc{\Ki}{}{i}}{\trunc{\Ki}{i}{}}}(\CoxS{}{}{i})\in
	\sfAut{F_{\Ki}}$\footnote{$\trunc{\Ki}{}{i}$ and $\trunc{\Ki}{i}{}$ denote the truncations 
		of $\Ki$ at $(\dgr,\{i\})$ and $(\{i\},\emptyset)$, respectively, see~Definition~\ref{ss:elem-seq}(\ref{def:truncation}).}.
	\item {\bf Coproduct identity.}
	For any $i\in D$, the following holds in 
	$\Aut\left(\Fi\ten \Fi\right)$
	\begin{equation}\label{eq:coxcoprod-cat}
		J_i^{-1}\circ 
		\Fi(c_i)\circ\Delta(\CoxS{}{}{i})\circ J_i=
		c_{\emptyset}\circ \CoxS{}{}{i}\ten \CoxS{}{}{i}
	\end{equation}
	where $J_i$ is the tensor structure on $\Fi$ and $c_i, c_{\emptyset}$ are
	the opposite braidings in $\C_i$ and $\C_{\emptyset}$, respectively.\footnote{
		Given a braided monoidal category with braiding $\beta$, we set $\beta^{\scs\operatorname{op}}_{X,Y}:=\beta_{Y,X}^{-1}$.
	} More specifically, for any $V,W\in\C_i$, the following diagram is commutative
	\[
	\xymatrix@C=1.5cm{
		\Fi(V)\ten \Fi(W) \ar[d]_{J_i^{V,W}} \ar[r]^{\CoxS{}{, V}{i}\ten \CoxS{}{, W}{i}} & \Fi(V)\ten \Fi(W)  \ar[r]^{c_{\emptyset}} & \Fi(W)\ten \Fi(V)
		\ar[d]^{J_i^{W,V}}\\
		\Fi(V\ten W) \ar[r]_{\CoxS{}{, V\ten W}{i}} & \Fi(V\ten W)  \ar[r]_{\Fi(c_i)} & \Fi(W\ten V)  
	}
	\]
\end{itemize}

\noindent\remark\;
The identity \eqref{eq:coxcoprod-cat} relates the 
failure of $(F_i,J_i)$ to be a braided monoidal functor and that of $S_i$
to be a monoidal isomorphism. That is, if \eqref{eq:coxcoprod-cat} holds,
$S_i$ is monoidal if and only if $J_i$ is braided. Conversely, if $S_i$ is
monoidal and $J_i$ is braided, then \eqref{eq:coxcoprod-cat} automatically
holds.

\subsection{Morphisms}\label{ss:precox mor}

Let $\cCox{}$, $\cCox{}'$ be two braided Coxeter categories of type $(\dgr,\ulm)$. 
A $1$--morphism $\mCox{}:\cCox{}\to\cCox{}'$ consists of the following data.
\vspace{0.25cm}
\begin{itemize}\itemsep0.25cm
	\item {\bf Horizontal functors.} For any $B\subseteq\dgr$, a braided tensor functor $H_B: \C_B\to\C'_B$.
	\item {\bf Diagonal isomorphisms. } For any $B'\subseteq B\subseteq\dgr$ and
	$\F\in\Mns{B,B'}$, an isomorphism of tensor functors 
	\[
	\begin{tikzcd}
		\C_B
		\arrow[r, "H_B"]
		\arrow[d,"F_{\F}"']
		&
		\C'_{B}	
		\arrow[d, "F'_{\F}"]
		\arrow[dl, Rightarrow, "\gamma_{\F}"']\\
		\C_{B'}
		\arrow[r, "H_{B'}"']
		&
		\C'_{B'}	
	\end{tikzcd}
	\]
	such that $\DCPAC{}{\G}{\F}
	\circ\gamma_{\F}=\gamma_{\G}\circ(\DCPAC{}{\G}{\F})'$ as isomorphisms
	$F'_{\F}\circ H_B\Rightarrow H_{B'}\circ F_{\G}$.
\end{itemize}
These data are assumed to satisfy the following properties.
\vspace{0.25cm}
\begin{itemize}\itemsep0.25cm
	\item {\bf Normalisation.} If $\F=\{B\}$ is the unique element 
	in $\Mns{B,B}$, so that $F_\F=\id_{\C_B}$ and $F'_\F=\id_{\C'_B}$, then 
	$\gamma_{\F}=\id_{H_B}$.
	\item {\bf Vertical factorisation.} For any $B''\subseteq B'\subseteq B$,
	$\F\in\Mns{B,B'}$ and $\F'\in\Mns{B',B''}$, the following equality holds
	\[\gamma_{\F'\cup\F}\circ(\redasso{\F}{\F'})'=
	\redasso{\F}{\F'}\circ
	\left(\begin{array}{l}\gamma_{\F}\\\phantom{0}\circ\\ \gamma_{\F'}\end{array}\right)
	\]
	as isomorphisms $F'_{\F'}\circ F'_{\F}\circ H_{B}
	\Rightarrow H_{B''}\circ F_{\F'}\circ F_{\F}$.
	\item {\bf Generalised braid group invariance.} 
	The generalised braid group operator are preserved, \ie 
	for any $i\in D$, $\CoxS{}{}{i}\circ\gamma_
	{\emptyset i}=\gamma_{\emptyset i}\circ S_i'$ as isomorphisms $F'_i\circ
	H_i\Rightarrow H_\emptyset\circ F_i$.\\
\end{itemize}

Finally, let $\mCox{}^1,\mCox{}^2$ be two $1$--morphisms $\cCox{}\to\cCox{}'$. 
A $2$--morphism $\nCox{}: \mCox{}^1\Rightarrow\mCox{}^2$ is the datum, for any $B\subseteq\dgr$, of a natural transformation of braided tensor functors 
$v_B:H^1_B\Rightarrow H^2_B$ such that, for any $B'\subseteq B$ and 
$\F\in\Mns{B,B'}$, $\gamma_{\F}\circ v_B=v_{B'}\circ \gamma_{\F}$ as morphisms 
$F'_{\F}\circ H^1_B\Rightarrow H^2_{B'}\circ F_{\F}$. 

\subsection{Coxeter algebras and Coxeter categories}\label{ss:braid-cox-alg-vs-cat}
The notion of braided Coxeter category is tailored to describe the category of
representations of a braided Coxeter algebra.
In particular, let $\ACox{}$ be a diagrammatic bialgebra and $\ACox{}^{\ten \bullet, \flat}$ the 
corresponding cosimplicial bidiagrammatic algebra (cf. Proposition \ref{ss:bi-diag-cosim}). 
We have the following

\begin{proposition}\hfill
	\begin{enumerate}\itemsep0.25cm
		\item
		Let $\sCox{}=(\Phi_B,R_B, \Jg{}{\F}{},\DCPA
		{\F}{\G},\redasso{\F}{\F'}, S_i)$ be a braided Coxeter structure on $\ACox{}^{\ten \bullet, \flat}$ 
		(cf. Definition~\ref{ss:braid-cox-alg}). Then, $\sCox{}$ gives rise to a braided Coxeter category 
		$\Rep_{\sCox{}}(\ACox{})$ given by the following data
		\vspace{0.25cm}
		\begin{itemize}\itemsep0.25cm
			\item for any $B\subseteq\dgr$, the braided monoidal category $\Rep(\rda{A}{}{B})$ with
			associativity and commutativity constraints given, respectively, by the action 
			of $\Phi_B\in\rda{A}{\ten 3, B}{B}$ and $R_B\in\rda{A}{\ten 2, B}{B}$
			\item for any $B'\subseteq B$ and $\F\in\Mns{B,B'}$, 
			the {\em tensor} restriction functor $\Res_{\F}: \Rep(\rda{A}{}{B})\to\Rep(\rda{A}{}{B'})$, 
			with tensor structure given by the action of $\Jg{}{\F}{}\in\rda{A}{\ten 2, B'}{B}$
			\item for any $B'\subseteq B$ and $\F,\G\in\Mns{B,B'}$, the natural isomorphism of tensor
			functors $\Res_{\G}\Rightarrow\Res_{\F}$ given by the action of $\DCPA{\F}{\G}\in A_B^{B'}$
			\item for any chain of inclusions $B''\subseteq B'\subseteq B$, $\F\in\Mns{B,B'}$,
			and $\F'\in\Mns{B',B''}$, an isomorphism of tensor functors $\Res_{\F'}\circ \Res_{\F}
			\Rightarrow \Res_{\F'\cup\F}$ given by the action of $\redasso{\F}{\F'}\in A_B^{B''}$
			\item for any vertex $i$ of $\dgr$, the invertible operator in $\Aut(F_{\{i\}})$ given by the action of
			$S_i\in A_{i}$.
		\end{itemize}
		\item
		Let $\tCox{}=(u_{\F},\twF{B})$ be a twist in $\ACox{}$ (cf.~Definition~\ref{ss:twist-gauge-braided-Cox}).
		There is a canonical $1$--isomorphism of braided Coxeter categories $\mCox{\tCox{}}:\Rep_{\sCox{}}(\ACox{})
		\to\Rep_{\sCox{\tCox{}}}(\ACox{})$ given by the tensor equivalences $H_{\tCox{},B}=(\id_B,\twF{B}): \Rep_{\sCox{}}(A_B)
		\to\Rep_{\sCox{\tCox{}}}(A_B)$, with tensor structure given by the action of 
		$\twF{B}\in A_B^{\ten 2,B}$, and the tensor isomorphisms $\gamma_{\tCox{},\F}:H_{\tCox{},B'}\circ \Res_{\sCox{}, \F}\Rightarrow
		\Res_{\sCox{\tCox{}},\F}\circ H_{\tCox{},B}$, given by the action of $u_{\F}\in A_B^{B'}$.
		\item
		Let $a=\{a_B\}$ be a gauge in $\ACox{}$. There is a canonical $2$--isomorphism 
		$\nCox{a}: \mCox{\tCox{}}\Rightarrow \mCox{\tCox{\gCox{}}}$
		with natural braided tensor isomorphism $v_{a,B}: H_{\tCox{},B}\Rightarrow H_{\tCox{\gCox{}}, B}$ given by the action 
		of $a_B\in A_B^B$.
	\end{enumerate}
\end{proposition}

\subsection{Braid group representations}\label{ss:braid-group-rep-cat}

The following is a categorical analogue of Propositions~\ref{ss:braid-Cox-to-braid-rep}
and \ref{ss:twist-gauge-braided-Cox}.

\begin{proposition}
	Let $\cCox{}=(\C_B, F_{\F}, \DCPA{\F}{\G}, \redasso{\F}{\F'}, \CoxS{}{i}{})$ be a braided Coxeter category. Then, there is a family of representations
	\[\lambda^{\cCox{}}_{\F, b}:\BBm\times\Br{n}\to\sfAut{F_{\F}^{\boxtimes n}}\]
	labelled by $B\subseteq\dgr$, $\F\in\Mns{B}$, and $b\in\brac{n}$, which is uniquely determined by the conditions
	\vspace{0.25cm}
	\begin{itemize}\itemsep0.25cm
		\item $\lambda^{\cCox{}}_{\F, b}(\topS{i})=\sfAd{\redasso{\trunc{\F}{}{i}}{\trunc{\F}{i}{}}}
		(\CoxS{}{}{i})_{1\dots n}$ if $\{i\}\in\F$
		and $\lambda^{\cCox{}}_{\G, b}=\sfAd{\DCPA{\G}{\F}}_{1\dots n}\circ\lambda^{\cCox{}}_{\F, b}$.
		\item $\lambda^{\cCox{}}_{\F, b}(\topT{i})=(i\ i+1)\circ (R_{B})_{i,i+1}$ if $b= x_1\cdots (x_ix_{i+1})\cdots x_n$ and $\lambda^{\cCox{}}_{\F, b'}=\sfAd{\Phi_{B, b'b}}\circ\lambda^{\cCox{}}_{\F, b}$.\\
	\end{itemize}
	Let $\mCox{}:\cCox{}\to\cCox{}'$ be a $1$--isomorphism of braided Coxeter
	categories. Then, the representations $\lambda^{\cCox{}}_{\F, b}$ and
	$\lambda^{\cCox{}'}_{\F, b}$ are equivalent through the natural isomorphism 
	$\gamma_{\F}: F'_{\F}\circ H_{B}\Rightarrow F_{\F}$.
\end{proposition}


\section{Braided Coxeter structures on Kac--Moody algebras}\label{s:km-cox}
In this section, we describe the standard symmetric Coxeter category associated to a diagrammatic
symmetrisable \KM algebra, and its deformations. 

\newcommand{\Og}{\O_{\g}}
\newcommand{\Ointg}{\Oint_{\g}}

\newcommand{\Oinfg}{\O_{\infty,\g}}
\newcommand{\Oinfintg}{\Oint_{\infty,\g}}
\newcommand {\gweight}[1]{[\negthinspace[#1]\negthinspace]}

\subsection{Category $\Oinf$ representations}\label{ss:cat-O-diag}
Let $\g$ be a diagrammatic symmetrisable Kac--Moody algebra (cf.~\ref{ss:diag-KM}).
If $V$ is an $\h$--module and $\lambda\in\h^*$, we denote the corresponding
weight space of $V$ by
\[V[\lambda]=\{v\in V|\,h\,v=\lambda(h)v,\,h\in\h\}\]
and set $P(V)=\{\lambda\in\h^*|\,V{[\lambda]}\neq0\}$.
Recall that a $\g$--module $V$ is
\begin{enumerate}[label=(C\arabic*)]\itemsep0.25cm
	\item\label{cond:weight-dec} a {\it weight module} if $V=\bigoplus_{\lambda\in\h^*}
	V{[\lambda]}$.
	\item\label{cond:int} {\it integrable} if it is a weight module, and the elements $\{e_i,
	f_i\}_{i\in\bfI}$ act locally nilpotently.
	
	This implies that $\lambda(h_i)\in\IZ$ for any $\lambda\in P(V)$ and $i\in\bfI$,
	and that $V$ is completely reducible as a (possibly infinite) direct sum of simple
	\fd modules over $\sl{2}^{\alpha_i}=\langle e_i, h_i,f_i\rangle\subset\g$.
	\item\label{cond:Oinfty} in {\it category $\Oinfg$} if the action of $\bp{}$  is locally finite,
	\ie any $v\in V$ is contained in a \fd $\bp{}$--submodule of $V$.
	This is equivalent to $V$ being the direct sum of its generalised weight spaces,
	together with
	\begin{enumerate}[label=(C\arabic*$^\prime$), start=3]
	\item\label{cond:Oinfty-bis} for any $v\in\V$, $(U\np{})_{\beta}v=0$ for all but finitely many $\beta\in\sfQ_+$.
	\end{enumerate}
	\item\label{cond:O} in {\it category $\Og$} if it is a weight module with \fd weight
	spaces, such that
	\begin{equation}\label{eq:cone condition}
		P(V)\subseteq D(\lambda_1)\cup\cdots\cup D(\lambda_m)
	\end{equation}
	for some $\lambda_1,\ldots,\lambda_m\in\h^*$, where 
	$D(\lambda)=\{\mu\in\h^*\;|\;\mu\leqslant\lambda\}$ and
	$\mu\leqslant\lambda$ iff $\lambda-\mu\in\sfQ_+=\bigoplus_{i\in\bfI}\IN\alpha_i$.
\end{enumerate}

The categories $\Og\subset\Oinfg$ are symmetric tensor categories. Let  $\Ointg\subset
\Og$ and $\Oinfintg\subset\Oinfg$ be the full tensor subcategories of integrable representations. We
have the following inclusions
\[\xymatrix@C=.1cm@R=.1cm
{\Og&\subset&\Oinfg\\
	\cup&&\cup\\
	\Ointg&\subset&\Oinfintg
}\]

\noindent\remarks\hfill
	\begin{enumerate}\itemsep0.25cm
		\item Category $\O$ does not fit naturally within the framework of Coxeter categories since
		condition \ref{cond:O} is not stable under restriction to diagrammatic Lie subalgebras. It is
		therefore convenient to consider instead the categories $\Oint_{\infty, \g_B}$, $B\subseteq\dgr$, 
		with restriction functors $\Res_{B'B}:\Oint_{\infty, \g_B}\to\Oint_{\infty, \g_B'}$, $B'\subseteq B$.
		\item As pointed out in \cite[Sec.~13.9]{ATL1-2}, the lack of diagrammatic restriction functors at 
		the level of categories $\O$ can also be overcome by replacing the Lie subalgebras $\g_B$ with 
		the Levi subalgebras $\mathfrak{l}_B=\g_B+\h$. These, however, do not induce a
		diagrammatic structure on $U\g$ since $\mathfrak{l}_{B'}$ and $\mathfrak{l}_{B''}$ do not commute
		if $B'\perp B''$, and require a further modification of the framework.
	\end{enumerate}

\subsection{The symmetric Coxeter category $\OCox{\g}{\scsop{\sint}}$}\label{ss:km-sym-cox}
Let $W$ be the Weyl group of $\g$ with set of simple reflections $\{s_i\}_{i\in\bfI}\subset W$. 
Set $\ulm=(m_{ij})$, where $m_{ij}$ is the order of $s_is_j$ in $W$. Let 
\[\UOintz{\g}{1}=\End(\Oint_{\infty,\g}\to\vect)\]
be the algebra of endomorphism of 
the forgetful functor. Then, it is well--known that $U\g\subset\UOintz{\g}{1}$, \ie the objects in $\Oint_{\infty,\g}$ separate
points in $U\g$, and $\texp{i}\in(\UOintz{\g}{1})^{\times}$ where $\texp{i}=\exp(e_i)\cdot\exp(-f_i)\cdot\exp(e_i)$.
Since the triple exponential operators satisfy the generalised braid relations \eqref{eq:gen-braid},
we obtain a homomorphism $\Br{W}\to(\UOintz{\g}{1})^{\times}$
given by 
$\topS{i}\mapsto\texp{i}$ (cf.~Remarks~\ref{ss:flatness} and \ref{ss:gen-braid-gp-diag}).
The following is straightforward.

\begin{proposition}
	There is a canonical $(\redasso{}{},\DCPA{}{})$--strict symmetric Coxeter category 
	$\OCox{\g}{\scsop{\sint}}$ of type $(\dgr,\ulm)$ given by the following data.
	\vspace{0.25cm}
	\begin{itemize}\itemsep0.25cm
		\item For any $B\subseteq\dgr$, the symmetric monoidal category $\Oint_{\infty, \g_B}$.
		\item For any $B'\subseteq B$, the restriction functor $\Res_{B'B}:\Oint_{\infty, \g_B}\to
		\Oint_{\infty, \g_{B'}}$ with the trivial tensor structure.
		\item For any $i\in\dgr$, the operator $\CoxS{\OCox{}{}}{}{i}=\texp{i}\in(\UOintz{\g}{1})^{\times}$.
	\end{itemize}
\end{proposition} 
\begin{pf}
	It is enough to observe that the operator $\CoxS{\OCox{}{}}{}{i}$ is group--like and therefore satisfies 
	the coproduct identity \eqref{eq:coxcoprod-cat}, which for the symmetric category $\Oint_{\infty, 
		\g_i}$ reduces precisely to the condition $\Delta(\CoxS{\OCox{}{}}{}{i})=\CoxS{\OCox{}{}}{}{i}\ten\CoxS{\OCox{}{}}{}{i}$.
\end{pf}

\newcommand{\Ohg}{\O^{\hbar}_{\g}}
\newcommand{\Ohinfg}{\O^{\hbar}_{\infty,\g}}
\newcommand{\Ohintg}{\O^{\hbar,\sint}_{\g}}
\newcommand{\Ohinfintg}{\O^{\hbar,\sint}_{\infty,\g}}
\newcommand{\hgt}{\operatorname{ht}}

\subsection{Deformation category $\Oinf$ representations}\label{ss:cat-O-diag-h}\label{ss:def-catO}
We shall be interested in deformations of the symmetric Coxeter category $\OCox{\g}{\scsop{\sint}}$.
To this end, consider the deformation parameter $\hbar$ and set $\nablah=\hbar/2\pi\iota$
(cf.~\ref{ss:casimir-conn}). 
Let $\tfV$ be the category of topologically free $\hext{\IC}$--modules.
A $\g$--module $\V\in\tfV$ 
is called 
\begin{enumerate}[label=(D\arabic*)]\itemsep0.25cm
	\item\label{cond:weight-dec-h} a {\it weight module} if $\V=\bigoplus_{\lambda\in\h^*}\V[\lambda]$,\footnote{Note
		that the eigenvalues of the action of $\h$ on $\V$ are required to lie in $\h^*\subsetneq
		\h^*\fml$.} where $\bigoplus$ is the direct sum in $\tfV$, \ie the completion of
	the algebraic direct sum in the $\hbar$--adic topology.
	\item\label{cond:int-h} {\it integrable} if it is a weight module and, for any $i\in\bfI$ and $v\in\V$,
	$\lim_{n\to\infty}e_i^n v=0=\lim_{n\to\infty}f^n_iv$, where the limit is taken in the $\hbar$--adic
	topology.
	
	This implies that $\V$ is complete reducible as a (possibly infinite) direct sum of indecomposable 
	finite--rank modules over $\sl{2}^{\alpha_i}=\langle e_i, h_i,f_i\rangle$.
	\item\label{cond:Oinftyh} in {\it category $\Ohinfg$} if the action of $\bp{}$ on $\V/\hbar^n\V$
	is locally finite for any $n\geqslant 0$.	
	This is equivalent to $\V$ being the $\hbar$--adic direct sum of its generalised weight spaces, and 
	\begin{enumerate}[label=(D\arabic*$^\prime$), start=3]\itemsep0.25cm
		\item\label{cond:Oinftyh-bis} for any $v\in\V$, $\displaystyle\lim_{\hgt(\beta)\to\infty}(U\np{})_{\beta}v=0$.
	\end{enumerate}
	\item\label{cond:O-h} in category $\Ohg$ if it is a weight representation with finite--rank
	weight spaces, and such that $P(\V)$ satisfies \eqref{eq:cone condition}.
\end{enumerate}

It is easy to see that $\V$ is a weight (resp. integrable, in category $\Ohinfg$) module in $\tfV$
if and only if $\V/\hbar^n\V$ is a weight (resp. integrable, in category $\Oinfg$) module in $\vect$
for any $n\geqslant 0$.\\

We denote by $\Ohintg\subset\Ohg$ and $\Ohinfintg\subset\Ohinfg$ 
the full tensor subcategories of integrable representations.
We shall describe the deformations of $\OCox{\g}{\scsop{\sint}}$ arising
from braided Coxeter structures on the cosimplicial lax bidiagrammatic algebra $\UCoxOint{\g}{\bullet}$
of endomorphisms of the forgetful functor  from $\O^{\hdef, \sint}_{\infty,\g}$
to $\tfV$.

\subsection{The cosimplicial algebra $\UOint{\g}{\bullet}$}\label{sss:Ugh-cosimp}
Let
\[\mathsf{f}:\O^{\hdef, \sint}_{\infty,\g}\to\tfV
\aand
\UOint{\g}{n}=\sfEnd{\ff^{\boxtimes n}}\]
be the forgetful functor and endomorphisms of its $n$th tensor power. By   \cite[Thm.~3.1]{ATL5} the category 
$\O^{\hdef, \sint}_{\infty,\g}$ separates points in $U\g\fml$. Thus,
we get a natural embedding  $\hext{U\g^{\otimes n}}
\subset\UOint{\g}{n}$.
The tower of algebras $\{\UOint{\g}{n}\}_{n\geqslant 0}$ is a cosimplicial 
algebra with the face and degeneration morphisms described in \ref{ss:cosim}.
Moreover, there is a canonical embedding of cosimplicial algebras $\hext{U\g^{\ten \bullet}}\subset\UOint{\g}{\bullet}$.

\subsection{Bidiagrammatic structure on $\UOint{\g}{\bullet}$}\label{ss:diagr-struc-U}

For $B'\subseteq B\subseteq\dgr$ and $n\geqslant 0$, let
\[\sff_{B'B}:\O^{\hdef, \sint}_{\infty, \g_B}\to
\O^{\hdef, \sint}_{\infty, \g_{B'}}
\aand
\UOint{\g,BB'}{n}=\sfEnd{\sff^{\boxtimes n}_{B'B}}
\]
be the restriction functor and the algebra of endomorphisms of its $n$th tensor power.
In particular, we have $\sff_{\emptyset B}=\sff_{B}$ and $\UOint{\g,B\emptyset}{n}=
\UOint{\g_B}{n}$.
Note that 
$\hext{(U\g_B^{\ten n})^{\g_{B'}}}\subset\UOint{\g,BB'}{n}$. 
The collection of algebras $\{\UOint{\g,BB'}{n}\;|\; B'\subseteq B\}$ gives rise to a
lax bidiagrammatic algebra (cf.~\ref{ss:bi-diag-alg} and ~\cite[Sec.~8.6]{ATL1-2})
with the following structural 
morphisms.
\begin{itemize}\itemsep0.25cm
	\item For any $C\subseteq B$, $C'\subseteq B'$, with 
	$C\subseteq C'\subseteq B'\subseteq B$, the identity
	\[\ff_{CC'}\circ\ff_{C'B'}\circ\ff_{B'B}=\ff_{CB}\]
	induces a canonical morphism of algebras
	$\UOint{\g,B'C'}{n}\to\UOint{\g,BC}{n}$.
	\item For any 
	$C_1\subseteq B_1\perp B_2\supset C_2$, the identities $\g_{B_1\sqcup B_2}=\g_{B_1}\oplus
	\g_{B_2}$ and $\g_{C_1\sqcup C_2}=\g_{C_1}\oplus\g_{C_2}$ imply that the natural morphism 
	$\UOint{\g,B_1C_1}{n}\ten\UOint{\g,B_2C_2}{n}\to\UOint{\g,B_1\sqcup B_2}{n}$ factors through the image of $\UOint{\g,(B_1\sqcup B_2)(C_1\sqcup C_2)}{n}$
	in $\UOint{\g,B_1\sqcup B_2}{n}$.
\end{itemize}
We denote by $\UCoxOint{\g}{\bullet}$ the resulting lax bidiagrammatic cosimplicial algebra.

\subsection{Braided Coxeter structures on $\UCoxOint{\g}{\bullet}$}\label{ss:from-end-to-cat}
The following is an analogue of Proposition~\ref{ss:braid-cox-alg-vs-cat}.

\begin{proposition}\hfill
	\begin{enumerate}\itemsep0.25cm
		\item
		Let $\sCox{}=(\Phi_B,R_B, \Jg{}{\F}{},\DCPA{\F}{\G},\redasso{\F}{\F'}, S_i)$ be a braided 
		Coxeter structure on $\UCoxOint{\g}{\bullet}$ (cf.~Definition~\ref{ss:braid-cox-alg}). Then,
		$\sCox{}$ gives rise to a Coxeter category $\OCox{\sCox{}}{\scsop{\hdef,\sint}}$ given by
		the following data
		\vspace{0.25cm}
		\begin{itemize}\itemsep0.25cm
			\item For any $B\subseteq\dgr$, the braided monoidal category
			$\O^{\hdef, \sint}_{\infty, \g_B}$
			with associativity and commutativity constraints given, respectively, by $\Phi_B\in\UOint{\g,BB}{3}$ and $R_B\in\UOint{\g,BB}{2}$
			\item For any $B'\subseteq B$ and $\F\in\Mns{B,B'}$, 
			the restriction functor $\Res_{\F}:  \O^{\hdef, \sint}_{\infty, \g_B}\to\O^{\hdef, \sint}_{\infty, \g_{B'}}$ 
			with tensor structure given by $\Jg{}{\F}{}\in\UOint{\g,B'B}{2}$
			\item For any $B'\subseteq B$ and $\F,\G\in\Mns{B,B'}$, the isomorphism of tensor
			functors $\Res_{\G}\Rightarrow\Res_{\F}$ given by $\DCPA{\F}{\G}\in\UOint{\g,B'B}{1}$
			\item For any chain of inclusions $B''\subseteq B'\subseteq B$, $\F\in\Mns{B,B'}$,
			and $\F'\in\Mns{B',B''}$, the isomorphism of tensor functors $\Res_{\F'}\circ \Res_{\F}
			\Rightarrow \Res_{\F'\cup\F}$ given by $\redasso{\F}{\F'}\in\UOint{\g,B''B}{1}$
			\item For any vertex $i$ of $\dgr$, the invertible operator $S_i\in\left(\UOint{\g,\{i\}}{1}\right)^{\times}$
		\end{itemize}
		\item
		Let $\tCox{}=(u_{\F},\twF{B})$ be a twist in $\UCoxOint{\g}{\bullet}$ (cf.~Definition~\ref{ss:twist-gauge-braided-Cox}).
		Then, $\tCox{}$ gives rise to a $1$--isomorphism of braided Coxeter categories $\mCox{\tCox{}}:\OCox{\sCox{}}{\scsop{\hdef,\sint}}\to\OCox{\sCox{\tCox{}}}{\scsop{\hdef,\sint}}$ given by the tensor equivalences
		\[H_{\tCox{},B}=(\id_B,\twF{B}): \OCox{\sCox{}, B}{\scsop{\hdef,\sint}}
		\to\OCox{\sCox{\tCox{}}, B'}{\scsop{\hdef,\sint}}\]
		with tensor structure given by the action of 
		$\twF{B}\in\UOint{\g,BB}{2}$, and the tensor isomorphisms $\gamma_{\tCox{},\F}:H_{\tCox{},B'}\circ \Res_{\sCox{}, \F}\Rightarrow
		\Res_{\sCox{\tCox{}},\F}\circ H_{\tCox{},B}$, given by the action of $u_{\F}\in\UOint{\g,BB'}{1}$.
		\item
		Let $a=\{a_B\}$ be a gauge in $\ACox{}$. Then, $a$ gives rise to a $2$--isomorphism 
		$\nCox{a}: \mCox{\tCox{}}\Rightarrow \mCox{\tCox{\gCox{}}}$
		with natural braided tensor isomorphism $v_{a,B}: H_{\tCox{},B}\Rightarrow H_{\tCox{\gCox{}}, B}$ given by the action 
		of $a_B\in\UOint{\g,BB}{1}$.
	\end{enumerate}
\end{proposition}

\section{Double holonomy and the category $\OCox{\g,\nabla}{\scsop{\hdef, \sint}}$}\label{s:km-holo}
In this section. we prove that the braided Coxeter structure $\sCox{\nabla}$ on the extended double
holonomy algebra $\DBLHAH{\Rs{}}{\bullet,\scsop{ext}}$ arising from the monodromy data of the KZ
and Casimir connections (Theorem~\ref{thm:holo-cox}) gives rise to a braided Coxeter structure on
$\UCoxOint{\g}{\bullet}$, and therefore to a braided Coxeter category $\OCox{\g,\nabla}{\scsop{\hdef,
\sint}}$.

\subsection{From the extended double holonomy algebra $\DBLHAH{\Rs{}}{\bullet,\scsop{ext}}$ to $\UCoxOint{\g}{\bullet}$}\label{ss:dblh-U}

\begin{proposition}
	There is a 	canonical morphism of cosimplicial lax diagrammatic
	algebras $\xi_W^{\scsop{\bullet}}:\DBLHAH{\Rs{}}{\bullet,\scsop{ext}}\to\UCoxOint{\g}{\bullet}$. 
\end{proposition}

The construction of $\xi$ is carried out in \ref{sss:hol-U}--\ref{sss:dblh-km}.
Recall that $\nablah=\hbar/2\pi\iota$.

\subsubsection{The holonomy algebra $\DCPHA{\rootsys}$ and $\UOint{\g}{1}$}\label{sss:hol-U}
Let $\rootsys$ be the root system of the diagrammatic Kac--Moody algebra $\g$
and $\DCPHA{\rootsys}$ the corresponding holonomy algebra 
with diagrammatic subalgebras $\DCPHA{\Rs{},B}$, $B\subseteq\dgr$  (cf.~\ref{ss:holonomy} 
and \ref{ss:comp-diag-sub}). Recall that $U\g$ naturally embeds in $\UOint{\g}{1}$.
We have the following

\begin{lemma}\hfill
	\begin{enumerate}\itemsep0.25cm
		\item
		There is a morphism of algebras $\xi_{\rootsys}:\DCPHA{\rootsys}\to\UOint{\g}{1}$ defined by
		\[
		\xi_{\rootsys}(\Kh{\alpha}{})=\nablah\cdot\Ku{\alpha}{+}
		\]
		where $\Ku{\alpha}{+}=\sum_{i=1}^{\rsm{\alpha}}e_{-\alpha}^{(i)}e_\alpha^{(i)}$ is the normally--ordered, truncated
		Casimir operator \eqref{eq:K +}. $\xi_{\rootsys}$ is compatible with the grading, and therefore extends
		to a morphism $\DCPHAH{\rootsys}\to\UOint{\g}{1}$.
		\item
		For any $B\subseteq\dgr$, the restriction of $\xi_{\rootsys}$ to $\DKHAH{B}{}\subseteq\DCPHAH{\rootsys}$ coincides
		with the morphism $\xi_{\Rs{B}}:\DKHAH{B}{}\to\UOint{\g,B}{1}$. In particular, $\xi_{\rootsys}$ is 
		a morphism of lax diagrammatic algebras.
	\end{enumerate}
\end{lemma}

\begin{pf}
	(1) follows from the commutation relations proved in \ref{ss:flatness}. 
	(2) is clear.
\end{pf}

\subsubsection{The holonomy algebra $\DKHA{}{n}$ and $\UOint{\g}{n}$}\label{sss:holo-to-Ug}

Let $r\in\UOint{\g}{2}$ be the classical $r$--matrix of $\g$, \ie in the notation of \ref{ss:casimir-conn}. 
\begin{equation}\label{eq:r-mx}
	r=\sum_{\alpha\in\Rs{+}}\sum_{i=1}^{\dim\g_{\alpha}}e_{-\alpha}^{(i)}\ten e_{\alpha}^{(i)}
	+\sum_{j=1}^{\dim\h}x_j\ten x^j
\end{equation}
where $\{x_j\}$, $\{x^j\}$ are dual bases of $\h$ with respect to the inner product $\iip{\cdot}{\cdot}$.
Note that, if $|\Rs{+}|=\infty$, $r\not\in U\g^{\ten 2}$. For any $n\geqslant 2$ and $1\leq i\neq j\leq n$,
set
\[\Omega^{ij}=r^{ij}+r^{ji}\in\UOint{\g}{n}\]

Let $\DKHA{}{n}$ be the holonomy algebra introduced in \ref{ss:KZ-holo}. The following result is well--known
(see \eg~\cite{e3}).
\begin{lemma}
There is a morphism of algebras $\xi^{\scs{n}}: \DKHA{}{n}\to\UOint{\g}{n}$ defined by
	\[
	\xi^{\scs{n}}(\Th{ij})=\nablah\cdot\Omega^{ij} 
	\]
	$\xi^{\scs{n}}$ is compatible with the cosimplicial structure and the grading on $\DKHA{}{n}$ given 
	by $\deg\Th{ij}=1$, and therefore extends to a morphism of cosimplicial algebras $\xi^{\scs{\bullet}}:\DKHAH{}{\bullet}
	\to\UCoxOint{\g}{\bullet}$. 
\end{lemma}

\noindent\remark\;
An element $\varphi\in\UOint{\g}{n}$ is $\g$--invariant if $[\varphi,\Delta^{(n)}(x)]=0$ for any $x\in\g$.
Since each $\Omega^{ij}$ is $\g$--invariant, it follows that $\mathsf{im}(\xi^{\scs{n}})\subseteq(\UOint
{\g}{n})^{\g}$.

\subsubsection{The root refinement $\DBLHA{}{\bullet,\scsop{\Rs{}}}$ and $\UCoxOint{\g}{\bullet}$}\label{sss:holo-refinement}

We now discuss the relation between the algebra $\DKHA{}{n,\scsop{\Rs{}}}$, which is a
root refinement of $\DKHA{}{n}$ (cf.~\ref{ss:R-KZ-holo}), and the algebra $\UOint{\g}{n}$.
For any $\alpha\in\rootsys$, set 
\[
r_{\alpha}=\sum_{a=1}^{\dim\g_{\alpha}}(e_{-\alpha})_{a}\ten (e_{\alpha})^{a}
\in\g_{-\alpha}\otimes\g_\alpha
\]
where $\{(e_{-\alpha})_{a}\},\{(e_{\alpha})^{a}\}$ are dual bases of
$\g_{-\alpha}$, $\g_{\alpha}$, and
\[
\Omega_{0}=\sum_{a=1}^{\dim\h}x_{a}\ten x^{a}\in\h\otimes\h
\]
where $\{x_{a}\},\{x^a\}$ are dual bases of $\h$. The following is clear.

\begin{lemma}
	There is a morphism of algebras $\xi^{\scs{n},\scsop{\Rs{}}}: \DKHA{}{n,\scsop{\Rs{}}}\to\UOint{\g}{n}$ defined by
	\[
	\xi^{\scs{n},\scsop{\Rs{}}}(\Tdh{ij}{0})=\nablah\cdot\Omega^{ij}_{0}\qquad
	\xi^{\scs{n},\scsop{\Rs{}}}(\Trdh{ij}{\alpha})=\nablah\cdot r^{ij}_{\alpha}
	\]
	$\xi^{\scs{n},\scsop{\Rs{}}}$ is compatible with the cosimplicial structure and the 
	grading on $\DKHA{}{n,\scsop{\Rs{}}}$, and therefore extends to a morphism of cosimplicial 
	algebras $\xi^{\scsop{\bullet},\scsop{\Rs{}}}:\DKHAH{}{\bullet,\scsop{\Rs{}}}\to\UCoxOint{\g}{\bullet}$.
\end{lemma}

\subsubsection{The extended double holonomy algebra $\DBLHAH{\Rs{}}{\bullet}$ and $\UCoxOint{\g}{\bullet}$}\label{sss:dblh-km}

Recall that, with respect to the root refinement $\DKHA{}{n,\scsop{\Rs{}}}$, the algebra $\DBLHA{\rootsys}
{n}$ is endowed with additional generators $\{\Kdh{\alpha}{(n)},\Kdh{\alpha}{k}\}_{\substack{\alpha\in\Rs{+}
\\ 1\leq k\leq n}}$ (cf.~\ref{ss:doubleholo}). These should be thought of as the elements of $\UOint{\g}{n}$
given by, respectively
\[\nablah\cdot\Delta^{(n)}(\Ku{\alpha}{+})\aand
\nablah\cdot\left(1^{\otimes (k-1)}\otimes\Ku{\alpha}{+}\otimes 1^{\otimes (n-k)}\right)\]
where $\Ku{\alpha}{+}$ is Casimir operator \eqref{eq:K +}. 
Specifically, for any $B\subseteq\dgr$, set
\[
\Omega_{0,B}=\sum_{a=1}^{\dim\h_B}x_{B,a}\ten x_B^{a}\in\h_B\otimes\h_B
\]
where $\{x_{B,a}\},\{x_B^a\}$ are dual bases of $\h_B$. Then, the following holds

\begin{lemma}\hfill
\begin{enumerate}
\item There is a morphism of lax diagrammatic algebras $\xi_{\rootsys}^{\scs{n}}: \DBLHA{\rootsys}{n}\to\UOint{\g}{n}$ defined by
\begin{gather*}
\xi^{\scs{n}}_{\rootsys}(\Tdh{ij}{0,B}) =\nablah\cdot\Omega^{ij}_{0,B}
\qquad\qquad
\xi^{\scs{n}}_{\rootsys}(\Trdh{ij}{\alpha}) =\nablah\cdot r^{ij}_{\alpha}
\\[2ex]
\xi_{\rootsys}^{\scs{n}}(\Kdh{\alpha}{(n)}) =\nablah\cdot\Delta^{(n)}(\Ku{\alpha}{+})
\qquad\qquad
\xi_{\rootsys}^{\scs{n}}(\Kdh{\alpha}{k})=\nablah\cdot(1^{\ten (k-1)}\ten\Ku{\alpha}{+}\ten 1^{\ten (n-k)})
\end{gather*}
\item $\xi^{\scs{n}}_{\rootsys}$ is compatible with the cosimplicial structure, the action of $\SS_n\ltimes\h^{\oplus n}$,
and the grading on $\DBLHA{\rootsys}{n}$ given by $\deg\Tdh{}{}=\deg\Trdh{}{}=\deg\Kdh{}{}=1$. It therefore extends
to a morphism of cosimplicial lax diagrammatic algebras $\xi^{\scsop{\bullet}}_{\rootsys}:\DBLHAH{\rootsys}{\bullet}\to\UCoxOint{\g}{\bullet}$.
\item The following holds
	\[
	\xi_{\rootsys}^{\scs{n}}((\Kdh{\alpha}{k})^{(m)})=
	\nablah\cdot(1^{\ten(k-1)}\ten\Delta^{(m)}(\Ku{\alpha}{+})\ten 1^{\ten(n-m-k+1)})
	\]
	where $(\Kdh{\alpha}{k})^{(m)}$ are defined in \eqref{eq:coproduct-K}.
\end{enumerate}
\end{lemma}
\begin{pf}
(1)
	The relations satisfied by $\Tdh{ij}{}, 1\leqslant i<j\leqslant n$, follow from the commutativity of the diagram
	\[
	\xymatrix{\DBLHA{\rootsys}{n} \ar[r]^{\xi_{\rootsys}^{\scs{n}}} & \UOint{\g}{n}\\ \DKHA{}{n} \ar[u]^{\iota_{\rootsys}^{\scs{n}}} \ar[ur]_{\xi^{\scs{n}}} &}
	\]
	where the vertical arrow $\iota_{\rootsys}^{\scs{n}}$ is the natural morphism from $\DKHA{}{n}$ 
	in $\DBLHA{\rootsys}{n}$. The relations \eqref{eq:Casimir-relations} are satisfied by the elements 
	$\Ku{\alpha}{+}\in\UOint{\g}{n}$ (cf. Theorem~\ref{ss:flatness}). 
	The invariance relations $[\Tdh{ij}{},\Kdh{\alpha}{(n)}]=0$ \eqref{eq:invariance}
	follow from the $\g$--invariance of $\Omega^{ij}=r^{ij}+r^{ji}$ in $\UOint{\g}{n}$.
	The coproduct identity \eqref{eq:Kn coproduct} holds in $\UOint{\g}{n}$ since
\[	\Delta^{(n)}(\Ku{\alpha}{+})=\Delta^{(n)}\left(\sum_{i=1}^{\rsm{\alpha}}e_{-\alpha}^{(i)}e_{\alpha}^{(i)}\right)
	=\sum_{i<j}(r^{ij}_{\alpha}+r^{ij}_{-\alpha})+\Ku{\alpha}{0}
\]
where $\Ku{\alpha}{0}=\sum_{i=1}^n(1^{\ten k-1}\ten\Ku{\alpha}{+}\ten 1^{\ten n-k})$ 
is a weight zero element. 
	
(2)--(3) are clear.
\end{pf}

Through the action of the braid group $\Br{W_B}$ on any object in $\O_{\infty,B}^{\hdef,\sint}$, we readily lift 
the collection of the morphisms $\xi_{\Rs{B}}^{\scs{n}}:\DBLHA{\Rs{},B}{n}\to\UOint{\g,B}{n}$, $B\subseteq\dgr$ and $n\geqslant 0$,
to the extended double holonomy algebras (cf.~\ref{def:ext-holo-alg}, \ref{ss:hol-cox} and \ref{ss:hol-braid-cox})
\begin{equation}
	\rda{(\DBLHA{\Rs{}}{n,\scsop{ext}})}{}{B}=
	\Br{W_B}\ltimes\left(\DKHA{B}{n}\ten({S}\h'_B)^{\ten n}\right)
\end{equation}
and we obtain a morphism of cosimplicial lax diagrammatic algebras $\xi^{\scs{\bullet}}_{W}:\DBLHAH{\Rs{}}{\bullet,\scsop{ext}}\to\UCoxOint{\g}{\bullet}$.

\subsection{The braided Coxeter category $\OCox{\g,\nabla}{\scsop{\hdef, \sint}}$}\label{ss:from-hol-to-KM}

\begin{theorem}\label{thm:hol-to-KM}
	\hfill
	\begin{enumerate}\itemsep0.25cm
		\item Let $\sCox{\nabla}=(\Phig{\nabla}{B}{}, \Rg{\nabla}{B}{}, \Jg{\nabla}{\F}{}, \DCPAC{\nabla}{\F}{\G}, \CoxS{\nabla}{}{i})$
		be the $\redasso{}{}$-strict braided Coxeter structure on $\DBLHAH{\Rs{}}{\bullet,\scsop{ext}}$ given by Theorem~\ref{thm:holo-cox}.
		Then, then datum of
		\[
		\sCox{\nabla, \g}=(\Phig{\nabla}{B}{,\g}, \Rg{\nabla}{B}{,\g}, 
		\Jg{\nabla}{\F}{,\g}, \DCg{\nabla}{\F}{\G}{,\g}, \Sg{\nabla}{}{i}{,\g})
		\]
		where
		\begin{align*}
			\Phig{\nabla}{B}{,\g}=\xi^3_{\Rs{}}(\Phig{\nabla}{B}{}), \quad
			\Rg{\nabla}{B}{,\g}=\xi^2_{\Rs{}}(\Rg{\nabla}{B}{}), \quad
			\Jg{\nabla}{\F}{,\g}=\xi^2_{\Rs{}}(\Jg{\nabla}{\F}{}), \quad
			\DCg{\nabla}{\F}{\G}{,\g}=\xi^1_{\Rs{}}(\DCg{\nabla}{\F}{\G}{}),
		\end{align*}
		and $\Sg{\nabla}{}{i}{,\g}=\xi^1_W(\CoxS{\nabla}{}{i})$,
		is an $\redasso{}{}$--strict braided Coxeter structure on $\UCoxOint{\g}{\bullet}$.
		\item There is a braided Coxeter category $\OCox{\g,\nabla}{\scsop{\hdef, \sint}}$ with the following data.
		\vspace{0.25cm}
		\begin{itemize}\itemsep0.25cm
			\item For any $B\subseteq\dgr$, the category $\O^{\hdef, \sint}_{\infty, \g_B}$, $B\subseteq\dgr$, 
			with braided monoidal structure given by 
			$\Phig{\nabla}{B}{,\g}$ and $\Rg{\nabla}{B}{,\g}$
			\item For any $B'\subseteq B$ and $\F\in\Mns{B,B'}$, the standard restriction functor 
			$\Res_{\F}:\O^{\hdef, \sint}_{\infty, \g_B}\to\O^{\hdef, \sint}_{\infty, \g_{B'}}$
			with tensor structure given by $\Jg{\nabla}{\F}{,\g}$.
			\item For any $B'\subseteq B$ and $\F,\G\in\Mns{B,B'}$, the natural isomorphism of tensor
			functors $\Res_{\G}\Rightarrow\Res_{\F}$ given by the De Concini--Procesi associator
			$\DCg{\nabla}{\F}{\G}{,\g}$. 
			\item For any $i\in\dgr$, the monodromy operator $\Sg{\nabla}{}{i}{,\g}$.
		\end{itemize}
	\end{enumerate}	
\end{theorem}

\begin{pf}
	We shall verify that $\sCox{\nabla,\g}$ satisfy the properties (a)--(e) of Definition~\ref{ss:braid-cox-alg}
	\wrt the cosimplicial lax bidiagrammatic structure on $\UCoxOint{\g}{\bullet}$. By construction,
	$\sCox{\nabla,\g}$ is the image of a braided pre--Coxeter structure
	$\sCox{\nabla}^{\scsop{pre}}$ in $\DBLHAH{\Rs{}}{\bullet}$ through 
	the morphism $\xi^{\scs{\bullet}}_{\Rs{}}:\DBLHAH{\Rs{}}{\bullet}\to\UCoxOint{\g}{\bullet}$.
	Although $\xi^{\scs{\bullet}}_{\Rs{}}$ is a morphism of cosimplicial lax diagrammatic 
	algebras, it does not preserve the invariant subalgebras, as the condition of being invariant 
	in $\UOint{\g}{n}$ is generally stronger than being
	invariant in $\DBLHA{\Rs{}}{n}$.
	For instance, while the element $\Kdh{\alpha_i,1}{}$ is obviously central in $\DBLHA{\Rs{},i}{1}$,
	the normally ordered Casimir operator $\Ku{\alpha_i}{+}=\xi^1_{\Rs{}}(\Kdh{\alpha_i,1}{})$ 
	is not $\sl{2}^{\alpha_i}$--invariant. Therefore, proving that
	$\sCox{\nabla}^{\scsop{pre}}$ is a braided pre--Coxeter structure in 
	$\UCoxOint{\g}{\bullet}$ reduces to showing that the defining elements of $\sCox{\nabla,\g}$
	satisfy the necessary invariance properties in $\UCoxOint{\g}{\bullet}$.
	
	Note that, for any $B\subseteq\dgr$, 
	\[
	\Phig{\nabla}{B}{,\g}=\Phi^{\nabla}(\hbar\cdot\Omega_{B,12},\hbar\cdot\Omega_{B,23})
	\aand
	\Rg{\nabla}{B}{,\g}=\exp(\hbar/2\cdot\Omega_B)
	\]
	are clearly $\g_B$--invariant since $\Omega_B\in\UOint{\g,BB}{2}$.
	It remains to prove that the relative twist $\Jg{\nabla}{\F}{,\g}$ and 
	the De Concini--Procesi associator $\DCg{\nabla}{\F}{\G}{,\g}$, 
	corresponding to the maximal nested sets $\F,\G\in\Mns{B,B'}$, 
	are $\g_{B'}$--invariant. To this end, it is enough to observe that the 
	coefficients of the equations defining $\Jg{\nabla}{\F}{,\g}$ and $\DCg{\nabla}{\F}{\G}{,\g}$ 
	specialise to $\g_{B'}$--invariant elements in $\UOint{\g,B}{1}$ and $\UOint{\g,B}{2}$, 
	which follows as in \cite[Thm.~1.33]{vtl-4} and \cite[App.~B.4]{vtl-6}.
	Finally, by Proposition~\ref{ss:from-end-to-cat}, (2) follows from (1). 	
\end{pf}

\noindent\remark\; Note that the operators $\Phig{\nabla}{B}{,\g}$, $\Rg{\nabla}{B}{,\g}$, 
$\Jg{\nabla}{\F}{,\g}$, and $\DCg{\nabla}{\F}{\G}{,\g}$ are well--defined on category 
$\O_{\infty}$ $\g$--modules and therefore give rise to a braided \emph{pre}--Coxeter
category $\OCox{\g,\nabla}{\scsop{\hdef}}$.

\section{Quantum Kac--Moody algebras and the category $\OCox{\Uhg,\Rmx, \qWS{}}{\scsop{\sint}}$}\label{s:km-qg} 

In this section, we describe the standard braided Coxeter category $\OCox{\Uhg,\Rmx, \qWS{}}{\scsop{\sint}}$
associated to a quantised Kac--Moody algebra $U_\hbar\g$,  
which encodes the action of the universal $R$--matrix and Lusztig's quantum Weyl group operators \cite{lusztig-94}
on integrable highest weight $\Uhg$--modules. We then recall the main result of \cite[Thm.~13.9]{ATL1-2}, 
which provides a description of $\OCox{\Uhg,\Rmx, \qWS{}}{\scsop{\sint}}$ in terms of integrable highest 
weight $\g$--modules.


\subsection{The Drinfeld--Jimbo quantum group \cite{drinfeld-86, j}}\label{ss:DJ-qg}

\newcommand{\qKi}[2]{q_{#1}^{{#2}\cor{#1}}} 

Let $\g$ be a symmetrisable Kac--Moody algebra. Set $q=\exp(\hbar/2)$
and $q_i= q^{\symd{i}}$, $i\in\bfI$. 
The Drinfeld--Jimbo quantum group of $\g$ is 
the algebra $\DJ{\g}$ over $\hext{\IC}$ topologically generated by 
$\h$ and the elements $\{E_i, F_i\}_{ i\in\bfI}$, subject to the relations $[h,h']=0$, 
\begin{align*}
	[h, E_i]=\alpha_i(h)E_i
	\qquad
	[h, F_i]=-\alpha_i(h)F_i
	\qquad
	[E_i, F_j]=\drc{ij}\frac{\qKi{i}{}-\qKi{i}{-}}{q_i-q_i^{-1}}
\end{align*}
for any $h,h'\in\h$, $i,j\in\bfI$, and the $q$--Serre relations
\begin{align}
	\sum_{m=0}^{1-a_{ij}}(-1)^m{1-a_{ij}\brack m}_i X_i^{1-a_{ij}-m}X_jX_i^{m}&=0
\end{align}
for $X=E,F$, $i\neq j\in\bfI$, where $\displaystyle [n]_i=\frac{q_i^n-q_i^{-n}}{q_i-q_i^{-1}}$ 
and, for any $k\leqslant n$,
\begin{align}
	[n]_i!= [n]_i\cdot [n-1]_i\cdots[1]_i
	\aand
	{n\brack k}_i=\frac{[n]_i!}{[k]_i!\cdot [n-k]_i}	
\end{align}

We consider on $\DJ{\g}$ the Hopf algebra structure with coproduct
\[\Delta(h)=
h\ten1+1\ten h\qquad
\Delta(E_i)=E_i\ten \qKi{i}{}+1\ten E_i\qquad
\Delta(F_i)=F_i\ten 1+\qKi{i}{-}\ten F_i\]
counit $\varepsilon(h)=\varepsilon(E_i)=\varepsilon(F_i)=0$, and antipode $S(h)=-h$, $S(E_i)=-
E_i\qKi{i}{-}$, and $S(F_i)=-\qKi{i}{}F_i$ for any $h\in\h$ and $i\in\bfI$.

\newcommand{\OinfintUhg}{\O_{\infty,\Uhg}^{\sint}}
\newcommand{\OinfUhg}{\O_{\infty,\Uhg}}
\newcommand{\OUhg}{\O_{\Uhg}}

Define weight, integrable, category $\Oinf$ and $\O$ modules for $\Uhg$ in $\tfV$
analogously to Section \ref{ss:cat-O-diag-h}, and denote by
\[\OinfintUhg\subset\OinfUhg
\aand
\O_{\Uhg}^{\sint}\subset\OUhg\]
the subcategories of integrable modules.\footnote{In particular a representation
	$\V$ of $\Uhg$ is in category $\Oinf$ if the action of $\Uhbp{}$ on $\V/\hbar^n\V$ is locally finite for any $n\geqslant 0$. Note that the analogue of the condition \ref{cond:Oinftyh-bis} holds.} 


\subsection{The universal $R$--matrix}\label{ss:R-matrix}

The Hopf algebra $\DJ{\g}$ is quasitriangular (cf. \cite{drinfeld-86, lusztig-94}).
Namely, let $\Uhnp$ (resp. $\Uhnm$) be the subalgebra generated by $E_i, i\in\bfI$ 
(resp. $F_i, i\in\bfI$), and set $\Uhbpm{}=\Uhnpm{}\hext{U\h}$. By \cite{drinfeld-86}, there is a 
unique non--degenerate Hopf pairing 
$\iip{\cdot}{\cdot}_{\D}:\Uhbm{}\ten\Uhbp{}\to\IC(\negthinspace(\hbar)\negthinspace)$ 
such that $\iip{1}{1}_{\D}= 1$,
\[
\begin{array}{cccc}
	\iip{h}{h'}_{\D} =\displaystyle{\frac{1}{\hbar}\iip{h}{h'}} &&&
	\iip{F_i}{E_j}_{\D}=\displaystyle{\frac{\delta_{ij}}{q-q^{-1}}}
\end{array}
\]
and zero otherwise. Let $\{x_a\}, \{x^a\}\subset\h$ be dual bases. 
Note that the pairing respects the weight decomposition in $\Uhbpm{}$. 
For any $\mu\in\sfQ_+$, let $\{X^\pm_{\mu,p}\}_p\subset\Uhnpm[\pm\mu]$ be dual bases 
with respect to $\iip{\cdot}{\cdot}_{\D}$ and set $\Theta_{\mu}=\sum_{p}X^-_{\mu,p}\ten X^+_{\mu,p}$. Then, $\Uhg$ is a quasitriangular Hopf algebra with $R$--matrix
\begin{align}
	\Rmx= q^{\Omega^0}\cdot\Theta\in\Uhbm{}\wh{\ten}\Uhbp{}
\end{align}
where $\Omega^0=\sum_a x_a\ten x^a$ and $\Theta=\sum_{\mu>0}\Theta_{\mu}$,
that is, $\Rmx$ satisfies the intertwining property $\Delta^{\op}(x)=\Rmx\cdot\Delta(x)\cdot
\Rmx^{-1}$, as well as the cabling identities
\begin{equation}
	\Delta\ten1(\Rmx)=\Rmx_{13}\cdot \Rmx_{23}
	\aand
	1\ten\Delta(\Rmx)=\Rmx_{13}\cdot \Rmx_{12}
\end{equation}
from which the Yang--Baxter equation $\Rmx_{12}\cdot \Rmx_{13}\cdot \Rmx_{23}
= \Rmx_{23}\cdot \Rmx_{13}\cdot \Rmx_{12}$ follows. The action of the $R$--matrix
on a tensor product of representations in $\O_{\infty,\Uhg}$ is well--defined and induces
a braiding.

\subsection{Quantum Weyl group operators}\label{ss:qWg}
Let $\hOinfint$ be the category of integrable $\Uhg$--modules in category $\hOinf$,
\ie the action of the elements $E_i, F_i$, $i\in\bfI$, is locally nilpotent mod $\hbar^n$ for
any $n\geqslant 0$. Let $V\in\hOinfint$. For any $i \in\bfI$, the operator $\qWSk{i}$ is defined on $V$ 
as follows\footnote{The operators $\qWSk{i}$ are well--defined on any integrable $\Uhg$--module. Note that
	in the notation of \cite[Sec.~5.2]{lusztig-94} $\qWSk{i}$ coincides with the operator $T''_{i,+1}$.}
\cite{lusztig-94, kirillov-reshetikhin-90, levendorskii-soibelman-91}: for any $v_{\mu}\in V[\mu]$, 
\begin{equation}\label{eq:qwg-op}
	\qWSk{i}(m)=\sum_{\substack{a,b,c \in \IZ_{\geqslant 0}\\ a-b+c = -\mu(\cor{i})} }(-1)^b q_i^{b-ac} E_{i}^{(a)} F_{i}^{(b)} E_{i}^{(c)}\cdot v_{\mu}
\end{equation} 
where $X_i^{(a)}= X^a/[a]_i!$, $X=E,F$.
Clearly we have $\qWSk{i}(V[\mu]) \subseteq V[s_i(\mu)]$. By \cite[Sec.~39.4]{lusztig-94},
the operators $\qWSk{i}$, $i\in I$, induce an action of the generalised braid group $\Br{W}$ on $V\in\hOinfint$, which recovers the action by triple exponentials described in \ref{ss:km-sym-cox} at $\hbar=0$.
By construction, $\qWSk{i}$ is an element of the completion of $\Uhg$ with respect to 
the category $\hOinfint$, \ie $\qWSk{i}\in\Aut(\hOinfint\to\tfV)$. By \cite[Sec.~37.1]{lusztig-94}, 
the action of the  operators $\qWSk{i}$ induces an algebra automorphism of $\Uhg$, which 
we denote by the same symbol, such that, for any $u\in\Uhg$ and $v\in V\in\hOinfint$, one 
has $\qWSk{i}(u\cdot m)=\qWSk{i}(u)\cdot \qWSk{i}(m)$. Moreover, for any $h\in\h$, 
$\qWSk{i}(h)=s_i(h)$.\\

The operator $\qWSk{i}$ allows to recover the universal $R$--matrix
as a multiplicative coboundary. Indeed, by \cite[Sec.~5.3]{lusztig-94} (see also \cite[Sec.~4.10]{appel-vlaar-21} 
which follows our conventions), the operator $\qWSk{i}$ satisfies the coproduct identity 
\begin{equation}\label{eq:coprod-id-qg-theta}
	\Delta(\qWSk{i})=(\qWSk{i}\ten\qWSk{i})\cdot\Theta_i
\end{equation}
where $\Rmx_{i}\in\Uhbm{i}\wh{\ten}\Uhbp{i}$ is the universal $R$--matrix of the subalgebra 
$U_{\hbar}\g_i$ corresponding to the simple root $\root{i}$ and $\Rmx_{i}=q_i^{\cor{i}\ten
	\cor{i}/2}\cdot\Theta_i$. Note that
\[q_i^{\cor{i}\ten\cor{i}/2}=q_i^{-\cor{i}^2/4}\ten q_i^{-\cor{i}^2/4}\cdot\Delta\left(q_i^{\cor{i}^2/4}\right)\]
Set $\qWS{i}= q_i^{{\cor{i}^2}/{4}}\cdot\qWSk{i}=\qWSk{i}\cdot q_i^{{\cor{i}^2}/{4}}$. Note that
the operators $\qWS{i}$, $i\in\bfI$, also satisfy the generalised braid relations \eqref{eq:gen-braid}. 
Moreover, we have
\begin{align}
	(\qWS{i}\ten\qWS{i})^{-1}\cdot\Delta(\qWS{i})
	=&(\qWSk{i}\ten\qWSk{i})^{-1}\cdot q_i^{-\cor{i}^2/4}\ten q_i^{-\cor{i}^2/4}\cdot \Delta\left(q_i^{\cor{i}^2/4}\right)\cdot\Delta(\qWSk{i})\\
	=&(\qWSk{i}\ten\qWSk{i})^{-1}\cdot q_i^{\cor{i}\ten\cor{i}/2}\cdot \Delta(\qWSk{i})\\
	=&q_i^{\cor{i}\ten\cor{i}/2}\cdot (\qWSk{i}\ten\qWSk{i})^{-1}\cdot \Delta(\qWSk{i})\\
	=&q_i^{\cor{i}\ten\cor{i}/2}\cdot\Theta_i=\Rmx_i
\end{align}
Therefore, from the identity $\Delta(\qWS{i})=\Rmx_{i,21}\cdot\Delta^{\op}(\qWS{i})\cdot \Rmx_{i,21}^{-1}$,
we get the coproduct identity
\begin{equation}\label{eq:coprod-id-qg}
	\Delta(\qWS{i})=(\qWS{i}\ten\qWS{i})\cdot\Rmx_{i}=\Rmx_{i}^{21}\cdot(\qWS{i}\ten\qWS{i})
\end{equation}
We shall refer to both $\qWSk{i}$ and $\qWS{i}$ as the quantum Weyl group operators of $\Uhg$.\\

\noindent\remark\; By \cite[Sec.~5.2]{lusztig-94}, the squares of the operators $\qWSk{i}$ and $\qWS{i}$ are particularly 
simple and related to the quantum Casimir element of the quantum algebra 
$U_{\hbar}\sl{2}^{\alpha_i}=\langle E_i, F_i, \cor{i}\rangle\subset U_{\hbar}\g$. 
Recall that, since 
$\O_{\infty,U_{\hbar}\sl{2}}^{\sint}$ is semisimple, an element in $\Aut(\O_{\infty,U_{\hbar}\sl{2}}^{\sint}\to\tfV)$
is uniquely determined by its action on the indecomposable finite--rank representations.
Let $\Cu{\hbar,i}$ be the quantum Casimir operator, acting on the irreducible representation of rank $d+1$ 
as $\symd{i}\cdot d\cdot(d+2)/2$, and set $\ku{\hbar, i}= \Cu{\hbar,i}-\symd{i}\cdot\cor{i}^2/2$.
Then, we have
\[
\qWSk{i}^2=\exp(\pi\iota\cor{i})\cdot q^{\ku{\hbar,i}}
\aand
\qWS{i}^2=\exp(\pi\iota\cor{i})\cdot q^{\Cu{\hbar, i}}
\]
Note that $\exp(\pi\iota\cor{i})$ is central, \ie it commutes with the action of 
$U_{\hbar}\g_i$ and therefore so is $\qWS{i}^2$.


\subsection{The braided Coxeter category $\OCox{\Uhg,\Rmx, \qWS{}}{\scsop{\sint}}$}\label{ss:qg-cox}

Integrable highest weight representations of quantum Kac--Moody algebras give rise to a
braided Coxeter category. Namely, let $\g$ be a diagrammatic Kac--Moody algebra with 
labelled Dynkin diagram $(\dgr,\ulm)$ and Cartan subalgebras $\h_B\subseteq\h$, 
$B\subseteq\dgr$. The quantum group $\Uhg$ 
is a 
bidiagrammatic Hopf algebra, with subalgebras 
\[\Uhg_B=\left\langle \h_B, E_i, F_i\;|\; i\in B\right\rangle\]
$B\subseteq\dgr$, and the corresponding diagrammatic invariants. 
These induce restriction functors $\hRes_{B'B}:\hOint_{\infty, \Uhg_B}\to
\hOint_{\infty, \Uhg_{B'}}$ with $B'\subseteq B$, and yield the following
analogue of Proposition~\ref{ss:km-sym-cox}.

\begin{proposition}\label{pr:ORS}
	There is a canonical $(\redasso{}{},\DCPA{}{})$--strict braided Coxeter category 
	$\OCox{\Uhg,\Rmx, \qWS{}}{\scsop{\sint}}$ of type $(\dgr, \ulm)$ given by the following data.
	\vspace{0.25cm}
	\begin{itemize}\itemsep0.25cm
		\item For any $B\subseteq\dgr$, the braided monoidal category $\hOint_{\infty, \Uhg_B}$ with 
		braiding induced by the action of the universal $R$--matrix $\Rmx_B\in\Uhbm{B}\wh{\ten}\Uhbp{B}$.
		\item For any $B'\subseteq B$, the restriction functor $\hRes_{B'B}:\hOint_{\infty, \Uhg_B}\to\hOint_{\infty, \Uhg_{B'}}$.
		\item For any $i\in\dgr$, the quantum Weyl group operator $\qWS{i}\in\Aut(\hOint_{\infty, U_{\hbar}\g_i}\to\tfV)$.
	\end{itemize}
\end{proposition}
\begin{pf}
	It is enough to observe that the quantum Weyl group operators $\qWS{i}$ satisfy
	the coproduct identity \eqref{eq:coxcoprod-cat}, which for the braided monoidal 
	category $\Oint_{\infty, U_{\hbar}\g_i}$ reduces precisely to the equation \eqref{eq:coprod-id-qg}.
\end{pf}

\noindent\remark\; Note that the action of the $R$--matrix on category $\O_{\infty}$
modules, together with the corresponding restriction functors, gives rise to a braided \emph{pre}--Coxeter 
category $\OCox{\Uhg,\Rmx}{}$.

\part{The monodromy theorem}\label{part-four}

This final part is devoted to the proof of the main result of the paper. The material is organized as follows.
In Section \ref{s:diag-LBA}, we introduce the 
notions of a split diagrammatic Lie bialgebra $\b$, its Drinfeld--Yetter modules $\hDrY{\b}{}$, 
and the symmetric pre--Coxeter category $\DYCox{\b}{}$. In Section \ref{s:univ-alg}, 
we introduce the $\PROP$ of universal Drinfeld--Yetter modules over a split diagrammatic Lie
bialgebra and the {\em universal} algebra $\DUA{\dgr}{\bullet}$ which controls the deformation 
of $\DYCox{\b}{}$ as a braided pre--Coxeter category. In Section \ref{s:diag-HA}, we describe
similar results for split diagrammatic quantum enveloping algebras and their {\em admissible}
Drinfeld--Yetter modules. We review the construction of an explicit equivalence 
of braided pre--Coxeter categories $\DYCox{\b}{\Phi}\to\DYCox{\Q(\b)}{\adm}$, given in 
\cite[Thm.~10.10]{ATL1-2}, where $\b$ is a split diagrammatic Lie bialgebra, $\DYCox{\b}{\hdef,\Phi}$
denotes a deformation of $\DYCox{\b}{}$ depending upon the choice of a Lie associator 
$\Phi$, $\Q(\b)$ is the Etingof--Kazhdan quantisation of $\b$ \cite{ek-1}, and
$\DYCox{\Q(\b)}{\adm}$ denotes the braided pre--Coxeter category arising from 
admissible Drinfeld--Yetter $\Q(\b)$--modules. In Section \ref{s:univ-Cox}, we fix a 
diagrammatic Kac--Moody algebra $\g$ with root system $\Rs{}$. We introduce the $\PROP$ 
of universal Drinfeld--Yetter modules over a split diagrammatic Lie bialgebra graded 
over $\Rs{+}$, modelled over $\bm{}$. Its {universal} algebra $\DUA{\Rs{}}{\bullet}$ controls 
the deformation of $\DYCox{\bm{}}{}$. 
In Section \ref{s:main-thm}, we prove that the monodromy data of 
the KZ and Casimir connections are encoded by a universal structure on $\DUA{\Rs{}}{\bullet}$. 
We rely on the rigidity of $\DUA{\Rs{}}{\bullet}$, proved in \cite[Thm.~15.15]{ATL2}, to obtain 
an equivalence of braided pre--Coxeter categories $\DYCox{\bm{}}{\hdef,\Phi}\to\DYCox{\Uhbm{}}{\adm}$, 
which finally yields the equivalence $\OCox{\g,\nabla}{\scsop{\hdef, \sint}}\to\OCox{\Uhg,\Rmx, \qWS{}}{\scsop{\sint}}$.

\section{From category $\O$ to Drinfeld--Yetter modules}\label{s:diag-LBA}

We review the notion of diagrammatic Lie bialgebra introduced
in \cite{ATL1-2}, and the fact that their Drinfeld--Yetter modules
give rise to a braided pre--Coxeter category. In the case of a
diagrammatic Kac--Moody algebra, this recovers category
$\OCox{\g}{\scsop{\sint}}$ as a full subcategory of Drinfeld--Yetter modules
over its negative Borel subalgebra.


\subsection{Lie bialgebras}\label{ss:lba}

A Lie bialgebra is a triple $(\b,[\,,\,]_{\b}, \delta_{\b})$ where $(\b,[\,,\,]_\b)$
is a Lie algebra, $(\b,\delta_\b)$ a Lie coalgebra, and the cobracket $\delta_
\b:\b\to\b\otimes\b$ satisfies the cocycle condition
\[\delta_\b([X,Y]_\b)=\ad(X)\,\delta_\b(Y)-\ad(Y)\,\delta_b(X)\]

\subsection{Drinfeld double \cite{drinfeld-86}}\label{ss:drinf-double}

The Drinfeld double of a Lie bialgebra $(\b,[\,,\,]_{\b},\delta_{\b})$ is the
Lie algebra $\gb$ defined as follows. As a vector space, $\gb=\b\oplus
\b^*$. The duality pairing $\b^*\ten\b\to\sfk$ extends uniquely to a symmetric,
non--degenerate bilinear form $\iip{\cdot}{\cdot}$ on $\gb$, \wrt which
both $\b$ and $\b^*$ are isotropic subspaces. The Lie bracket on $\gb$ is
defined as the unique bracket which coincides with $[\,,\,]_{\b}$ on $\b$,
with $\delta_{\b}^t$ on $\b^*$, and is compatible with $\iip{\cdot}{\cdot}$,
\ie satisfies $\iip{[x,y]}{z}=\iip{x}{[y,z]}$ for all $x,y,z\in\gb$. The mixed
bracket of $x\in\b$ and $\phi\in\b^*$ is then given by
\[[x,\phi]
=\sfad^*(x)(\phi)+\phi\ten\id_\b\circ\delta(x)
\]
where $\sfad^*$ is the coadjoint action of $\b$ on $\b^*$.
Note that $(\gb,\b,\b^*)$ is a Manin triple \cite{drinfeld-86,ek-1}, 
and any such triple arises this way.

Similarly, if $\b$ is a Lie bialgebra which is $\IN$--graded with \fd components,
and such that the bracket and cobracket are homogeneous of degrees 0 and
$d\in\IZ$ respectively,\footnote{In the sequel, we shall abusively refer to such
	a $\b$ as an $\IN$--graded Lie bialgebra.} the restricted double of $\b$ is
defined as $\grb=\b\oplus\b^{\star}[d]$, where $\b^{\star}[d]_n=(\b_{-n+d})^*$,
and is a restricted Manin triple.

The restricted double $\grb$ (and, in particular, the double of a \fd Lie bialgebra)
is additionally endowed with a Lie bialgebra structure, with cobracket
\[\delta_{\grb}(X)=[X\otimes 1+1\otimes X,r]\]
where $r$ is the canonical element in $\b\wh{\otimes}\b^\star$, with $\wh{\otimes}$
the completion of the tensor product \wrt the grading, and $\delta_{\grb}=\delta_
\b-\delta_{\b^\star}$.

\subsection{Drinfeld--Yetter modules \cite{drinfeld-86, ek-2}}\label{ss:DY}
A \DYt module over a Lie bialgebra $\b$ is a triple $(V,\pi_V,\pi_V^*)$, where
$(V,\pi_V)$ is a left $\b$--module, $(V,\pi_V^*)$ a right $\b$--comodule, and
the maps $\pi_V:\b\otimes V\to V$ and $\pi_V^*:V\to\b\otimes V$ satisfy the
following compatibility in $\End(\b\ten V)$
\begin{equation}\label{eq:actcoact}
	\pi_V^*\circ\pi_V-\id_{\b}\ten\pi_V\circ(12)\circ\id_{\b}\ten\pi_V^*=
	[\cdot,\cdot]_{\b}\ten\id_V\circ\id_{\b}\ten\pi_V^*-\id_{\b}\ten\pi_V\circ\delta_{\b}\ten\id_V
\end{equation}

The category $\DrY{\b}$ of such modules is a symmetric tensor category. 
For any $V,W\in\DrY{\b}$, the action and coaction on the tensor product 
$V\ten W$ are defined, respectively, by
\begin{align*}
	\pi_{V\ten W}	&=\pi_{V}\ten\id_W+\id_V\ten\pi_{W}\circ(1\,2)\otimes\id_W\\
	\pi^*_{V\ten W}	&=\pi^*_{V}\ten\id_W+(1\,2)\otimes\id_W\circ\id_V\ten\pi^*_{W}
\end{align*}
The associativity constraints are trivial, and the braiding is defined by
$\beta_{VW}=(1\,2)$. 

\subsection{Representations of the Drinfeld double}\label{ss:drinf-db-rep}

The category $\DrY{\b}$ is canonically isomorphic to the category $\Eq
{\gb}$ of {\it equicontinuous} $\gb$--modules \cite{ek-1}, \ie those endowed
with a locally finite $\b^*$--action. This condition yields a functor $E:\Eq{\gb}
\to\DrY{\b}$, which assigns to any $V\in\Eq{\gb}$ the \DYt $\b$--module
$(V,\pi,\pi^*)$, where $\pi$ is the restriction of the action of $\gb$ to $\b$,
and the coaction $\pi^*$ is given by
\begin{equation}\label{eq:DY-to-equi}
\pi^*(v)=\sum_{i}b_i\ten b^i\,v\in\b\ten V
\end{equation}
where $\{b_i\},\{b^i\}$ are dual bases of $\b$ and $\b^*$. The inverse
functor is obtained by letting $\phi\in\b^*\subset\gb$ act on $V\in\DrY
{\b}$ by $\phi\otimes\id_V\circ\pi^*$.

If $\b$ is $\IN$--graded with \fd homogeneous components, the formulae
defining $E$ similarly give rise to an isomorphism $E\res$ between the
category $\Eq{\grb}$ of equicontinuous modules over the restricted double
of $\b$ and $\DrY{\b}$. Moreover, the categories $\Eq{\gb}$ and $\Eq{\grb}$
are isomorphic, since any locally finite action of $\b^\star$ extends
uniquely to one of $\b^*$, and the following diagram is commutative
\[\xymatrix@C=1.5cm{
	\Eq{\gb}\ar[rr]\ar[rd]_{E}&&\Eq{\grb}\ar[ld]^{E\res}\\
	&\DrY{\b}&
}\]

\subsection{Symmetrisable Kac--Moody algebras}\label{ss:km-LBA}
\newcommand{\hz}{\h_{\z}}

Let $\g$ be a symmetrisable Kac--Moody algebra with opposite Borel 
subalgebras $\bpm{}\subset\g$ (cf.~\ref{ss:km-recap}). The identifications
$(\bmp{})^{\star}\simeq\bpm{}$ give rise to a Lie bialgebra structure on
$\bpm{}$ and $\g$, which is compatible with the grading. 
Specifically, consider the Lie algebra $\gtwo=\g\oplus \hz$, with $\hz=\h$,
and endow it with the inner product 
\[\iip{\cdot}{\cdot}\two=
\iip{\cdot}{\cdot}\oplus-\left\iip
{\cdot}{\cdot}\right|_{\hz\times \hz}\]
Let $\pi_0:\g\to\h$ be the projection arising from the root space decomposition,
and $\btwo_\pm\subset\gtwo$ the subalgebra
\[\btwo_\pm=\{(X,h)\in\bpm{}\oplus\hz|\,\pi_0(X)=\pm h\}\]
Note that the projection $\gtwo\to\g$ onto the first component restricts
to an isomorphism $\btwo_\pm\to\bpm{}$ with inverse $\bpm{}\ni X\to
(X,\pm\pi_0(X))\in\btwo_\pm$. The following is easily seen to hold (cf.~\cite{drinfeld-86, ek-6}).

\begin{proposition}\label{pr:double g}\hfill
	\begin{enumerate}\itemsep0.25cm
		\item $(\gtwo, \btwo_-, \btwo_+)$ is a restricted Manin triple.
		In particular, $\btwo_\mp$ and $\gtwo$ are Lie bialgebras,
		with cobracket $\delta_{\btwo_\mp}=[\cdot,\cdot]_{\btwo_\pm}^t$
		and $\delta_{\gtwo}=\delta_{\btwo_-}-\delta_{\btwo_+}$.
		\item The central subalgebra $0\oplus \hz\subset\gtwo$ is a coideal,
		so that the projection $\gtwo\to\g$ induces a Lie bialgebra structure
		on $\g$ and $\b_\mp$. 
		\item The Lie bialgebra structure on $\g$ is given by 
		\begin{equation}\label{eq:std-cobracket-km}
			\delta|_{\h}=0
			\qquad
			\delta(e_i)=\symd{i} \cor{i}\wedge e_i
			\qquad
			\delta(f_i)=\symd{i} \cor{i}\wedge f_i
		\end{equation}
		\item The projection $\gtwo\to\g$ maps the canonical elements $r_{\gtwo}\in\btwo_-
		\wh{\otimes}\btwo_+$ and $\Omega_{\gtwo}=r_{\gtwo}+r_{\gtwo}^{21}\in
		\gtwo\wh{\otimes}\gtwo$ corresponding to the inner product $\iip{\cdot}{\cdot}\two$  
		to
		\[r_\g=\sum_i x_i\otimes x^i+\half{1}\sum_a t_a\otimes t^a\]
		and the canonical element $\Omega_\g\in\g\wh{\otimes}\g$ corresponding to 
		$\iip{\cdot}{\cdot}$,  where $\{x_i\},\{x^i\}$ are homogeneous dual bases of $\n_-,
		\n_+$, and $\{t_a\},\{t^a\}$ are dual bases of $\h$.\footnote{The $\half{1}$ factor
		in $r_\g$ arises because if $t,t'\in\h$ and $t_+=(t,t)$ and $t'_-=(t',-t')$ are the
		corresponding elements of $\btwo_\pm$, then $\iip{t_+}{t'_-}\two=2\iip{t}{t'}$.}
\end{enumerate}
\end{proposition}

\subsection{From category $\O$ to Drinfeld--Yetter modules}\label{ss:from-O-to-DY}
By Proposition~\ref{ss:km-LBA} and \ref{ss:drinf-db-rep}, the category
of \DYt modules over $\bm{}$ is equivalent to the category $\E_{\gtwo}$ of $\gtwo
$--modules which carry a locally finite action of $\btwo_+\subset\gtwo$. This implies the following, cf.~\cite[Prop.~12.8]{ATL1-2}.

\begin{proposition}
	\hfill
	\begin{enumerate}\itemsep0.25cm
		\item Category $\O_{\infty,\g}$ is isomorphic to the full tensor subcategory
		of $\E_{\gtwo}$ consisting of those modules carrying a trivial action of
		$\h_{\z}$. 
		\item Under the equivalence $\E_{\gtwo}\simeq\DrY{\bm{}}$, $\O_
		{\infty, \g}$ is isomorphic to the full tensor subcategory of $\DrY{\bm{}}$
		consisting of those modules $V$ such that the action $\rho_V$ and the
		coaction $\rho^*_V$ of $\h$ on $V$ coincide under $\iip{\cdot}{\cdot}_{\h}$, \ie
		\begin{equation}\label{eq:DY-ss-0}
			\rho_V=
			\iip{\cdot}{\cdot}_{\h}\ten\id_V\circ\id_{\h}\ten\rho_V^*
		\end{equation}
	\end{enumerate}
\end{proposition}


\subsection{Split pairs and restriction functors }\label{ss:sLBA}

\DYt modules cannot be pulled back under morphisms of Lie bialgebras since modules
are contravariant and modules are covariant \wrt such morphisms. This can be rectified,
however, by considering a different notion of morphism of Lie bialgebras.

A {\it split pair} of Lie bialgebras $(\b,\a)$ \cite{ATL1-1} is the datum of two Lie
bialgebras $\a,\b$, together with Lie bialgebra morphisms $i:\a\to
\b$ and $p:\b\to\a$ such that $p\circ i=\id_{\a}$. For any split pair of 
Lie bialgebras $(\b,\a)$, there is a monoidal restriction functor $\Res_{\a,\b}:\DrY{\b}\to\DrY{\a}$ 
defined by
\[\Res_{\a,\b}(V,\pi_V,\pi^*_V)=
(V,\pi_V\circ i\ten\id_V,p\ten\id_V\circ\pi^*_V)\]
Moreover, if $\a\hookrightarrow\b\hookrightarrow\c$ is a chain of split
embeddings, then $\Res_{\a,\b}\circ\Res_{\b,\c}=\Res_{\a,\c}$.
Note that, under the identification of $\DrY{\b},\DrY{\a}$ with the categories
of equicontinuous modules over the doubles $\gb$ and $\ga$
respectively, $\Res_{\a,\b}$ coincides with the pullback functor
corresponding to the morphism $i\oplus p^t:\ga\to\gb$.


\subsection{Diagrammatic Lie bialgebras \cite{ATL1-2}}\label{ss:diag-LBA}

By analogy with Section \ref{ss:functorial-diag-alg}, a diagrammatic Lie (bi)algebra is a monoidal
functor from $\PD$ to the category of Lie (bi)algebras. Specifically, a {\it diagrammatic Lie (bi)algebra}
$\b$ is the datum of
\begin{itemize}
	\item a diagram $\dgr$
	\item for any $B\subseteq \dgr$, a Lie (bi)algebra $\b_{B}$
	\item for any $B'\subseteq B$, a Lie (bi)algebra morphism $i_{BB'}
	:\b_{B'}\to\b_{B}$
\end{itemize}
such that 
\begin{itemize}
	\item for any $B\subseteq \dgr$, $i_{BB}=\id_{\b_B}$
	\item for any $B''\subseteq B'\subseteq B$, $i_{BB'}\circ i_{B'B''}=i_{BB''}$
	\item for any $B_1\perp B_2$, 
	\[i_{(B_1\sqcup B_2)B_1}+i_{(B_1\sqcup B_2)B_2}:\b_{B_1}\oplus\b_{B_2}\to\b_{B_1\sqcup B_2}\]
	is an isomorphism of Lie (bi)algebras. 
\end{itemize}
The above properties imply in particular that $\b_{\emptyset}=0$,
and that $U\b$ is a diagrammatic algebra, with $(U\b)_B=U\b_B$
(cf.\;\ref{ss:diag-alg}).

A morphism $\varphi:\b\to\c$ of diagrammatic Lie (bi)algebras with the
same underlying diagram $\dgr$ is a collection of Lie (bi)algebra morphisms
$\varphi_B:\b_B\to\c_B$ labelled by the subdiagrams $B\subseteq
\dgr$ such that, for any $B'\subseteq B$, $\varphi_B\circ i^\b_{BB'}=
i^\c_{BB'}\circ\varphi_{B'}$.

\subsection{Split diagrammatic Lie bialgebras \cite{ATL1-2}}\label{ss:diag-sLBA}

A diagrammatic Lie (bi)algebra $\b$ is {\it split} if there are Lie (bi)algebra
morphisms $p_{B'B}:\b_{B}\to\b_{B'}$ for any $B'\subseteq B$, such
that $p_{B'B}\circ i_{BB'}=\id_{\b_{B'}}$, and

\begin{itemize}
	\item for any $B\subseteq D$, $p_{BB}=\id_{\b_B}$
	\item for any $B''\subseteq B'\subseteq B$, $p_{B''B'}\circ p_{B'B}=p_{B''B}$
	\item for any $B_1\perp B_2$ 
	\[p_{B_1(B_1\sqcup B_2)}\oplus p_{B_2(B_1\sqcup B_2)}:\b_{B_1\sqcup B_2}\to\b_{B_1}\oplus\b_{B_2}\]
	is an isomorphism of Lie (bi)algebras, and is the inverse of 
	$i_{(B_1\sqcup B_2)B_1}+i_{(B_1\sqcup B_2)B_2}$.
\end{itemize}

A morphism $\varphi:\b\to\c$ of split diagrammatic Lie (bi)algebras with
the same underlying diagram is one of the underlying diagrammatic
Lie (bi)algebras such that, for any $B'\subseteq B$, $p^\c_{B'B}\circ
\varphi_B=\varphi_{B'}\circ p^\b_{B'B}$.

The following is clear.

\begin{proposition}\label{pr:precox on bialg}
	Let $\b$ be a split diagrammatic Lie bialgebra. Then, there is an
	$(\redasso{}{},\DCPA{}{})$--strict
	symmetric pre--Coxeter category $\DYCox{\b}{}{}$ defined by the following data
	\begin{itemize}\itemsep0.25cm
		\item For any $B\subseteq\dgr$, the symmetric monoidal category $\DrY{\b_B}$.
		\item For any $B'\subseteq B$, the restriction functor $\Res_{\b_{B'},\b_{B}}$.
	\end{itemize}
\end{proposition}

\noindent\remark\;
A split diagrammatic Lie bialgebra $\b=\{b_B\}_{B\subseteq D}$ gives rise 
to a diagrammatic Manin triple $\gb=\{\g_{\b_B}\}_{B\subseteq D}$,
which will be referred to as the double of $\b$, and any such triple arises
this way (cf.~\cite[Sec.~5]{ATL1-2}).
Similarly, if $\b$ is an $\IN$--graded split diagrammatic Lie bialgebra 
with \fd homogeneous components (\ie for any $B\subseteq D$, $\b_
B$ is $\IN$--graded, with \fd homogeneous components and, for any
$B'\subseteq B$, the morphisms $i_{BB'}$ and $p_{B'B}$ are homogeneous
of degree 0), one can similarly define a diagrammatic Lie bialgebra
$\gb\res$, with $(\gb\res)_B=\g_{\b_B}\res$, endowed with
a canonical morphism of diagrammatic Lie bialgebras $\b\to\gb\res$.

\subsection{Diagrammatic Kac--Moody algebras and split structures}\label{ss:km-sLBA}

Let $\g$ be a diagrammatic \KM algebra with Dynkin diagram $\dgr$ and Cartan 
subalgebras $\h_B$, $B\subseteq\dgr$ (cf.~\ref{ss:diag-KM}). Then $\g$ is a diagrammatic 
Lie bialgebra where, for any $B\subseteq\dgr$, $\g_B=\<e_i, f_i, \h_B\;|\; i\in B\>$.

The diagrammatic structure on $\g$ determines a {\em split}
diagrammatic one on $\bpm{}$ as follows. For any $B
\subseteq \dgr$, let $\bpm{B}=\bpm{}\cap\g_B$ be the
subalgebras generated by $\{\h_B,e_i\;|\; i\in B\}$ and $
\{\h_B,f_i\;|\; i\in B\}$ respectively. If $B'\subseteq B$, let 
$i_{\pm,BB'}:\bpm{B'}\to\bpm{B}$ be the standard embedding, 
and regard $p_{\pm,B'B}=i_{\mp,BB'}^t$ as a map $\bpm{B}\to\bpm{B'}$
via the identifications $(\bmp{C})^\star\cong\bpm{C}$
given by the inner product. Then,
$\ker(p_{\pm, B'B})$ is a Lie subalgebra in $\bpm{B}$, 
and therefore $\{p_{\pm,B'B}\}$ give the required splitting 
of the Lie bialgebra $\bpm{}$.

Note that the restriction of $i_{\pm,BB'}$ to $\h_{B'}$ is the embedding $\h_{B'}
\hookrightarrow\h_B$, while $p_{\pm,B'B}:\h_B\to\h_{B'}$ is the projection corresponding
to the decomposition $\h_B=\h_{B'}\oplus\h_{B'}^\perp$.

\subsection{The symmetric Coxeter category $\DYCox{\bm{}}{\scsop{\sint}}$}\label{ss:Cox-DY}

We describe the Drinfeld--Yetter analogue of the symmetric Coxeter category
$\OCox{\g}{\scsop{\sint}}$ from \ref{ss:km-sym-cox}. Let $\hDrY{\bm{}}{\sint }$ be the category of 
{\em integrable} Drinfeld--Yetter $\bm{}$--modules, \ie $\h$--diagonalisable, 
endowed with a locally nilpotent action of the elements $\{f_i\}_{i\in \dgr}\subseteq\bm{}$, 
and satisfying \eqref{eq:DY-ss-0}, so as to give rise to integrable $\g$--modules
under the correspondence described in Proposition~\ref{ss:from-O-to-DY}.
Thus, the generalised braid group $\Br{W}$ acts on the objects in 
$\hDrY{\bm{}}{\sint }$ via the triple exponential operators $\texp{i}$, $i\in\bfI$. 
Moreover, $\O^{\sint}_{\infty,\g}$ identifies with a full braided tensor subcategory of 
$\hDrY{\bm{}}{\sint }$. The following is straightforward.

\begin{proposition}
	There is a canonical $(\redasso{}{},\DCPA{}{})$--strict symmetric Coxeter category 
	$\DYCox{\bm{}}{\scsop{\sint}}$ of type $(\dgr,\ulm)$ given by the following data.
	\begin{itemize}\itemsep0.25cm
		\item For any $B\subseteq D$, the symmetric monoidal category $\hDrY{\bm{B}}{\sint}$.
		\item For any $B'\subseteq B$, the restriction functor $\Res_{B'B}:\hDrY{\bm{B}}{\sint }
		\to\hDrY{\bm{B'}}{\sint }$.
		\item For any $i\in D$, the operator $\CoxS{\DYCox{}{}}{}{i}=\texp{i}$.
	\end{itemize}
	Moreover, $\OCox{\g}{\scsop{\sint}}$ naturally identifies with a subcategory in $\DYCox{\bm{}}{\scsop{\sint}}$.
\end{proposition}
%

\subsection{Deformations of $\DYCox{\bm{}}{\scsop{\sint}}$}\label{ss:def-Cox-DY}

We shall be interested in deformations of the symmetric Coxeter category $\DYCox{\bm{}}{\scsop{\sint}}$.
It is clear that the results of \ref{ss:def-catO} extend from category $\O_{\infty}$ $\g$--modules
to \DYt $\bm{}$--modules. Indeed, let $\hDrY{\bm{}}{\hdef,\sint}$ 
denote the category of deformation integrable Drinfeld--Yetter $\bm{}$--modules.
Since $\O^{\hdef, \sint}_{\infty,\g}$ identifies with a full 
braided tensor subcategory of $\hDrY{\bm{}}{\hdef, \sint}$, 
the algebra $\UDYint{\bm{}}{n}$ of the endomorphisms of the forgetful 
functor $(\hDrY{\bm{}}{\hdef,\sint})^{\boxtimes n}\to\tfV$ is endowed with a canonical
morphism $\UDYint{\bm{}}{n}\to\UOint{\g}{n}$.

Note that the restriction functors preserve the subcategories 
$\O^{\hdef, \sint}_{\infty,\g_B}$, $B\subseteq\dgr$, therefore we obtain a cosimplicial lax 
bidiagrammatic algebra $\UCoxDYint{\bm{}}{\bullet}$. By restriction to category $\Oint_{\infty}$
modules, we obtain a canonical morphism $\varphi^{\scs{\bullet}}_{\g}:\UCoxDYint{\bm{}}{\bullet}\to\UCoxOint{\g}{\bullet}$.
This yields the following analogue of Proposition~\ref{ss:from-end-to-cat}.

\begin{proposition}\hfill
	\begin{enumerate}\itemsep0.25cm
		\item
		Every braided Coxeter structure $\sCox{}$ on $\UCoxDYint{\bm{}}{\bullet}$ gives rise 
		to a canonical braided Coxeter category $\DYCox{\sCox{}}{\scsop{\hdef,\sint}}$ on deformation integrable 
		Drinfeld--Yetter modules.
		\item 
		By restriction to integrable category $\O_{\infty}$ modules, $\sCox{}$ defines a braided Coxeter structure 
		on $\UCoxOint{\g}{\bullet}$. The corresponding category $\OCox{\sCox{}}{\scsop{\hdef,\sint}}$, defined as in Proposition~\ref{ss:from-end-to-cat}, identifies with a subcategory of $\DYCox{\sCox{}}{\scsop{\hdef,\sint}}$.
	\end{enumerate}
\end{proposition} 

\section{Universal pre--Coxeter structures on diagrammatic Lie bialgebras}\label{s:univ-alg}

We review the definition of the diagrammatic $\PROP$s
$\dLBA{\dgr}$, $\MDY{\dgr}{n}$ and the universal algebra 
$\DUA{\dgr}{\bullet}$ introduced in \cite{ATL2}. The latter is a 
universal analogue of the cosimplicial bidiagrammatic algebra
$U\gb^{\otimes \bullet}$ given by the enveloping algebra of the double of 
a split diagrammatic Lie bialgebra.

\subsection{$\PROP$s \textbf{\cite{La,mac,ee,ATL1-2}}}\label{ss:prop-intro}

A $\PROP$ is a $\sfk$--linear, strict, symmetric monoidal category
$\sfP$ whose objects are the non--negative integers, and such that
$[n]\ten[m]=[n+m]$. In particular, $[0]$ is the unit object and $[n]=[1]^
{\ten n}$ carries an action of the symmetric group ${\mathfrak S}_n$. A morphism of $\PROP$s is a symmetric monoidal
functor $\G:\sfP\to\sfQ$ which is the identity on objects,
and is endowed with the trivial tensor structure
\[\id:\G([m]_{\sfP})\otimes\G([n]_{\sfP})=
[m]_{\sfQ}\otimes[n]_{\sfQ}=[m+n]_{\sfQ}=\G([m+n]_{\sfP})\]

Fix henceforth a complete bracketing $b_n$ on $n$ letters for
any $n\geqslant 2$, and set $\bfb=\{b_n\}_{n\geqslant 2}$. A {\it module}
over $\sfP$ in a symmetric monoidal category $\N$ is a symmetric
monoidal functor $(\G,J):\sfP\to\N$ such that\footnote{In a monoidal
	category $(\C,\otimes)$, $V^{\otimes n}_{b_n}$ denotes the $n$--fold
	tensor product of $V\in\C$ bracketed according to $b_n$. For example
	$V^{\otimes 3}_{(\bullet\bullet)\bullet}=(V\otimes V)\otimes V$.}
\begin{equation*}
	\G([n])=\G([1])^{\otimes n}_{b_n}
\end{equation*}
and the following diagram is commutative
\begin{equation}\label{eq:J Phi}
	\xymatrix{
		\G([m])\otimes \G([n]) \ar@{=}[d]\ar[rr]^{J_{[m],[n]}} && \G([m+n])\ar@{=}[d]\\
		\G([1])^{\otimes m}_{b_m}\otimes \G([1])^{\otimes n}_{b_n}\ar[rr]_{\Phi}&& \G([1])^{\otimes(m+n)}_{b_{m+n}}
	}
\end{equation}
where $\Phi$ is the associativity constraint in $\N$.

A {\it morphism} of modules over
$\sfP$ is a natural transformation of functors. The category of $\sfP
$--modules in $\N$ is denoted by $\Fun_{\bfb}^\otimes(\sfP,\N)$.\\

\noindent\example
Let $\preLA$ be the $\PROP$ generated by a morphism $\mu:[2]\to[1]$,
subject to the relations
\[\mu\circ(\id_{[2]}+(1\,2))=0
\aand \mu\circ(\mu\ten\id_{[1]})\circ(\id_{[3]}+(1\,2\,3)+(3\,1\,2))=0\]
as morphisms $[2]\to[1]$ and $[3]\to[1]$ respectively. Let $\preLA(\sfk)$ be the category of 
Lie algebras over a field $\sfk$. Note that there is a canonical isomorphism of categories 
$\Fun_{\bfb}^{\ten}(\preLA,\vect_{\sfk})\to\preLA(\sfk)$, which assigns to a functor $\G$
the Lie algebra $\G([1])$ with bracket $\G(\mu):\G(1)\ten \G([1])=\G([2])\to \G([1])$.
We denote by $\mathsf{LBA}$ the analogous $\PROP$ corresponding to Lie bialgebras. 

\subsection{The Karoubi envelope}\label{ss:karoubi}

Recall that the Karoubi envelope of a category $\C$ is the category $\Kar
{\C}$ whose objects are pairs $(X,\pi)$, where $X\in\C$ and $\pi:X\to X$
is an idempotent. The morphisms in $\Kar{\C}$ are defined as
\[
{\Kar{\C}}((X,\pi), (Y,\rho))=\{f\in\C(X,Y)\;|\; \rho\circ f=f=f\circ\pi\}
\]
with $\id_{(X,\pi)}=\pi$. In particular, ${\Kar{\C}}((X,\id),(Y,\id))={\C}(X,Y)$, 
so that the functor $\C\to\Kar{\C}$ which maps $X\mapsto(X,\id)$ and
$f\mapsto f$ is fully faithful.

Every idempotent in $\Kar{\C}$ splits canonically. Namely,
if $q\in\Kar{\C}((X,\pi),(X,\pi))$ satisfies $q^2=q$, the maps
\[i=q:(X,q)\to(X,\pi)\aand p=q:(X,\pi)\to(X,q)\]
satisfy $i\circ p=q$ and $p\circ i=\id_{(X,q)}$. 

If $\sfP$ is a $\PROP$, we denote by $\cKar{\sfP}$ the closure
under infinite direct sums of the Karoubi completion of $\sfP$. 
By a slight abuse of terminology, in the following we still refer to $\cKar{\sfP}$ as a $\PROP$.
If $\N$ is a symmetric monoidal category, a {\it module} over $\cKar
{\sfP}$ in $\N$ is a symmetric monoidal functor $\cKar{\sfP}\to\N$
such that the composition $\sfP\to\cKar{\sfP}\to\N$ is a module
over $\sfP$. We denote the category of such modules by $\Fun
_{\bfb}^\otimes(\cKar{\sfP},\N)$. It is clear that, if $\N$ is Karoubi
complete and closed under infinite direct sums, the pull--back functor
\[\Fun_{\bfb}^\otimes(\cKar{\sfP},\N)\to\Fun_{\bfb}^\otimes(\sfP,\N)\]
is an equivalence of categories. 

\subsection{Colored $\PROP$s}\label{ss:colored-prop}

A \emph{colored} $\PROP$ $\mathsf{P}$ is a $\sfk$--linear, strict,
symmetric monoidal category whose objects are finite sequences
over a set $\sfA$, \ie
\[\mathsf{Obj}(\mathsf{P})=\coprod_{n\geqslant0}\sfA^n\]
with tensor product given by the concatenation of sequences, and
tensor unit given by the empty sequence. Modules over a colored
$\PROP$ $\mathsf{P}$ and its closure $\ul{\sfP}$ are defined as in
\ref{ss:prop-intro} and \ref{ss:karoubi}, respectively.


\subsection{Diagrammatic $\PROP$s}\label{ss:diag-prop}

Let $\dgr$ be a non--empty diagram. We denote by $\dLBA{\dgr}$ the $\PROP$
generated by a Lie bialgebra object $([1],\mu,\delta)$ with a Lie bialgebra idempotent 
$\dmp{B}:[1]\to[1]$ for any $B\subseteq \dgr$ subject to the relations
\begin{itemize}
	\item $\dmp{\dgr}=\id_{[1]}$
	\item for any $B'\subseteq B$, $\dmp{B'}\circ\dmp{B}=\dmp{B'}=\dmp{B}\circ\dmp{B'}$ 
	\item  for any $B_1\perp B_2$, $\dmp{B_1\sqcup B_2}=\dmp{B_1}+\dmp{B_2}$. 
\end{itemize}
The above relations imply in particular that $\dmp{\emptyset}=0$, and that
$\dmp{B'}\circ\dmp{B''}=0=\dmp{B''}\circ\dmp{B'}$ for any
$B'\perp B''$ since if $p,q$ are idempotents, $p+q$ is an
idempotent if and only if $pq=0=qp$.\\

\noindent\remark\; Note that a module over $\mathsf{LBA}_\dgr$ in $\N$ (or equivalently
a module over its Karoubi completion $\dLBA{\dgr}$) is the same as a split diagrammatic 
Lie bialgebra in $\N$, as defined in \ref{ss:diag-sLBA}.

\subsection{Universal Drinfeld--Yetter modules}\label{ss:DY-prop}

Given a diagram $\dgr$ and $n\geqslant0$, the category $\MDY{\dgr}{n}$
is the colored $\uPROP$ generated by $n+1$ objects, $\ACDY{1}$ and
$\{\VCDY{k}\}_{k=1}^n$, and morphisms
\begin{itemize}
	\item $\dmp{B}:\ACDY{1}\to\ACDY{1}$, $B\subseteq \dgr$\\
	\item $\mu:\ACDY{2}\to\ACDY{1}$, $\delta:\ACDY{1}\to\ACDY{2}$\\[.05ex]
	\item $\pi_k:\ACDY{1}\ten \VCDY{k}\to \VCDY{k}$ and $\pi_k^*:\VCDY{k}\to\ACDY{1}\ten \VCDY{k}$
\end{itemize}
such that 
\begin{itemize}
	\item $(\ACDY{1},\{\dmp{B}\}_{B\subseteq \dgr},\mu,\delta)$ is an $\dLBA{\dgr}$--module 
	in $\MDY{\dgr}{n}$\\
	\item every $(\VCDY{k},\pi_k,\pi_k^*)$ is a \DYt module over $\ACDY{1}$
\end{itemize}
In particular, $\MDY{\dgr}{0}=\dLBA{\dgr}$.\\

\noindent\remark\;
If $\N$ is a $\sfk$--linear symmetric monoidal category, 
$\MDY{\dgr}{n}$--modules in $\N$ are isomorphic to the category whose
objects are tuples $(\b;V_1,\ldots,V_n)$ consisting of a split diagrammatic
Lie bialgebra $\b$ in $\N$, and $n$ \DYt modules $V_1,\ldots,V_n\in\N$
over $\b$. For any such tuple, we shall refer to the corresponding functor
$\G_{(\b;V_1,\ldots,V_n)}:\MDY{\dgr}{n}\to\N$ as its \emph{realisation functor}.  

\subsection{Universal algebras}\label{ss:univ-alg}
For $B\subseteq\dgr$ and $n\geqslant0$, set 
\[\RDYUA{B}{n}=\pEnd{\MDY{B}{n}}{\VCDY{1}\ten\cdots\ten\VCDY{n}}\]
Let $\b$ be a split diagrammatic Lie bialgebra and $\gb$ its Drinfeld double. The
algebra $\RDYUA{B}{n}$ is a universal analogue of $U\g_B^{\ten n}$.
Specifically, let $\UDYz{\b_B}{n}$ be the algebra of endomorphisms of the forgetful
functor $(\hDrY{\b_B}{})^{\boxtimes n}\to\vect$, and $\UCoxDYz{\b}{\bullet}$ the corresponding  cosimplicial lax
bidiagrammatic algebra. Then, the following holds \cite[Prop.~8.5 and 8.9]{ATL1-2}.

\begin{proposition}
	\hfill
	\begin{enumerate}\itemsep0.25cm
		\item For any $B'\subseteq B$, there is a canonical morphism of algebras 
		$\sfi_{BB'}^n:\RDYUA{B'}{n}\to\RDYUA{B}{n}$. The algebras $\{\RDYUA{B}{n}\}_{B\subseteq\dgr}$
		and morphisms $\{\sfi_{BB'}^n\}_{B'\subseteq B\subseteq\dgr}$ give rise to a diagrammatic algebra $\DUA{\dgr}{n}$ for any $n\geqslant 0$.
		\item For any $B'\subseteq B$, there is a canonical invariant subalgebra $\RDYUA{BB'}{n}\subset
		\RDYUA{B}{n}$, yielding a bidiagrammatic structure on $\DUA{\dgr}{n}$.
		\item For any $B\subseteq\dgr$, there is a canonical cosimplicial structure on the tower of algebras $\{\RDYUA{B}{n}\}_{n\geqslant 0}$, 
		which is compatible with the morphisms $\sfi_{BB'}^n$ and the invariant subalgebras, yielding a cosimplicial bidiagrammatic 
		structure $\DUA{\dgr}{\bullet}$. 
		\item Let $\b$ be a split diagrammatic Lie bialgebra. The realisation functors induce a canonical morphism of  cosimplicial lax bidiagrammatic algebras $\rho_{\b}^{\bullet}:\DUA{\dgr}{\bullet}\to\UCoxDYz{\b}{\bullet}$.
	\end{enumerate}
\end{proposition}

We describe  the diagrammatic subalgebras and morphisms $\sfi_{BB'}^n$
in \ref{ss:diag-univ-sub}, the subalgebras of invariants in \ref{ss:univ-inv}, the
cosimplicial structure in \ref{ss:univ-cosimp}, and the morphisms $\rho_{\b}^n:
\RDYUA{B}{n}\to\UDYz{\b_B}{n}$ in \ref{ss:real-funct-end}.

\subsection{Diagrammatic subalgebras}\label{ss:diag-univ-sub}
For any $B'\subseteq B$, there is a canonical realisation functor
$\RDY{B'}{n}\to\RDY{B}{n}$ which sends the object $[1]_{B'}$ in
$\RDY{B'}{n}$ to the Lie bialgebra $\theta_{B'}([1]_B)=([1]_B,\dmp
{B'})$ in $\RDY{B}{n}$, and each $(\VCDY{B',k},\pi_{B',k},\pi_{B',
	k}^*)$ to
\[\Res_{\theta_{B'}([1]_B),[1]_B}(\VCDY{B,k},\pi_{B,k},\pi_{B,k}^*)=
(\VCDY{B,k},\pi_{B,k}\circ \theta_{B'}\otimes\id,\theta_{B'}\otimes
\id\circ\pi_{B,k}^*)\]
where $\theta_{B'}$ is regarded both as the split injection $([1]_B,
\dmp{B'})\to[1]_B$ and projection $[1]_B\to([1]_B,\dmp{B'})$ (cf.
\ref{ss:karoubi}). The functor induces a homomorphism $\sfi_{BB'}:
\RDYUA{B'}{n}\to\RDYUA{B}{n}$, and it is clear that $\sfi_{BB}=
\id_{\RDYUA{B}{n}}$, and $\sfi_{BB'}\circ\sfi_{B'B''}=\sfi_{BB''}$
for any $B''\subseteq B'\subseteq B$.\\

\noindent\remark\label{re:inj D}
We show in \cite{ATL2} that the homomorphism
$\sfi_{BB'}:\RDYUA{B'}{n}\to\RDYUA{B}{n}$ is injective. We shall therefore
regard $\RDYUA{B'}{n}$ as a subalgebra of $\RDYUA{B}{n}$ and, for $x\in
\RDYUA{B'}{n}$, write $x\in\RDYUA{B}{n}$ instead of $\sfi_{BB'}(x)\in\RDYUA
{B}{n}$. Moreover, $\{\RDYUA{B}{n}\}_{B\subseteq \dgr}$ is a diagrammatic algebra, 
since multiplication induces an isomorphism $\RDYUA{B_1\sqcup B_2}{n}
\cong\RDYUA{B_1}{n}\ten\RDYUA{B_2}{n}$ \cite[Prop. 10.6 (4)]{ATL2}.

\subsection{Invariants}\label{ss:univ-inv}

For any pair of subdiagrams $B'\subseteq B$, denote by $\RDYUA {BB'}
{n}\subseteq\RDYUA{B}{n}$ the subalgebra of elements which commute
with the diagonal action and coaction of $[\b_{B'}]=([1],\dmp{B'})$ on
$\VDY{1}\otimes\cdots\otimes\VDY{n}$. Note that, by \cite[Lemma~8.4]{ATL1-2},
$\RDYUA{BB'}{n}$ commutes with the diagonal action of $\RDYUA{B'}{}$
on $\VDY{1}\otimes\cdots\otimes\VDY{n}$, which is given by 
\[ \RDYUA{B'}{}\ni x\longrightarrow x_{1,2,\ldots,n}=\Delta^{n-1}_1\circ\cdots\circ\Delta^2_1\circ\Delta^1_1(x)\in\RDYUA{B'}{n}\]

\subsection{Cosimplicial structure}\label{ss:univ-cosimp}

For every $B\subseteq \dgr$, 
$n\geqslant1$ and $i=0,\dots, n+1$, there are faithful functors
\[
\D_{i}^n:\MDY{B}{n}\to\MDY{B}{n+1}
\]
mapping $[1]$ to $[1]$, and given by
\[\D^n_0(\VDY{k})=\VDY{k+1}
\aand
\D^n_{n+1}(\VDY{k})=\VDY{k}\]
for $1\leqslant k\leqslant n$, and, for $1\leqslant i\leqslant n$,
\[\D^n_i(\VDY{k})=
\left\{\begin{array}{cc}
	\VDY{k} 				& 1\leqslant k\leqslant i-1	\\[1.1ex]
	\VDY{i}\otimes \VDY{i+1} 	& k=i					\\[1.1ex]
	\VDY{k+1}				& i+1\leqslant k\leqslant n	
\end{array}\right.\]
and $\E_n^{(i)}:\MDY{B}{n}\to\MDY{B}{n-1}$
\[
\E_n^{(i)}=\G_{([1], \VDY{1},\dots,\VDY{i-1},\1,\VDY{i+1},\dots, \VDY{n-1})}
\]
where $\1$ is the tensor unit in $\MDY{B}{n}$, regarded as trivial
\DYt module. These induce algebra homomorphisms
\[
\Delta_i^n:\RDYUA{B}{n}\to\RDYUA{B}{n+1}
\]
which are universal analogues of the insertion/coproduct maps on $U\g
_{\b_B}^{\otimes n}$. They endow the tower $\{\RDYUA{B}{n}\}_{n\geqslant
	0}$ with the structure of a cosimplicial algebra, with Hochschild differential
$d^n=\sum_{i=0}^{n+1}(-1)^i\Delta_i^n:\RDYUA{B}{n}\to\RDYUA{B}{n+1}$.
This structure is compatible with the maps $\{\sfi_{BB'}\}_{B'\subseteq B
	\subseteq D}$ and invariants.

\subsection{Realisation functors and endomorphisms}\label{ss:real-funct-end}

Let $B\subseteq D$. For any $n$--tuple $\{V_k,\pi_k,\pi_
k^*\}_{k=1}^ n$ of \DYt modules over $\b_B$, let
\[\G_{(\b_B;V_1,\dots, V_n)}:\MDY{B}{n}\longrightarrow{\vect}\]
be the corresponding realisation functor. We have the following \cite[Prop.~8.7]{ATL1-2}.

\begin{proposition}\hfill
	\begin{enumerate}\itemsep0.25cm
		\item 
		There is an algebra homomorphism
		\[\DYrho{\b_B}{n}:\RDYUA{B}{n}\to\UDYz{\b_B}{n}\]
		which assigns to any $T\in\RDYUA{B}{n}$, and any $V_1,\ldots,V_n\in
		\DrY{\b_B}$ the endomorphism $\G_{(\b_B;V_1,\dots, V_n)}(T)\in\End_
		\sfk(V_1\otimes\cdots\otimes V_n)$.
		\item The collection of homomorphisms $\{\DYrho{\b_B}{n}\}_{B\subseteq\dgr}$
		is a morphism of cosimplicial bidiagrammatic algebras $\DYrho{\b}{\bullet}:\DUA{\dgr}{\bullet}\to\UCoxDYz{\b}{\bullet}$.
	\end{enumerate}
\end{proposition}

\subsection{Gradings and completions}\label{ss:grading}
Let $B\subseteq\dgr$. The $\uPROP$ $\MDY{B}{n}$ has a natural $\IN$--bigrading 
given by $\deg(\sigma)=(0,0)=\deg(\dmp{B'})$ for any $\sigma\in\SS_N$ and 
$B'\subseteq B$,
\[\deg(\mu)=(1,0)=\deg(\pi_{\VDY{k}})
\aand
\deg(\delta)=(0,1)=\deg(\pi_{\VDY{k}}^*)\]
for any $1\leqslant k\leqslant n$.
The algebra $\RDYUA{B}{n}$ inherits this bigrading and is concentrated
in bidegrees $(N,N)$, since a degree $(p,q)$ morphism with source
$\VCDY{1}\ten\cdots\ten\VCDY{n}$ is easily seen to map to $[1]^{\otimes
	(q-p)}\otimes\VCDY{1}\ten\cdots\ten\VCDY{n}$. For any $a,b\in\IN$, the
corresponding $\IN$--grading determined by mapping $(1,0),(0,1)$ to
$a,b$ respectively yields the same graded completion $\CRDYUA{B}
{n}$ of $\RDYUA{B}{n}$, so long as $a+b>0$. For definiteness, we
set $a=0$ and $b=1$.

Note that the morphisms $\sfi_{BB'}^n$ and the cosimplicial structure
are compatible with grading, thus yielding a cosimplicial lax bidiagrammatic
algebra $\CDUA{\dgr}{\bullet}$ given by the collection of the invariant subalgebras
$\CRDYUA{BB'}{n}\subseteq\CRDYUA{B}{n}$, $B'\subseteq B$, defined as in \ref{ss:univ-inv}.

\subsection{Universal pre--Coxeter structures}\label{ss:univ-fiber}

Let $\b$ be a split diagrammatic Lie bialgebra and $\gb$ its 
Drinfeld double. By Proposition \ref{ss:univ-alg}, 
$\DUA{\dgr}{\bullet}$ is a 
universal version of the cosimplicial bidiagrammatic algebra $U\gb^{\otimes
\bullet}$. In a similar vein, its completion $\CDUA{\dgr}{\bullet}$ is a 
universal analogue of the trivial deformation $\hext{U\gb^{\otimes \bullet}}$. 
Namely, let $\UDY{\b_B}{n}$ be the algebra of endomorphisms of the forgetful functor $(\hDrY{\b_B}{\hdef})^{\boxtimes n}\to{\tfV}$ and $\UCoxDY{\b}{\bullet}$ 
the corresponding  cosimplicial lax bidiagrammatic algebra. We have the following
\cite[Sec.~9.7 and Prop.~9.8]{ATL1-2}.

\begin{proposition}\hfill
	\begin{enumerate}\itemsep0.25cm
		\item
		There is a canonical morphism of  cosimplicial lax bidiagrammatic algebras
		$\wh{\rho}_{\b}^{\bullet}:\CDUA{\dgr}{\bullet}\to\UCoxDY{\b}{\bullet}$.	
		\item
		A braided pre--Coxeter structure $\pCox{}=(\Phi_{B}, R_{B}, J_{\F}, \DCPA{\F}{\G}, 
		\redasso{\F}{\F'})$ on $\CDUA{\dgr}{\bullet}$ is \emph{universal}, \ie 
		for any split diagrammatic Lie bialgebra $\b$, it induces one on $\UCoxDY{\b}{\bullet}$
		through $\wh{\rho}_{\b}^{\bullet}$. We denote the resulting braided pre--Coxeter category
		by $\DYCox{\b,\pCox{}}{}$.
	\end{enumerate}
\end{proposition}

For the reader's convenience, we recall the construction of 
$\wh{\rho}_{\b}^{\bullet}$ from \cite[Sec.~9.7]{ATL1-2}.
Let $\c$ be a Lie bialgebra and $\hDrY{\c}{\hdef}$ the category
of \DYt $\c$--modules in $\tfV$. Scaling the coaction on $V\in\hDrY{\c}{\hdef}$ 
by $\hbar$ yields an isomorphism between $\hDrY{\c}{\hdef}$ and the 
category $\hDrY{\hextsub{\c}}{\adm}$ of \DYt modules over the Lie bialgebra $\hextsub{\c}=(\hext{\c},[\cdot,\cdot],\hbar\delta)$, whose
coaction is divisible by $\hbar$. We denote by $\cU{\c}{n}$ the algebra
of endomorphisms of the $n$--fold tensor power of the forgetful functor
$\ff_\c:\hDrY{\c}{\hdef}\to{\tfV}$. Note that $\UDY{\c}{n}$ identifies canonically 
with the analogous completion defined for $\hDrY{\hextsub{\c}}{\adm}$.\\ 

In the case of the split diagrammatic Lie bialgebra $\b$, the realisation functors 
\[\G_{(\hextsub{\b}_B;V_1,\dots, V_n)}:\MDY{B}{n}\longrightarrow{\tfV}\]
corresponding to $V_1,\ldots,V_n\in\hDrY{\hextsub{\b}_B}{\adm}\cong \hDrY
{\b_B}{\hdef}$ induce a homomorphism $\wh{\rho}_{\b}^{n}:\DUA{\dgr}{n}\to\UDY{\b}{n}$
which naturally extends to $\CDUA{\dgr}{n}$. Finally, note that, if $B'\subseteq B$, the subalgebra 
of $[\b_{B'}]$--invariants in $\CRDYUA{BB'}{n}\subset\CRDYUA{B}{n}$ is 
mapped by $\wh{\rho}_{\b_B}^{n}$ to elements in $\UDY{\b_B}{n}$
commuting with the diagonal (co)action of $\b_{B'}$.

\subsection{Distinguished elements in $\DUA{\dgr}{\bullet}$}\label{ss:distinguished-elements}

	There are two distinguished families of elements in $\DUA{\dgr}{n}$, namely
	\begin{equation}
		\Kp{i}{B}=	\pi_{\VDY{i}}\circ\dmp{B}\ten\id_{\ten\VDY{}}\circ \pi^*_{\VDY{i}}
		\aand
		\rp{ij}{B}= \pi_{\VDY{i}}\circ\dmp{B}\ten\id_{\ten\VDY{}}\circ \pi^*_{\VDY{j}}
	\end{equation}
	where $1\leqslant i\neq j\leqslant n$ and $B\subseteq\dgr$. Note that, for a split 
	diagrammatic Lie bialgebra $\b$, under the equivalence between Drinfeld--Yetter $\b$--modules and equicontinuous $\gb$--modules described in \ref{ss:drinf-db-rep}, one has
	\begin{equation}
		\wh{\rho}_{\b}(\Kp{i}{B})=\hbar\sum_k (b_k)^{(i)}\cdot (b^k)^{(i)}
		\aand
		\wh{\rho}_{\b}(\rp{ij}{B})=\hbar\sum_k (b_k)^{(i)}\cdot (b^k)^{(j)}
	\end{equation}
	where $\{b_k\},\{b^k\}$ are dual bases of $\b_B$ and $\b_B^*$. Therefore, the 
	algebra $\DUA{\dgr}{\bullet}$ contains the universal analogues of the 
	$r$--matrices and normally ordered Casimir elements of the Drinfeld doubles
	$\g_{\b_B}$, $B\subseteq\dgr$.
	 
	Set $\Tp{ij}{B}=\rp{ij}{B}+\rp{ji}{B}$. As in Lemma~\ref{sss:holo-to-Ug}, 
	we obtain a morphism of algebras $\xi_{B}^n:\DKHAH{}{n}\to\CDUA{B}{n}$ given by
	$\xi_{B}^{n}(\Th{ij})=\Tp{ij}{B}$. Therefore, any universal associator $\Phi\in\DKHAH{}{3}$
	is naturally realised in $\CDUA{B}{3}$ as $\Phi_B=\xi_{B}^{3}(\Phi)$. Note also that, 
	if $\Phi$ is a Lie associator, then for any $B_1\perp B_2$ one has $\Phi_{B_1\sqcup B_2}=\Phi_{B_1}\cdot
	\Phi_{B_2}$. In the following, we shall be interested in braided pre--Coxeter structure on 
	$\CDUA{\dgr}{\bullet}$ whose diagrammatic associators $\Phi_B$ are uniformly determined by a
	fixed Lie associator $\Phi\in\DKHAH{}{3}$\footnote{In \cite[Sec.~10.1]{ATL1-2}, we consider 
	a larger class of \emph{factorisable} associators.} and
	$R_B=\exp(\Tp{}{B}/2)$.

\section{Quantisation of diagrammatic Lie bialgebras}\label{s:diag-HA}

\newcommand{\hA}{\mathfrak{A}}
\newcommand{\hC}{\mathfrak{B}}
\newcommand{\BonC}{\triangleright}
\newcommand{\ConB}{\triangleleft}
\newcommand{\dcp}[2]{#1\negthinspace\BonC\negthinspace\negthinspace\ConB\negthinspace#2}

In this section, we review the notion of admissible \DYt module over a quantised universal
enveloping algebra (QUE) introduced in \cite{ATL1-1}. The category of such modules over a split
diagrammatic QUE $\hC$ gives rise to a braided  pre--Coxeter category $\DYCox{\hC}{\adm}$. 
When $\hC$ is the \nEK quantisation of a split diagrammatic Lie bialgebra $\b$, we outline the
construction of a Tannakian equivalence between $\DYCox{\hC}{\adm}$ and a braided pre--Coxeter
category of deformation \DYt modules over $\b$ arising from the universal diagrammatic algebra
$\CDUA{\dgr}{\bullet}$ obtained in \cite{ATL1-1,ATL1-2}.

\subsection{Drinfeld--Yetter modules over a Hopf algebra \cite{ek-2,Y}}\label{ss:DY Hopf}

A \DYt module over a Hopf algebra $\hC$ is a triple $(\V,\pi_{\V},\pi_{\V}^*)
$, where $(\V,\pi_\V)$ is a left $\hC$--module, $(\V,\pi^*_\V)$ a right $\hC$--comodule,
and the maps $\pi_\V:\hC\otimes\V\to\V$ and $\pi_\V^*:\V\to \hC\otimes\V$
satisfy the following compatibility condition:
\[\pi_{\V}^*\circ\pi_{\V}=
m^{(3)}\otimes\pi_{\V} \circ (1\,3)(2\,4) \circ
S^{-1}\otimes\id^{\otimes 4}\circ \Delta^{(3)}\otimes\pi_{\V}^*\]
where $m^{(3)}:\hC^{\otimes 3}\to \hC$ and $\Delta^{(3)}:\hC\to \hC^{\otimes 3}$
are the iterated multiplication and comultiplication respectively, and $S:\hC
\to \hC$ is the antipode.

The category $\DrY{\hC}$ of such modules is a braided monoidal
category. For any $\V,\W\in\DrY{\hC}$, the action and coaction on
the tensor product  $\V\ten \W$ are defined by
\[
\pi_{\V\ten \W}=\pi_{\V}\ten\pi_{\W}\circ (2\,3)\circ\Delta\ten\id_{\V\otimes\W}
\quad\mbox{and}\quad
\pi^*_{\V\ten \W}= m^{21}\ten\id_{\V\otimes\W}\circ(2\,3)\circ\pi^*_{\V}\ten\pi^*_{\W}
\]
The associativity constraints are trivial, and the braiding is 
$\beta_{\V\W}= (1\,2)\circ R_{\V\W}$, where the $R$--matrix
$R_{\V\W}\in\End(\V\ten \W)$ is defined by
\[R_{\V\W}=\pi_{\V}\ten\id_\W\circ(1\,2)\circ\id_\V\ten\pi_{\W}^*\]
The linear map $R_{\V\W}$ is invertible, with inverse
\[R_{\V\W}^{-1}=
\pi_{\V}\ten\id_\W\circ S\ten \id_{\V\otimes\W}\circ(1\,2)\circ\id_\V\ten\pi_{\W}^*\]

\subsection{The finite quantum double \cite{drinfeld-86}}\label{ss:DY-qD}

Let $\hC$ be a finite--dimensional Hopf algebra, and $\hC^{\circ}$ the dual
Hopf algebra $\hC^*$ with opposite coproduct. The quantum double of $\hC$
is the unique quasitriangular Hopf algebra $(D\hC, R)$ such that 1) $D\hC=
\hC\otimes \hC^\circ$ as vector spaces 2) $\hC$ and $\hC^\circ$ are Hopf subalgebras
of $D\hC$ and 3) $R$ is the canonical element in $\hC\ten \hC^{\circ}\subset
D\hC\ten D\hC$. The category $\Rep D\hC$ is readily seen to be canonically isomorphic, as 
a braided monoidal category, to $\DrY{\hC}$ (see \eg~\cite[Appendix A]{ATL1-1}).

\subsection{Quantum double for QUEs}\label{ss:dualityQUE}

The construction of the quantum double can be adapted to
quantised universal enveloping algebras (QUE). Recall that
a QUE is a topological Hopf algebra $\hC$ over $\hext{\IC}$ which 
reduces modulo $\hbar$ to an enveloping algebra $U\b$ for 
some Lie bialgebra $\b$, and is such that, for any $x\in\b$,
\[
\delta(x)=\frac{\Delta(\wt{x})-\Delta^{21}(\wt{x})}{\hbar}\,\mod\hbar
\]
where $\wt{x}\in \hC$ is any lift of $x$. A QUE is of finite type if the
underlying Lie bialgebra $\b$ is finite--dimensional. In this case,
the dual $\hC^*$ is a {quantised formal
	series Hopf algebra} (QFSH), \ie a topological Hopf algebra over
$\hext{\IC}$ which reduces modulo $\hbar$ to $\wh{S\b}=\prod_nS^n\b$. 
Conversely, the dual of a QFSH of finite type is a QUE (cf. \cite
{drinfeld-86, gav} or \cite[Sec. 2.19]{ATL1-1}).

If $\hC$ is a QUE, set
\[\hC'=
\left\{b\in \hC\;\left|\; (\id-\iota\circ\varepsilon)^{\ten n}\circ\Delta^{(n)}(b)\in\hbar^n\hC^{\ten n}
\;\text{for any $n\geqslant 0$}\right.\right\}\]
where $\Delta^{(n)}:\hC\to \hC^{\otimes n}$ is the iterated coproduct.
Then, $\hC'$ is a Hopf subalgebra of $\hC$, and a QFSH. In particular,
if $\hC$ is of finite type, $\hC^\vee=(\hC')^*$ is a QUE. 
As in \ref{ss:DY-qD}, $(\hC, \hC^\vee)$ is a matched pair of Hopf algebras
\cite[A.5]{ATL1-1}.
The double cross product $D\hC=\dcp{\hC}{\; \hC^\vee}$ is a 
quasitriangular QUE, whose $R$--matrix is the canonical element 
$R\in \hC'\ten \hC^\vee$, and underlying Lie bialgebra the Drinfeld
double $\gb=\b\oplus\b^*$. 

This construction extends to the case of \emph{finitely $\IN$--graded}
QUEs, \ie $\IN$--graded Hopf algebras $\hC=\bigoplus_{n\geqslant 0}
\hC_n$ such that $\hC_0$ is a QUE of finite type, and each $\hC_n$ is a
finitely generated $\hC_0$--module. Note that such a QUE is a quantisation
of an $\IN$--graded Lie bialgebra with \fd components
and cobracket of degree $d=0$ (cf. \ref{ss:drinf-double}).
Moreover, $\hC'=\bigoplus_{n\geqslant 0}(\hC'\cap \hC_n)$ is also graded, 
and its \emph{restricted dual} $\hC^\star=\bigoplus_{n\geqslant 0}(\hC'\cap \hC_n)^*$
is a finitely $\IN$--graded QUE quantising the restricted dual Lie bialgebra $\b^\star$.
The double cross product $(D\hC)\res=\dcp{\hC\;}{\;\;\hC^\star}$
is called the \emph{restricted} quantum double of $\hC$. $(D\hC)\res$
is a quasitriangular, finitely $\IZ$--graded QUE whose $R$--matrix
is the canonical element  in the graded completion of $\hC'\ten \hC^\star$,
and underlying Lie bialgebra is the restricted Drinfeld double $\gb
\res=\b\oplus\b^\star$.\\

\noindent\example
Let $\g$ be a symmetrisable Kac--Moody algebra. It is well--known (cf. \cite{drinfeld-86}
or \cite[13.1]{ATL1-2}) that the quantum group $\Uhg$ is isomorphic to a quotient
of the restricted quantum double of $\Uhbm{}$. This isomorphism yields the universal
$R$--matrix $\Rmx\in\Uhbm{}\wh{\ten}\Uhbp{}$ described in \ref{ss:R-matrix}, and
reduces modulo $\hbar$ to the classical isomorphism described in Proposition~\ref{ss:km-LBA}.

\subsection{Admissible \DYt modules over a QUE \cite{ATL1-1}}\label{ss:adm QUE}

If $\hC$ is a QUE, the categories of Drinfeld--Yetter $\hC$--modules and modules
over $(D\hC)\res$ are not equivalent, even when $\hC$ is of finite type. This motivates
the following definition, due to P. Etingof.

A \DYt module $(\V,\pi_\V,\pi^*_\V)$ over $\hC$ is {\it admissible} if
the coaction $\pi^*_\V:\V\to\hC\otimes\V$ factors through $\hC'\otimes\V$, where
$\otimes$ is the $\hbar$--adic tensor product, and $\hC'$ is endowed
with topology induced by the $\hbar$--adic topology on $\hC$, so that $\hC'\otimes
\V\subset\hC\otimes\V$.\footnote{Note that the induced topology on $\hC'$ coincide
its QFSH topology.} We denote the category of such modules by $\aDrY{\hC}$. If $\B$ is a
quantisation of $\b$, the category $\aDrY{\hC}$ reduces modulo $\hbar$
to $\DrY{\b}$. Moreover, we observe in \cite[Sec.~6.4]{ATL1-2} that, if
$\hC$ is a finitely $\IN$--graded QUE, there is a canonical isomorphism
between $\aDrY{\hC}$ and the category of $(D\hC)\res$--modules with
a locally finite action of $\hC^\star$.\\

\noindent\example\;
Let $\g$ be a symmetrisable Kac--Moody algebra. In analogy with Proposition~\ref{ss:from-O-to-DY},
one can identify $\O_{\infty, \Uhg}$ with a full tensor subcategory of $\hDrY{\Uhbm{}}{\adm}$
whose objects satisfy the condition \eqref{eq:DY-ss-0}. Similarly, $\O_{\infty, \Uhg}^{\sint}$
identifies with a subcategory of \emph{integrable} admissible Drinfeld--Yetter $\Uhbm{}$--modules
(cf.~\cite[Sec.~13.3]{ATL1-2}).

\subsection{Diagrammatic Hopf algebras \cite{ATL1-2}}\label{ss:diag-HA}
By analogy with Section \ref{ss:functorial-diag-alg}, a diagrammatic Hopf algebra
is a monoidal functor from $\PD$ to the category of Hopf bialgebras. Specifically,
a {\it diagrammatic} Hopf algebra is the datum of
\begin{itemize} 
	\item a diagram $\dgr$
	\item for any $B\subseteq \dgr$, a Hopf algebra $\hC_B$
	\item for any $B'\subseteq B$, a morphism of Hopf algebras
	$i_{BB'}:\hC_{B'}\to \hC_B$
\end{itemize}
such that
\begin{itemize}
	\item for any $B\subseteq \dgr$, $i_{BB}=\id_{\hC_B}$
	\item for any $B''\subseteq B'\subseteq B$, $i_{BB'}\circ i_{B'B''}=i_{BB''}$
	\item for any $B_1\sqcup B_2$, 
	\[m_{B_1\sqcup B_2}\circ i_{(B_1\sqcup B_2)B_1}\ten i_{(B_1\sqcup B_2)B_2}:\hC_{B_1}\ten \hC_{B_2}\to \hC_{B_1\sqcup B_2}\]
	is an isomorphism of Hopf algebras, where $m_{B_1\sqcup B_2}$ is the multiplication
	of $\hC_{B_1\sqcup B_2}$.
\end{itemize}

Diagrammatic QUEs are defined similarly.

\subsection{Split diagrammatic Hopf algebras \cite{ATL1-2}}\label{ss:diag-sHA}

Recall that a split pair of Hopf algebras is the datum of two Hopf algebras
$\hA,\hC$ together with Hopf algebra morphisms $\hA\xrightarrow{i} \hC\xrightarrow
{p}\hA$ such that $p\circ i=\id_\hA$ \cite[Sec. 4.6]{ATL1-1}. 
A {\it split diagrammatic} Hopf algebra is a diagrammatic Hopf algebra
$\hC=\{\hC_B\}_{B\subseteq D}$, together with Hopf algebra morphisms
$p_{B'B}:\hC_{B}\to \hC_{B'}$ for any $B'\subseteq B$, such that $p_{B'B}\circ i
_{BB'}=\id_{\hC_{B'}}$ and 
\begin{itemize}
	\item for any $B$, $p_{BB}=\id_{\b_B}$
	\item for any $B''\subseteq B'\subseteq B$, $p_{B''B'}\circ p_{B'B}=p_{B''B}$
	\item for any $B_1\perp B_2$, $p_{B_1(B_1\sqcup B_2)}\ten p_{B_2(B_1\sqcup B_2)}\circ\Delta_{B_1\sqcup B_2}:\hC_{B_1\sqcup B_2}\to \hC_{B_1}\ten \hC_{B_2}$
	is a morphism of Hopf algebras, and the inverse of $m_{B_1\sqcup B_2}\circ i_{(B_1\sqcup B_2)B_1}\ten i_{(B_1\sqcup B_2)B_2}$.
\end{itemize}

Split diagrammatic QUEs are defined similarly.\\

\noindent\remark\; Note that, if $\hC$ is a split diagrammatic  Hopf algebra, where $\hC_B$ are 
finitely $\IN$--graded QUE, there is a diagrammatic QUE $(D\hC)\res$ with $(D\hC)\res_B=(D\hC_
B)\res$, endowed  with a canonical embedding of diagrammatic Hopf algebras
$\hC\to (D\hC)\res$.\\

\noindent\example\;
Let $\g$ be a diagrammatic Kac--Moody algebra. The algebra $\Uhbm{}$ is a finitely $\IN$--graded
split diagrammatic QUE and therefore $\Uhg$, as a quotient of $(D\Uhbm{})\res$, is a finitely $\IZ$--graded diagrammatic QUE.

\subsection{Drinfeld--Yetter modules over split diagrammatic Hopf algebras}\label{ss:DYdiag}

If $\hA\leftrightarrows \hC$ is a split pair of Hopf algebras, there is a
monoidal restriction functor $\Res_{\hA,\hC}:\DrY{\hC}\to\DrY{\hA}$ given by 
\[\Res_{\hA,\hC}(\V,\pi_\V,\pi^*_\V)=
(\V,\pi_\V\circ i\otimes\id_\V,p\otimes\id_\V\circ\pi_\V^*)\]
If $\hA,\hC$ are QUEs, $\Res_{\hA,\hC}$ restricts to a functor $\aDrY{\hC}
\to\aDrY{\hA}$.

\begin{proposition}\label{pr:diagr HA}
	Let $\hC$ be a split diagrammatic Hopf algebra. Then, there is
	an $(\redasso{}{},\DCPA{}{})$--strict braided pre--Coxeter category $\DYCox{\hC}{}$
	defined by the following data
	\begin{itemize}
		\item For any $B\subseteq \dgr$, the braided monoidal
		category $\DrY{\hC_B}$.
		\item For any $B'\subseteq B$, the restriction functor 
		$\Res_{\hC_{B'},\hC_{B}}:\DrY{\hC_B}\to \DrY{\hC_{B'}}$.
	\end{itemize}
\end{proposition}

In the case of a split diagrammatic QUE, we have a braided pre--Coxeter
subcategory $\DYCox{\hC}{\adm}$ given by admissible Drinfeld--Yetter modules.

\subsection{Quantisation of diagrammatic Lie bialgebras}\label{ss:ek-quantisation}

In \cite{ek-1,ek-2}, Etingof and Kazhdan construct a quantisation functor
$\Q$ from the category of Lie bialgebras to the category of QUEs. We observe
in \cite[Prop.~6.8]{ATL1-2} that $\Q$ respects direct sums, \ie for any Lie bialgebras 
$\a,\b$, there is an isomorphism of Hopf algebras $J_{\a,\b}:\Q(\a)\ten
\Q(\b)\to\Q(\a\oplus\b)$. It follows that the quantisation of a (split) diagrammatic Lie 
bialgebra is a (split) diagrammatic QUE. Thus, for any split diagrammatic Lie bialgebra
$\b$, we have the braided pre--Coxeter category $\DYCox{\Q(\b)}{\adm}$, which
reduces modulo $\hbar$ to the category $\DYCox{\b}{}$ defined in \ref{ss:diag-sLBA}.\\

\noindent\example\; Let $\g$ be a symmetrisable Kac--Moody algebra. By \cite{ek-6},
there are isomorphisms $\Q(\bpm{})\simeq\Uhbpm{}$ and $\Q(\g)\simeq\Uhg$. 
In \cite[Prop.~13.6]{ATL1-2}, we observe that, in the case of a diagrammatic Kac--Moody algebra, 
the isomorphisms preserve the (split) diagrammatic structure.

\subsection{Universal structures arising from quantisation}\label{ss:diag-EK-equivalence}

Let $\b$ be a Lie bialgebra and $\Phi$ a Lie associator. In \cite{ek-1}, Etingof and Kazhdan
define an equivalence of braided monoidal categories $H_{\b}:\hDrY{\b}{\Phi}\to\aDrY{\Q(\b)}$,
where $\hDrY{\b}{\Phi}$ denotes the \emph{Drinfeld category}, \ie deformation Drinfeld--Yetter
$\b$--modules with associativity and commutativity constraints given by $\Phi_{\b}=
\wh{\rho}_{\b}^3(\Phi)$ and $R_B=\exp(\hbar/2\cdot\Omega_{B})$.

In \cite{ATL1-2}, this result is extended to a split diagrammatic Lie bialgebra $\b$ with underlying
diagram $\dgr$. Specifically, the following holds.

\begin{theorem}\cite[Thm.~10.2 and 10.10]{ATL1-2}
	\hfill
	\begin{enumerate}\itemsep0.25cm
		\item 
		Let $\Phi$ be a Lie associator. There is a canonical $\DCPA{}{}$--strict braided 
		pre--Coxeter structure $\pCox{\Phi}^\gstr$ on $\CDUA{\dgr}{\bullet}$ which 
		is trivial in degree zero, and is such that $\Phi_B=\xi_{B}^3(\Phi)$ for any 
		$B\subseteq \dgr$ (cf.~ Remark~\ref{ss:univ-fiber}--(2)).
		\item 
		Set 
		$\DYCox{\b}{\hdef,\Phi,\gstr}=\DYCox{\b,\pCox{\Phi}^\gstr}{\hdef}$.
		There is a canonical equivalence of braided pre--Coxeter categories 
		\[
		\mCox{\b}:\DYCox{\b}{\hdef,\Phi,\gstr}\longrightarrow\DYCox{\Q(\b)}{\adm}
		\]
		whose diagrammatic equivalences are given by the Etingof--Kazhdan functors
		$H_{\b_B}:\hDrY{\b_B}{\hdef,\Phi_B}\to\aDrY{\Q(\b_B)}$, $B\subseteq D$.
	\end{enumerate}
\end{theorem}

\noindent\remark\;
The main ingredients of the pre--Coxeter structure $\pCox{\Phi}^\gstr$ and equivalence $\mCox{\b}$ are the following.
\begin{enumerate}\itemsep0.1cm
	\item For any $B'\subseteq B$, the tensor structure $J_{B'B}^{\Phi}$ on the restriction functor
	$\Res_{B'B}:\hDrY{\b_B}{\hdef,\Phi_B}\to\hDrY{\b_{B'}}{\hdef,\Phi_{B'}}$ and the vertical join $\redasso
	{B'}{B''B}:\Res_{B''B'}\circ\Res_{B'B}\Rightarrow\Res_{B''B}$, are constructed in \cite[Thm.~1.5]
	{ATL1-1}, and determine the $\DCPA{}{}$--strict braided pre--Coxeter structure $\pCox{\Phi}^{\gstr}$.
	\item The horizontal equivalences $\hDrY{\b_B}{\hdef,\Phi_B}\to\aDrY{\Q(\b_B)}$ of braided tensor categories
	are the \nEK Tannakian equivalences $H_{\b_B}$.
	\item The diagonal isomorphism of tensor functors $\gamma_{B'B}:H_{B'}\circ\Res_{B'B}\Rightarrow\Res_{B'B}^{\hbar}\circ H_{B}$, $B'\subseteq B$,
	are constructed in \cite[Thm.~1.7]{ATL1-1}.
\end{enumerate}
Note that, by Proposition~\ref{prop:strictness-univ-cox}, we obtain an $\redasso{}{}$--strict braided Coxeter structure $\pCox{\Phi}^{\astr}$ and the corresponding category
$\DYCox{\b}{\hdef,\Phi,\astr}$, which is canonically equivalent to $\DYCox{\b}{\hdef,\Phi,\gstr}$ 
and therefore to $\DYCox{\Q(\b)}{\adm}$ via $\mCox{\b}$.

\subsection{Universality}\label{ss:univ-equivalence}

The category $\DYCox{\b}{\hdef,\Phi,\gstr}$ is {\em universal} in that its essential data
are described by the diagrammatic $\PROP$s $\MDY{\dgr}{n}$, $n\geqslant 0$.
The category $\DYCox{\Q(\b)}{\adm}$ and the equivalence $\mCox{\b}:\DYCox
{\b}{\hdef,\Phi,\gstr}\longrightarrow\DYCox{\Q(\b)}{\adm}$ are also {universal} as we
briefly explain below. For further details, we refer the reader to \cite[Sec.~6.17]
{ATL1-1} and \cite[Sec.~10.7]{ATL1-2}. 

Let $\MDY{\mathsf{QUE}}{\adm}$ be the $\PROP$ describing an admissible \DYt
module over a QUE. The category $\hDrY{\Q(\b)}{\adm}$ is isomorphic to that of
realisation functors from $\MDY{\mathsf{QUE}}{\adm}$ to $\tfV$. It follows that the
essential data defining the braided pre--Coxeter category $\DYCox{\Q(\b)}{\adm}$ is 
entirely encoded by the diagrammatic $\PROP$s $\MDY{\mathsf{QUE},\dgr}{n, \adm}$
describing $n$ admissible Drinfeld--Yetter modules over a split diagrammatic QUE. 
Therefore, the braided pre--Coxeter structure on $\DYCox{\Q(\b)}{\adm}$ is clearly universal, induced by the standard
braided pre--Coxeter structure of the \emph{quantum}
universal diagrammatic algebra $\CDUA{\dgr}{\hbar,\bullet}$ naturally associated to $\MDY{\mathsf{QUE},\dgr}{\bullet, \adm}$ (as in \ref{ss:univ-alg}).

The universality of the equivalence $\mCox{\b}:\DYCox{\b}{\hdef,\Phi,\gstr}\longrightarrow\DYCox{\Q(\b)}{\adm}$
is more subtle. Roughly, this means that every datum listed in Remark~\ref{ss:diag-EK-equivalence}
admits a suitable universal counterpart. For instance, the Etingof--Kazhdan functor $H_{\b}$ 
with its tensor structure arises as the \emph{pullback} of a morphism of topological $\PROP$s ${\mathsf{H}}:\MDY{\mathsf{QUE}}{\adm}\to\wh{\MDY{}{}}_{\mathsf{LBA}}$, where the latter 
is a graded completion of the $\PROP$ ${\MDY{\mathsf{LBA}}{}}$ describing a Drinfeld--Yetter module over a Lie bialgebra (cf.~\ref{ss:DY-prop}) and $H$ depends upon the choice of
a universal associator $\Phi$ and a universal twist $J^{\Phi}$.

The restriction functors are similarly obtained through morphisms of 
$\PROP$s involving a \emph{universal split pair}. Namely, let $\MDY{\mathsf{LBA, sp}}{}$ (resp. $\MDY{\mathsf{QUE, sp}}{\adm}$) denote the $\PROP$s 
describing a \DYt module over a split pair of Lie bialgebras
$[\a]\to[\b]$ (resp. over a split pair of QUEs $[A]\to[B]$). 
Given a split pair of Lie bialgebras $\a\to\b$, we realise the restriction functor $\hDrY{\b}{\hdef,\Phi}\to\hDrY{\a}{\hdef,\Phi}$ as a morphism of $\PROP$s $\wh{\MDY{}{}}_{\mathsf{LBA}}\to\wh{\MDY{}{}}_{\mathsf{LBA,sp}}$ 
mapping the generating objects of $\mathsf{LBA}$ to $[\a]$, depending upon the
upon the choice of a Lie associator $\Phi$ and a universal \emph{relative} twist 
$J^{\Phi}_{[\a],[\b]}$.
Finally, we prove that the natural isomorphism $\gamma$ is also universal, \ie 
it is induced by a natural isomorphism
\begin{equation*}
	\xymatrix@C=2cm{
		\MDY{\mathsf{QUE}}{\adm} \ar[r] \ar[d] & 
		\wh{\MDY{}{}}_{\mathsf{LBA}} \ar[d]
		\ar@{=>}[dl]_{\gamma_{[\a],[\b]}}
		\\
		\MDY{\mathsf{QUE, sp}}{\adm} \ar[r] & \wh{\MDY{}{}}_{\mathsf{LBA, sp}}
	}
\end{equation*}

\noindent\remark\;
Let $\MDY{\mathsf{UE_{cP}}}{\adm}$ be the $\PROP$ describing an admissible 
Drinfeld--Yetter module over a co--Poisson universal enveloping algebra, so that the 
category $\hDrY{\b}{\hdef}\simeq\hDrY{\hext{U\b}}{\adm}$ 
is equivalent to that of realisation functors from $\MDY{\mathsf{UE_{cP}}}{\adm}$ to $\tfV$.
Restricting the above constructions to $\MDY{\mathsf{UE_{cP}}}{\adm}$, we obtained in 
\cite[Sec.~6.17]{ATL1-1} an alternative proof of the invertibility of the Etingof--Kazhdan 
functor $H_{\b}$.

\section{Universal Coxeter structures on Kac--Moody algebras}\label{s:univ-Cox}

We enhance the results of Section \ref{s:univ-alg} by introducing the $\uPROP$ $\dLBA{\rootsys}$ as a
refinement of $\dLBA{\dgr}$ modelled over the set of non--negative roots of a 
\KM algebra. The corresponding universal algebra $\DUA{\rootsys}{\bullet}$ interpolates
between $\OCox{\Rmx{},\qWS{}}{\hbar,\sint}$ and $\OCox{\nabla}{\sint}$.
Specifically, we will prove in Section~\ref{s:main-thm} that it is endowed with
morphisms $\DUA{\dgr}{\bullet}\to\DUA{\rootsys}{\bullet}\leftarrow\DBLHA{\rootsys}{{\bullet}}$,
and therefore contains the data defining both categories.


\subsection{$\rootsys$--graded diagrammatic Lie bialgebras}\label{ss:root-lba}

Let $\g$ be a symmetrisable Kac--Moody algebra with Cartan subalgebra 
$\h\subset\g$, Dynkin diagram $\dgr$, and root system $\rootsys\subset\h^*$.
For any $B\subseteq\dgr$, we denote by $\Rs{B}\subseteq\rootsys$ the 
corresponding root subsystem. Recall that, for any $\alpha\in\Rs{}$ and
$B\subseteq\dgr$, we write $\alpha\perp B$ if $\supp(\alpha)\perp B$.

Let $\dLBA{\rootsys}$ be the $\uPROP$ generated by a Lie bialgebra object $[1]$,
\ie a module over $\mLBA{}$ with bracket $\mu:[2]\to[1]$ and cobracket $\delta:[1]\to[2]$,
and two sets of projectors
\vspace{0.25cm}
\begin{itemize}\itemsep0.25cm
	\item {\bf Weight projectors:} 
	a complete set of orthogonal idempotents\footnote{If $|\rootsys|=\infty$, the completeness relation
		$\dmp{0}+\sum_{\alpha\in\Rs{+}}\dmp{\alpha}=\id_{[1]}$ is imposed by considering an appropriate
		completion of the $\PROP$ $\LBA$ (cf.~\cite[Sec.~9.1]{ATL2}).} 
	\[\dmp{\alpha}:[1]\to[1]\qquad\alpha\in\Rs{+}\sqcup\{0\}\]
	\item {\bf Diagrammatic projectors:} 
	a family of idempotents 
	\[\dmp{0,B}:[1]\to[1]\qquad B\subseteq\dgr\]
\end{itemize}
such that the following relations hold.
\vspace{0.25cm}
\begin{itemize}\itemsep0.25cm
	\item {\bf Normalisation:} $\dmp{0,\dgr}=\dmp{0}$.
	\item {\bf $\Rs{}$--grading:} for any $\alpha\in\Rs{+}$,
	\begin{eqnarray*}
		\dmp{\alpha}\circ\mu&=&\sum_{\beta+\gamma=\alpha}\mu\circ\dmp{\beta}\ten\dmp{\gamma}\\
		\delta\circ\dmp{\alpha}&=&\sum_{\beta+\gamma=\alpha}\dmp{\beta}\ten\dmp{\gamma}\circ\delta
	\end{eqnarray*}
	where the sums run over all ordered pairs $(\beta,\gamma)\in\Rs{+}$ such that $\beta+\gamma=\alpha$.
	Moreover, $\dmp{0}\circ\mu=0=\mu\circ\dmp{0}\ten\dmp{0}$ and $\delta\circ\dmp{0}=0=\dmp{0}\circ\dmp{0}\circ\delta$.
	\item {\bf Nestedness:} for any $B'\subseteq B\subseteq\dgr$, 
	\begin{equation*}
		\dmp{0,B'}\circ\dmp{0,B}=\dmp{0,B'}=\dmp{0,B}\circ\dmp{0,B'}
	\end{equation*}
	and, for any $B_1\perp B_2$,
	\begin{equation*}
		\dmp{0,B_1\sqcup B_2}=\dmp{0,B_1}+\dmp{0,B_2}
	\end{equation*}
	In particular, $\dmp{0,\emptyset}=0$ and $\dmp{0,B_1}\circ\dmp{0,B_2}=0=\dmp{0,B_2}\circ\dmp{0,B_1}$ for any
	$B_1\perp B_2$.
	\item {\bf Support:} for any $\alpha\in\Rs{+}$ and $B\subseteq\dgr$,
	\begin{align*}
		\mu\circ\dmp{0,B}\ten\dmp{\alpha}=&
		\left\{
		\begin{array}{cll}
			0 & \text{if $\alpha\perp B$}\\[1.1ex]
			\mu\circ\dmp{0}\ten\dmp{\alpha} & \mbox{if }\alpha\in\Rs{B,+}
		\end{array}
		\right.
		\\[1.2ex]
		\dmp{0,B}\ten\dmp{\alpha}\circ\delta=&
		\left\{
		\begin{array}{cll}
			0 &  \text{if $\alpha\perp B$}\\[1.1ex]
			\dmp{0}\ten\dmp{\alpha}\circ\delta &  \mbox{if }\alpha\in\Rs{B,+}
		\end{array}
		\right.
	\end{align*}
\end{itemize}

\subsection{Remarks}\label{ss:root-lba-rks}

\begin{enumerate}\itemsep0.25cm
	\item 
	In \cite[Sec.~12.7]{ATL2}, we introduced a refinement of the $\PROP$ $\mathsf{LBA}$ 
	associated to a \emph{diagrammatic partial semigroup} $\sfS$ \cite[Sec.~9]{ATL2}. 
	$\dLBA{\rootsys}$ is a special case of this construction and arises when $\sfS=\Rs{+}$.
	
	\item 
	A module over $\dLBA{\rootsys}$ (in a Karoubi complete category) is a Lie bialgebra $(\c,[\cdot,\cdot], \delta)$ 
	carrying some extra structure. The weight projectors induce a decomposition $\c=\c_0\oplus\bigoplus_{\alpha\in\Rs{+}}
	\c_{\alpha}$. This is compatible with the Lie algebra structure in that, for any $\beta,\gamma\in\Rs{+}$,
	$[\c_{\beta},\c_{\gamma}]\subseteq\c_{\beta+\gamma}$, whenever $\beta+\gamma\in\Rs{+}$, and 
	$[\c_{\beta},\c_{\gamma}]=0$ otherwise. Moreover, $[\c_0,\c_{\beta}]\subseteq\c_{\beta}$ and $[\c_0,\c_0]=0$. 
	The compatibility with the Lie coalgebra structure is similar. 
	
	The diagrammatic projectors lead instead to a split diagrammatic structure on $\c$. Indeed, note that, for any 
	$B\subseteq \dgr$, the morphism
	\begin{equation}\label{eq:dmp B}
	\dmp{B}=\dmp{0,B}+\sum_{\alpha\in\Rs{B,+}}\dmp{\alpha}:[1]\to[1]
	\end{equation}
	is a Lie bialgebra idemopotent \ie $\dmp{B}^2=\dmp{B}$,
	\[
	\dmp{B}\circ\mu=\mu\circ\dmp{B}\ten\dmp{B}
	\aand
	\delta\circ\dmp{B}=\dmp{B}\ten\dmp{B}\circ\delta
	\]
	In particular, 
	$\c$ is a split diagrammatic Lie bialgebra with $\c_B=\dmp{B}(\c)$, $B\subseteq \dgr$.
	
	\item 
	If $\g$ is a diagrammatic Kac--Moody algebra, the Borel subalgebras $\bpm{}$ are
	modules over $\dLBA{\rootsys}$. Namely, for any $B\subseteq\dgr$, the idempotent $\dmp{0,B}$ 
	corresponds to the splitting $\h=\h_B\oplus\h_B^{\perp}$, while the 
	idempotents $\dmp{\alpha}$, $\alpha\in\Rs{+}\sqcup\{0\}$, arise from the root space decomposition 
	$\bpm{}=\h\oplus\bigoplus_{\alpha\in\Rs{+}}\g_{\pm\alpha}$. In particular, for any $B\subseteq\dgr$, 
	we have $\bpm{B}=\dmp{B}(\bpm{})$.
\end{enumerate}
%

\subsection{Universal Drinfeld--Yetter modules}\label{ss:root-DY}

Proceeding as in \ref{ss:DY-prop} and \ref{ss:univ-alg}, we introduce the $\uPROP$s of universal
Drinfeld--Yetter modules $\MDY{\rootsys}{n}$ and the universal algebras $\RDYUA{\rootsys}{n}$ 
associated with $\dLBA{\rootsys}$.

The category $\MDY{\rootsys}{n}$, $n\geqslant0$, is the colored $\uPROP$ generated by
$n+1$ objects, $\ACDY{1}$ and $\{\VCDY{k}\}_{k=1}^n$, and morphisms
\vspace{0.25cm}
\begin{itemize}\itemsep0.25cm
	\item $\dmp{\alpha}:\ACDY{1}\to\ACDY{1}$, $\alpha\in\Rs{+}\sqcup\{0\}$, and 
	$\dmp{0,B}:\ACDY{1}\to\ACDY{1}$, $B\subseteq\dgr$
	\item $\mu:\ACDY{2}\to\ACDY{1}$, $\delta:\ACDY{1}\to\ACDY{2}$
	\item $\pi_k:\ACDY{1}\ten \VCDY{k}\to \VCDY{k}$, $\pi_k^*:\VCDY{k}\to\ACDY{1}\ten \VCDY{k}$
\end{itemize}
such that 
\vspace{0.25cm}
\begin{itemize}\itemsep0.25cm
	\item $(\ACDY{1}, \dmp{\alpha},\dmp{0,B},\mu,\delta)$ is an $\dLBA{\rootsys}$--module 
	in $\MDY{\rootsys}{n}$
	\item every $(\VCDY{k},\pi_k,\pi_k^*)$ is a \DYt module over $\ACDY{1}$
\end{itemize}
In particular, $\MDY{\rootsys}{0}=\dLBA{\rootsys}$.

Similarly to \ref{ss:grading}, we consider on $\MDY{\rootsys}{n}$ the $\IN$--grading given by
$\deg(\sigma)=0$ for any $\sigma\in\SS_N$, $\deg(\mu)=0=\deg(\pi_{\VDY{k}})$ and 
$\deg(\delta)=1=\deg(\pi_{\VDY{k}}^*)$ for any $1\leqslant k\leqslant n$, and
finally $\deg(\dmp{\alpha})=0=\deg(\dmp{0,B})$, for any $\alpha\in\Rs{+}$ and $B\subseteq\dgr$.
This yields the universal algebra 
\[\RDYUA{\rootsys}{n}=\pEnd{\MDY{\rootsys}{n}}{\VCDY{1}\ten\VCDY{2}\ten\cdots\ten\VCDY{n}}\]
and its completion $\CRDYUA{\rootsys}{n}$.

\subsection{The universal algebra $\DUA{\rootsys}{\bullet}$}\label{ss:root-univ-alg}

The algebra $\RDYUA{\rootsys}{n}$ has a canonical diagrammatic structure, arising
from the projectors $\{\dmp{B}\}_{B\subseteq\dgr}$ \eqref{eq:dmp B}. 
Namely, for any $B\subseteq\dgr$, we set $\MDY{\rootsys, B}{n}=\MDY{\Rs{B}}{n}$ 
and $\RDYUA{\rootsys, B}{n}=\RDYUA{\Rs{B}}{n}$.  For any $B\subseteq B'$, there is 
a canonical realisation functor
\[\G_{\dmp{B}[1],\VCDY{1},\dots,\VCDY{n}}:\RDY{\rootsys, B}{n}\to\RDY{\rootsys, B'}{n}\]
which sends the object $[1]_{B}$ in $\RDY{\rootsys,B}{n}$ to the Lie bialgebra object
$\dmp{B}[1]_{B'}=([1]_{B'},\dmp{B})$ in $\RDY{\rootsys, B'}{n}$. This induces a 
morphism of algebras $\sfi_{\rootsys, B'B}^n:\RDYUA{\rootsys,B}{n}\to\RDYUA{\rootsys,B'}{n}$.

The following is an analogue of Proposition \ref{ss:univ-alg} (cf.~\cite[Prop.~12.4]{ATL2}).
\begin{proposition}
	\hfill
	\begin{enumerate}\itemsep0.25cm
		\item For any $n\geqslant 0$, the algebras $\{\RDYUA{\rootsys, B}{n}\}_{B\subseteq\dgr}$
		and morphisms $\{\sfi_{\rootsys, BB'}^n\}_{B'\subseteq B\subseteq\dgr}$ give rise to a diagrammatic 
		algebra $\DUA{\rootsys}{n}$.
		\item The invariant subalgebras $\{\RDYUA{\rootsys, BB'}{n}\subset
		\RDYUA{\rootsys, B}{n}\;|\;B'\subseteq B\}$ yield a bidiagrammatic structure on $\DUA{\rootsys}{n}$.
		\item For any $B\subseteq\dgr$, there is a canonical cosimplicial structure on the tower of algebras $\{\RDYUA{\rootsys, B}{n}\}_{n\geqslant 0}$, 
		which is defined as in \ref{ss:univ-cosimp}, is compatible with the morphisms $\sfi_{BB'}^n$ 
		and preserves the invariant subalgebras, yielding a cosimplicial bidiagrammatic 
		structure $\DUA{\rootsys}{\bullet}$.
	\end{enumerate}
\end{proposition}

\noindent\remark\; 
The morphisms $\sfi_{\rootsys, BB'}^n$ and the cosimplicial structure
are compatible with grading, thus yielding a cosimplicial lax bidiagrammatic
algebra $\CDUA{\rootsys}{\bullet}\supset\DUA{\rootsys}{\bullet}$ given by the collection 
of the invariant subalgebras $\CRDYUA{\rootsys, BB'}{n}\subseteq\CRDYUA{\rootsys, B}{n}$,
$B'\subseteq B$.

\subsection{From $\DUA{\dgr}{\bullet}$ to $\DUA{\rootsys}{\bullet}$}\label{ss:from-dgr-to-rootsys}
As pointed out in \ref{ss:root-lba-rks} (2), the generating object in $\dLBA{\rootsys}$ is a 
split diagrammatic Lie bialgebra, with diagrammatic structure given by the projectors
$\{\dmp{B}\}_{B\subseteq\dgr}$ \eqref{eq:dmp B}.
This yields canonical realisation functors 	
\[
\dLBA{\dgr}\to\dLBA{\rootsys}
\aand
\MDY{\dgr}{n}\to\MDY{\rootsys}{n}\quad (n\geqslant0)
\]
and morphisms of algebras $\iota_{\rootsys}^n:\RDYUA{\dgr}{n}\to\RDYUA{\rootsys}{n}$,
$n\geqslant0$. One readily checks that these 
preserve the diagrammatic subalgebras, the invariant subalgebras, the cosimplicial structure,
and the grading, thus giving rise to the morphisms of cosimplicial (lax) bidiagrammatic algebras
$\DUA{\dgr}{\bullet}\to\DUA{\rootsys}{\bullet}$ and $\CDUA{\dgr}{\bullet}\to\CDUA{\rootsys}
{\bullet}$.

\subsection{Universal pre--Coxeter structures for Kac--Moody algebras}\label{ss:univ-pre-cox-km}

Let $\g$ be a diagrammatic Kac--Moody algebra with root system $\Rs{}$ and 
Borel subalgebras $\bpm{}\subseteq\g$ and $\UCoxOint{\g}{\bullet}$ the  cosimplicial lax
bidiagrammatic algebra arising from deformation category $\O_{\infty}$ integrable $\g$--modules
defined in \ref{ss:def-catO}.

By \ref{ss:root-lba-rks} (3), the 
Lie bialgebras $\bpm{}$ are modules over $\dLBA{\rootsys}$. Therefore, for any
$n$--tuple $\{V_k,\pi_k,\pi_k^*\}_{k=1}^n$ of \DYt $\bpm{}$--modules, there is a 
canonical realisation functor 
\[\G_{(\bpm{},V_1,\dots, V_n)}:\MDY{\rootsys}{n}\longrightarrow{\vect{}}\]
sending $[1]\mapsto\bpm{}$, and $\VDY{k}\mapsto V_k$. 

Let $\hDrY{\bpm{}}{\hdef,\sint}$ be the category of deformation
integrable Drinfeld--Yetter $\bpm{}$--modules as defined in \ref{ss:Cox-DY}.
Let $\DYA{\bpm{}}{n}$ be the algebra of endomorphisms of the forgetful functor $(\hDrY{\bpm{}}{\hdef, \sint})^{\boxtimes n}\to\tfV$ and
$\UDYint{\bpm{},\h}{n}\subseteq\UDYint{\bpm{}}{n}$ the subalgebra of 
$\h$--invariant (or \emph{weight--zero}) elements. Proceeding as in \ref{ss:univ-fiber}, 
we obtain a canonical morphism of algebras $\rho_{\bpm{}}^n:\CRDYUA{\rootsys}{n}\to\UDYint{\bpm{}}{n}$ 
induced by the realisation functors $\G_{(\bpm{},V_1,\dots, V_n)}$. We observed in
\cite[Remark~15.12]{ATL2} that the morphism $\rho_{\bpm{}}^n$ factors through 
the weight--zero subalgebra $\UDYint{\bpm{},\h}{n}$.

Let $\UCoxDYint{\bpm{}}{\bullet}$ be the  cosimplicial  lax bidiagrammatic algebra 
corresponding to $\UDYint{\bpm{}}{n}$ (cf.~\ref{ss:def-Cox-DY}). We have the following
analogue of Proposition~\ref{ss:univ-fiber}.

\begin{proposition}
	Let $\g$ be a diagrammatic Kac--Moody algebra with root system $\Rs{}$ and Borel 
	subalgebras $\bpm{}\subseteq\g$. 
	\vspace{0.25cm}\begin{enumerate}\itemsep0.25cm
		\item
		The realisation functors induce a canonical morphism of  cosimplicial lax bidiagrammatic 
		algebras ${\rho}_{\bpm{}}^{\bullet}:\CDUA{\rootsys}{\bullet}\to\UCoxOint{\bpm{}}{\bullet}$.
		\item 
		Every braided pre--Coxeter structure $\pCox{}=(\Phi_{B}, R_{B},\Jg{}{\F}{}, \DCPA{\F}{\G}, 
		\redasso{\F}{\F'})$ on $\CDUA{\rootsys}{\bullet}$ yields the following.
		\vspace{0.25cm}\begin{enumerate}\itemsep0.25cm
			\item 
			A weight--zero braided pre--Coxeter structure $\pCox{\bpm{}}$ on $\UCoxOint{\bpm{}}{\bullet}$ through the morphism ${\rho}_{\bpm{}}^{\bullet}:\CDUA{\rootsys}{\bullet}\to\UCoxDYint{\bpm{}}{\bullet}$.
			\item 
			A braided pre--Coxeter category $\DYCox{\pCox{}}{\pm}$ on deformation 
			integrable Drinfeld--Yetter $\bpm{}$--modules, defined by $\pCox{\bpm{}}$ through 
			Proposition~\ref{ss:def-Cox-DY} (1).
			\item 
			A braided pre--Coxeter category $\OCox{\pCox{}}{}$ on deformation 
			integrable category $\O_{\infty}$ $\g$--modules, defined by $\pCox{\bm{}}$ through 
			Proposition~\ref{ss:def-Cox-DY} (2).
		\end{enumerate}
	\end{enumerate}
\end{proposition}

We say that a braided pre--Coxeter structure on $\UCoxOint{\bpm{}}{\bullet}$ is \emph{universal} if it
is lifted from one on $\CDUA{\rootsys}{\bullet}$ as in (2)--(a) above.\\

\noindent\remark\; 
Note that elements in $\DUA{\rootsys}{\bullet}$ act on \emph{any} Drinfeld--Yetter $\bpm{}$--modules (in particular, category $\O_{\infty}$ $\g$--modules) 
without any requirement of integrability. Therefore, the categories from (b) and (c) above
can be similarly defined without the requirement of integrability.

\subsection{Universal Coxeter structures for Kac--Moody algebras \cite[Def.~15.12]{ATL2}}\label{ss:univ-cox-km}

A braided Coxeter structure $\sCox{}=(\Phi_{B}, R_{B}, \Jg{}{\F}{}, \DCPA{\F}{\G},\redasso{\F}{\F'}, S_i)$ of type $(\dgr,\ulm)$ on $\UCoxDYint{\bm{}}{\bullet}$ (or equivalently on $\UCoxOint{\g}{\bullet}$) is \emph{universal} if 
\vspace{0.25cm}\begin{enumerate}\itemsep0.25cm 
	\item $\sCox{}$ is {\em supported} on $\CDUA{\rootsys}{\bullet}$, \ie the underlying 
	braided pre--Coxeter structure $\sCox{}^{\scsop{pre}}=
	(\Phi_{B}, R_{B}, \Jg{}{\F}{}, \DCPA{\F}{\G}, 
	\redasso{\F}{\F'})$ arises from a braided pre-Cox structure on $\CDUA{\rootsys}{\bullet}$
	via Proposition~\ref{ss:univ-pre-cox-km}
	\item the local monodromies $S_i$ have the form
	\begin{equation}
		S_i=\texp{i}\cdot\ul{S}_i
	\end{equation}
	where $\texp{i}=\exp(e_i)\cdot\exp(-f_i)\cdot\exp(e_i)$
	and
	$\ul{S}_i\in\hext{U\sl{2}^{\alpha_i}}$
	is $\h$--invariant with $\ul{S}_i=1\mod\hbar$.
\end{enumerate}


\noindent\remark\;  Note that, by \ref{ss:Cox-DY}, $\texp{i}$ and $\ul{S}_i$ act on integrable Drinfeld--Yetter $\b^- _{i}$--modules.


\section{Proof of the monodromy theorem}\label{s:main-thm}

\subsection{} 
\label{ss:mon-thm}

The following is the main result of this paper.

\begin{theorem}
Let $\g$ be a diagrammatic Kac--Moody algebra with negative Borel subalgebra
$\b^-$.
\begin{enumerate}\itemsep0.25cm
\item
The monodromy data of the joint KZ--Casimir connection gives rise to a braided
Coxeter category $\DYCox{\bm{},\nabla}{\hdef,\sint}$ on deformation integrable
Drinfeld--Yetter modules over $\bm{}$, which extends the braided Coxeter category
$\OCox{\g,\nabla}{\scsop{\hdef, \sint}}$ given by Theorem \ref{thm:hol-to-KM}.
\item
The $R$--matrix and quantum Weyl group operators of $\Uhg$ give rise to a
braided Coxeter category $\DYCox{\Uhbm{},\Rmx,\qWS{}}{\adm,\sint}$ on integrable
admissible \DYt modules over $\Uhbm{}$, which extends the braided Coxeter
category $\OCox{\Uhg,\Rmx,\qWS{}}{\sint}$ given by Proposition \ref{pr:ORS}.
\item There is a canonical equivalence of braided Coxeter categories 
\[\mCox{\bm{}}:\DYCox{\bm{},\nabla}{\hdef,\sint}
\to\DYCox{\Uhbm{},\Rmx,\qWS{}}{\adm,\sint}\]
which preserves category $\O_{\infty}$  modules, and
restricts to an equivalence of braided Coxeter categories $\mCox{\g}:
\OCox{\g,\nabla}{\scsop{\hdef,\sint}}\to\OCox{\Uhg,\Rmx,\qWS{}}{\sint}$.
\item The equivalence $\mCox{\bm{}}$ is obtained as follows.
\vskip 0.25cm
\begin{enumerate}
\itemsep 0.25cm
\item The structure $\DYCox{\bm{},\nabla}{\hdef,\sint}$ is universal, that is arises from a canonical
braided pre--Coxeter structure $\sCox{\nabla}$ on the universal root diagrammatic algebra $\CDUA
{\Rs{}}{\bullet}$ introduced in \ref{ss:root-univ-alg}, via the realisation morphism associated to $\b^-$.
\item 
The structure $\DYCox{\Uhbm{},\Rmx,\qWS{}}{\adm,\sint}$ is universal, that is arises from the standard braided pre--Coxeter structure $\sCox{\Rmx,\qWS{}}^{\hbar}$ on the quantum universal
diagrammatic algebra $\CDUA{\dgr}{\hbar,\bullet}$ introduced in 
\ref{ss:univ-equivalence}.

\item
There is a canonical braided pre--Coxeter structure $\sCox{\Rmx,\qWS{}}$ on the universal diagrammatic
algebra $\CDUA{\dgr}{\bullet}$ introduced in 
\ref{ss:univ-alg}, together with a canonical universal equivalence
\[\mCox{\bm{}}':\DYCox{\bm{},\sCox{\Rmx,\qWS{}}}{\hdef,\sint}\to
\DYCox{\Uhbm{},\sCox{\Rmx,\qWS{}}^{\hbar}}{\adm,\sint}\]
\item The braided pre--Coxeter structures $\sCox{\nabla}$ and $\sCox{\Rmx,\qWS{}}$ are related by a
unique twist, which yields an equivalence
\[\mCox{\bm{}}'':
\DYCox{\bm{},\sCox{\nabla}}{\hdef,\sint} 
 \to
 \DYCox{\bm{},\sCox{\Rmx,\qWS{}}}{\hdef,\sint}\]
\item The equivalence $\mCox{\bm{}}$ is given by the composition
\[\xymatrix@C=1.5cm{
\DYCox{\bm{},\nabla}{\hdef,\sint}\ar[rr]^{\mCox{\bm{}}} && \DYCox{\Uhbm{},\Rmx,\qWS{}}{\adm,\sint} \\
 \DYCox{\bm{},\sCox{\nabla}}{\hdef,\sint} \ar@{=}[u] \ar[dr]_{\mCox{\bm{}}''} && \DYCox{\Uhbm{},\sCox{\Rmx,\qWS{}}^{\hbar}}{\adm,\sint}\ar@{=}[u] \\
 &\DYCox{\bm{},\sCox{\Rmx,\qWS{}}}{\hdef,\sint} \ar[ur]_{\mCox{\bm{}}'}&
}\]
where the vertical equalities follow, respectively, from (a) and (b).

\end{enumerate}
\end{enumerate}
\end{theorem}

In particular, we obtain the following.

\begin{theorem}
Let $V$ be an integrable category $\O_{\infty}$ $\g$--module, and $\V\in\hOinfint$ a quantum deformation
of $V$. Then, the $W$--equivariant monodromy of the Casimir connection on $\hext{V}$ is equivalent to the
quantum Weyl group action of the braid group $\Br{W}$ on $\V$.
\end{theorem}

\subsection{Remark} 

As explained in \ref{ss:qG-to-km}, $\mCox{\bm{}}$ (resp. $\mCox{\g}$) hold more generally as equivalences
of \emph{pre}--Coxeter structures for arbitrary \DYt $\bm{}$--modules (resp. category $\O_{\infty}$ $\g$--modules)
without any requirement on integrability.

The proof of Theorem~\ref{ss:mon-thm} is carried out in the rest of this section. 
In Sections~\ref{ss:holo-univ-1}--\ref{ss:from-holo-to-prop}, we prove that the double holonomy algebra
$\DBLHAH{\nabla}{\bullet}$ maps to the universal algebra $\CDUA{\rootsys}{\bullet}$. Then, 
(1) and (4a) are proved in Section \ref{ss:nabla-universality}; 
(2) and (4b) are proved in Section \ref{ss:DY-qG};
(4c) and (4d) are proved in Sections~
 \ref{ss:qG-to-km}, and \ref{ss:twist-equiv}, 
respectively. Thus, (4e) and the first statement in (3) follow. 
Finally, the second statement in (3) is proved in \ref{ss:main-thm-cat-O}.

\subsection{From $\DBLHAH{\rootsys}{\bullet}$ to $\CDUA{\rootsys}{\bullet}$}\label{ss:holo-univ-1}

In Section~\ref{ss:dblh-U}, we constructed a morphism of cosimplicial lax diagrammatic
algebras $\xi_{\rootsys}^{\scs{\bullet}}:\DBLHAH{\rootsys}{\bullet}\to\UCoxOint{\g}{\bullet}$
and used it to define a braided Coxeter structure on $\UCoxOint{\g}{\bullet}$ encoding 
the monodromy data of the joint KZ--Casimir connection.
We prove in Proposition~\ref{ss:from-holo-to-prop} that $\xi_{\rootsys}^{\scs{\bullet}}$ 
factors through the universal algebra $\CDUA{\rootsys}{\bullet}$ introduced in \ref{ss:root-DY}, \ie there is a canonical
morphism $\eta_{\rootsys}^{\scs{\bullet}}:\DBLHAH{\rootsys}{\bullet}\to\CDUA{\rootsys}
{\bullet}$ which fits in a commutative diagram
	\begin{equation}
		\begin{tikzcd}
			\DBLHAH{\rootsys}{\bullet}
			\arrow[r, "\xi^{\scs{\bullet}}_{\rootsys}"]
			\arrow[d, "\eta_{\rootsys}^{\scs{\bullet}}"']
			&
			\UCoxOint{\g}{\bullet}\\
			\CDUA{\rootsys}{\bullet}
			\arrow[r, "\rho^{\scs{\bullet}}_{\bm{}}"']
			&
			\UCoxDYint{\bm{}}{\bullet}
			\arrow[u, "\varphi^{\scs{\bullet}}_{\g}"']
		\end{tikzcd}
	\end{equation}
where $\rho^{\scs{\bullet}}_{\bm{}}$ is the realisation morphism from \ref{ss:univ-pre-cox-km},
and $\varphi^{\scs{\bullet}}_{\g}$ is given by restriction from Drinfeld--Yetter $\bm{}$--modules
to category $\O_{\infty}$ $\g$--modules, as described in \ref{ss:def-Cox-DY}.

\subsection{Arc diagrams in $\MDY{\rootsys}{\bullet}$}\label{ss:holo-univ-2}

The elements in $\DUA{\rootsys}{\bullet}$ may conveniently be represented
in terms of string and arc diagrams, which we read as morphisms from left to
right. In $\MDY{\rootsys}{n}$, we represent $\id_{[1]}$ with a line and each 
$\id_{\VDY{i}}$ with a bold line. The bracket $\mu:[2]\to[1]$ and the cobracket
$\delta:[1]\to[2]$ are represented, respectively, by the diagrams
\begin{align}
	\begin{tikzpicture}[scale=0.75]
		\node at (0,0) {
			\begin{tikzpicture}[scale=0.75]
				\node (S1) at (-1,1) {};
				\node (S2) at (-1,-1) {};
				\node (T) at (1,0){};
				\draw (S1)--(0,0)--(T);
				\draw (S2)--(0,0);
			\end{tikzpicture}
		};
		\node at (3,0) {and};
		\node at (6,0) {
		\begin{tikzpicture}[scale=0.75]
			\node (S1) at (1,1) {};
			\node (S2) at (1,-1) {};
			\node (T) at (-1,0){};
			\draw (S1)--(0,0)--(T);
			\draw (S2)--(0,0);
		\end{tikzpicture}
		};
	\end{tikzpicture}
\end{align}
Set $\VDY{}=\VDY{1}\ten\cdots\ten\VDY{n}$. The action $\pi_{\VDY{i}}:[1]\ten\VDY{}\to\VDY{}$
and the coaction $\pi^*_{\VDY{i}}:\VDY{}\to[1]\ten\VDY{}$ on the $i$th component of $\VDY{}$
are represented, respectively, by the diagrams
\begin{align}
	\begin{tikzpicture}[scale=0.75]
		\node at (0,-0.1) {
			\begin{tikzpicture}[scale=0.75]
				\node at (-0.5,0) {$\VDY{i}$};
				\draw[very thick] (0,0)--(2,0);
				\node[circle] (L) at (0,1) {};
				\node at (1,0.5) {$\vdots$};
				\node at (1,-0.5) {$\vdots$};
				\node (S) at (0.25,0) {};
				\node (C) at (1,1) {};
				\node (T) at (1.75,0) {};
				\draw plot[smooth,  tension=1] coordinates {(L) (C) (T)};
			\end{tikzpicture}
		};
		\node at (3,-0.25) {and};
		\node at (6,-0.1) {
			\begin{tikzpicture}[scale=0.75]
				\node at (-0.5,0) {$\VDY{i}$};
				\draw[very thick] (0,0)--(2,0);
				\node[circle] (L) at (2,1) {};
				\node at (1,0.5) {$\vdots$};
				\node at (1,-0.5) {$\vdots$};
				\node (S) at (0.25,0) {};
				\node (C) at (1,1) {};
				\node (T) at (1.75,0) {};
				\draw plot[smooth,  tension=1] coordinates {(S) (C) (L)};
			\end{tikzpicture}
		};
	\end{tikzpicture}
\end{align}
Finally, the idempotents $\dmp{\star}:[1]\to[1]$, where the label 
$\star$ is either $\alpha\in\Rs{+}$ or  $(0,B)$ with $B\subseteq\dgr$, 
are represented by the diagram
\vspace{0.25cm}
\begin{align}
	\begin{tikzpicture}[scale=0.75]
	\node[rounded corners, draw=black] (L) at (0,0) {$\star$};
	\draw (-1,0)--(L)--(1,0);
	\end{tikzpicture}
\end{align}

\subsection{Relations in $\DUA{\rootsys}{\bullet}$}\label{ss:holo-univ-3}

As in \ref{ss:distinguished-elements}, there are two distinguished families of
elements in $\DUA{\rootsys}{n}$, namely
\begin{equation}
\Kp{i}{\star}=\pi_{\VDY{i}}\circ\dmp{\star}\ten\id_{\ten\VDY{}}\circ \pi^*_{\VDY{i}}
\aand
\rp{ij}{\star}=\pi_{\VDY{i}}\circ\dmp{\star}\ten\id_{\ten\VDY{}}\circ \pi^*_{\VDY{j}}
\end{equation}
where $1\leqslant i\neq j\leqslant n$, and $\star$ is either $\alpha\in\Rs{+}$ or $(0,B)$
with $B\subseteq\dgr$. These correspond, respectively, to the diagrams
\begin{align}
	\begin{tikzpicture}[scale=0.75]
		\node at (0,-0.15) {
			\begin{tikzpicture}[scale=0.75]
				\node at (-0.5,0) {$\VDY{i}$};
				\draw[very thick] (0,0)--(2,0);
				\node[rounded corners, draw=black] (L) at (1,1) {};
				\node at (1,0.5) {$\vdots$};
				\node at (1,-0.5) {$\vdots$};
				\node (S) at (0.25,0) {};
				\node (T) at (1.75,0) {};
				\draw plot[smooth,  tension=2] coordinates {(S) (L) (T)};
				\node[rounded corners, draw=black, fill=white] at (L) {$\star$};
			\end{tikzpicture}
		};
		\node at (3,-0.25) {and};
		\node at (6,0.15) {
			\begin{tikzpicture}[scale=0.75]
				\node at (-0.5,0) {$\VDY{i}$};
				\node at (-0.5,-1) {$\VDY{j}$};
				\draw[very thick] (0,0)--(2,0);
				\draw[very thick] (0,-1)--(2,-1);
				\node[rounded corners, draw=black] (L) at (1,0.75) {};
				\node (S) at (0.25,-1) {};
				\node (T) at (1.75,0) {};
				\draw plot[smooth,  tension=2] coordinates {(S) (L) (T)};
				\node[rounded corners, draw=black, fill=white] at (L) {$\star$};
				\node at (1,-0.4) {$\vdots$};
			\end{tikzpicture}
		};
	\end{tikzpicture}
\end{align}
Similarly to \ref{ss:distinguished-elements}, it follows from the definition of ${\rho}^{\bullet}
_{\bm{}}$ in \ref{ss:univ-pre-cox-km} and Proposition \ref{pr:double g} (4) that
\begin{xalignat*}{5}
\varphi^{n}_{\g}\circ{\rho}^{n}_{\bm{}}(\Kp{i}{\alpha})&=\hbar\cdot\Ku{\alpha}{+,i}
&&&
\varphi^{n}_{\g}\circ{\rho}^{n}_{\bm{}}(\rp{ij}{\alpha})&=\hbar\cdot r_{\alpha}^{ij}\\
\varphi^{n}_{\g}\circ{\rho}^{n}_{\bm{}}(\Kp{i}{0,B})&=\half{\hbar}\sum_k (t_k)^{(i)}\cdot (t^k)^{(i)}
&&&
\varphi^{n}_{\g}\circ{\rho}^{n}_{\bm{}}(\rp{ij}{0,B})&=\half{\hbar}\sum_k (t_k)^{(i)}\cdot (t^k)^{(j)}
\end{xalignat*}
where $\{t_k\},\{t^k\}$ are dual bases of $\h_B$.


\begin{lemma}\label{lem:comm-rel}
\vspace{0.25cm}
The following holds.
\begin{enumerate}\itemsep0.25cm
		\item For any $B\subseteq\dgr$ and $\alpha\in\Rs{+}$, $\left[\Kp{i}{0,B},\Kp{i}{\alpha}\right]=0$.
		\item For any $B\subseteq\dgr$ and $\alpha\in\Rs{B,+}$,
		$\left[\Kp{i}{\alpha}, \sum_{\beta\in\Rs{B,+}}\Kp{i}{\beta}\right]=0$.
	\end{enumerate}
\end{lemma}

\begin{pf}
(1) follows from the identities
\begin{align}
	\begin{tikzpicture}[scale=0.75]
	\node at (1,-0.4) {
		\begin{tikzpicture}[scale=0.75]
			\draw[very thick] (0,0)--(4,0);
			\node[rounded corners, draw=black] (L) at (1,0.75) {};
			\node (S) at (0.25,0) {};
			\node (T) at (1.75,0) {};
			\draw plot[smooth,  tension=2] coordinates {(S) (L) (T)};
			\node[rounded corners, draw=black, fill=white] at (L) {$0$};
			\node[rounded corners, draw=black] (L2) at (3,0.75) {};
			\node (S2) at (2.25,0) {};
			\node (T2) at (3.75,0) {};
			\draw plot[smooth,  tension=2] coordinates {(S2) (L2) (T2)};
			\node[rounded corners, draw=black, fill=white] at (L2) {$\alpha$};
		\end{tikzpicture}
	};
	\node at (3.5,-0.25) {=};
	\node at (5,-0.025) {
		\begin{tikzpicture}[scale=0.75]
			\draw[very thick] (0,0)--(2,0);
			\node[rounded corners, draw=black] (L1) at (1,0.75) {};
			\node (S1) at (0.5,0) {};
			\node (T1) at (1.5,0) {};
			\draw plot[smooth,  tension=2] coordinates {(S1) (L1) (T1)};
			\node[rounded corners, draw=black, fill=white] at (L1) {$\alpha$};
			\node[rounded corners, draw=black] (L2) at (1,1.5) {};
			\node (S2) at (0.25,0) {};
			\node (T2) at (1.75,0) {};
			\draw plot[smooth,  tension=2] coordinates {(S2) (L2) (T2)};
			\node[rounded corners, draw=black, fill=white] at (L2) {$0$};
		\end{tikzpicture}
	};
	\node at (6.5,-0.25) {=};
	\node at (9,-0.4) {
		\begin{tikzpicture}[scale=0.75]
			\draw[very thick] (0,0)--(4,0);
			\node[rounded corners, draw=black] (L) at (1,0.75) {};
			\node (S) at (0.25,0) {};
			\node (T) at (1.75,0) {};
			\draw plot[smooth,  tension=2] coordinates {(S) (L) (T)};
			\node[rounded corners, draw=black, fill=white] at (L) {$\alpha$};
			\node[rounded corners, draw=black] (L2) at (3,0.75) {};
			\node (S2) at (2.25,0) {};
			\node (T2) at (3.75,0) {};
			\draw plot[smooth,  tension=2] coordinates {(S2) (L2) (T2)};
			\node[rounded corners, draw=black, fill=white] at (L2) {$0$};
		\end{tikzpicture}
	};
	\end{tikzpicture}
\end{align}
(2) 
Let $\iota_{\rootsys}^n\colon\DUA{\dgr}{n}\to\DUA{\rootsys}{n}$ be the 
morphism defined in \ref {ss:root-univ-alg}. Then, for any $B\subseteq\dgr$, one has 
\[\iota_{\rootsys}^n(\Kp{i}{B})=\Kp{i}{0,B}+\sum_{\beta\in\Rs{B,+}}\Kp{i}{\beta}\]
where 
$\Kp{i}{B}\in\RDYUA{B}{n}$ is defined in \ref{ss:distinguished-elements}.
In \cite[Prop.~9.8]{ATL2}, we proved that $\sum_i\Kp{i}{B}$ is central in $\RDYUA{B}{n}$. The same proof applies to $\iota_{\rootsys}^n(\Kp{i}{B})$ in $\RDYUA{\rootsys, B}{n}$. 
The result then follows from (1).
\end{pf}

Clearly, the identity (2) above can be regarded as a $tt$--relation \eqref{eq:K-tt} with respect
to a diagrammatic root subsystem. Proceeding along the same lines, one shows the standard 
$tt$--relations hold in $\RDYUA{\rootsys}{n}$.

\begin{proposition}\label{prop:prop-tt-rel}
For any rank $2$ subsystem $\Psi\subset\Rs{+}$ and $\alpha\in\Psi$,
$\left[\Kp{i}{\alpha}, \sum_{\beta\in\Psi}\Kp{i}{\beta}\right]=0$.
\end{proposition}

\subsection{The morphism $\eta_{\rootsys}^{\scs{\bullet}}:\DBLHAH{\rootsys}{\bullet}\to\CDUA{\rootsys}{\scs{\bullet}}$}
\label{ss:from-holo-to-prop}
For any $n\geqslant 2$ and $1\leq i\neq j\leq n$, define $\Tp{ij}{\star}\in\RDYUA{\rootsys}{n}$ by
$\Tp{ij}{\star}=\rp{ij}{\star}+\rp{ji}{\star}$.
\begin{proposition}\label{pr:holo-Casimir}
The assignments
	\begin{align*}
		\eta_{\rootsys}^{n}(\Trdh{ij}{+\alpha})&=\frac{1}{2\pi\iota}\rp{ij}{\alpha} 
		\qquad\qquad
		\eta_{\rootsys}^{n}(\Kdh{\alpha}{i})=\frac{1}{2\pi\iota}\Kp{i}{\alpha}
		\qquad\qquad
		\eta_{\rootsys}^{n}(\Tdh{ij}{0,B})=\frac{1}{2\pi\iota}\Tp{ij}{0,B}\\
		\eta_{\rootsys}^{n}(\Trdh{ij}{-\alpha})&=\frac{1}{2\pi\iota}\rp{ji}{\alpha}
		\qquad\qquad
		\eta_{\rootsys}^{n}(\Kdh{\alpha}{(n)})=\frac{1}{2\pi\iota}\Delta^{(n)}(\Kp{}{\alpha})
	\end{align*}
uniquely extends to a morphisms of algebras $\eta_{\rootsys}^{n}:\DBLHA{\rootsys}{n}
\to\RDYUA{\rootsys}{n}$ compatible with the cosimplicial structure, the diagrammatic
	structure, and the natural $\IN$--gradings.

	The corresponding morphism of cosimplicial lax diagrammatic algebras 
	$\eta_{\rootsys}^{\scs{\bullet}}:\DBLHAH{\rootsys}{\bullet}\to\CDUA{\rootsys}{\bullet}$
	give rise to the commutative diagram
		\begin{equation}\label{eq:holo-prop-rep}
		\begin{tikzcd}
			\DBLHAH{\rootsys}{\bullet}
			\arrow[r, "\xi^{\scs{\bullet}}_{\rootsys}"]
			\arrow[d, "\eta_{\rootsys}^{\scs{\bullet}}"']
			&
			\UCoxOint{\g}{\bullet}\\
			\CDUA{\rootsys}{\bullet}
			\arrow[r, "\rho^{\scs{\bullet}}_{\bm{}}"']
			&
			\UCoxDYint{\bm{}}{\bullet}
			\arrow[u, "\varphi^{\scs{\bullet}}_{\g}"']
		\end{tikzcd}
	\end{equation}
	where $\xi_{\rootsys}^{\scs{\bullet}}$ and $\rho^{\scs{\bullet}}_{\bm{}}$ are the realisation morphism 
	from \ref{ss:dblh-U} and \ref{ss:real-funct-end}, respectively, and $\varphi^{\scs{\bullet}}_{\g}$ is given 
	by the restriction from integrable Drinfeld--Yetter $\bm{}$--modules to integrable category $\O_{\infty}$ 
	$\g$--modules, described in \ref{ss:def-Cox-DY}.
\end{proposition}

\noindent\remark\;
It is clear that, at this stage, it is not necessary to work with integrable modules. Namely, let 
$\UCoxO{\g}{\bullet}$ and $\UCoxDY{\bm{}}{\bullet}$ be, respectively, the completions with 
respect to deformation category $\O_{\infty}$ $\g$--modules and Drinfeld--Yetter $\bm{}$--modules (cf.~\ref{ss:def-catO}
and \ref{ss:def-Cox-DY}). Note that there are canonical maps 
$\UCoxO{\g}{\bullet}\to\UCoxOint{\g}{\bullet}$ and $\UCoxDY{\bm{}}{\bullet}\to\UCoxDYint{\bm{}}{\bullet}$,
given by restriction to integrable modules.
One readily checks that the maps $\xi^{\scs{\bullet}}_{\rootsys}$, $\rho^{\scs{\bullet}}_{\bm{}}$,  $\varphi^{\scs{\bullet}}_{\g}$
factor through $\UCoxDY{\bm{}}{\bullet}$ and $\UCoxO{\g}{\bullet}$, yielding a commutative diagram as in \eqref{eq:holo-prop-rep}.\\

\begin{pf}
The commutativity of \eqref{eq:holo-prop-rep} is verified by direct inspection. Note that the scaling 
factor in the definition of $\eta_{\rootsys}^{n}$ is chosen so to guarantee the commutativity of 
\eqref{eq:holo-prop-rep} and it is determined by the relation $\hbar=2\pi\iota\nablah$.
It remains to check that the linear map $\eta_{\rootsys}^{n}$ preserves the relations from 
Definition~\ref{ss:doubleholo}.

The symmetry and locality relations \eqref{eq:symmetry} and \eqref{eq:locality-1}, \eqref{eq:locality-2}, clearly holds in $\RDYUA{\rootsys}{n}$,
as they involve string diagrams insisting on distinct thick lines. The orthogonality relations
\eqref{eq:orthogonality-1} follow from the $\Rs{}$--grading relations in $\dLBA{\rootsys}$ (cf.~\ref{ss:root-lba}). 
Indeed, it is enough to observe that, if $\alpha\perp\beta$, one has
\begin{align}
	\begin{tikzpicture}[scale=0.75]
		\node at (0.5,-0.125) {
			\begin{tikzpicture}[scale=0.9]
				\node[rounded corners, draw] (S1) at (-1,0.5) {$\alpha$};
				\node[rounded corners, draw] (S2) at (-1,-0.5) {$\beta$};
				\node (T) at (1,0){};
				\draw (-1.5, 0.5)--(S1)--(0,0)--(T);
				\draw (-1.5, -0.5)--(S2)--(0,0);
			\end{tikzpicture}
		};
		\node at (3,0) {$=\;0\;=$};
		\node at (5.5,-0.125) {
			\begin{tikzpicture}[scale=0.9]
				\node[rounded corners, draw] (S1) at (1,0.5) {$\alpha$};
				\node[rounded corners, draw] (S2) at (1,-0.5) {$\beta$};
				\node (T) at (-1,0){};
				\draw (1.5, 0.5)--(S1)--(0,0)--(T);
				\draw (1.5, -0.5)--(S2)--(0,0);
			\end{tikzpicture}
		};
	\end{tikzpicture}
\end{align}
Therefore, actions and coactions labelled by $\alpha$ and $\beta$ commute, \ie
\begin{align}
	\begin{tikzpicture}[scale=0.75]
		\node at (0,0) {
			\begin{tikzpicture}[scale=0.75]
				\draw[very thick] (0,0)--(3,0);
				\node (L1) at (0,1) {};
				\node (C1) at (0.75,0.75) {};
				\node (T1) at (1.25,0) {};
				\draw plot[smooth,  tension=1] coordinates {(L1) (C1) (T1)};
				\node[rounded corners, draw, fill=white] at (C1) {$\alpha$};
				\node (L2) at (3,1) {};
				\node (C2) at (2.25,0.75) {};
				\node (T2) at (1.75,0) {};
				\draw plot[smooth,  tension=1] coordinates {(L2) (C2) (T2)};
				\node[rounded corners, draw, fill=white] at (C2) {$\beta$};
			\end{tikzpicture}
		};
		\node at (2.5,0) {$=$};
		\node at (5,0) {
			\begin{tikzpicture}[scale=0.75]
				\draw[very thick] (0,0)--(3,0);
				\node (L1) at (0,1) {};
				\node (C1) at (0.75,0.75) {};
				\node (T1) at (1.75,0) {};
				\draw plot[smooth,  tension=1] coordinates {(L1) (C1) (T1)};
				\node[rounded corners, draw, fill=white] at (C1) {$\alpha$};
				\node (L2) at (3,1) {};
				\node (C2) at (2.25,0.75) {};
				\node (T2) at (1.25,0) {};
				\draw plot[smooth,  tension=1] coordinates {(L2) (C2) (T2)};
				\node[rounded corners, draw, fill=white] at (C2) {$\beta$};
			\end{tikzpicture}
		};
	\end{tikzpicture}
\end{align}
and
\begin{align}
	\begin{tikzpicture}[scale=0.75]
		\node at (0,0) {
			\begin{tikzpicture}[scale=0.75]
				\draw[very thick] (0,0)--(2,0);
				\node (L1) at (0,1) {};
				\node (C1) at (0.75,0.75) {};
				\node (T1) at (1.25,0) {};
				\draw plot[smooth,  tension=1] coordinates {(L1) (C1) (T1)};
				\node[rounded corners, draw, fill=white] at (C1) {$\alpha$};
				\node (L2) at (0,1.75) {};
				\node (C2) at (1,1.5) {};
				\node (T2) at (1.75,0) {};
				\draw plot[smooth,  tension=1] coordinates {(L2) (C2) (T2)};
				\node[rounded corners, draw, fill=white] at (C2) {$\beta$};
			\end{tikzpicture}
		};
		\node at (2,0) {$=$};
		\node at (4,0) {
			\begin{tikzpicture}[scale=0.75]
				\draw[very thick] (0,0)--(2,0);
				\node (L1) at (0,1) {};
				\node (C1) at (0.75,0.75) {};
				\node (T1) at (1.75,0) {};
				\draw plot[smooth,  tension=1] coordinates {(L1) (C1) (T1)};
				\node[rounded corners, draw, fill=white] at (C1) {$\alpha$};
				\node (L2) at (0,1.75) {};
				\node (C2) at (1,1.5) {};
				\node (T2) at (1.25,0) {};
				\draw plot[smooth,  tension=1] coordinates {(L2) (C2) (T2)};
				\node[rounded corners, draw, fill=white] at (C2) {$\beta$};
			\end{tikzpicture}
		};
		\node at (9,0) {
			\begin{tikzpicture}[scale=0.75]
				\draw[very thick] (0,0)--(2,0);
				\node (L1) at (2,1) {};
				\node (C1) at (1.25,0.75) {};
				\node (T1) at (0.75,0) {};
				\draw plot[smooth,  tension=1] coordinates {(L1) (C1) (T1)};
				\node[rounded corners, draw, fill=white] at (C1) {$\alpha$};
				\node (L2) at (2,1.75) {};
				\node (C2) at (1,1.5) {};
				\node (T2) at (0.25,0) {};
				\draw plot[smooth,  tension=1] coordinates {(L2) (C2) (T2)};
				\node[rounded corners, draw, fill=white] at (C2) {$\beta$};
			\end{tikzpicture}
		};
		\node at (11,0) {$=$};
		\node at (13,0) {
		\begin{tikzpicture}[scale=0.75]
			\draw[very thick] (0,0)--(2,0);
			\node (L1) at (2,1) {};
			\node (C1) at (1.25,0.75) {};
			\node (T1) at (0.25,0) {};
			\draw plot[smooth,  tension=1] coordinates {(L1) (C1) (T1)};
			\node[rounded corners, draw, fill=white] at (C1) {$\alpha$};
			\node (L2) at (2,1.75) {};
			\node (C2) at (1,1.5) {};
			\node (T2) at (0.75,0) {};
			\draw plot[smooth,  tension=1] coordinates {(L2) (C2) (T2)};
			\node[rounded corners, draw, fill=white] at (C2) {$\beta$};
		\end{tikzpicture}
		};
	\end{tikzpicture}
\end{align}
It follows that any two arc diagrams labelled, respectively, by $\alpha$ and $\beta$ 
clearly commute. The orthogonality relations \eqref{eq:orthogonality-2} are proved similarly, 
by relying on the nestedness and support relations in $\dLBA{\Rs{}}$.\\

The proof of the KZ relations \eqref{eq:KZ-relations} is standard. Let 
$\MDY{}{n}$ be the $\PROP$ describing $n$ Drinfeld--Yetter modules over a Lie bialgebra. 
One observes that
the operator
\begin{align}
	\begin{tikzpicture}[scale=0.7]
		\node at (-1,-0.5) {
			\begin{tikzpicture}
				\draw[very thick] (-1,0.25)--(1,0.25);
				\draw[very thick] (-1,-0.25)--(1,-0.25);
				\draw[fill=white] (-0.5,0.5) -- (0.5,0.5) -- (0.5,-0.5) -- (-0.5,-0.5) -- (-0.5,0.5);
				\node (O) at (0,0) {$\Tp{}{}$};
			\end{tikzpicture}
		};
		\node at (1.5,-0.5) {$=$};
		\node at (4,0) {
			\begin{tikzpicture}
				\draw[very thick] (-1,0.25)--(1,0.25);
				\draw[very thick] (-1,-0.25)--(1,-0.25);
				\node (L) at (0,0.75) {};
				\node (S) at (-0.75,-0.25) {};
				\node (T) at (0.75,0.25) {};
				\draw plot[smooth,  tension=2] coordinates {(S) (L) (T)};
			\end{tikzpicture}
		};
		\node at (6,-0.5) {+};
		\node at (8,0) {
			\begin{tikzpicture}
				\draw[very thick] (-1,0.25)--(1,0.25);
				\draw[very thick] (-1,-0.25)--(1,-0.25);
				\node (L) at (0,0.75) {};
				\node (S) at (-0.75,0.25) {};
				\node (T) at (0.75,-0.25) {};
				\draw plot[smooth,  tension=2] coordinates {(S) (L) (T)};
			\end{tikzpicture}
		};
	\end{tikzpicture}
\end{align}
is invariant, \ie it commutes with the action and the coaction on $\VDY{1}\ten\VDY{2}$
\begin{align}
	\begin{tikzpicture}[scale=0.75]
		\node at (-1,-0.125) {
			\begin{tikzpicture}
				\draw[very thick] (-1,0.25)--(1,0.25);
				\draw[very thick] (-1,-0.25)--(1,-0.25);
				\node (O) at (0,0) {};
				\node (C) at (-0.5,0.75) {};
				\node (L) at (-1,1) {};
				\draw plot[smooth,  tension=1] coordinates {(L) (C) (O)};
				\draw[fill=white] (-0.5,0.5) -- (0.5,0.5) -- (0.5,-0.5) -- (-0.5,-0.5) -- (-0.5,0.5);
			\end{tikzpicture}
		};
		\node at (1.5,-0.5) {$=$};
		\node at (4,0) {
			\begin{tikzpicture}[scale=0.75]
				\draw[very thick] (-1,0.25)--(1,0.25);
				\draw[very thick] (-1,-0.25)--(1,-0.25);
				\node (C) at (-0.5,0.75) {};
				\node (L) at (-1,1) {};
				\node (T) at (0,0.25) {};
				\draw plot[smooth,  tension=1] coordinates {(L) (C) (T)};
			\end{tikzpicture}
		};
		\node at (6,-0.5) {+};
		\node at (8,0) {
			\begin{tikzpicture}[scale=0.75]
				\draw[very thick] (-1,0.25)--(1,0.25);
				\draw[very thick] (-1,-0.25)--(1,-0.25);
				\node (C) at (-0.5,0.75) {};
				\node (L) at (-1,1) {};
				\node (T) at (0,-0.25) {};
				\draw plot[smooth,  tension=1] coordinates {(L) (C) (T)};
			\end{tikzpicture}
		};
		\node at (-1,-3.125) {
			\begin{tikzpicture}[scale=0.75]
				\draw[very thick] (-1,0.25)--(1,0.25);
				\draw[very thick] (-1,-0.25)--(1,-0.25);
				\node(O) at (0,0) {};
				\node (C) at (0.5,0.75) {};
				\node (L) at (1,1) {};
				\draw plot[smooth,  tension=1] coordinates {(L) (C) (O)};
				\draw[fill=white] (-0.5,0.5) -- (0.5,0.5) -- (0.5,-0.5) -- (-0.5,-0.5) -- (-0.5,0.5);
			\end{tikzpicture}
		};
		\node at (1.5,-3.5) {$=$};
		\node at (4,-3) {
			\begin{tikzpicture}[scale=0.75]
				\draw[very thick] (-1,0.25)--(1,0.25);
				\draw[very thick] (-1,-0.25)--(1,-0.25);
				\node (C) at (0.5,0.75) {};
				\node (L) at (1,1) {};
				\node (T) at (0,0.25) {};
				\draw plot[smooth,  tension=1] coordinates {(L) (C) (T)};
			\end{tikzpicture}
		};
		\node at (6,-3.5) {+};
		\node at (8,-3) {
			\begin{tikzpicture}[scale=0.75]
				\draw[very thick] (-1,0.25)--(1,0.25);
				\draw[very thick] (-1,-0.25)--(1,-0.25);
				\node (C) at (0.5,0.75) {};
				\node (L) at (1,1) {};
				\node (T) at (0,-0.25) {};
				\draw plot[smooth,  tension=1] coordinates {(L) (C) (T)};
			\end{tikzpicture}
		};
	\end{tikzpicture}
\end{align}
Therefore, the operator $\Tp{12}{}$ on $\VDY{1}\ten\VDY{2}\ten\VDY{3}$ commutes 
with $\Tp{13}{}+\Tp{23}{}$, since the latter is the operator
\begin{align}
	\begin{tikzpicture}[scale=1]
		\node at (-2.5,0) {
			\begin{tikzpicture}[scale=1]
				\draw[very thick] (-1,0.25)--(1,0.25);
				\draw[very thick] (-1,0)--(1,0);
				\draw[very thick] (-1,-0.25)--(1,-0.25);
				\node (T) at (0.5,0.125) {};
				\node (C) at (0,0.75) {};
				\node (S) at (-0.75,-0.25) {};
				\draw plot[smooth,  tension=1] coordinates {(S) (C) (T)};
				\draw[fill=white] (0.25,0.375) -- (0.75,0.375) -- (0.75,-0.125) -- (0.25,-0.125) -- (0.25,0.375);
			\end{tikzpicture}
		};
		\node at (0,-0.25) {+};
		\node at (2.5,0) {
			\begin{tikzpicture}[scale=1]
				\draw[very thick] (-1,0.25)--(1,0.25);
				\draw[very thick] (-1,0)--(1,0);
				\draw[very thick] (-1,-0.25)--(1,-0.25);
				\node (T) at (-0.5,0.125) {};
				\node (C) at (0,0.75) {};
				\node (S) at (0.75,-0.25) {};
				\draw plot[smooth,  tension=1] coordinates {(S) (C) (T)};
				\draw[fill=white] (-0.25,0.375) -- (-0.75,0.375) -- (-0.75,-0.125) -- (-0.25,-0.125) -- (-0.25,0.375);
			\end{tikzpicture}
		};
	\end{tikzpicture}
\end{align}
 For any $B\subseteq\dgr$, we consider the canonical morphism of $\PROP$s $\MDY{}{n}\to\MDY{\Rs{}}{n}$, mapping the Lie bialgebra object $[1]$ in $\MDY{}{n}$ to the Lie bialgebra $([1],\dmp{B})$ in 
$\MDY{\Rs{}}{n}$. This shows that $\Tp{}{B}$ commutes with the action and coaction of $[\b_{B'}]$ for any $B'\subseteq B$, and the diagrammatic KZ relations \eqref{eq:KZ-relations} follow.

The weight zero relations follow from the fact that the Lie bialgebras $([1],\dmp{0,B})$
are abelian.

By Lemma~\ref{lem:comm-rel} (3), the operators $\Kp{i}{\alpha}$ and 
$\Delta^{(n)}(\Kp{}{\alpha})$ satisfy the Casimir relations \eqref{eq:Casimir-relations}.
Finally, it is clear that
\[
\Delta^{(n)}(\Kp{}{\alpha})=\sum_{i<j}\Tp{ij}{\alpha}+\sum_{i=1}^n\Kp{i}{\alpha}
\]
so that \eqref{eq:Kn coproduct} and \eqref{eq:invariance} holds in $\DUA{\Rs{}}{n}$.\\

The algebra maps $\eta_{\rootsys}^n:\DBLHA{\Rs{}}{n}\to\RDYUA{\Rs{}}{n}$
clearly preserve the cosimplicial structure, the diagrammatic subalgebras, and the natural
grading. The result follows.
\end{pf}

\subsection{Proof of Theorem~\ref{ss:mon-thm} (1) and (3a)}\label{ss:nabla-universality}
We shall prove the following

\begin{theorem}
Let $\sCox{\nabla}=(\Phig{\nabla}{B}{}, \Rg{\nabla}{B}{}, \Jg{\nabla}{\F}{}, \DCg{\nabla}{\F}{\G}{}, \Sg{\nabla}{}{i}{})$ be the $\redasso{}{}$--strict braided Coxeter 
structure on the extended double holonomy algebra $\DBLHAH{\Rs{}}{\bullet,\scsop{ext}}$ defined in Theorem~\ref{thm:holo-cox}. 
\begin{enumerate}[label=(\alph*)]\itemsep0.25cm
\item The datum of  
\[\sCox{\nabla}^{\scsop{pre,\eta}}=(\Phig{\nabla}{B}{,\eta}, \Rg{\nabla}{B}{,\eta}, \Jg{\nabla}{\F}{,\eta}, \DCg{\nabla}{\F}{\G}{,\eta})\] 
where
\begin{align*}
\Phig{\nabla}{B}{,\eta}=\eta_{\Rs{}}^3(\Phig{\nabla}{B}{}), \quad
\Rg{\nabla}{B}{,\eta}=\eta_{\Rs{}}^2(\Rg{\nabla}{B}{}), \quad
\Jg{\nabla}{\F}{,\eta}=\eta_{\Rs{}}^2(\Jg{\nabla}{\F}{}), \quad
\DCg{\nabla}{\F}{\G}{,\eta}=\eta_{\Rs{}}^1(\DCg{\nabla}{\F}{\G}{}),
\end{align*}
is a braided pre--Coxeter structure on $\CDUA{\Rs{}}{\bullet}$.
\item 
Through the realisation morphisms $\CDUA{\Rs{}}{\bullet}\to\UCoxDYint{\bm{}}{\bullet}\to\UCoxOint{\g}{\bullet}$ (cf.~Section~\ref{ss:from-holo-to-prop}),
 $\sCox{\nabla}^{\scsop{pre,\eta}}$ induces on $\UCoxOint{\g}{\bullet}$ 
the braided Coxeter structure arising from the joint KZ--Casimir connection
defined in Theorem~\ref{thm:hol-to-KM}.
\end{enumerate}
\end{theorem}

\begin{pf}
	Part (b) follows from the commutativity of the diagram \eqref{eq:holo-prop-rep}.
	For part (a), we proceed as in the proof of Theorem~\ref{thm:hol-to-KM}.
	We shall verify that $\sCox{\nabla}^{\scsop{pre,\eta}}$ satisfy the properties
	(a)--(e) from Definition~\ref{ss:braid-cox-alg} with respect to the cosimplicial 
	bidiagrammatic structure on $\CDUA{\Rs{}}{\bullet}$. By construction,
	$\sCox{\nabla}^{\scsop{pre,\eta}}$ is the image of a braided pre--Coxeter structure
	$\sCox{\nabla}^{\scsop{pre}}$ in $\DBLHAH{\Rs{}}{\bullet}$ through 
	the morphism $\eta_{\Rs{}}^{\bullet}:\DBLHAH{\Rs{}}{\bullet}\to\CDUA{\Rs{}}{\bullet}$
	defined in \ref{ss:from-holo-to-prop}. Although $\eta_{\Rs{}}^{\bullet}$ is a morphism 
	of cosimplicial diagrammatic algebras, it does not preserve the invariant subalgebras,
	as the condition of being invariant in $\RDYUA{\Rs{}}{n}$ is generally stronger than being
	invariant in $\DBLHA{\Rs{}}{n}$. Therefore, proving that
	$\sCox{\nabla}^{\scsop{pre,\eta}}$ is a braided pre--Coxeter structure in 
	$\CDUA{\Rs{}}{\bullet}$ reduces to showing that the elements $\Phig{\nabla}{B}{,\eta}$, $\Rg{\nabla}{B}{,\eta}$, $\Jg{\nabla}{\F}{,\eta}$, and $\DCg{\nabla}{\F}{\G}{,\eta}$ satisfy
	the necessary invariance properties.\\
	
	By definition, $\Rg{\nabla,\eta}{B}{}=\exp(\Tp{12}{B}/2)\in\CRDYUA{\Rs{},B}{2}$ 
	and, by Theorem~\ref{ss:KZ-associator}, the associator
	$\Phig{\nabla,\eta}{B}{}\in\CRDYUA{\Rs{},B}{3}$ is the exponential of a Lie series in 
	$\Tp{12}{B}$ and $\Tp{23}{B}$. As observed in \ref{ss:from-holo-to-prop}, the 
	operator $\Tp{i,i+1}{B}$ is $[1]_B$--invariant in $\RDYUA{\Rs{}, B}{n}$, therefore so are $\Rg{\nabla,\eta}{B}{}\in\CRDYUA{\Rs{},BB}{2}$
	and $\Phig{\nabla,\eta}{B}{}\in\CRDYUA{\Rs{},BB}{3}$.
	The invariance of the relative twists and the De Concini--Procesi associators 
	is obtained as in \cite[Thm.~1.33]{vtl-4} and \cite[App.~B.4]{vtl-6}.
	Namely, it is enough to observe that the \emph{relative} Casimir operators, 
	which provide the coefficients of the differential equations defining 
	$\Jg{\nabla}{\F}{}$ and $\DCg{\nabla}{\F}{\G}{}$ in $\DBLHAH{\Rs{}}{\bullet}$,
	specialise in $\RDYUA{\Rs{}}{1}$ to elements with the necessary invariant 
	properties.\\
	
	For any $B'\subseteq B\subseteq\dgr$, set $\Kp{}{BB'}=\sum_{\beta\in\Rs{B,+}\setminus\Rs{B',+}}\Kp{}{\beta}$. 
	We shall prove that $\Kp{}{BB'}$ commutes with the action and the coaction of the 
	universal Lie subbialgebra $[1]_{B'}=([1],\dmp{B'})$. Note that the elements
	$\Kp{}{\beta}$ are weight zero, \ie for any $\beta\in\Rs{+}$, 
	we have
		\begin{align}\label{eq:invariance-rel-K-zero}
		\begin{tikzpicture}[scale=0.75,baseline=(current  bounding  box.center)]
			\node at (2,0) {
				\begin{tikzpicture}[scale=0.75]
					\node (S1) at (-1,1) {};
					\node (L1) at (-0.5,0.75) {};
					\node (T1) at (-0.25,0) {};
					\draw plot[smooth,  tension=1] coordinates {(S1) (L1) (T1)};
					\node[rounded corners, draw=black, fill=white] at (L1) {$0$};
					\draw[very thick] (-1,0)--(2,0);
					\node[rounded corners, draw=black] (L2) at (1,0.75) {};
					\node (S2) at (0.25,0) {};
					\node (T2) at (1.75,0) {};
					\draw plot[smooth,  tension=2] coordinates {(S2) (L2) (T2)};
					\node[rounded corners, draw=black, fill=white] at (L2) {$\beta$};
				\end{tikzpicture}
			};
			\node at (4,0) {$=$};
			\node at (6,0.4) {
				\begin{tikzpicture}[scale=0.75]
					\draw[very thick] (-0.5,0)--(2.5,0);
					\node (L1) at (1.5,1.5) {};
					\node (S1) at (-0.25,1.5) {};
					\node (T1) at (2.25,0) {};
					\draw plot[smooth,  tension=1] coordinates {(S1) (L1) (T1)};
					\node[rounded corners, draw=black, fill=white] at (1,1.65) {$0$};
					\node (L2) at (1,0.75) {};
					\node (S2) at (0,0) {};
					\node (T2) at (2,0) {};
					\draw plot[smooth,  tension=2] coordinates {(S2) (L2) (T2)};
					\node[rounded corners, draw=black, fill=white] at (L2) {$\beta$};
				\end{tikzpicture}
			};
		\end{tikzpicture}
	\end{align}
	Let $\alpha\in\Rs{B',+}$ and $\beta\in\Rs{B,+}\setminus\Rs{B',+}$. Note that 
	$\alpha-\beta$ is never a positive root and we have
	\begin{align}\label{eq:invariance-rel-K}
		\begin{tikzpicture}[scale=0.75,baseline=(current  bounding  box.center)]
			\node at (2,0) {
				\begin{tikzpicture}[scale=0.75]
					\node (S1) at (-1,1) {};
					\node (L1) at (-0.5,0.75) {};
					\node (T1) at (-0.25,0) {};
					\draw plot[smooth,  tension=1] coordinates {(S1) (L1) (T1)};
					\node[rounded corners, draw=black, fill=white] at (L1) {$\alpha$};
					\draw[very thick] (-1,0)--(2,0);
					\node[rounded corners, draw=black] (L2) at (1,0.75) {};
					\node (S2) at (0.25,0) {};
					\node (T2) at (1.75,0) {};
					\draw plot[smooth,  tension=2] coordinates {(S2) (L2) (T2)};
					\node[rounded corners, draw=black, fill=white] at (L2) {$\beta$};
				\end{tikzpicture}
			};
			\node at (4,0) {$=$};
			\node at (6,0) {
				\begin{tikzpicture}[scale=0.75]
					\node (S1) at (-1,1) {};
					\node (L1) at (-0.5,0.75) {};
					\node (T1) at (-0.25,0) {};
					\draw[very thick] (-1,0)--(2,0);
					\node[rounded corners, draw=black] (L2) at (1,0.75) {};
					\node (S2) at (0.25,0) {};
					\node (T2) at (1.75,0) {};
					\draw plot[smooth,  tension=1] coordinates {(S1) (L1) (S2)};
					\node[rounded corners, draw=black, fill=white] at (L1) {$\alpha$};
					\draw plot[smooth,  tension=1.5] coordinates {(T1) (L2) (T2)};
					\node[rounded corners, draw=black, fill=white] at (L2) {$\beta$};
				\end{tikzpicture}
			};
			\node at (8,0) {$+$};
			\node at (10,0.235) {
				\begin{tikzpicture}[scale=0.75]
					\draw[very thick] (-1,0)--(2,0);
					\node (S1) at (-1,1.25) {};
					\node (L1) at (1,1.25) {};
					\node (T1) at (1.5,0.215) {};
					\node (L2) at (0.25,0.5) {};
					\node (S2) at (-0.75,0) {};
					\node (T2) at (1.75,0) {};
					\draw plot[smooth,  tension=1] coordinates {(S1) (L1) (T1)};
					\node[rounded corners, draw=black, fill=white] at (-0.5,1.25) {$\alpha$};
					\draw plot[smooth,  tension=1.5] coordinates {(S2) (L2) (T2)};
					\node[rounded corners, draw=black, fill=white] at (L2) {$\beta-\alpha$};
				\end{tikzpicture}
			};
		\node at (1,-2.25) {$=$};
		\node at (2.5,-2) {
			\begin{tikzpicture}[scale=0.75]
				\draw[very thick] (-0.25,0)--(2,0);
				\node[rounded corners, draw=black] (L1) at (1,0.75) {};
				\node (S1) at (-0.25,0.75) {};
				\node (T1) at (1.5,0) {};
				\draw plot[smooth,  tension=1] coordinates {(S1) (L1) (T1)};
				\node[rounded corners, draw=black, fill=white] at (L1) {$\alpha$};
				\node (L2) at (1,1.5) {};
				\node (S2) at (0.25,0) {};
				\node (T2) at (1.75,0) {};
				\draw plot[smooth,  tension=2] coordinates {(S2) (L2) (T2)};
				\node[rounded corners, draw=black, fill=white] at (L2) {$\beta$};
			\end{tikzpicture}
		};
		\node at (4,-2.25) {$+$};
		\node at (6,-2) {
			\begin{tikzpicture}[scale=0.75]
				\draw[very thick] (-0.5,0)--(2.5,0);
				\node (L1) at (1.5,1.5) {};
				\node (S1) at (-0.25,1.5) {};
				\node (T1) at (2.25,0) {};
				\draw plot[smooth,  tension=1] coordinates {(S1) (L1) (T1)};
				\node[rounded corners, draw=black, fill=white] at (1,1.6) {$\alpha$};
				\node (L2) at (1,0.75) {};
				\node (S2) at (0,0) {};
				\node (T2) at (2,0) {};
				\draw plot[smooth,  tension=2] coordinates {(S2) (L2) (T2)};
				\node[rounded corners, draw=black, fill=white] at (L2) {$\beta-\alpha$};
			\end{tikzpicture}
		};
		\node at (8,-2.25) {$-$};
		\node at (10,-2) {
			\begin{tikzpicture}[scale=0.75]
				\draw[very thick] (-0.5,0)--(2.5,0);
				\node[rounded corners, draw=black] (L1) at (1,0.75) {};
				\node (S1) at (-0.5,0.75) {};
				\node (T1) at (2,0) {};
				\draw plot[smooth,  tension=1] coordinates {(S1) (L1) (T1)};
				\node[draw=black, fill=white,rounded corners] at (L1) {$\alpha$};
				\node (L2) at (1,1.5) {};
				\node (S2) at (-0.25,0) {};
				\node (T2) at (2.25,0) {};
				\draw plot[smooth,  tension=2] coordinates {(S2) (L2) (T2)};
				\node[draw=black, fill=white, rounded corners] at (L2) {$\beta-\alpha$};
			\end{tikzpicture}
		};
		\end{tikzpicture}
	\end{align}
where the second and third summands appear if and only if $\beta-\alpha\in\Rs{B,+}$. 
Summing over all positive roots $\beta\in\Rs{B,+}\setminus\Rs{B',+}$, the first and 
third summands cancel out. Namely, if $\beta-\alpha\in\Rs{B,+}$, then the third summand in 
the equation \eqref{eq:invariance-rel-K} for $\beta$ cancels out with the first summand in the 
equation \eqref{eq:invariance-rel-K} for $\beta-\alpha$. On the other hand, assume that 
$\beta+\alpha\in\Rs{B,+}$. Then, the first summand in the equation \eqref{eq:invariance-rel-K} 
for $\beta$ cancels out with the third summand in the equation \eqref{eq:invariance-rel-K} for $\beta+\alpha$. Finally, if $\beta+\alpha\not\in\Rs{B,+}$, then 
\begin{align}\label{eq:invariance-rel-K-2}
	\begin{tikzpicture}[scale=0.75,baseline=(current  bounding  box.center)]
		\node at (2,0) {
		\begin{tikzpicture}[scale=0.75]
			\draw[very thick] (-0.25,0)--(2,0);
			\node[rounded corners, draw=black] (L1) at (1,0.75) {};
			\node (S1) at (-0.25,0.75) {};
			\node (T1) at (1.5,0) {};
			\draw plot[smooth,  tension=1] coordinates {(S1) (L1) (T1)};
			\node[rounded corners, draw=black, fill=white] at (L1) {$\alpha$};
			\node (L2) at (1,1.5) {};
			\node (S2) at (0.25,0) {};
			\node (T2) at (1.75,0) {};
			\draw plot[smooth,  tension=2] coordinates {(S2) (L2) (T2)};
			\node[rounded corners, draw=black, fill=white] at (L2) {$\beta$};
		\end{tikzpicture}
		};
		\node at (3.75,0) {$=$};
		\node at (6,0) {
			\begin{tikzpicture}[scale=0.75]
				\draw[very thick] (0,0)--(2.5,0);
				\node (L1) at (1.5,1.5) {};
				\node (S1) at (0,1.5) {};
				\node (T1) at (2.25,0) {};
				\draw plot[smooth,  tension=1] coordinates {(S1) (L1) (T1)};
				\node[rounded corners, draw=black, fill=white] at (1,1.6) {$\alpha$};
				\node (L2) at (1.125,0.75) {};
				\node (S2) at (0.25,0) {};
				\node (T2) at (2,0) {};
				\draw plot[smooth,  tension=2] coordinates {(S2) (L2) (T2)};
				\node[rounded corners, draw=black, fill=white] at (L2) {$\beta$};
			\end{tikzpicture}
		};
	\end{tikzpicture}
\end{align}
Therefore, by \eqref{eq:invariance-rel-K-zero}, the operator $\Kp{}{BB'}$ commutes with 
the action of $[1]_{B'}$. The invariance of $\Kp{}{BB'}$ under the coaction of $[1]_{B'}$ is 
proved similarly.
\end{pf}

Therefore, $\sCox{\nabla}$ induces an $\redasso{}{}$--strict
universal braided Coxeter structure on $\UCoxDYint{\bm{}}{\bullet}$, which we denote by
$\sCox{\nabla}^{\astr}$ and, by Proposition~\ref{ss:def-Cox-DY}, yields a braided Coxeter category $\DYCox{\bm{},\nabla}{\hdef,\sint}$.

\subsection{Proof of Theorem~\ref{ss:mon-thm} (2) and (4b)}\label{ss:DY-qG}
In Proposition~\ref{ss:qg-cox}, we described the $(\redasso{}{},\DCPA{}{})$--strict braided Coxeter
category $\OCox{\Uhg,\Rmx, \qWS{}}{\scsop{\sint}}$ arising from the action of the $R$--matrix 
and the quantum Weyl group operators of the quantum group $\Uhg$ on category $\O_{\infty}$
integrable $\Uhg$--modules. 
In analogy with the classical case (cf.~\ref{ss:from-O-to-DY} and \ref{ss:Cox-DY}), this
extends to admissible Drinfeld--Yetter $\Uhbm{}$--modules. Namely, the quantum 
group $\Uhg$ is isomorphic, as diagrammatic QUEs, to the quotient of the restricted
quantum double of $\Uhbm{}$. Therefore, any admissible Drinfeld--Yetter module 
$(V,\rho_V,\rho_V)$ satisfying 
\begin{equation}\label{eq:DY-ss-0-q}
	\rho_{\V}=
	\iip{\cdot}{\cdot}_{\h}\ten\id_{\V}\circ\id_{\h}\ten\rho_{\V}^*
\end{equation}
is naturally a module over $\Uhg$. In particular, this allows to recover 
category $\O_{\infty,\Uhg}$ as a braided tensor subcategory of $\hDrY{\Uhbm{}}{\adm}$.
We say that a deformation Drinfeld--Yetter $\bm{}$--module is \emph{integrable} if it satisfies \eqref{eq:DY-ss-0-q} and \ref{cond:int-h}. Similarly for $\Uhbm{}$. Let $\hDrY{\Uhbm{}}{\adm,\sint}$ 
be the category of integrable admissible Drinfeld--Yetter $\Uhbm{}$--modules.
Then, the generalised braid group $\Br{W}$ acts on 
the objects in $\hDrY{\Uhbm{}}{\adm,\sint}$ via the quantum Weyl group operators 
$\qWS{i}$, $i\in\bfI$. By relying on the split diagrammatic structure of $\Uhbm{}$, we
obtain the following extension of Proposition~\ref{ss:qg-cox}.

\begin{proposition}
	There is a $(\redasso{}{},\DCPA{}{})$--strict braided Coxeter category 
	$\DYCox{\Uhbm{},\Rmx,\qWS{}}{\adm,\sint}$  
	of type $(\dgr, \ulm)$ given by the following data.
	\vspace{0.25cm}
	\begin{itemize}\itemsep0.25cm
		\item For any $B\subseteq\dgr$, the braided monoidal category 
		$\hDrY{\Uhbm{B}}{\adm,\sint}$.
		\item For any $B'\subseteq B$, the restriction functor $\hRes_{B'B}:
		\hDrY{\Uhbm{B}}{\adm,\sint}\to\hDrY{\Uhbm{B'}}{\adm,\sint}$.
		\item For any $i\in\dgr$, the quantum Weyl group operator $\qWS{i}\in\Aut(\hDrY{\Uhbm{i}}{\adm,\sint}\to\tfV)$.
	\end{itemize}
	Moreover, $\OCox{\Uhbm{},\Rmx,\qWS{}}{\sint}$ naturally identifies with a subcategory
	of $\DYCox{\Uhbm{},\Rmx,\qWS{}}{\adm,\sint}$.
\end{proposition}

Finally, it follows as in \ref{ss:univ-equivalence} that the braided pre--Coxeter structure  $\DYCox{\Uhbm{},\Rmx}{\adm}$ is universal and induced by the
standard braided pre--Coxeter structure on $\CDUA{\dgr}{\hbar,\bullet}$.

\subsection{Proof of Theorem~\ref{ss:mon-thm} (4c)}\label{ss:qG-to-km}

Let $\Phi^{\nabla}$ be the KZ associator. Since $\Phi^{\nabla}$ is a Lie associator
by Theorem~\ref{ss:KZ-associator}, Theorem~\ref{ss:diag-EK-equivalence} yields
a universal braided pre--Coxeter structure $\pCox{\Phi^{\nabla}}^{\gstr}$ on $\CDUA
{\dgr}{\bullet}$, and therefore an equivalence of braided pre--Coxeter categories
\[
\mCox{\bm{}}^{\scsop{pre}}:\DYCox{\bm{},\pCox{\Phi^{\nabla}}^{\gstr}}{\hdef}
\longrightarrow\DYCox{\Q(\bm{})}{\adm}
\]
By \cite{ek-6} and \cite[Prop.~13.6]{ATL1-2}, the split diagrammatic QUEs
$\Q(\bm{})$ and $\Uhbm{}$ are isomorphic, thus yielding an equivalence 
of braided pre--Coxeter categories 
$\DYCox{\Q(\bm{})}{\adm}\simeq\DYCox{\Uhbm{},\Rmx}{\adm}$. 

\begin{lemma}
	The composite equivalence 
		\begin{equation*}
		\begin{tikzcd}
			\DYCox{\bm{},\pCox{\Phi^{\nabla}}^{\gstr}}{\hdef} \arrow[r, "\mCox{\bm{}}^{\scsop{pre}}"]&\DYCox{\Q(\bm{})}{\adm}\simeq\DYCox{\Uhbm{},\Rmx}{\adm}
		\end{tikzcd}
	\end{equation*}
	preserves integrability.
\end{lemma}

\begin{pf}
Recall that, for $\g=\mathfrak{sl}_2$, integrability is equivalent to complete reducibility
as a possibly infinite direct sum of (indecomposable) finite--rank modules. Since the
equivalence commutes with direct sums and preserves the rank, the result is clear in
this case.
	
For any $i\in\bfI$, set $\bm{i}=\langle f_i, h_i\rangle\subset\mathfrak{sl}_2^{\alpha_i}$. 
By \cite[Thm.~1.7]{ATL1-1}, there is a commutative diagram 
of functors
\begin{equation}\label{cd:EK-res-i}
	\begin{tikzcd}
		\hDrY{\bm{}}{\hdef,\Phi} \arrow[r, "H_{\bm{}}"]\arrow[d]& \aDrY{\Q(\bm{})}\arrow[d]\\
		\hDrY{\bm{}}{\hdef,\Phi_{i}} \arrow[r, "H_{\bm{i}}"']& \aDrY{\Q{\bm{i}}}\\
	\end{tikzcd}
	\vspace{-0.5cm}
\end{equation}
where the horizontal arrows are the Etingof--Kazhdan equivalences and the vertical arrows
are restrictions. Then, the result follows, since restrictions preserve integrability and the
isomorphism $\Q(\bm{})\simeq\Uhbm{}$ is split diagrammatic.
\end{pf}

This allows to enhance $\mCox{\bm{}}^{\scsop{pre}}$ to
an equivalence of braided Coxeter categories
\[
\mCox{\bm{}}':\DYCox{\bm{},\sCox{\Rmx,\qWS{}}^{\gstr}}{\hdef,\sint}
\longrightarrow\DYCox{\Uhbm{},\Rmx,\qWS{}}{\adm,\sint}
\]
where $\sCox{\Rmx,\qWS{}}^{\gstr}$ is a universal braided Coxeter structure
which extends $\pCox{\Phi^{\nabla}}^{\gstr}$, \ie
$\sCox{\Rmx,\qWS{}}^{\gstr,\scsop{pre}}=\pCox{\Phi^{\nabla}}^{\gstr}$.

\subsection{Proof of Theorem~\ref{ss:mon-thm}~(4d)}\label{ss:twist-equiv}
By the discussion above, we now have an $\redasso{}{}$--strict universal braided Coxeter 
structure $\sCox{\nabla}^{\astr}$, arising from the monodromy data and supported on 
$\CDUA{\Rs{}}{\bullet}$, and a $\DCPA{}{}$--strict braided Coxeter structure
$\sCox{\Rmx,\qWS{}}^{\gstr}$, arising from the quantum group $\Uhg$ and supported
on $\CDUA{\dgr}{\bullet}$. Note that, by construction, $\sCox{\nabla}^{\astr}$ and $\sCox{\Rmx,\qWS{}}^{\gstr}$ already share the same associators and $R$--matrices.
The proof of Theorem~\ref{ss:mon-thm} (3b) amounts to showing that $\sCox{\nabla}^{\astr}$
and $\sCox{\Rmx,\qWS{}}^{\gstr}$ are twist equivalent. More precisely, we prove the
following
\begin{theorem}
The universal structures $\sCox{\nabla}^{\astr}$ and $\sCox{\Rmx,\qWS{}}^{\gstr}$ are 
twist equivalent (cf.~\ref{ss:twist-gauge-braided-Cox}) with respect to a twist of the form $\tCox{}=\tCox{}'\cdot\tCox{}''$, where
\begin{enumerate}\itemsep0.25cm
	\item $\tCox{}'$ is uniquely determined by a tuple of grouplike elements in
	$\wh{S}\h_i$, $i\in\dgr$
	\item $\tCox{}''$ is a unique universal twist supported on $\CDUA{\Rs{}}{\bullet}$
\end{enumerate}	
\end{theorem}
This is achieved in two steps, which 
rely heavily on the fact that both structures are universal and supported on $\CDUA{\Rs{}}{\bullet}$.
Indeed, we proved in \cite{ATL2} that braided pre--Coxeter structures on $\CDUA{\Rs{}}{\bullet}$
are rigid. Specifically, we have the following

\begin{theorem}\cite[Thm.~13.4]{ATL2}
Let $\pCox{k}$, $k=1,2$, be two $\redasso{}{}$--strict braided pre--Coxeter structures 
on $\CDUA{\Rs{}}{\bullet}$. Then, there exists a twist $\tCox{}$ such that $\pCox{2}=
(\pCox{1})_{\tCox{}}$. Moreover, $\tCox{}$ is unique up to a unique gauge.
\end{theorem}

Note that, by Proposition~\ref{ss:strict-pre-Cox}, $\sCox{\Rmx,\qWS{}}^{\gstr}$ is canonically 
twist equivalent to an $\redasso{}{}$--strict universal braided Coxeter structure 
$\sCox{\Rmx,\qWS{}}^{\astr}$. 
Let $\sCox{\nabla}^{\astr,\scsop{pre}}$ and $\sCox{\Rmx,\qWS{}}^{\astr,\scsop{pre}}$  
be the braided pre--Coxeter structures underlying $\sCox{\nabla}^{\astr}$ and $\sCox{\Rmx,\qWS{}}^{\astr}$, respectively. The result above determines a universal twist 
$\tCox{}''$, unique up to a unique universal gauge, such that $\sCox{\Rmx,\qWS{}}^{\gstr,\scsop{pre}}
=(\sCox{\nabla}^{\astr,\scsop{pre}})_{\tCox{}''}$.\\

However, at the level of braided Coxeter structures, we need a further correction,
since the local monodromy operators are determined by the underlying universal 
structure in $\CDUA{\rootsys}{\bullet}$ only up to a unique Cartan--valued gauge. 
More precisely, we have the following

\begin{proposition}\cite[Cor.~15.13]{ATL2}
	Up to a unique gauge transformation determined by a tuple of grouplike
	elements in $\hext{S\h_i}$, $i\in\dgr$, a braided pre--Coxeter structure on 
	$\CDUA{\rootsys}{\bullet}$ can be lifted to at most one universal braided 
	Coxeter structure on $\UCoxDYint{\bm{}}{\bullet}$.
\end{proposition}

Therefore, this yields a canonical twist $\tCox{}'$ such that 
$\sCox{\Rmx,\qWS{}}^{\gstr}=(\sCox{\nabla}^{\astr})_{\tCox{}}$
with $\tCox{}=\tCox{}'\cdot \tCox{}''$. The twist $\tCox{}$ 
induces an equivalence of braided Coxeter categories 
$\mCox{\bm{}}'':\DYCox{\bm{},\nabla}{\hdef,\sint}\to\DYCox{\bm{},\sCox{\Rmx,\qWS{}}^{\gstr}}{\hdef,\sint}$ and therefore
\[
\mCox{\bm{}}=\mCox{\bm{}}'\circ\mCox{\bm{}}'':
\DYCox{\bm{},\nabla}{\hdef,\sint}\to\DYCox{\Uhbm{},\Rmx,\qWS{}}{\adm,\sint}
\]

\subsection{Proof of Theorem~\ref{ss:mon-thm}~(3)}\label{ss:main-thm-cat-O}

There remains to show that the equivalence $\mCox{\bm{}}$ preserves category
$\O_{\infty}$ modules, and therefore restricts to an equivalence of braided Coxeter
categories $\mCox{\g}:\OCox{\g,\nabla}{\scsop{\hdef,\sint}}\to\OCox{\Uhg,\Rmx,
\qWS{}}{\sint}$.

\begin{lemma}
	The functor 
	\begin{equation}\label{eq:DrY-EK-equiv}
	\begin{tikzcd}
		\hDrY{\bm{}}{\hdef,\Phi} \arrow[r, "H_{\bm{}}"]& \aDrY{\Q(\bm{})}\simeq\aDrY{\Uhbm{}}
	\end{tikzcd}
	\end{equation}
	restricts to an equivalence of categories $\Ohinfg\to\OinfUhg$.
\end{lemma}

\begin{pf}
	In analogy with Proposition~\ref{ss:from-O-to-DY}, category $\Ohinfg$ identifies
	with the subcategory of deformation Drinfeld--Yetter modules over $\bm{}$ satisfying
	condition \eqref{eq:DY-ss-0}. An analogous characterization holds for $\OinfUhg$. 
	Since the equivalence \eqref{eq:DrY-EK-equiv} is the identity on Drinfeld--Yetter
	$\h$--modules, condition \eqref{eq:DY-ss-0} is automatically preserved, and the
	result follows.
\end{pf}

This concludes the proof of Theorem \ref{ss:mon-thm}.
 

\appendix


\section{The $W$--equivariant Casimir connection of an affine Kac--Moody algebra}\label{s:coda}

In this appendix, we construct two explicit $W$--equivariant corrections of the Casimir connection
\[
\nabla=d-A=d-\nablah\sum_{\alpha\in\Rs{+}}\frac{d\alpha}{\alpha}\Ku{\alpha}{+}
\]
where $\Rs{+}$ is the set of positive roots of an affine Lie algebra and $\Ku{\alpha}{+}$
is the normally ordered Casimir operator (cf.~\ref{ss:casimir-conn}). These extensions provide an 
affine analogue of the $W$--equivariant Casimir connections $\nabla=d-A_{\Ku{}{}}$ and 
$\nabla=d-A_{C}$ with
\begin{align}
A_{\Ku{}{}}=\frac{\nablah}{2}\sum_{\alpha\in\Rs{+}}\frac{d\alpha}{\alpha}\Ku{\alpha}{}
\aand
A_{C}&=\frac{\nablah}{2}\sum_{\alpha\in\Rs{+}}\frac{d\alpha}{\alpha}	C_{\alpha}
\end{align}
where $\Rs{+}$ is the set of positive roots of a finite--dimensional simple Lie algebra and $\Ku{\alpha}{}$
(resp. $C_{\alpha}$) is the truncated (resp. full) Casimir element of $\sl{2}^{\alpha}$. 
More precisely, we prove the following. 

\begin{theorem}
Let $\g$ be an affine Lie algebra with Cartan subalgebra $\h$. Then, there are two explict
closed $1$--forms $A_{\h}$ and $A_{S^2\h}$ valued, respectively, in $\h$ and $S^2\h$,
such that the following holds.
\vspace{0.25cm}\begin{enumerate}\itemsep0.25cm
\item\label{thm:correction1} The connection $\nabla=d-A_{\Ku{}{}}$, with $A_{\Ku{}{}}= A+A_{\h}$, 
is flat and $W$--equivariant.
Moreover, for any $i\in\bfI$, $\Res_{\alpha_i=0}A_{\Ku{}{}}=\frac{\nablah}{2}\cdot\Ku{i}{}$, 
where $\Ku{i}{}$ is the truncated Casimir element of $\sl{2}^{\alpha_i}$.
\item\label{thm:correction2} The connection $\nabla=d-A_{C}$, with $A_{C}= A_{}+A_{\h}+A_{S^2\h}$, is flat and $W$--equivariant.
Moreover, for any $i\in\bfI$, $\Res_{\alpha_i=0}A_C=\frac{\nablah}{2}\cdot C_i$, 
where $C_i$ is the full Casimir element of $\sl{2}^{\alpha_i}$.
\end{enumerate} 
\end{theorem}

The construction of the forms $A_{\h}$ and $A_{S^2\h}$ is given in \ref{ss:form1} and \ref{ss:form2}, respectively.
The proof of (1) and (2) is given in \ref{ss:correction1} and \ref{ss:correction2}, respectively.

\subsection{The form $A_{\h}$}\label{ss:form1}
For any $\delta\in\IC^{\times}$, set
\[
\Psi^{\pm}_{\delta}(x)=\sum_{n>0}\left(\frac{1}{\pm x+n\delta}-\frac{1}{n\delta}\right)=\Psi^{\mp}_{\delta}(-x)
\]
One verifies easily that $\Psi^{\pm}_{\delta}$ satisfies the following properties:
\begin{itemize}
\item[(i)] $\Psi^{\pm}_{\delta}(x)$ is holomorphic on $\IC\setminus\IZ_{\neq0}\delta$
\item[(ii)] $\displaystyle\Psi^+_{\delta}(x+\delta)=\Psi^+_{\delta}(x)-\frac{1}{x+\delta}$
\item[(iii)] $\displaystyle\Psi^-_{\delta}(x+\delta)=\Psi^-_{\delta}(x)-\frac{1}{x}$
\end{itemize}
Set $\Psi^{\pm}=\Psi^{\pm}_1$ and $\Psi=\Psi^++\Psi^-$.\\

Let $\g$ be an affine Kac--Moody of rank $\ell+1$ associated to the minimal realisation, 
with Cartan subalgebra $\h\subset\g$ and root system $\rootsys$. Let $\fdim{\g}$ be the 
corresponding finite dimensional Lie algebra with Cartan subalgebra $\fdim{\h}\subset\g$ and 
root system $\fdim{\rootsys}$, so that
\[
\h=\fdim{\h}\oplus\IC c\oplus\IC d
\]
where $\fdim{\h}\subset\fdim{\g}$, $c=\sum_{i=0}^{\ell} a_i^{\vee}\cor{i}$ is the canonical central
element and $d$ satisfies $\alpha_i(d)=\delta_{i,0}$. Let  $\iip{\cdot}{\cdot}$ be the normalized non--degenerate 
bilinear form on $\h$, and $\nu:\h\to\h^*$ the isomorphism induced by $\iip{\cdot}{\cdot}$ (cf.~\ref{ss:km-recap}). 
Let $\delta=\sum_{i=0}^\ell a_i\alpha_i$ be the minimal imaginary root. We set 
\begin{equation}
A_{\h}=\nablah\cdot\left(\sum_{\beta\in\fpr} A_{\beta}\left(\frac{\beta}{\delta}\right)
+\rho^{\vee}\frac{d\delta}{\delta}\right)
\end{equation}
where
\begin{equation}
A_{\beta}\left(\frac{\beta}{\delta}\right)=
\frac{1}{2}\left[
\left(
\frac{\delta}{\beta}+\Psi\left(\frac{\beta}{\delta}\right)
\right)\hinv{\beta}-\frac{\beta}{\delta}\left(
2+\Psi\left(\frac{\beta}{\delta}\right)
\right)c
\right]d\left(\frac{\beta}{\delta}\right)
\end{equation}
and $\rho^{\vee}\in\h$ is a fixed solution of 
\begin{equation}\label{eq:defB}
	\alpha_i(\rho^{\vee})=1\qquad\qquad i=0,1,\dots, \ell
\end{equation}
\subsection{Proof of (\ref{thm:correction1})}\label{ss:correction1}

We shall prove that the form $A+A_{\h}$ defines a flat and $W$--equivariant connection with residues
$\Res_{\alpha_i=0}A+A_{\h}=\nablah\cdot\left(f_ie_i+\frac{1}{2}\cor{i}\right)=\frac{\nablah}{2}\cdot \Ku{i}{}$. To this end, we explain below 
that the formula for $A_{\h}$ naturally arises by imposing the equivariance condition with respect to the
extended Weyl group in the case of affine rank one, \ie $\g=\wh{\sl{2}}$, and then by extending it to the 
higher rank case. In \ref{sss:correction-sl2}--\ref{sss:correction-sl2-end}, we prove the case of
affine rank one. The proof for the general case is carried out in \ref{sss:correction-affine}
--\ref{sss:correction-affine-end}.

\subsubsection{The case of affine rank one}\label{sss:correction-sl2}

Set $\g=\wh{\sl{2}}$ and assume that $A_{\h}$ has the
form
\begin{align*}
A_{\h}&=\nablah\cdot\left(\left(\Sfd{\theta}h+\Tfd{\theta}c\right)\Diffd{\theta}+B(\delta)d\delta\right)\\
&=\nablah\cdot\left(\frac{1}{\delta}\left(\Sfd{\theta}h+\Tfd{\theta}c\right)d{\theta}
-\frac{\theta}{\delta}\left(\Sfd{\theta}h+\Tfd{\theta}c\right)d{\delta}+B(\delta)d\delta\right)
\end{align*}
where $\theta=\alpha_1$ and $h=\cor{1}$.
In particular, $A_{\h}$ is closed.

Let $W^{\scriptscriptstyle{\operatorname{ext}}}$ the \emph{extended} Weyl group, \ie
$W^{\scriptscriptstyle{\operatorname{ext}}}= W\rtimes\mathsf{Aut}(\dgr)$, where
$\Aut(\dgr)$ denotes the group of diagram automorphisms of the Dynkin diagram of $\g$. 
Then, the form $A+A_{\h}$ is $W^{\scriptscriptstyle{\operatorname{ext}}}$--equivariant if and only if
\begin{equation}\label{eq:equi1}
	s_1^*A_{\h}=A_{\h}-\nablah\cdot\frac{h}{\theta}d\theta 
	\aand
	\gamma^*A_{\h}=A_{\h}
\end{equation}
where $s_1$ is the simple reflection on $\theta$ and $\gamma$ is induced by the symmetry
of the Dynkin diagram of $\wh{\sl{2}}$. In particular, we have
\begin{equation}
\begin{array}{lll}
s_1(\theta)=-\theta & s_1(\delta)=\delta & s_1(\Lambda)=\Lambda\\
\gamma(\theta)=-\theta+\delta & \gamma(\delta)=\delta & 
\gamma(\Lambda)=\frac{\theta}{2}-\frac{c}{4}+\Lambda\\
\end{array}
\end{equation}

\subsubsection{}
Set $z=\theta/\delta$. The condition \eqref{eq:equi1} is equivalent to the system of equations
\begin{align}
\label{eq:S1} S(-z)&\,=\,S(z)-\frac{1}{z}\\
\label{eq:T1} -T(-z)&\,=\,T(z)\\
\label{eq:S2} S(1-z)&\,=\,S(z)\\
\label{eq:T2} T(z)+T(1-z)&\,=\,-S(1-z)
\end{align}
and
\begin{align}
\label{eq:Sd1} \frac{z}{\delta}S(-z)+(s_1^*B(\delta))_{(h)}&\,=\, -\frac{z}{\delta}S(z)+B(\delta)_{(h)}\\
\label{eq:Td1} \frac{z}{\delta}T(-z)+(s_1^*B(\delta))_{(c)}&\,=\,-\frac{z}{\delta}T(z)+B(\delta)_{(c)}\\
\label{eq:Sd2} -\frac{z}{\delta}S(1-z)+(\gamma^*B(\delta))_{(h)}&\,=\,-\frac{z}{\delta}S(z)+B(\delta)_{(h)}\\
\label{eq:Td2} \frac{z}{\delta}\left[S(1-z)+T(1-z)\right]+(\gamma^*B(\delta))_{(c)}&\,=\,-\frac{z}{\delta}T(z)+
B(\delta)_{(c)}
\end{align}
where the subscripts $X_{(h)}$, $X_{(c)}$ denote the components along $h$ and $c$, respectively.\\

\subsubsection{}
If $S(z),T(z)$ are functions satisfying \eqref{eq:S1}, \eqref{eq:T1}, \eqref{eq:S2}, \eqref{eq:T2},
then $B(\delta)$ is bound to satisfy
\begin{equation*}
s_1^*B(\delta)\,=\,B(\delta)-\frac{h}{\delta}
\aand
\gamma^*B(\delta)\,=\,B(\delta)
\end{equation*}
The general solution is easily computed to be
\begin{equation*}\label{eq:genB}
B(\delta)=\frac{1}{\delta}\left(\frac{h}{2}+2d+f(\delta)c\right)
\end{equation*}
where $f(\delta)$ is any function in $\delta$. In particular, the condition \eqref{eq:defB} is satisfied.
Note however that $B(\delta)$ is not supported in $\h'$ (cf. Theorem~\ref{ss:holo-Cox-correction}).

\subsubsection{}
Let $S(z)$ be a given function satisfying \eqref{eq:S1} and \eqref{eq:S2}. 
We aim to find two polynomials $p(z)$ 
and $q(z)$ such that $T'(z)= p(z)S(z)+q(z)$ satisfies \eqref{eq:T1}, \eqref{eq:T2}.
In terms of $p$ and $q$, the latter conditions are equivalent to the system
\begin{alignat*}{2}
p(z)+p(-z)&=\,\phantom{+}0\qquad\qquad\phantom{...}q(z)+q(-z)&&=\,\frac{1}{z}p(-z)\\
p(z)+p(1-z)&=\,-1\qquad\qquad q(z)+q(1-z)&&=\,0
\end{alignat*}
A solution is given by $p(z)=-z$ and $q(z)=\frac{1}{2}-z$.
Therefore, given $S(z)$,  the function $T(z)$ has the form
\begin{equation*}\label{eq:genT}
T(z)=-z\left(S(z)+1\right)+\frac{1}{2}+E(z)
\end{equation*}
where $E(z)$ is any function satisfying $E(-z)=-E(z)$ and $E(z)=-E(1-z)$.

\subsubsection{}
Finally, we need to solve the equations \eqref{eq:S1} and \eqref{eq:S2}, 
which are equivalent to the system $S(-z)=S(z)-\frac{1}{z}$ and 
$S(z+1)=S(z)-\frac{1}{z}$. A particular solution is given by the function
\begin{equation*}
S(z)=\frac{1}{2}\left(\frac{1}{z}+\Psi(z)\right)
\end{equation*}
Therefore, the general solution is given by the formula
\begin{equation*}\label{eq:gene}
S(z)=\frac{1}{2}\left(\frac{1}{z}+\Psi(z)\right)+e(z)
\end{equation*}
where $e(z)$ is any function satisfying $e(-z)=e(z)$ and $e(z+1)=e(z)$.

\subsubsection{}\label{sss:correction-sl2-end}

Setting $e=E=f=0$, we get, for $\g=\wh{\sl{2}}$,
\begin{equation*}\label{eq:sl2sol}
A_{\h}=
\frac{\nablah}{2}\left[
\left(
\frac{\delta}{\theta}+\Psi\left(\frac{\theta}{\delta}\right)
\right)h-\frac{\theta}{\delta}\left(
2+\Psi\left(\frac{\theta}{\delta}\right)
\right)c
\right]d\left(\frac{\theta}{\delta}\right)+\nablah\left(\frac{h}{2}+2d\right)\frac{d\delta}{\delta}
\end{equation*}
and the resulting connection $\nabla=d-(A+A_{\h})$ is flat and $W$--equivariant. A simple
 computation shows that 
 \[
 \Res_{\theta=0}A+A_{\h}=\frac{\nablah}{2}\cdot\Ku{\theta}{}\cdot d\theta
 \aand
 \Res_{\theta=\delta}A+A_{\h}=\frac{\nablah}{2}\cdot\Ku{\delta-\theta}{}\cdot d(\delta-\theta)
 \]

\subsubsection{The general case}\label{sss:correction-affine}

Let now $\g$ be an affine Kac--Moody algebra and set 
\begin{equation*}
A_{\h}=\nablah\cdot\left(\sum_{\beta\in\fpr} A_{\beta}\left(\frac{\beta}{\delta}\right)
+B\frac{d\delta}{\delta}\right)
\end{equation*}
where
\begin{equation*}
A_{\beta}\left(\frac{\beta}{\delta}\right)=
\frac{1}{2}\left[
\left(
\frac{\delta}{\beta}+\Psi\left(\frac{\beta}{\delta}\right)
\right)\hinv{\beta}-\frac{\beta}{\delta}\left(
2+\Psi\left(\frac{\beta}{\delta}\right)
\right)c
\right]d\left(\frac{\beta}{\delta}\right)
\end{equation*}
with $\hinv{\beta}=\nu^{-1}(\beta)$ and $B\in\h$. 
We shall prove that there exists $B\in\h$ such that $A_{\Ku{}{}}=A+A_{\h}$
satisfies (\ref{thm:correction1}). Note that the form $A_{\beta}(\beta/\delta)$ satisfies
\begin{align}
\label{eq:formA1} A_{-\beta}\left(\frac{-\beta}{\delta}\right)&=A_{\beta}\left(\frac{\beta}{\delta}\right)
-\frac{\hinv{\beta}}{\beta}d\beta +\frac{\hinv{\beta}}{\delta}d\delta\\ 
\label{eq:formA2} A_{-\beta+\delta}\left(\frac{-\beta+\delta}{\delta}\right)&=A_{\beta}\left(\frac{\beta}{\delta}\right)
\end{align}
as proved in the case $\g=\wh{\sl{2}}$.

\subsubsection{}
For every $i=1,\dots,\ell$, the simple reflection $s_i$ permutes the elements in 
$\fpr\setminus\{\alpha_i\}$, and
\begin{align*}
s_i^*\left(\sum_{\beta\in\fpr} A_{\beta}\left(\frac{\beta}{\delta}\right)\right)&=
\sum_{\substack{\beta\in\fpr\\ \beta\neq\alpha_i}} A_{\beta}\left(\frac{\beta}{\delta}\right)
+A_{-\alpha_i}\left(\frac{-\alpha_i}{\delta}\right)\\
&=\sum_{\beta\in\fpr} A_{\beta}\left(\frac{\beta}{\delta}\right)-\frac{\cor{i}}{\alpha_i}d\alpha_i+\frac{\cor{i}}{\delta}d\delta
\end{align*}
where the second equality follows from \eqref{eq:formA1}.
Therefore, the form $A+A_{\h}$ is $\mathring{W}$--equivariant if and only if $s_i(B)=B-\cor{i}$ and
$\alpha_i(B)=1$.

\subsubsection{}

Let $\beta\in\fpr$.
It follows from \eqref{eq:formA2} that
\[
A_{-(\theta-\beta)+\delta}\left(\frac{-(\theta-\beta)+\delta}{\delta}\right)=A_{\theta-\beta}\left(\frac{\theta-\beta}{\delta}\right)
\]
Therefore
\begin{equation*}
s_0^*\left(\sum_{\beta\in\fpr} A_{\beta}\left(\frac{\beta}{\delta}\right)\right)=
\sum_{\beta\in\fpr\setminus\{\theta\}} A_{\beta}\left(\frac{\beta}{\delta}\right)+A_{-\theta+2\delta}\left(\frac{-\theta+2\delta}{\delta}\right)
\end{equation*}
By \eqref{eq:formA1} and \eqref{eq:formA2}, 
\begin{align*}
A_{-\theta+2\delta}\left(\frac{-\theta+2\delta}{\delta}\right)=&A_{\theta-\delta}\left(\frac{\theta-\delta}{\delta}\right)\\
=&A_{\delta-\theta}\left(\frac{\delta-\theta}{\delta}\right)-\frac{\cor{0}}{\alpha_0}d\alpha_0+\frac{\cor{0}}{\delta}d\delta\\
=&A_{\theta}\left(\frac{\theta}{\delta}\right)-\frac{\cor{0}}{\alpha_0}d\alpha_0+\frac{\cor{0}}{\delta}d\delta
\end{align*}
Therefore, $s_0^*(A+A_{\h})=A+A_{\h}$ if and only if $s_0(B)=B-\cor{0}$ and $\alpha_0(B)=1$.

\subsubsection{}\label{sss:correction-affine-end}

Finally, we conclude that, for any $B\in\h$ satisfying $\alpha_i(B)=1$, 
$i=0,1,\dots,\ell$, there is a flat and $W$--equivariant connection
$A_{\Ku{}{}}=A+A_{\h}$, where
\begin{equation}
A_{\h}=\nablah\cdot\left(\sum_{\beta\in\fpr} A_{\beta}\left(\frac{\beta}{\delta}\right)
+B\frac{d\delta}{\delta}\right)
\end{equation}
Moreover, its residues
$\Res_{\alpha_i=0}A_{\kappa}=\frac{1}{2}\Ku{i}{}d\alpha_i$
are given by the truncated Casimir elements. This completes
the proof of (\ref{thm:correction1}).

\subsection{The form $A_{S^2\h}$}\label{ss:form2}

We shall extend the $W$--equivariant connection $\nabla=d-A_{\Ku{}{}}$ with a closed, 
$W$--equivariant form $A_{S^2\h}$ with values in $S^2\h$, so that 
\begin{equation*}
\Res_{\alpha_i=0}A_C=\frac{1}{2}C_id\alpha_i
\end{equation*}
where $A_C=A+A_{\h}+A_{S^2\h}$.
This provides an affine analogue of the Casimir connection of a finite--dimensional
simple Lie algebra with coefficients $C_{\alpha}$. To this end, we set
\begin{equation}
A_{S^2\h}=\nablah\sum_{\beta\in\fpr}\frac{\pi}{2}\cot\left(\pi\frac{\beta}{\delta}\right)
\left(\hinv{\beta}-\frac{\beta}{\delta}c\right)^2
d\left(\frac{\beta}{\delta}\right)
\end{equation}

\subsection{Proof of (\ref{thm:correction2})}\label{ss:correction2}

As before, we first consider the case $\g=\wh{\sl{2}}$. We have
\begin{equation*}
A_{S^2\h}=\nablah\frac{\pi}{2}\cot\left(\pi\frac{\theta}{\delta}\right)\left(h-\frac{\theta}{\delta}c\right)^2
d\left(\frac{\theta}{\delta}\right)
\end{equation*}
Then, $A_{S^2\h}$ is closed with residues
\begin{eqnarray}
\Res_{\theta=0} A_{S^2\h} &=&\frac{\nablah}{2}\cdot h^2\cdot d\theta\\
\Res_{\theta=\delta} A_{S^2\h} &=& \frac{\nablah}{2}\cdot(-h+c)^2\cdot d(\delta-\theta)
\end{eqnarray}
Moreover, $A_{S^2\h}$ is $W$--equivariant since we have
\begin{equation*}
s_1^*A_{S^2\h}=\nablah\frac{\pi}{2}\cot\left(\pi\frac{-\theta}{\delta}\right)\left(-h-\frac{-\theta}{\delta}c\right)^2
d\left(\frac{-\theta}{\delta}\right)=A_{S^2\h}
\end{equation*}
and
\begin{equation*}
s_0^*A_{S^2\h}=\nablah\frac{\pi}{2}\cot\left(\pi\frac{-\theta+2\delta}{\delta}\right)\left(-h+2c-\frac{-\theta+2\delta}{\delta}c\right)^2
d\left(\frac{-\theta+2\delta}{\delta}\right)=A_{S^2\h}
\end{equation*}

Let now $\g$ be an affine Kac--Moody algebra and
\begin{equation}
A_{S^2\h}=\nablah\sum_{\beta\in\fpr}\frac{\pi}{2}\cot\left(\pi\frac{\beta}{\delta}\right)\left(\hinv{\beta}-\frac{\beta}{\delta}c\right)^2
d\left(\frac{\beta}{\delta}\right)
\end{equation}
Clearly, $A_{S^2\h}$ is closed with the required residues. Moreover, for any element of the Weyl group $w\in W$, we have 
\footnote{
	In the case of $\g=\wh{\sl{2}}$, one has
	\begin{align*}
		\pi&\cot\left(\pi\frac{\theta}{\delta}\right)d\left(\frac{\theta}{\delta}\right)=
		\delta\left[\frac{1}{\theta}+\sum_{n>0}\left(\frac{1}{\theta+n\delta}-\frac{1}{-\theta+n\delta}\right)\right]
		d\left(\frac{\theta}{\delta}\right)=\\
		&=\frac{\delta}{2}\left(\frac{1}{\theta}-\frac{1}{s_1(\theta)}\right)d\left(\frac{\theta}{\delta}\right)+
		\frac{\delta}{2}\sum_{n>0}\left(\frac{1}{\theta+n\delta}-\frac{1}{s_1(\theta+n\delta)}\right)
		d\left(\frac{\theta+n\delta}{\delta}\right)\\
		&\hspace{1.545in}+\frac{\delta}{2}\sum_{n>0}\left(\frac{1}{s_1(-\theta+n\delta)}-\frac{1}{-\theta+n\delta}\right)d\left(-\frac{-\theta+n\delta}{\delta}\right)=\\
		&=\frac{\delta}{2}\sum_{\alpha\in\prr}\left(\frac{1}{\alpha}-\frac{1}{s_1(\alpha)}\right)
		d\left(\frac{\alpha}{\delta}\right)
	\end{align*}
	Similarly for higher rank $\g$ and $w\in W$.
}
\begin{align*}
\sum_{\beta\in\fpr}\frac{\pi}{2}\cot\left(\pi\frac{\beta}{\delta}\right)\left(\hinv{\beta}-\frac{\beta}{\delta}c\right)^2d\left(\frac{\beta}{\delta}\right)=
\frac{\delta}{4}\sum_{\alpha\in\prr}\left(\frac{1}{\alpha}-\frac{1}{w(\alpha)}\right)
\left(\hinv{\alpha}-\frac{\alpha}{\delta}c\right)^2d\left(\frac{\alpha}{\delta}\right)
\end{align*}

\noindent\remark\;
The expression of the form $A_{S^2\h}$ for $\g=\wh{\sl{2}}$ has been computed as in \ref{sss:correction-sl2}--\ref{sss:correction-sl2-end}.
We set 
\begin{equation}
A_{S^2\h}=\nablah(A_{\theta}(\theta,\delta)d\theta+A_{\delta}(\theta,\delta)d\delta)
\end{equation}
with $A_{\theta}(\theta,\delta)=S(\theta,\delta)h^2+T(\theta,\delta)hc+U(\theta,\delta)c^2$ and 
similarly for $A_{\delta}$. By imposing the $W$--equivariance (for a fixed value $\delta\in\IC^*$) 
we obtain a sistem of difference equation in $\theta$ for the functions $S,T,U$ which is easily solved with 
functions of the form $p(z)\cot(z)$, where $p$ is a polynomial. 
More specifically, $S,U$ are odd functions in $\theta$ and $T$ is an even function in $\theta$ 
such that
\begin{eqnarray}
S(\theta+\delta)&=&S(\theta)\\
T(\theta+\delta)&=&T(\theta)-2S(\theta)\\
U(\theta+\delta)&=&U(\theta)+S(\theta)-T(\theta)
\end{eqnarray}
The system above encodes the invariance with respect to the translation $\theta\mapsto\theta-\delta$.
Finally, the condition $dA=0$ gives a formula for $A_{\delta}$.
Namely, we obtain to a general solution of the form
\begin{equation}
A_{S^2\h}=\nablah\left(\frac{\pi}{2}\cot\left(\pi\frac{\theta}{\delta}\right)\left(h-\frac{\theta}{\delta}c\right)^2
d\left(\frac{\theta}{\delta}\right) + B(\delta)d\delta\right)
\end{equation}
where $B(\delta)$ is any $W$--equivariant function (which is therefore chosen to be equal to zero).


\end{document}